\documentclass[a4paper,11pt]{amsart}
\usepackage{amsmath}

\usepackage{amssymb, amsthm, xcolor,enumerate, mathtools}

\usepackage{verbatim}
\usepackage{enumitem}
\usepackage[hidelinks]{hyperref}

\usepackage{graphicx}
\usepackage{tikz, tikz-cd}
\usepackage{color}
\usepackage{xr}

\usepackage{booktabs}

\definecolor{cutcolour}{RGB}{0,100,0}

\newtheorem{theorem}{Theorem}
\newtheorem{lemma}[theorem]{Lemma}
\newtheorem{proposition}[theorem]{Proposition}
\newtheorem{corollary}[theorem]{Corollary}
\theoremstyle{definition}
\newtheorem{definition}[theorem]{Definition}
\newtheorem{notation}[theorem]{Notation}

\newtheorem{question}[theorem]{Question}
\newtheorem*{remark*}{Remark}
\newtheorem{remark}[theorem]{Remark}
\newtheorem{example}[theorem]{Example}

\renewcommand{\labelenumi}{\theenumi}
\renewcommand{\theenumi}{\rm{(\roman{enumi})}}

\numberwithin{equation}{section}
\numberwithin{theorem}{section}

\usepackage{thmtools}
\usepackage{thm-restate}

\makeatletter
\def\namedlabel#1#2{\begingroup
	#2
    \def\@currentlabel{#2}
    \label{#1}\endgroup
}
\makeatother

\newcommand{\Z}{\mathcal Z}
\newcommand{\R}{\mathcal R}
\newcommand{\F}{\mathcal F}

\newcommand{\C}{\mathbb C}
\newcommand{\N}{\mathbb N}

\newcommand{\tr}{\mathrm{tr}}
\newcommand{\id}{\mathrm{id}}
\newcommand{\M}{\mathcal{M}}

\newcommand{\barotimes}{\bar\otimes}

\newcommand{\completion}[2]{\overline{#1}^{#2}}

\let\oldtocsubsection=\tocsubsection
\renewcommand{\tocsubsection}[2]{\hspace{.75cm}\oldtocsubsection{#1}{#2}}

\DeclareRobustCommand{\SkipTocEntry}[5]{}

\usepackage{perpage}
\newcounter{mparcnt}
\MakePerPage{mparcnt}

\title{Tracially Complete $C^*$-algebras}

\author[J.\ Carri\'on et al.]{Jos\'e R.\ Carri\'on}
\address{Jos\'e R.\ Carri\'on, Department of Mathematics, Texas Christian
  University, Fort Worth, Texas 76129, USA.}
\email{j.carrion@tcu.edu}

\author[]{Jorge Castillejos}
\address{Jorge Castillejos, Instituto de Matem\'aticas, Unidad Cuernavaca, Universidad Nacional Autonoma de M\'exico, Cuernavaca, M\'exico}
\email{jorge.castillejos@im.unam.mx}

\author[]{Samuel Evington}
\address{Mathematical Institute, University of M\"unster, Einsteinstrasse 62, 48149 M\"unster, Germany}
\email{evington@uni-muenster.de}

\author[]{James Gabe}
\address{James Gabe, Department of Mathematics and Computer Science
University of Southern Denmark, Denmark}
  \email{gabe@imada.sdu.dk}

\author[]{Christopher Schafhauser}
\address{Christopher Schafhauser, Department of  Mathematics, University of Ne-\linebreak braska-Lincoln, Lincoln, Nebraska 68508, USA}
\email{cschafhauser2@unl.edu}

\author[]{Aaron Tikuisis}
\address{Aaron Tikuisis, Department of Mathematics and Statistics, University of
  Ottawa, 585 King Edward, Ottawa, ON, K1N 6N5, Canada}
\email{aaron.tikuisis@uottawa.ca}

\author[]{Stuart White}
\address{Stuart White, Mathematical Institute, University of Oxford,
  Oxford, OX2 6GG, United Kingdom}
\email{stuart.white@maths.ox.ac.uk}

\thanks{This work was supported by:
NSF grant DMS-2451675 (Carri\'on);
long term structural funding – a Methusalem grant of the Flemish Government (Castillejos); UNAM-PAPIIT IN108126 (Castillejos);
Deutsche Forschungsgemeinschaft (DFG, German Research Foundation) – Project-ID 427320536 – SFB 1442 (Evington); 
Germany's Excellence Strategy EXC 2044 390685587  Mathematics M{\"u}nster: Dynamics–Geometry–Structure (Evington);  
ERC Advanced Grant 834267 - AMAREC (Evington);
Engineering and Physical Sciences Research Council [Grant Refs: EP/R025061/1, EP/R025061/2, and  EP/X026647/1] (Evington, White); 
Australian Research Council grant DP180100595 (Gabe); 
NSF grant DMS-2000129 (Schafhauser); 
NSERC Discovery Grant (Tikuisis), All Souls College visiting professorship (Tikuisis).
For the purpose of Open Access, the authors have applied a CC BY public copyright licence to any Author Accepted Manuscript (AAM) version arising from this submission.}

\begin{document}

\begin{abstract}
We introduce a new class of operator algebras -- tracially complete $C^*$-algebras -- as a vehicle for transferring ideas and results between $C^*$-algebras and their tracial von Neumann algebra completions. We obtain structure and classification results for amenable tracially complete $C^*$-algebras satisfying an appropriate version of Murray and von Neumann's property $\Gamma$ for II$_1$ factors. In a precise sense, these results fit between Connes' celebrated theorems for injective II$_1$ factors and the unital classification theorem for separable simple nuclear $C^*$-algebras. The theory also underpins arguments for the known parts of the Toms--Winter conjecture.
\end{abstract}

\maketitle

\addtocontents{toc}{\SkipTocEntry}
\section*{Overview of results}

\renewcommand{\thetheorem}{\Alph{theorem}} 

There are two major classes of self-adjoint operator algebras: $C^*$-algebras and von Neumann algebras. Despite significant differences in their theory and their stages of development, techniques for transferring structure between these two classes are important on both sides.
Examples of this are found in the deep connections between nuclearity of a $C^*$-algebra $A$ and injectivity of its enveloping von Neumann algebra $A^{**}$ (\cite{Co76,Choi-Effros76,Choi-Effros77,Effros-Lance77}), which are obtained through Connes' celebrated work on injectivity and hyperfiniteness of von Neumann algebras. Other prominent examples of structural transfer between operator algebras include \cite{Brown11,Haagerup83,MS12,MS14,Ozawa04,OP10}.

The focus of this paper is on operator algebras with a trace, by which we will always mean a tracial state. One of Murray and von Neumann's very early results is that a finite factor has a unique tracial state (\cite{MvN36, MvN37}).
In contrast, a simple unital stably finite $C^*$-algebra can have many tracial states. Indeed, by results of Blackadar (\cite{Bl80}) and Goodearl (\cite{Go77}), any Choquet simplex may arise as the trace simplex $T(A)$ of a simple unital approximately finite dimensional (AF) $C^*$-algebra $A$.  While a finite von Neumann algebra $\mathcal M$ always has sufficiently many traces, in that for any non-zero positive element $x\in\M$, there is a trace $\tau$ with $\tau(x)\neq 0$, the existence of traces is much more subtle in the $C^*$-setting. A deep theorem of Haagerup (\cite{Ha14}), together with \cite{Blackadar-Handelman, Blackadar-Rordam92}, is needed to obtain the corresponding result for a unital stably finite exact $C^*$-algebras, and without exactness, it is open whether traces necessarily exist.

Associated to every trace $\tau$ on a operator algebra $A$ are a 2-seminorm $\|x\|_{2,\tau}=\tau(x^*x)^{1/2}$, a GNS representation $\pi_\tau\colon A \rightarrow \mathcal B(\mathcal H_\tau)$, and a finite von Neumann algebra $\pi_\tau(A)''$.  
These objects are closely related. Indeed, the unit ball of $\pi_\tau(A)''$ is the $\|\cdot\|_{2,\tau}$-completion of the unit ball of $A$. In \cite{Oz13}, motivated by developments on the Toms--Winter conjecture (\cite{KR14,TWW15,Sa12}), Ozawa considers the $C^*$-algebra $\completion{A}{\rm u}$ obtained by completing the unit ball of $A$ with respect to the uniform 2-seminorm
\begin{equation*}
    \|a\|_{2,T(A)}\coloneqq\sup_{\tau \in T(A)} \|a\|_{2,\tau}.
\end{equation*}
When $A$ has a unique trace $\tau$, the uniform tracial completion $\completion{A}{\rm u}$ coincides with the von Neumann algebra $\pi_\tau(A)''$. 

When $T(A)$ has infinitely many extreme points, the uniform tracial completion $\completion{A}{\rm u}$ is no longer a von Neumann algebra. Ozawa's theory of $W^*$-bundles (see Section~\ref{sec:WStarBundles} below) axiomatises the structure of $\completion{A}{\rm u}$ in the special case that the extreme points of $T(A)$ are weak$^*$-closed. The techniques developed in this paper allow us to generalise Ozawa's theory to arbitrary trace simplices.  In Theorem \ref{Main-A}, the condition of tensorially absorbing the Jiang--Su algebra (from \cite{JS99}) is one of the two major hypotheses in the unital classification theorem for separable simple nuclear $C^*$-algebras (\cite{GLN20a,GLN20b,EGLN15,TWW17,CGSTW}; see \cite[Section 1.1]{CGSTW} for a discussion of this property and examples).  In particular the theorem applies to all $C^*$-algebras covered by this classification.

\begin{theorem}\label{Main-A}
Let $A$ and $B$ be unital separable nuclear $C^*$-algebras that absorb the Jiang--Su algebra tensorially.  Then $\completion{A}{\rm u}\cong \completion{B}{\rm u}$ if and only if $A$ and $B$ have affinely homeomorphic spaces of tracial states.\footnote{The isomorphism is to be interpreted as $(\completion{A}{\rm u}, T(A)) \cong (\completion{B}{\rm u},T(B))$ in the category of tracially complete $C^*$-algebras as defined in Section~\ref{sec:tracially-complete-defs}. In this case, this amounts to a uniform $2$-norm continuous $C^*$-isomorphism with a uniform $2$-norm continuous inverse.\label{FootnoteThmA}}
\end{theorem}

Note that, by contrast, if $A$ is a separable nuclear $C^*$-algebra with no finite dimensional quotients, then the finite part of the bidual 
$A^{**}_{\mathrm{fin}}$ \emph{only} remembers the number of extreme traces in $T(A)$ by results that go back to Connes and Elliott (\cite{Co76, El76b}).  Indeed, $A^{**}_{\mathrm{fin}}$ is type $\mathrm{II}_1$ as $A$ has no finite dimensional representations, and the minimal direct summands of $A^{**}_{\mathrm{fin}}$ are in bijection with the extremal traces of $A$, so the assertion follows from \cite[Theorem~4.3]{El76b}.
  
Over the last decade, major advances in the Elliott classification programme for simple nuclear $C^*$-algebras with traces\footnote{Purely infinite $C^*$-algebras and type III factors are out of the scope of this paper.}  have culminated in two major themes.
\begin{description}[leftmargin = \parindent, labelindent = 0 cm]
\item[Classification] The finite case of the \emph{unital classification theorem} (\cite[Corollary D]{CETWW}, using \cite{EGLN15,GLN20a,GLN20b, TWW17}; an abstract proof is given in \cite{CGSTW}), which correspond to the uniqueness of the injective II$_1$ factor (combining Connes' equivalence of injectivity and hyperfiniteness from \cite{Co76} with Murray and von Neumann's uniqueness of the hyperfinite II$_1$ factor). 
\item[Structure] Structural theorems (\cite{Wi10,Wi12,Ti14,MS12,MS14,SWW15,BBSTWW, CETWW, CE}), which combine to prove most of the \emph{Toms--Winter conjecture} (\cite[Conjecture~9.3]{WZ10}, cf.\ \cite{TW09,ET08}) for $C^*$-algebras.  These results mirror aspects of Connes' equivalence of injectivity and hyperfiniteness for von Neumann II$_1$ factors (\cite{Co76}).
\end{description}
 With the benefit of hindsight, the advances above have been heavily driven by a subtle interplay between nuclear $C^*$-algebras $A$ and their uniform tracial completions $\completion{A}{\rm u}$. This goes back to Matui and Sato's breakthrough work on the Toms--Winter regularity conjecture (\cite{MS12,MS14}).
 In a nutshell, these `von Neumann techniques' allow for direct access to Connes' work in $C^*$-arguments.  This is true both conceptually and technically.

In this paper, we provide an abstract framework for studying uniform tracial completions:  \emph{tracially complete $C^*$-algebras}, consisting of a unital $C^*$-algebra together with a distinguished set of tracial states such that the operator norm unit ball is complete in the resulting uniform $2$-norm. The motivating example is the pair $\big(\completion{A}{\rm u}, T(A)\big)$, consisting of the uniform tracial completion of a $C^*$-algebra $A$ together with the set of traces arising from traces on $A$. We regard tracially complete $C^*$-algebras as a bridge between tracial von Neumann algebras and tracial $C^*$-algebras, allowing one to solve problems about tracial $C^*$-algebras in three stages.
\begin{enumerate}
\item Apply von Neumann algebra theory in each tracial GNS representation.\label{item:intro:overallplan1}
\item Obtain results for uniform tracial completions  via local-to-global arguments.\label{item:intro:overallplan2}
\item Pull results back from the tracial completion to the $C^*$-level.\label{item:intro:overallplan3}
\end{enumerate}

The main objective of the paper is to obtain structure and classification theorems for tracially complete $C^*$-algebras -- i.e.\ results that fit into step \ref{item:intro:overallplan2} of the programme above. These theorems fit between Connes' Theorem (together with Murray and von Neumann's uniqueness of the hyperfinite II$_1$ factor) and the structure and classification theorems in the setting of simple nuclear $C^*$-algebras.  Our work is organised into the following topics.

\begin{description}[leftmargin=\parindent, labelindent = 0cm]
\item[Amenability] We will show that the two natural definitions of amenability, one in terms of amenability of all tracial von Neumann algebra completions and the other defined by a uniform 2-norm version of the completely positive approximation property, agree.
\item[Factoriality] We identify a subclass of algebras -- the (type II$_1$) \emph{factorial tracially complete $C^*$algebras} -- in which each GNS representation coming from an extreme point of the distinguished traces is a (type II$_1$) factor. Uniform tracial completions of $C^*$-algebras are always factorial, and Ozawa's $W^*$-bundles are examples of type II$_1$ factorial tracially complete $C^*$-algebras precisely when their fibres are II$_1$ factors.
\item[Regularity] We introduce the McDuff property and property $\Gamma$ for tracially complete $C^*$-algebras. This is simultaneously motivated by the following implications due to Connes for a II$_1$ factor (\cite{Co76}):
\begin{equation*}
\begin{tikzcd}[column sep=3ex]
\textrm{injective}\ar[r,Rightarrow, 
line width=.3mm, double distance=.7mm] & \begin{array}{c}\text{injective with } \\ \text{property }\Gamma\end{array} \ar[r,Rightarrow, line width=.3mm, double distance=.7mm] & \begin{array}{c}\text{injective} \\ \text{and McDuff}\end{array}\ar[r,Rightarrow, line width=.3mm, double distance=.7mm] & \text{hyperfinite,}
\end{tikzcd}
\end{equation*}
as well as by work on regularity properties (e.g.\ Jiang--Su stability and uniform property $\Gamma$) for $C^*$-algebras (\cite{CETWW,CETW,Ro04,MS12}).
\item[Hyperfiniteness] We consider inductive limits of finite dimensional tracially complete $C^*$-algebras, in the spirit of Murray and von Neumann's hyperfinite von Neumann algebras and Bratteli's AF $C^*$-algebras. 

\item[Concrete models] For each metrisable Choquet simplex $X$, we construct a separable hyperfinite type II$_1$ factorial tracially complete $C^*$-algebra $(\mathcal R_X,X)$ in the same spirit as Blackadar's and Goodearl's constructions of a simple unital AF $C^*$-algebra with trace simplex affinely homeomorphic to $X$.  Unlike the case of AF $C^*$-algebras, the resulting tracially complete $C^*$-algebra is independent of all choices made in the construction. When $X$ is an $n$-dimensional simplex, $(\mathcal R_X,X)$ is isomorphic to the direct sum of $n+1$ copies of the hyperfinite II$_1$ factor with its entire trace simplex.  When $X$ is a
metrisable Bauer simplex (i.e.\ the extreme boundary $\partial_e X$ is compact), $(\mathcal R_X, X)$ can be identified with the trivial $W^*$-bundle over $\partial_e X$ with fibre the hyperfinite II$_1$ factor $\mathcal R$.  

\end{description}
With this setup, our structure and classification theorems for amenable tracially complete $C^*$-algebras are as follows.

\begin{theorem}[Structure theorem]\label{InformalStructureThm}
For an amenable type {\rm II}$_1$ factorial tracially complete $C^*$-algebra, property $\Gamma$, the McDuff property, and hyperfiniteness are equivalent.
\end{theorem}
\begin{theorem}[Classification theorem]\label{InformalClassification}
Any amenable factorial tracially complete $C^*$-algebra which is separable (in the uniform 2-norm) and has property $\Gamma$ is isomorphic to the hyperfinite model $(\mathcal R_X,X)$ with the corresponding Choquet simplex $X$ of traces.
\end{theorem}

The equivalence of property $\Gamma$ and the McDuff property in Theorem \ref{InformalStructureThm} was established in \cite{CETW} for the uniform tracial completions of separable nuclear $C^*$-algebras with no finite dimensional representations.  

As with the classification programme for $C^*$-algebras and Murray and von Neumann's uniqueness of the hyperfinite II$_1$ factor before that, our classification theorem is powered by classification results for $^*$-homomorphisms satisfying a suitable notion of amenability together with an Elliott-style intertwining argument.  In the setting of tracially complete $C^*$-algebras amenability is given in terms of the completely positive approximation property in the uniform 2-norm or, equivalently, in terms of  uniformly amenable traces -- we call such maps \emph{tracially nuclear}.\footnote{We state Theorem \ref{IntroThmClassMap} for tracially complete $C^*$-algebras satisfying the regularity hypothesis of property $\Gamma$. In the main body, we will work with a more general hypothesis -- complemented partitions of unity (discussed in Section \ref{sec:intro:ltg2} below) -- so that that the classification of tracially nuclear $^*$-homomorphisms generalises the classification of weakly nuclear $^*$-homomorphisms from $C^*$-algebras into finite von Neumann algebras.}

\begin{theorem}[Classification of tracially nuclear $^*$-homomorphisms]\label{IntroThmClassMap}

Con\-si\-der tracially complete $C^*$-algebras $(\mathcal M, X)$ and $(\mathcal N, Y)$ with $(\mathcal M, X)$ being $\|\cdot\|_{2, X}$-separable and $(\mathcal N, Y)$ being factorial with property $\Gamma$.  Then a morphism $\phi \colon (\mathcal M, X) \rightarrow (\mathcal N, Y)$ is tracially nuclear if and only if the induced map $\phi^* \colon Y \to X$ takes values in the uniformly amenable traces on $\mathcal M$.  Further, any continuous affine $\gamma\colon Y\to X$ taking values in the uniformly amenable traces on $\mathcal M$ is induced by a tracially nuclear map $\phi\colon (\mathcal M,X)\to(\mathcal N,Y)$, which is unique up to approximate unitary equivalence in the uniform 2-norm on $(\mathcal N, Y)$.
\end{theorem}

An important special case of the previous theorem is given by taking $(\mathcal M, X)\coloneqq \big(\mathbb C^2, T(\mathbb C^2)\big)$, giving a classification of projections. In this case we can do better: unitary equivalence and Murray--von Neumann subequivalence of projections in factorial tracially complete $C^*$-algebras with property $\Gamma$ are determined by the designated set of traces (see Theorem \ref{thm:comparison}).\footnote{This also holds with complemented partitions in place of property $\Gamma$, but we emphasise that we do not know Theorem \ref{IntroThmCtsProj} for general II$_1$ factors $\M$ unless $X$ is totally disconnected. It is open whether Theorem \ref{IntroThmCtsProj} holds when $X= [0,1]$ and $\M$ is a free group factor.}  We highlight the following special case of this in the language of von Neumann algebras, showing that Murray and von Neumann's foundational classification of projections can be performed continuously in II$_1$ factors with property $\Gamma$.  It is obtained by applying our classification of projections to the trivial $W^*$-bundle over $X$ with fibre $\M$.

\begin{theorem}\label{IntroThmCtsProj}
Let $K$ be a compact Hausdorff space and let $\M$ be a {\rm II}$_1$ factor with property $\Gamma$.  Suppose that $p,q\colon K\to \M$ are projection-valued functions which are continuous with respect to $\|\cdot\|_2$.
\begin{enumerate}
    \item There is a $\|\cdot\|_2$-continuous $v\colon K\to\M$ with $v(x)^*v(x)=p(x)$ and $v(x)v(x)^*\leq q(x)$ for all $x\in K$ if and only if $\tau(p(x))\leq\tau(q(x))$ for all $x\in K$.  
    \item There is a $\|\cdot\|_2$-continuous $u\colon K\to\M$ taking values in the unitary group of $\M$ with $u(x)p(x)u(x)^*=q(x)$ for all $x\in K$ if and only if $\tau(p(x))=\tau(q(x))$ for all $x\in K$. 
\end{enumerate}
\end{theorem}
The counterpart existence result to Theorem \ref{IntroThmCtsProj} is immediate (and does not require property $\Gamma$). Fix a copy of $L^\infty([0,1])\subseteq \M$ with Lebesgue trace.  Given a continuous $f\colon K\to [0,1]$,  the projection-valued function $p(x)\coloneqq \chi_{[0,f(x)]}$ is $\|\cdot\|_2$-continuous and has $\tau(p(x))=f(x)$ for all $x\in K$.

\addtocontents{toc}{\SkipTocEntry}
\subsection*{Acknowledgements}
We'd like to thank Ilijas Farah, Bradd Hart, Ilan Hirshberg, G{\'a}bor Szab{\'o}, and Andrea Vaccaro for discussions and Grigorios Kopsacheilis, Julian Kranz, Kiefer Mommaerts, Valerie Morris, Mikkel Munkholm, and Pawel Sarkowicz, and Andrea Vaccaro for comments on earlier versions of the paper.  Parts of the research in this paper were undertaken at the American Institute of Mathematics as part of the SQuaRE \emph{von Neumann techniques in the classification of $C^*$-algebras} (2019-2023), the \emph{Workshop on $C^*$-algebras: structure and dynamics} (Sde Boker, 2022), and the Fields Institute as part of the \emph{Thematic program on operator algebras and their applications} (2023).  We thank the organisers and funders of these meetings.

\tableofcontents 
\addtocontents{toc}{\protect\setcounter{tocdepth}{2}} 

\renewcommand{\thetheorem}{\arabic{theorem}}
\numberwithin{theorem}{section}

\section{Introduction}

We will now describe the contents of the paper in more detail, starting with the definition of tracially complete $C^*$-algebras as a generalisation of $W^*$-bundles. We then turn to the key technical challenge for working with tracially complete $C^*$-algebras: local-to-global transfer. That is, how to upgrade properties that approximately hold in each $\|\cdot\|_{2,\tau}$-seminorm for each $\tau$, to the same statement in the uniform trace norm. We end the introduction with an overview of the proof of Theorems~\ref{Main-A}--\ref{IntroThmClassMap} and a description of how tracially complete $C^*$-algebras fit into the structure and classification theory for simple nuclear $C^*$-algebras.

\subsection{Tracially complete \texorpdfstring{$C^*$}{C*}-algebras}

Abstractly, we define a tracially complete $C^*$-algebra to be a pair $(\mathcal M,X)$ where $\mathcal M$ is a $C^*$-algebra and $X$ is a (weak$^*$) compact and convex subset of $T(\mathcal M)$ such that the unit ball of $\mathcal M$ is complete in the uniform 2-norm $\|\cdot\|_{2,X}$ induced by $X$.  This definition is designed to encompass both tracial von Neumann algebras and Ozawa's uniform tracial completions.  While in the overview above we only considered the uniform tracial completion of a $C^*$-algebra $A$ with respect to all of its traces, Ozawa's work considers uniform tracial completions with respect to any compact convex subset $X\subset T(A)$.  We write $\completion{A}{X}$ for the uniform tracial completion with respect to $X$ (and, in particular, use the notation $\completion{A}{T(A)}$ rather than $\completion{A}{\rm u}$) in the rest of the paper. Every trace $\tau\in X$ will extend uniquely to a $\|\cdot\|_{2,X}$-continuous trace on $\completion{A}{X}$; these extensions are the designated set of traces on $\completion{A}{X}$ in the definition of a tracially complete $C^*$-algebra.  With a minor abuse of notation, the extension of $\tau \in X$ to $\completion{A}{X}$ is still denoted $\tau$, and the set of such extensions is still denoted $X$.  

Given a tracial von Neumann algebra $(\M,\tau)$,\footnote{A tracial von Neumann algebra $(\M,\tau)$ is von Neumann algebra $\M$ equipped with a faithful normal trace $\tau$.} it is typically necessary in classification results to consider all (normal) traces on $\mathcal M$ instead of only $\tau$ itself.  For example, the specified trace $\tau$ will classify projections in $\M$ only when it is the unique trace on $\mathcal M$ -- i.e.\ only when $\M$ is a factor.  By analogy, we would like to work with a subclass of tracially complete $C^*$-algebras $(\mathcal M, X)$, where the traces in $X$ have the potential to provide enough information for classification.  In the case when $X$ is a face in $T(\mathcal M)$, the set $X$ is precisely the set of $\|\cdot\|_{2, X}$-continuous traces on $\mathcal M$ (and conversely) -- see Proposition~\ref{prop:face}.  We call such tracially complete $C^*$-algebras \emph{factorial}. In this way, a tracial von Neumann algebra $(\mathcal M,\tau)$ is factorial as a tracially complete C$^*$-algebra precisely when $\M$ is a factor. The name is further justified by the fact that a $W^*$-bundle is factorial as a tracially complete $C^*$-algebra if and only if each fibre is a factor (Proposition~\ref{prop:WStarBundleToTC}) and, more generally, a tracially complete $C^*$-algebra is factorial if and only if every extreme point of $X$ gives rise to a factor representation of $\mathcal M$ (Proposition~\ref{prop:factorial}).

The examples of greatest interest -- uniform tracial completions of $C^*$-algebras -- are automatically factorial.  Further, as well as being necessary to classify projections by the distinguished traces, we have found factoriality to be an important technical condition when applying local-to-global arguments. Accordingly, we have come to regard factorial tracially complete $C^*$-algebras as the fundamental objects in the class of tracially complete $C^*$-algebras.

When working with factorial tracially complete $C^*$-algebras $(\M,X)$ (such as uniform tracial completions of $C^*$-algebras), difficulties increase with the complexity of $X$.
The space $X$ is always always a Choquet simplex as it is a closed face in the Choquet simplex $T(\M)$.  When $X$ is finite dimensional or, more generally, a Bauer simplex, a factorial tracially complete $C^*$-algebra over $X$ has extra structure and is easier to analyse. This is illustrated schematically in Figure \ref{fig1}.
\begin{figure}[hh]
    \centering
\includegraphics[width=7cm]{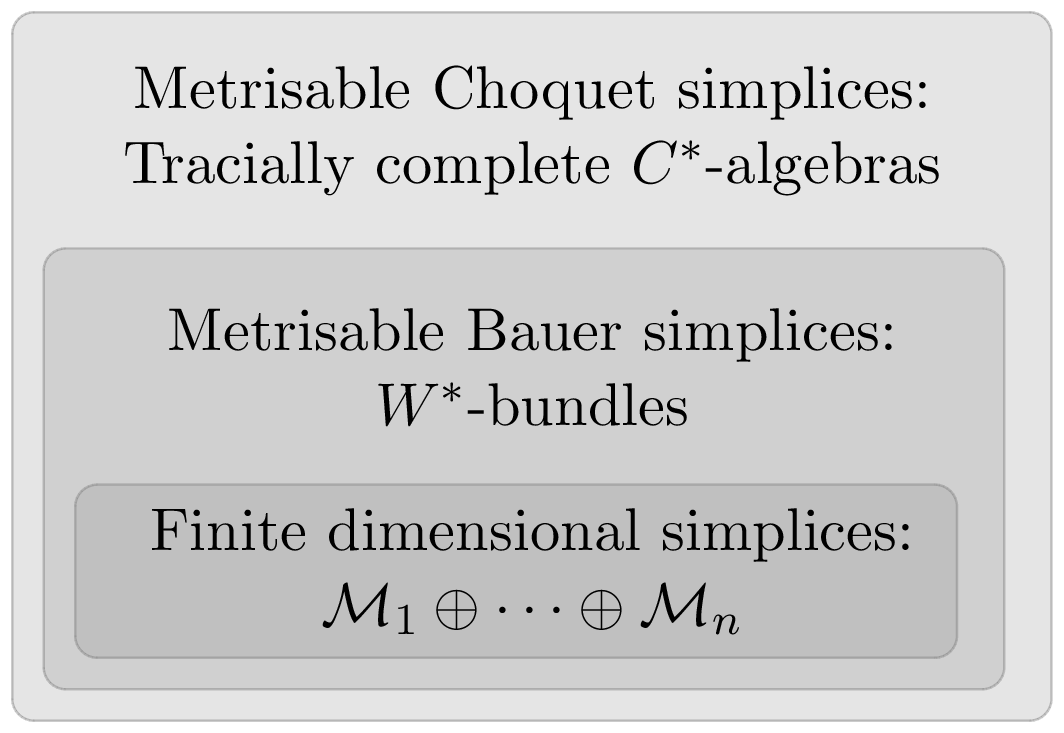}

    \caption{Factorial tracially complete C$^*$-algebras}\label{fig1}
\end{figure}

Firstly, when $X$ is a finite dimensional simplex, $\M$ is just a finite direct sum $\mathcal M_1\oplus\dots\oplus \M_n$ of factors.  In this case, Theorems \ref{InformalStructureThm} and \ref{InformalClassification} are Connes' Theorem\footnote{Theorem~\ref{IntroThmClassMap} also follows directly from Connes' Theorem (see Proposition \ref{prop:uniqueness-to-vNa}), though this is most often stated in the literature when $A$ is nuclear.} (\cite{Co76}), and regularity (i.e.\ property $\Gamma$) is automatic from amenability (again by Connes' work; see \cite[Corollary 2.2]{Co76}).

The next threshold of complexity occurs when $X$ is a \emph{Bauer simplex}: that is, when the extreme boundary $\partial_eX$ is compact. Ozawa paid particular attention to this situation in \cite{Oz13}, introducing the abstract notion of (the section algebra of) a \emph{continuous $W^*$-bundle} over a compact Hausdorff space $K$ (see Section \ref{sec:WStarBundles} for the precise definition).  The most basic examples of $W^*$-bundles are trivial bundles.  The trivial $W^*$-bundle over $K$ with fibre a tracial von Neumann algebra $(\mathcal M,\tau)$ is given by
\begin{equation}
C_\sigma(K,\mathcal M)\!\coloneqq\! \{f\colon\! K\to \mathcal M:f\text{ is }\|\cdot\|\text{-bounded and }\|\cdot\|_{2,\tau}\text{-continuous}\},
\end{equation}
together with the conditional expectation $E \colon C_\sigma(K, \mathcal M) \rightarrow C(K)$ given by composing with $\tau$.

Ozawa showed that if $A$ is a $C^*$-algebra whose tracial state space $T(A)$ is a Bauer simplex, then the uniform tracial completion $\completion{A}{T(A)}$ naturally has the structure of a continuous $W^*$-bundle over $\partial_eT(A)$ whose fibre at $\tau \in \partial_e T(A)$ is the von Neumann factor $\pi_\tau(A)''$. Moreover, he proved a `trivialisation theorem' (\cite[Theorem 15]{Oz13}), characterising when $W^*$-bundles whose fibres are the hyperfinite II$_1$ factor are trivial -- this amounts to asking when Connes' theorem can be established in a continuous fashion over $K$. Theorem \ref{Main-A} follows from Ozawa's trivialisation theorem whenever the $C^*$-algebras involved have Bauer trace simplices. Using Ozawa's methods, we show in Section \ref{sec:WStarBundles} that factorial tracially complete $C^*$-algebras $(\M,X)$ whose traces $X$ form a Bauer simplex have a natural structure of a $W^*$-bundle over $\partial_eX$, whose fibre at $\tau \in \partial_e X$ is the factor $\pi_\tau(\M)$.\footnote{By Ozawa's work, $\pi_\tau(\M)=\pi_\tau(\M)''$ -- see \cite[Theorem~11]{Oz13}.} 
In this way, Theorems~\ref{InformalStructureThm} and \ref{InformalClassification} also follow from Ozawa's trivialisation theorem in the case of factorial tracially complete $C^*$-algebras with a Bauer simplex of traces.

In general, the trace simplex of a (simple) $C^*$-algebra can be very far from Bauer. Indeed, any metrisable Choquet simplex can arise as the trace space of a separable $C^*$-algebra, including, for example, the Poulsen simplex characterised as the unique (non-trivial) metrisable Choquet simplex where the extreme points form a dense subset, and which   
stands at the opposite extreme to Bauer simplices, with a highly complex affine structure (see \cite{LindenstrassOlsenSternfeld}).  

We think of factorial tracially complete $C^*$-algebras $(\M,X)$ as providing a formalism for (the section algebra of) a bundle over a Choquet simplex $X$, with `fibres' coming from the GNS representations of $\M$ at points $\tau\in X$, which takes into account the affine relations between the fibres (see the discussion in Section \ref{sec:tracially-complete}).\footnote{At present, this is just intuition -- we would welcome a formal framework for viewing tracially $C^*$-algebras as affine bundles.} Outside the Bauer simplex setting, the subtle interaction between the affine structure of $X$ and the operator algebraic structure of $\M$ is quite challenging and leads to very different behaviour compared to $W^*$-bundles. To give one example, suppose that $(\M,X)$ is a factorial tracially complete $C^*$-algebra; if $X$ is a Bauer simplex then the centre of $\M$ can be identified with the continuous affine functions on $X$, whereas in the general case, the centre of $\M$ can be trivial.

\subsection{The trace problem} We take a brief interlude to discuss why we keep track of the designated collection of traces in our definition of a tracially complete $C^*$-algebra.   In a nutshell, while traces on a II$_1$ factor are automatically normal (as the unique trace is normal), the issue is whether the analogous statement holds for factorial tracially complete $C^*$-algebras holds. 

When $(\M,X)$ is a tracially complete $C^*$-algebra, all the traces in $X$ are evidently $\|\cdot\|_{2,X}$-continuous -- the appropriate notion of continuity in this setting. Moreover, all $\|\cdot\|_{2,X}$-continuous traces on $\M$ belong to the closed face in $T(\M)$ generated by $X$ (see Proposition \ref{prop:face}). Accordingly, when $(\M,X)$ is factorial, $X$ consists precisely of the $\|\cdot\|_{2,X}$ continuous traces on $\M$.  Are there any other traces? We regard this as a foundational problem in the theory of tracially complete $C^*$-algebras.

\begin{question}[Trace Problem]\label{Q:traces}
Let $(\M,X)$ be a factorial tracially complete $C^*$-algebra.  Are all traces on $\M$ automatically $\|\cdot\|_{2,X}$-continuous? Equivalently, is the inclusion $X\subseteq T(\M)$ an equality?
\end{question}

When $X$ is a finite dimensional simplex, $\mathcal M$ is a finite direct sum of factors, and hence the trace problem has a positive solution since all traces on $\mathcal M$ are normal.   In this paper, we resolve the trace problem for ultrapowers (and reduced products) of tracially complete $C^*$-algebras with property $\Gamma$ (Theorem \ref{thm:no-silly-traces}); property $\Gamma$ is discussed in Section \ref{sec:intro:ltg2} and Section \ref{sec:Gamma}.  This is in the spirit of various `no silly trace' results asserting that, under appropriate regularity conditions, traces on ultrapowers are generated by the limit traces.  Such results include \cite[Theorem 8]{Oz13}, \cite[Theorem~1.2]{NR16}, and \cite[Theorem~3.22]{BBSTWW}; see \cite[Theorem~A]{APRT23} for a very general $C^*$-algebra result.

The trace problem is particularly pertinent when we take tracial completions.  Given a $C^*$-algebra $A$, Ozawa's uniform tracial completion $\completion{A}{T(A)}$ gives rise to a factorial tracially complete $C^*$-algebra $\big(\completion{A}{T(A)},T(A)\big)$, but is $\completion{A}{T(A)}$ uniformly tracially complete with respect to \emph{all} its traces?  If there are additional traces on $\completion{A}{T(A)}$ which are not $\|\cdot\|_{2,T(A)}$-continuous, then there seems to be no reason why this should be the case.  A positive answer to the trace problem would ensure that the uniform tracial completion process stabilises.\footnote{While this paper was in preparation, the trace problem was resolved positively by the third-named author for tracially complete $C^*$-algebras with property $\Gamma$ (\cite{Ev24}). In particular, this means that for a $\Z$-stable $C^*$-algebra $A$, Ozawa's $\completion{A}{T(A)}$ has a complete unit ball in its uniform 2-norm.}

\subsection{Local-to-global: amenability}\label{Sect:IntroAmenability}

Our main approach to understanding the structure of tracially complete $C^*$-algebras is to pass from local properties at each trace, i.e.\ properties of the fibres $\pi_\tau(\mathcal M)''$, to global properties that hold uniformly over all traces.  Identifying the appropriate class of \emph{amenable} tracially complete $C^*$-algebras through completely positive approximations gives a first example of this idea.

Recall that nuclearity of a $C^*$-algebra $A$ is characterised through the completely positive approximation property (\cite{Ki77}): there is a net  $(F_i,\phi_i,\psi_i)_i$ consisting of finite dimensional $C^*$-algebras $F_i$ and completely positive and contractive (c.p.c.) maps
\begin{equation}
\begin{tikzcd}
    A\ar[r,"\psi_i"] & F_i\ar[r,"\phi_i"]&A
\end{tikzcd}
\end{equation}
such that $\|\phi_i(\psi_i(a))-a\|\to 0$ for all $a\in A$.  The directly analogous concept for a von Neumann algebra $\M$ is \emph{semidiscreteness}, which asks for 
c.p.c.\ maps
\begin{equation}\label{Intro:CPAP1}
\begin{tikzcd}
    \M\ar[r,"\psi_i"] & F_i\ar[r,"\phi_i"]&\M
\end{tikzcd}
\end{equation}
which approximate the identity on $\M$ in the point-weak$^*$ topology (rather than the point-norm topology).  When $(\M,\tau)$ is a tracial von Neumann algebra, (i.e.\ $\tau$ is a specified faithful normal trace on $\M$), a standard Hahn--Banach argument (recalled as Proposition~\ref{prop:tracialnuc=weaknuc} below) shows that this is equivalent to the completely positive approximation property in the $\|\cdot\|_{2,\tau}$-norm.  A large body of deep work, including Connes' theorem, gives a plethora of other conditions equivalent to the completely positive approximation property for $C^*$-algebras and to semidiscreteness for a von Neumann algebras. 

If we look for analogous conditions on a tracially complete $C^*$-algebra $(\M,X)$, one option is to work locally and ask for all the von Neumann fibres $\pi_\tau(\M)''$ for $\tau\in X$ to be  semidiscrete. Our local-to-global result for amenability allows us to go from such a local condition to a single system of completely positive approximations which works uniformly over all traces. 

\begin{theorem}\label{thm:introamenable}
Let $(\mathcal M,X)$ be a tracially complete $C^*$-algebra.  The following are equivalent:
\begin{enumerate}
        \item $(\M, X)$ is {}amenable, in the sense that the completely positive approximation property holds in the point-$\|\cdot\|_{2,X}$ topology;
        \item for all $\tau \in X$, $\pi_\tau(\M)''$ is semidiscrete, in the sense that the completely positive approximation property holds for $\pi_\tau(X)''$ in the point-weak$^*$ topology;
        \item every $\tau \in X$ is uniformly amenable in the sense of \cite[Definition~3.2.1]{Bro06}.
\end{enumerate}
\end{theorem}

In particular, via Connes' theorem, the uniform tracial completion of a nuclear $C^*$-algebra is {}amenable as a tracially complete $C^*$-algebra.\footnote{A naive argument for this fails in the same way that a direct proof that the bidual of a nuclear $C^*$-algebra is semidiscrete as a von Neumann algebra fails -- an extension of the approximations witnessing nuclearity of $A$ will not a priori approximate the identity map on the tracial completion in point-$\|\cdot\|_{2,X}$.}

Theorem \ref{thm:introamenable} generalises to $^*$-homomorphisms between $C^*$-algebras and tracially complete $C^*$-algebras (see Theorem \ref{thm:amenable}, which proves the first statement in Theorem \ref{IntroThmClassMap}) characterising the completely positive approximation property (with respect to the uniform 2-norm) in terms of pointwise amenability conditions.  We call such maps \emph{tracially nuclear}, and this is the appropriate amenability condition on $^*$-homomorphisms  for classification; see the discussion in Section \ref{Intro:StructureClass}.

Theorem \ref{thm:introamenable} and its generalisation to morphisms in Theorem \ref{thm:amenable} are both proved by means of a Hahn--Banach trick which we learnt from \cite{KranzPhd,GGKNV}.  The key point -- used by Ozawa in \cite{Oz13} -- is that the weak topology on the space $\mathrm{Aff}(X)$ of continuous affine functions on a Choquet simplex is given by pointwise convergence. 
Since the set of c.p.c.\ maps which factorise through finite dimensional $C^*$-algebras is closed under convex combinations, it is then possible to take convex combinations of a finite set of local approximations (obtained through compactness) to build a global approximation.

\subsection{Local-to-global: CPoU and property \texorpdfstring{$\Gamma$}{Gamma}}\label{sec:intro:ltg2}

The use of the Hahn--Banach theorem with $\mathrm{Aff}(X)$ described in the previous section provides a general local-to-global tool in the setting of tracially complete $C^*$-algebras for conditions witnessed by a system of affine equations over a convex set (see Lemma \ref{lem:affine-selection}).  However, the Hahn--Banach strategy is unlikely to apply to conditions which are not affine in nature.  To discuss this, let us consider the problem (which is open in general) of whether a type II$_1$ factorial tracially complete $C^*$-algebra $(\M,X)$ admits approximate projections which are approximately of trace $1/2$; i.e.\ given $\epsilon>0$, does there exist a positive contraction $p\in\M$ with $\|p-p^2\|_{2,X}<\epsilon$ and $|\tau(p)-1/2|<\epsilon$ for all $\tau\in X$?

As we think of tracially complete $C^*$-algebras as a kind of affine bundle, it is natural to approach local-to-global problems via partition of unity techniques.  For a $W^*$-bundle $\M$ with II$_1$ factor fibres over a compact Hausdorff space $K$, the algebra $C(K)$ embeds centrally into $\M$ in a way which identifies $X$ with the space $\mathrm{Prob}(K)$ of Radon probability measures on $K$.  It is straightforward to combine elements in $\M$ with a partition of unity in $C(K)$; however, such a direct approach does not preserve algebraic conditions such as (in our sample problem) being a projection.  To see where it fails, identify each point of $K$ with the trace on $\mathcal M$ corresponding to the point-mass measure on $K$, and for each point $\tau\in K$, fix a positive contraction $p_\tau\in\M$ such that $\pi_\tau(p_\tau)$ is a projection of trace $1/2$ in the fibre $\pi_\tau(\M)$ (which is automatically a II$_1$ factor). Fixing $\epsilon>0$, compactness gives an open cover $U_1,\dots,U_n$ of $K$ and positive contractions $p_1,\dots,p_n\in \M$ such that \begin{equation}
    \|p_{i}^2-p_{i}\|_{2,\tau}<\epsilon \quad \text{and} \quad \tau(p_i) \approx_\epsilon 1/2,\qquad \tau \in U_i.
\end{equation}
Taking a partition of unity $f_1,\dots,f_n\in C(K)$ subordinate to $U_1,\dots,U_n$, one can form $p\coloneqq \sum_{i=1}^nf_ip_{i}$. While this will have $\tau(p)\approx_\epsilon 1/2$ for all $\tau\in K$, there is no reason for $p$ to be an approximate projection in uniform 2-norm: partitions of unity from $C(K)$ do not interact well with multiplication due to the lack of orthogonality of the $f_i$.

The solution is to ask that the $f_i$ are approximate projections (giving rise to approximate orthogonality).  However, unless $K$ is zero dimensional, this requires that they come from $\M$ instead of $C(K)$, and we must weaken centrality to approximate centrality. If such $f_i$ can be found, then $p\coloneqq \sum_{i=1}^nf_i p_i$ will give the required approximate
projection in $\M$.\footnote{The reason this works is that $\tau(f_ip_i)\approx\tau(f_i)\tau(p_i)$ uniformly over  $\tau\in K$; see Proposition~\ref{prop:BauerDivideUnit}. For subtle reasons, it is essential for this approximation that $K$ is compact, i.e.\ that as a tracially complete $C^*$-algebra, the underlying simplex is Bauer.\label{fn:TracialFactorisation}}  In general, such partitions of unity need not exist (see Example~\ref{eg:free-group-bundle}), but when they do, they can be used for local-to-global results.  Ozawa's trivialisation theorem is a quintessential example: the passage from the existence of approximately central approximate projections of trace $1/2$ in a $W^*$-bundle with hyperfinite II$_1$ factor fibres to global triviality of the bundle (\cite[Theorem~15(ii)$\Rightarrow$(iii)$\Rightarrow$(i)]{Oz13}) is underpinned by such partition of unity arguments.  Subsequently, this strategy was made explicit and systematically used in \cite{BBSTWW} for local-to-global transfer in trivial $W^*$-bundles whose fibre is a McDuff II$_1$ factor.

Beyond the $W^*$-bundle setting, one needs even more control on the approximately central projections $f_i$ forming a partition of unity.
It is necessary to be able to uniformly control $\tau(f_ia)$ for certain $a$, whereas in the $W^*$-bundle setting, this control comes for free from knowing that $\tau(f_i)$ vanishes outside of $U_i$ (see Footnote~\ref{fn:TracialFactorisation}). A subset of the present authors together with Winter identified a solution in \cite{CETWW} by means of \emph{complemented partitions of unity} (CPoU) for a $C^*$-algebra.  The technical definition is designed to give the required control with respect to a given family of positive elements, which should be thought of as complementary to the tracial support of the $f_i$. See the second part of the introduction to \cite{CETWW} for a further discussion of this and the challenges that must be overcome outside the Bauer simplex setting.

In \cite{CETWW}, the concept of complemented partitions of unity was set out using the uniform tracial ultraproduct of a $C^*$-algebra, but as we describe in Section \ref{sec:CPoU}, this definition naturally lives at the level of tracially complete $C^*$-algebras. When complemented partitions of unity can be found, they give rise to a very general local-to-global transfer process.  We give a number of examples in Section \ref{sec:app-CPoU} (some of which are based on applications of CPoU for uniform tracial closures of nuclear $C^*$-algebras in \cite{CETW,CETW-classification}). 

As with the example described above of approximate projections, the local-to-global transfer process produces approximate properties in factorial tracially complete $C^*$-algebras with CPoU, i.e.\ properties holding up to a small error in uniform 2-norm.  In order to cleanly encode such approximate conditions, we develop the theory of reduced products (including ultraproducts) of tracially complete $C^*$-algebras in Section \ref{sec:InductiveLimitsETC}. The effect of the local-to-global argument is that the ultrapower of factorial tracially complete $C^*$-algebras with CPoU enjoys several of the fundamental properties of a finite von Neumann algebra: real rank zero, stable rank one, all unitaries are exponentials, and Murray--von Neumann comparison of projections is determined by traces.  Our solution to the trace problem for such reduced products (Theorem \ref{thm:no-silly-traces}) is obtained as a consequence of these results, avoiding the use of sums of commutators found in precursor results such as \cite[Section 3.4]{BBSTWW} (see Remark~\ref{rmk:commutators}).

Since CPoU arguments inevitably produce approximate conclusions, there is a more work to be done to achieve the exact classification of projections needed for Theorem \ref{IntroThmCtsProj}.  This is achieved by means of an intertwining argument based on explicit estimates showing that projections of the same trace which are close in 2-norm are approximately conjugate by a unitary close to the unit. This is in the spirit of similar perturbation results for finite von Neumann algebras and leads to the following result (proved as Theorems~\ref{thm:comparison} and~\ref{thm:existence-projections}).

\begin{theorem}\label{introthm-classprojections2}
Let $(\M,X)$ be a factorial type {\rm II}$_1$ tracially complete $C^*$-algebra with CPoU. 
\begin{enumerate}
    \item\label{intro-class-proj-unique} If $p, q \in \mathcal M$ are projections and $\tau(p) = \tau(q)$ for all $\tau \in X$, then $p$ and $q$ are unitarily equivalent.
    \item\label{intro-class-proj-exist} For any continuous affine $f\colon X\to[0,1]$ there is a projection $p\in\M$ with $\tau(p)=f(\tau)$ for all $\tau\in X$.
\end{enumerate}
\end{theorem}

The fundamental challenge is to determine when complemented partitions of unity can be found, i.e.\ when a factorial tracially complete $C^*$-algebra has CPoU.  The concept of \emph{uniform property $\Gamma$} for a $C^*$-algebra was introduced in \cite{CETWW} as a uniform 2-norm version of Murray and von Neumann's property $\Gamma$ for II$_1$ factors; just as McDuff II$_1$ factors have property $\Gamma$, so $\Z$-stable $C^*$-algebras have uniform property $\Gamma$.  Moreover, like CPoU, uniform property $\Gamma$ is most naturally a property of tracially complete $C^*$-algebras, and we say that $(\M,X)$ has property $\Gamma$ when there exist uniform 2-norm approximately central approximate projections, which approximately divide the trace of elements of $\M$ in half (Definition~\ref{def:UTCgamma} and Proposition~\ref{prop:Gamma}). The main technical result of \cite{CETWW} is that unital nuclear $C^*$-algebras with uniform property $\Gamma$ have complemented partitions of unity.  We strengthen the main result of \cite{CETWW} by removing the condition of nuclearity, allowing CPoU to be obtained from property $\Gamma$ in general.  As set out in Section \ref{intro:applications}, this forms a major ingredient in the forthcoming general framework for the classification of $^*$-homomorphisms (\cite{CGSTW2}).

\begin{restatable}{theorem}{gammacpou}
\label{introthmgammaimpliescpou}
Let $(\M,X)$ be a factorial tracially complete $C^*$-algebra with property $\Gamma$.  Then $(\M,X)$ has CPoU.  In particular, unital $C^*$-algebras with uniform property $\Gamma$ (e.g.\ unital $\Z$-stable $C^*$-algebras) have CPoU.
\end{restatable}

For a factorial tracially complete $C^*$-algebra $(\M,X)$, all of whose fibres $\pi_\tau(\M)''$ for $\tau \in \partial_e X$ have property $\Gamma$ as von Neumann algebras, obtaining property $\Gamma$ for $(\M,X)$ is itself a local-to-global transfer problem.  Theorem \ref{introthmgammaimpliescpou}, together with the local-to-global technology, shows that transferring property~$\Gamma$ from fibres to a global condition is to some extent a universal local-to-global problem.

The strategy to prove Theorem \ref{introthmgammaimpliescpou} follows the overall  framework used in \cite{CETWW}, which, in the language of this paper, proves Theorem \ref{introthmgammaimpliescpou} for the uniform tracial completion of a separable nuclear $C^*$-algebra.  The argument splits into two steps:
\begin{itemize}
    \item Obtain a weak form of CPoU in which all the approximate projections making up the partition of unity are replaced by contractions (Theorem~\ref{thm:weak-CPoU}) and are not orthogonal.  
    \item Use property $\Gamma$ to convert the weak form of CPoU to CPoU by means of orthogonalisation, projectionisation, and a maximality argument (see the discussion in the last section of the introduction to \cite{CETWW} for an outline).
\end{itemize}
The second step works generally as was foreshadowed in \cite[Lemma 3.7]{CETWW}. However, in \cite{CETWW}, nuclearity was instrumental in performing the first step (\cite[Lemma 3.6]{CETWW}) through a refined form of the completely positive approximation property from \cite{BCW16}.  In particular, Connes' theorem on the equivalence of injectivity and hyperfiniteness underpins these approximations.  Our proof of Theorem \ref{introthmgammaimpliescpou} establishes this weak form of CPoU in general (Theorem~\ref{thm:weak-CPoU}), a result which is already of independent interest (see \cite{SW23,WoutersPhd}).
 We do this by means of a Hahn--Banach-driven local-to-global transfer of the form described in Section \ref{Sect:IntroAmenability}. The point is that all finite von Neumann algebras (viewed as tracially complete $C^*$-algebras with respect to all their traces) satisfy CPoU. Taking suitable convex combinations of the elements witnessing CPoU in finitely many fibres gives rise to the required weak form of CPoU, and we do not need to rely on anything as deep as Connes' work.  Specialising to the uniform tracial completion of a nuclear $\mathcal Z$-stable $C^*$-algebra, this approach gives a much simpler overall argument for CPoU as compared with \cite{CETWW}.

\subsection{Structure and classification}\label{Intro:StructureClass}

Our structure and classification results (Theorems \ref{Main-A}--\ref{IntroThmClassMap}) are proved by a combination of local-to-global transfer in the same spirit as \cite{CETWW,CETW,CETW-classification} and Elliott-style intertwining arguments.  

For a simple separable nuclear $C^*$-algebra $A$, a local-to-global argument was given in \cite[Theorem~4.6]{CETW} to pass from uniform property $\Gamma$ to the uniform McDuff property via CPoU. The key point is that as $A$ is nuclear, all tracial von Neumann algebras $\pi_\tau(A)''$ associated to traces on $A$ are McDuff, and this can then be transformed into a global statement via CPoU.  This proves the equivalence of property $\Gamma$ and the McDuff property in Theorem~\ref{InformalStructureThm} for the uniform tracial closures $\big(\completion{A}{T(A)},T(A)\big)$ of such $C^*$-algebras. An identical argument can be used to obtain this equivalence for {}amenable factorial type II$_1$ tracially complete $C^*$-algebras, the only difference being the use of Theorem~\ref{introthmgammaimpliescpou} in place of the main result of \cite{CETWW}.  For the additional equivalence of hyperfiniteness in Theorem \ref{InformalStructureThm}, we go through classification (Theorems~\ref{InformalClassification} and~\ref{IntroThmClassMap}). Our local characterisations of amenability (Theorem~\ref{thm:introamenable}) ensure that hyperfinite tracially complete $C^*$-algebras are {}amenable (Theorem \ref{thm:hyperfinite-implies-semidiscrete}), so once we show hyperfinite factorial tracially complete $C^*$-algebras have CPoU (Theorem \ref{thm:hyperfinite-implies-CPoU}), the rest of the structure theorem (Theorem \ref{InformalStructureThm}) will follow from the classification theorem.  

As factorial finite dimensional tracially complete $C^*$-algebras have CPoU, and CPoU is preserved under inductive limits, it is easy to obtain CPoU for limits of factorial finite dimensional $C^*$-algebras.
It is similarly straightforward to show CPoU holds for a tracially complete $C^*$-algebra that is locally approximated by embedded factorial finite dimensional tracially complete subalgebras.
The problem comes when a factorial tracially complete $C^*$-algebra $(\M,X)$ has approximations by finite dimensional $C^*$-subalgebras in which not all traces extend to elements of $X$ -- and so they do not embed as tracially complete subalgebras. This is reminiscent of Murray and von Neumann's work (\cite{MvN43}) on the uniqueness of hyperfinite $\mathrm{II}_1$ factors, where they have to be concerned with approximating finite dimensional subalgebras which are not factors.  Our solution, found in Section \ref{sec:hyperfinite}, is conceptually the same as Murray and von Neumann's: reduce to the case that the building blocks can always be taken to be factorial.  To do this we show that for separable tracially complete $C^*$-algebras, local and inductive limit definitions of hyperfiniteness agree, from which we deduce that hyperfinite tracially complete $C^*$-algebras are tracial completions of AF $C^*$-algebras.

Turning to classification, the primary objective is the classification of tracially nuclear maps in Theorem \ref{IntroThmClassMap} (proved as Theorem~\ref{thm:classification-morphism}\ref{classif2}). As is standard for classification results for maps, this consists of two components: \emph{existence} of tracially nuclear $^*$-homomorphisms with specified behaviour at the level of traces and \emph{uniqueness} of such maps up to approximate unitary equivalence. The uniqueness aspect of Theorem \ref{IntroThmClassMap} (Theorem \ref{thm:uniqueness}) is a direct application of the corresponding uniqueness result for weakly nuclear $^*$-homomorphisms into finite von Neumann algebras by traces (a folklore consequence of Connes' theorem) and a CPoU-powered local-to-global argument.  This works in essentially the same fashion as the uniqueness result found in \cite[Theorem 2.2]{CETW-classification}, and the only difference is the increased generality of the statement.  

As with the corresponding theorems in the $C^*$-classification programme, the existence aspect of Theorem \ref{IntroThmClassMap} is obtained in two stages.  At the first pass, one only aims for an `approximate' result -- for a $C^*$-algebra $A$ and factorial tracially complete $C^*$-algebra $(\mathcal M,X)$, we want uniform 2-norm approximately multiplicative maps $A\rightarrow \mathcal M$ which approximately implement a given continuous affine map $\gamma\colon X\to T(A)$. For this to hold generally, it will be necessary for the traces $\gamma(X)\subseteq T(A)$ to satisfy suitable approximation properties -- in particular, we require the range of $\gamma$ to consist of \emph{hyperlinear} traces on $A$ (i.e.\ those factoring through the tracial ultrapower $\mathcal R^\omega$ of $\mathcal R$).\footnote{When $(\mathcal M, X)$ is the hyperfinite $\mathrm{II}_1$ factor with its unique trace $\tau_\mathcal R$, this approximate existence result is equivalent to hyperlinearity of the relevant trace on $A$.}  The approximate existence result uses another CPoU-powered local-to-global argument to patch together approximate morphisms into the fibres $\pi_\tau(\mathcal M)''$ for $\tau \in X$, which arise from composing approximate embeddings into $\mathcal R$ with a unital embedding $\mathcal R \rightarrow \pi_\tau(\mathcal M)''$.  Under the stronger hypothesis that the range of $\gamma \colon X \rightarrow T(A)$ consists of uniformly amenable traces, the approximate morphisms $A \rightarrow \mathcal M$ approximately realising $\gamma$ can be further arranged to be tracially nuclear.

While it would be possible (though technically somewhat awkward) to prove Theorem \ref{thm:existence} in a similar fashion to that of \cite[Theorem 2.6]{CETW-classification}, we instead take advantage of our classification of projections (Theorem~\ref{introthm-classprojections2}), and hence maps from finite dimensional algebras, 
to give a different and arguably more conceptual approach to these results.  The point is that as $(\M,X)$ is factorial, $X$ is a Choquet simplex, and so we can approximate $\gamma$ by affine maps $X\to Z_\lambda \to T(A)$ factoring through finite dimensional simplices $Z_\lambda$ using a result of Lazar and Lindenstrauss from the early 1970s (\cite{Lazar-Lindenstrauss71}).  Then one uses hyperlinearity to produce approximately multiplicative maps from $A$ into finite dimensional algebras which approximately realise the maps $Z_\lambda \to T(A)$ followed by the classification of projections to embed these finite dimensional algebras into $\M$ compatibly with the maps $X\to Z_\lambda$. 

At the second pass one aims for exact existence results by means of a one-sided Elliott intertwining argument. This is by now a standard technique, which is readily imported to the setting of tracially complete $C^*$-algebras and uniform 2-norms.\footnote{The one subtle point is that the intertwining by reparameterisation technique for constructing genuine morphisms from approximate morphisms (Theorem~\ref{thm:reparameterisation}) requires stability of unitaries in the uniform 2-norm.  We prove that this follows from CPoU in Corollary~\ref{cor:unitary-stable-relation}.} As ever, it is important that the uniqueness theorem is strong enough to cover approximately multiplicative maps. This gives rise to the exact existence result in Theorem \ref{thm:classification-morphism} and completes the proof of Theorem \ref{IntroThmClassMap}.

Classification results for uniformly tracially complete $C^*$-algebras are then obtained from Theorem \ref{IntroThmClassMap} using a two-sided Elliott intertwining argument. This classifies the family of amenable tracially complete $C^*$-algebras $(\mathcal M,X)$ satisfying both the domain and codomain hypotheses of Theorem \ref{IntroThmClassMap}. The amenability hypothesis is needed so that the relevant identity maps are tracially nuclear (and so fall within the scope of Theorem \ref{IntroThmClassMap}) as the intertwining argument will use uniqueness to compare the identity maps with the compositions of the maps obtained from the existence portion of Theorem \ref{IntroThmClassMap}.  This process (which is an instance of Elliott's abstract classification framework from \cite{El09}) yields Theorem \ref{InformalClassification}, which contains Theorem \ref{Main-A} as a special case.  See \cite[Section 6]{White:ICM} for a general description of the passage from approximate existence and uniqueness, to the classification of operator algebras via one-sided intertwining and then a symmetrisation of hypotheses.

Just as we often prefer to describe the outcomes of the local-to-global transfer procedure at the level of ultraproducts or sequence algebras in order to suppress explicit error tolerances, we do the same for approximately multiplicative maps. Accordingly, in the main body our approximate existence result and the corresponding uniqueness theorem are given in terms of exact classification results into reduced products as was done with the precursor results in \cite{CETW-classification}.

\subsection{Classification and the Toms--Winter conjecture}\label{intro:applications}

We end the introduction by discussing the role tracially complete $C^*$-algebras play in the structure and classification of simple nuclear stably finite $C^*$-algebras. These are some instances of step \ref{item:intro:overallplan3}, `pulling results back from the tracial completion to the $C^*$-level', of the scheme on page \pageref{item:intro:overallplan3}.

Given a simple unital $C^*$-algebra $B$ with $T(B)\neq\emptyset$, write $\mathcal M\coloneqq \completion{B}{T(B)}$, $B_\infty \coloneqq \ell^\infty(B) / c_0(B)$ for the norm approximate sequence algebra of $B$, and $\mathcal M^\infty$ for the uniform 2-norm approximate sequence algebra of $\M$, i.e.\ $\ell^\infty(\M)$ modulo the $\|\cdot\|_{2,T(B)}$-null sequences. By construction, there is a Kaplansky density type theorem -- the unit ball of $B$ is $\|\cdot\|_{2,T(B)}$-dense in the unit ball of $\M$ -- and this ensures that the canonical inclusion $B\to \mathcal M$ gives a surjection $B_\infty\to\mathcal M^\infty$.  The \emph{trace-kernel extension} is the short exact sequence induced by this surjection:\footnote{Ultrapower versions of the trace-kernel extension are also extensively used in the literature. We develop the theory in terms of reduced powers with respect to a free filter, simultaneously covering both cases.}
\begin{equation}
    0\longrightarrow J_B\longrightarrow B_\infty\longrightarrow \M^\infty\longrightarrow 0.
\end{equation}
The ideal $J_B$ is known as the \emph{trace-kernel ideal}, and it inherits regularity properties (such as separable $\Z$-stability and strict comparison) from corresponding properties of $B$.  

The main objective in the abstract approach to the unital classification theorem for simple nuclear $C^*$-algebras in \cite{CGSTW} is a classification of full unital $^*$-homomorphisms $A\to B_\infty$, where $A$ is a unital separable nuclear $C^*$-algebra satisfying the UCT and $B$ is a unital simple separable nuclear $\Z$-stable finite $C^*$-algebra (\cite[Theorem 1.1]{CGSTW}). The classification of $C^*$-algebras is obtained from this using Elliott intertwining arguments.  As set out in \cite[Section 1.3]{CGSTW} at a very high level,\footnote{There, the argument is given using the uniform tracial sequence algebra $B^\infty$; this is canonically isomorphic to $\M^\infty$ by a Kaplansky density argument.} the strategy for the classification of full approximate $^*$-homomorphisms falls into the three step plan from page \pageref{item:intro:overallplan3}.
\begin{enumerate}
\item Classify maps from $A$ into a finite von Neumann algebra by traces; this is a consequence of Connes' theorem.
\item Classify maps from $A$ into $\mathcal M^\infty$ by traces; this is the classification of approximate $^*$-homomorphisms from $A$ into $\M$ described in Section \ref{Intro:StructureClass}, and the instance used here is the case covered in the combination of \cite{CETWW,CETW-classification} (as $\M$ is the uniform tracial completion of a nuclear $\Z$-stable $C^*$-algebra).
\item Classify lifts of a full map $\theta\colon A\to \M^\infty$ back to $B_\infty$ in terms of $K$-theoretic data. This is the main task of \cite{CGSTW}, and a detailed outline of the strategy for this step is given there.
\end{enumerate}
In particular, none of the results in this paper are necessary for the abstract proof of the unital classification paper, though we do contend that, in hindsight, the classification results for maps $A\to\mathcal M$ given here make the three-step plan above more transparent.

To take stably finite classification beyond the setting of nuclear codomains, our Theorem~\ref{introthmgammaimpliescpou} will be crucial to step~\ref{item:intro:overallplan2}.  One objective is a stably finite version of Kirchberg's very general classification results for full morphisms $A\to B_\infty$, where $A$ is a separable exact $C^*$-algebras satisfying the UCT and $B$ is an $\mathcal O_\infty$-stable $C^*$-algebra (see \cite[Theorem 8.3.3]{Rordam-Book}, \cite[Theorems A and B]{Ga21}, and \cite{KirchbergLong}).
 In the forthcoming work \cite{CGSTW2}, a subset of us will extend the stably finite classification of morphisms to a level of generality corresponding to Kirchberg's framework. Sticking to the unital case, this will classify unital full nuclear $^*$-homomorphisms $A\to B_\infty$ for the domains $A$ as in Kirchberg's theorem, and where $B$ is unital, finite, $\Z$-stable, and has comparison of positive elements by bounded traces.\footnote{This will be discussed further in \cite{CGSTW2}, but it forces all quasitraces on $B$ to be traces.} Via intertwining, this entails a classification of unital $^*$-homomorphisms from $A$ to $B$ which map traces on $B$ to faithful traces on $A$.  Outside the setting where $B$ is nuclear, one cannot obtain CPoU for $B$ (or equivalently, for its uniform tracial completion) from \cite{CETWW} and Theorem \ref{introthmgammaimpliescpou} is essential.

Looking to the future, we hope that the breakthrough results in \cite{GS} on the classification of stronger outer actions of countable discrete amenable groups\footnote{The results of \cite{GS} are much more general than this, encompassing amenable isometrically shift absorbing actions of locally compact groups.} on Kirchberg algebras will, over time, have powerful stably finite counterparts.  Here, we expect that developing a suitable classification of group actions on classifiable tracially complete $C^*$-algebras (step \ref{item:intro:overallplan2}) will help break up the overall task into more manageable parts, particularly in the case where the underlying action on the trace space is complex.  Even more generally, one can imagine more general notions of quantum symmetries acting on uniform tracially complete $C^*$-algebras as a bridge towards studying such actions on classifiable $C^*$-algebras whose tracial state space is large.

On the structure side, the remaining open part of the Toms--Winter conjecture is intimately linked with tracially complete $C^*$-algebras. By now, for a simple separable non-elementary nuclear $C^*$-algebra $B$, the conditions of $\Z$-stability and finite nuclear dimension are known to coincide (\cite{Wi12,Wi10,Ti14} and \cite{CETWW,CE}, building on  \cite{BBSTWW,SWW15,MS14}). Back in 2004, R\o{}rdam showed that if $B$ is $\Z$-stable, then $B$ has strict comparison (\cite{Ro04}) and the remaining piece of the Toms--Winter conjecture is the converse.\footnote{A weaker conjecture, of the form that pure simple separable non-elementary nuclear $C^*$-algebras are $\Z$-stable, is discussed by Winter in \cite[Section 5.4]{Wi17}.  Here, pureness is the combination of strict comparison and a (tracial) divisibility condition on the Cuntz semigroup, which can be thought of as a combination of a weak uniqueness theorem and an existence theorem for positive elements in terms of their rank functions.\label{fn:WinterConjecture}}

The trace-kernel extension, as we view it today, has its origins in Matui and Sato's breakthrough work on the implication from strict comparison to $\Z$-stability (\cite{MS12}). A key fact is that it induces a short exact sequence at the level of central sequence algebras (see \cite[Theorem~3.3]{KR14})
\begin{equation}
    0\longrightarrow J_B\cap B'\longrightarrow B_\infty\cap B'\longrightarrow \M^\infty\cap \mathcal M'\longrightarrow 0,
\end{equation}
and tensorial absorption results for $B$ and $\M$ can be written in terms of the central sequence algebras $B_\infty\cap B'$ and $\M^\infty\cap \M'$.  Let us split the problem of whether strict comparison implies $\Z$-stability of $B$ into the three steps. 
\begin{enumerate}
    \item Injective type II$_1$ von Neumann algebras are McDuff as a consequence of Connes' theorem.\footnote{For II$_1$ factors this was a step along the road to Connes' theorem (this follows from 7$\Rightarrow$2 of \cite[Theorem~5.1]{Co76}, defining $\phi$ in 7 as the composition of conditional expectation and the trace), and it is folklore that it holds generally; a proof from hyperfiniteness to the McDuff property can be found as \cite[Proposition 1.6]{CETW}.  See also \cite{SW23b}, which obtains the equivariant McDuff property, extending results from Ocneanu beyond the factor setting.}
    \item Attempt to lift the McDuff property back from von Neumann algebras to $\M$.  This is an immediate consequence of the local-to-global argument, \emph{provided} one has CPoU. For the uniform tracial completion of a nuclear $C^*$-algebra with no finite dimensional representations, CPoU is equivalent to property $\Gamma$.
    \item Lift an embedding $M_n\to \mathcal M^\infty\cap \M'$ to an order zero map $\phi\colon M_n\to B_\infty\cap B'$. Such lifts always exist by projectivity of order zero maps from matrix algebras (\cite[Theorem~4.9]{Loring97}), and in fact, any $^*$-homomorphism from a separable $C^*$-algebra into $\mathcal M^\infty\cap \M'$ has an order zero lift by \cite[Propositions~4.5 and 4.6]{KR14}.  Matui and Sato's notion of property (SI) -- a kind of small-to-large comparison condition in $B_\infty\cap B'$  obtained from strict comparison and nuclearity of $B$ -- which is designed to ensure that $\phi$ gives rise to a copy of $\Z$ in $B_\infty\cap B'$ and hence to $\Z$-stability of $B$.
\end{enumerate}
Steps \ref{item:intro:overallplan1} and \ref{item:intro:overallplan3} work generally; the challenge is at step \ref{item:intro:overallplan2}.  This abstraction of Matui and Sato's strategy led to \cite[Theorem 5.6]{CETW}, showing that for unital $C^*$-algebras as in the Toms--Winter conjecture, $\Z$-stability is equivalent to the combination of strict comparison and uniform property $\Gamma$ (i.e.\ property $\Gamma$ for the tracial completion).\footnote{See \cite[Theorem~A]{CE2} for a non-unital statement.}

Thus the following question is fundamental; by \cite[Theorem 5.6]{CETW} a positive answer would resolve the Toms--Winter conjecture. The analogous result in the von Neumann setting is the first component of Connes' theorem: injective II$_1$ factors have property $\Gamma$ (which goes through the passage from failure of property $\Gamma$ to a spectral gap condition \cite[Corollary 2.2]{Co76}; see also the two new proofs \cite{Ma17,Ma23}).

\begin{question}\label{Q:semidiscreteGamma}
	Does every amenable type II$_1$ factorial tracially complete $C^*$-algebra satisfy property $\Gamma$?
\end{question} 

When the designated set of traces $X$ is reasonably small, Question \ref{Q:semidiscreteGamma} has a positive answer; indeed, in the unique trace case, this is due to Connes as mentioned above. We give a positive answer when $\partial_eX$ is compact and zero dimensional as Proposition~\ref{prop:zerodimgamma}. More generally, it will be shown in forthcoming work by the third- and fifth-named authors (\cite{EvingtonSchafhauser}) that the same holds for $\partial_eX$ compact and finite-dimensional.  Outside this setting, the tracial completions of simple nuclear $C^*$-algebras whose trace spaces have compact, finite-dimensional extreme boundary have property $\Gamma$ (proved in \cite{KR14,TWW15,Sa12}, though property $\Gamma$ was developed later). It remains mysterious whether one should expect a positive answer in general, or whether strict comparison for a $C^*$-algebra $A$ would imply property $\Gamma$ for its uniform tracial completion -- which, if true, would establish the Toms--Winter conjecture.  If neither of these situations hold, then it is reasonable to ask whether amenable tracially complete $C^*$-algebras with a suitable tracial divisibility property satisfy property $\Gamma$. In other words, if a version of Winter's tracial divisibility holds for an amenable tracially complete $C^*$-algebra, must it have property $\Gamma$? A positive answer to this question would resolve the modified Toms--Winter conjecture mentioned in Footnote~\ref{fn:WinterConjecture}.

We end with a comparison of the hypotheses in the classification of amenable algebras that we have discussed.

\

\begin{center}
\begin{tabular}{ll}\toprule
    Type of algebra & Regularity\\ \midrule
    Finite amenable von Neumann algebra & Automatic\\
    Amenable factorial II$_1$ tracially complete $C^*$-algebra & Open\\
    Simple separable non-elementary nuclear $C^*$-algebra & Not automatic\\ \bottomrule
\end{tabular}
\end{center}

\

With the hypotheses listed in the table, `regularity' for von Neumann algebras simply means the McDuff property; for tracially complete algebras it means either McDuff or property $\Gamma$ (equivalent in this setting).  For $C^*$-algebras as in the table, regularity can be interpreted as $\mathcal Z$-stability, and non-$\Z$-stable examples are known to exist.
The other key hypothesis to $C^*$-classification is the universal coefficient theorem, which is inherently topological, and therefore doesn't have von Neumann algebra or tracially complete counterparts; whether the UCT holds automatically for separable nuclear $C^*$-algebras remains a major challenge.

\section{Traces and Choquet simplices}\label{sec:prelim}

Choquet theory will play an important role throughout this paper, and we begin this preliminary section by recalling some definitions and theorems about Choquet simplices in Section~\ref{sec:simplices}.  The main examples of Choquet simplices in this paper are the spaces $T(A)$ of tracial states on a unital $C^*$-algebra $A$ with its weak$^*$ topology, along with the closed faces of $T(A)$.  We will discuss traces in more depth in Section~\ref{sec:traces-prelim} and recall some results related to their GNS representations and induced von Neumann algebras for later use.

\subsection{Choquet theory}\label{sec:simplices}  

We refer the reader to \cite{Alf71} and \cite{Go86} for general references on Choquet theory.  We collect some basic definitions here along with the results needed in the main body of the paper.

Let $X$ be a compact convex set in a locally convex space and write $\mathrm{Aff}(X)$ for the Banach space of all continuous affine functions $X \rightarrow \mathbb R$ equipped with the supremum norm.  Then $\mathrm{Aff}(X)$ is an Archimedean order unit space in the sense of \cite[Section~II.1]{Alf71} with the pointwise order and the order unit $1_{\mathrm{Aff}(X)}$, the constant function $1$.\footnote{
Briefly, an \emph{Archimedean order unit space} is a triple $V = (V, V_+, 1_V)$ where $V$ is a real vector space, $V_+ \subseteq V$ is a spanning cone, and $1_V \in V_+$ is a distinguished element $1_V \in V_+$, called the \emph{order unit}, such that \[\|v\| \coloneqq \inf \{ r > 0 : -r 1_V \leq v \leq r1_V \}, \qquad  v \in V, \] is a complete norm on $V$. (Note: in the literature, completeness is not always assumed as part of the definition, but for us it is.)}
Conversely, if $V$ is a Archimedean order unit space,
then the state space of $V$,
\begin{equation}
	S(V) \coloneqq \{ f \in V^* : \|f\| = f(1_V) = 1 \},
\end{equation} 
is a weak$^*$-compact convex subset of $V^*$.  By a result of Kadison \cite{Kadison51} (see also \cite[Chapter~7]{Go86}), the natural maps
\begin{equation}
	X \overset\cong\longrightarrow S(\mathrm{Aff}(X)) \qquad \text{and} \qquad V \overset\cong\longrightarrow \mathrm{Aff}(S(V))
\end{equation}
are isomorphisms. This gives an anti-equivalence between the categories of compact convex sets and Archimedean order unit spaces.  We refer to this result as \emph{Kadison duality}.  One important consequence is that the weak topology on ${\rm Aff}(X)$ is the topology of pointwise convergence. This is standard and used, for example, in \cite{Oz13} and \cite{GGKNV}.

\begin{proposition}\label{prop:pointwise-to-uniform}
	Let $X$ be a compact convex set.
 \begin{enumerate}
     \item All states on ${\rm Aff}(X)$ are given by point evaluations.\label{prop:pointwise-to-uniform1}
     \item If $(f_\lambda) \subseteq {\rm Aff}(X)$ is a net and $f \in {\rm Aff}(X)$, then $f_\lambda \rightarrow f$ weakly if and only if $f_\lambda(\tau) \rightarrow f(\tau)$ for all $\tau \in X$.\label{prop:pointwise-to-uniform2}
 \end{enumerate}
\end{proposition}

\begin{proof}
\ref{prop:pointwise-to-uniform1}.  This is a consequence of Kadison's duality.

\ref{prop:pointwise-to-uniform2}.  Note that bounded linear functionals on $\mathrm{Aff}(X)$ extend to bounded linear functionals on $C(X)$ by the Hahn--Banach theorem.  Since the dual of $C(X)$ is spanned by states on $C(X)$ and all states on $C(X)$ restrict to states on ${\rm Aff}(X)$, it follows that $S(\mathrm{Aff}(X))$ spans $\mathrm{Aff}(X)^*$.  The result now follows from \ref{prop:pointwise-to-uniform1}.
\end{proof}

For a compact convex set $X$, write $\mathrm{Prob}(X)$ for the set of Radon probability measures on $X$. Given $\mu\in\mathrm{Prob}(X)$, the \emph{barycentre} of $\mu$ is the unique point $x \in X$ such that
\begin{equation}
	\int_X f \, d\mu = f(x), \qquad f \in \mathrm{Aff}(X).
\end{equation}
Let $\partial_e X \subseteq X$ denote the set of extreme points of $X$.  We say a complex Radon measure $\mu$ is \emph{supported on} $\partial_e X$ if for every Baire measurable set\footnote{The Baire $\sigma$-algebra on a compact Hausdorff space $X$ is the smallest $\sigma$-algebra on $X$ such that every continuous function $X \rightarrow \mathbb C$ is measurable.  When $X$ is metrisable, this coincides with the Borel $\sigma$-algebra, but it is smaller in general.} \mbox{$E \subseteq X$} with $\partial_e X \subseteq E$, we have $|\mu|(X\setminus E) = 0$.\footnote{When $X$ is metrisable, the set $\partial_e X$ is a $G_\delta$-set, so $\mu \in \mathrm{Prob}(X)$ is supported on $\partial_e X$ if and only if $\mu(\partial_e X) = 1$.  In general, $\partial_e X$ is not Baire measurable (and hence not Borel measurable) -- see \cite[Section VII]{Bishop-deLeeuw} (and \cite[Lemma 4.1]{Bishop-deLeeuw} to identify $\partial_e X$ with $M(B)$).}  By the Choquet--Bishop--de Leeuw theorem (see \cite[Theorem~I.4.8 and Corollary~1.4.12]{Alf71}), for every $x \in X$, there exists $\mu \in \mathrm{Prob}(X)$ supported on $\partial_e X$ with barycentre $x$.  A compact convex set $X$ is called a \emph{Choquet simplex} if this measure $\mu$ is unique for every $x \in X$.\footnote{The Choquet--Bishop--de Leeuw theorem and the definition of Choquet simplex are often stated in terms of Baire measures in place of Radon measures.  The definitions are equivalent as every complex Baire measure on a compact Hausdorff space admits a unique extension to a Radon measure; the case of (finite) signed measures is \cite[Corollary~7.3.4]{Bogachev07}, and the complex version follows easily.}

In the finite dimensional setting, there is a unique Choquet simplex of every dimension.  Specifically, if $n \geq 0$ is an integer, then the $n$-dimensional Choquet simplex is given as the convex hull of an orthonormal set of $n + 1$ vectors in a Hilbert space.  

A Choquet simplex $X$ is called a \emph{Bauer simplex} if $\partial_e X$ is compact.  In this case, we view $\mathrm{Prob}(\partial_e X)$ as a compact convex set in the locally convex space $C(\partial_e X)^*$.  Note that there is an affine homeomorphism
\begin{equation}
	\mathrm{Prob}(\partial_e X) \overset\cong\longrightarrow X
\end{equation}
given by sending a measure to the barycentre of its canonical extension to $X$ (given by declaring that $X \setminus \partial_e X$ has measure zero).  Indeed, this map is bijective by the definition of a Choquet simplex, and it is easily seen to be a homeomorphism.

When $X$ is a metrisable Choquet simplex, a result of Lazar and Lind\-enstrauss in \cite[Corollary of Theorem~5.2]{Lazar-Lindenstrauss71} (see also \cite[Theorem~11.6]{Go86}) shows $X$ can be written as a projective limit of finite dimensional Choquet simplices.  As noted in \cite[Lemma~2.8]{Elliott-Niu15}, this implies $X$ satisfies the finite dimensional approximation property in Theorem \ref{thm:simplex-is-nuclear} below.  As we set out in Appendix \ref{sec:sep-choquet}, this result holds generally, i.e.\ without a metrisability assumption on $X$. Naturally, for results which only concern the separable situation ($C^*$-algebras which are separable in norm and tracially complete $C^*$-algebras which are separable in their uniform 2-norm), the metrisable version of Theorem \ref{thm:simplex-is-nuclear} will suffice.

\begin{restatable}{theorem}{simplexisnuclear}
	\label{thm:simplex-is-nuclear}
	If $X$ is a Choquet simplex, then there are nets of finite dimensional Choquet simplices $Z_\lambda$ and continuous affine maps
	\begin{equation}
		X \overset{\beta_\lambda}{\longrightarrow} Z_\lambda \overset{\alpha_\lambda}{\longrightarrow} X
	\end{equation}
	such that $\lim_\lambda \|f \circ \alpha_\lambda \circ \beta_\lambda - f \| = 0$ for all $f \in \mathrm{Aff}(X)$.
\end{restatable}

We end this subsection with some results about closed faces in a Choquet simplex.  A \emph{face} in a compact convex set $X$ is a convex set $F \subseteq X$ with the following property: 
for all $x_1, x_2 \in X$, if a non-trivial convex combination of $x_1$ and $x_2$ is in $F$, then both $x_1 $ and $x_2$ are in $F$.

The next two results can be viewed in analogy with the Hahn--Banach theorem.  The first is an extension result, and the second is a separation result.

\begin{theorem}[{\cite[Theorem~II.5.19]{Alf71}}]\label{thm:extending-affine}
	Let $F$ be a closed face of a Choquet simplex $X$.
	For every $f \in \mathrm{Aff}(F)$, there exists $\hat{f} \in \mathrm{Aff}(X)$ with $\hat{f}|_F = f$ and $\|\hat{f}\| = \|f\|$. 
\end{theorem}

The following is an easy consequence of \cite[Corollary~II.5.20]{Alf71}.  The first part follows the proof of \cite[Proposition~II.5.16]{Alf71}.

\begin{theorem}\label{thm:exposed}
	Let $F$ be a closed face in a Choquet simplex $X$.
	\begin{enumerate}
		\item\label{item:exposed1} If $S \subseteq X$ is an $F_\sigma$-set with $S \cap F = \emptyset$, then there is a continuous affine function $f \colon X \rightarrow [0, 1]$ such that $f(x) = 0$ for all $x \in F$ and $f(x) > 0$ for all $x \in S$.
		\item\label{item:exposed2} If $x \in X$ and $f(x) = 0$ for all $f \in \mathrm{Aff}(X)_+$ with $f|_F = 0$, then $x \in F$.
	\end{enumerate}
\end{theorem}

\begin{proof}
	By \cite[Corollary~II.5.20]{Alf71}, closed faces in Choquet simplices are relatively exposed, which means that \ref{item:exposed1} holds when $S$ is a point.  From here, \ref{item:exposed2} is immediate.
	
	In \ref{item:exposed1}, suppose first that $S$ is closed (and hence compact as $X$ is compact).  For $s \in S$, let $f_s \colon X \rightarrow [0, 1]$ be a continuous affine function such that $f_s(x) = 0$ for $x \in F$ and $f_s(s) > 0$.  As $S$ is compact and each $f_s$ is continuous, there are $s_1, \ldots, s_n \in S$ such that
	\begin{equation}
		\inf_{s \in S} \max_{1 \leq i \leq n} f_{s_i}(s) > 0.
	\end{equation}
	Set $f \coloneqq \frac1n \sum_{i = 1}^n f_{s_i}$.  To see the general case, consider an $F_\sigma$-set $S$ and write $S$ as the union of closed sets $(S_n)_{n=1}^\infty$.  For each $n \geq 1$, let $f_n \colon X \rightarrow [0, 1]$ be a continuous affine function such that $f_n(x) = 0$ for all $x\in F$ and $f_n(x) > 0$ for all $x \in S_n$.  Then set $f \coloneqq \sum_{n=1}^\infty 2^{-n} f_n$.
\end{proof}

The final result of this subsection concerns detecting closed faces. In general, if $X$ is a Choquet simplex and $S \subseteq \partial_e X$, then the convex hull of $S$, written  $\mathrm{co}(S)$, is a face in $X$.  However, the closed convex hull of $S$, written $\overline{\mathrm{co}}(S)$, need not be a face (see \cite[Theorem~1]{Alf64}).  The following result of Roy gives a replacement.

\begin{theorem}[{cf.\ \cite[Proposition~4.4]{Roy75}}]\label{thm:choquet-faces}
	If $X$ is a Choquet simplex and $F \subseteq X$ is a closed convex set, then $F$ is a face in $X$ if and only if $\partial_e F \subseteq \partial_e X$.
\end{theorem}

\begin{proof}
	The forward direction is clear.  Conversely, if $\partial_e F \subseteq \partial_e X$, then $F_0 \coloneqq \mathrm{co}(\partial_e F)$ is a face in $X$, and by the Krein--Milman theorem, $F$ is the closure of $F_0$.  Using again that $\partial_e F \subseteq \partial_e X$, we have that $F$ is a face in $X$ by \cite[Proposition~4.4]{Roy75}.
\end{proof}

\subsection{Traces and the GNS construction}\label{sec:traces-prelim}

By a \emph{trace} on a $C^*$-algebra, we will always mean a tracial state.  For a $C^*$-algebra $A$, let $T(A)$ denote the set of traces on $A$ equipped with the weak$^*$ topology.  Then $T(A)$ is convex.  We will typically be interested in $C^*$-algebras where $T(A)$ is compact -- this is the case for unital $C^*$-algebras, for example.  The following result is folklore.

\begin{theorem}\label{thm:traces-choquet}
	If $A$ is a $C^*$-algebra, then every compact face in $T(A)$ is a Choquet simplex.
\end{theorem}

\begin{proof}
When $A$ is unital, $T(A)$ is a Choquet simplex by  \cite[Theorem~3.1.18]{Sak98}.  By \cite[Proposition~10.9]{Go86}, every closed face in a Choquet simplex is a Choquet simplex.

When $A$ is non-unital with unitisation $A^\dagger$, traces on $A$ can be extended to traces on $A^\dagger$, and $\tau\in T(A^\dagger)$ induces a trace on $A$ if and only if $\|\tau|_A\|=1$.  In this way $T(A)$ can be viewed as a face in $T(A^\dagger)$, and the result follows from the unital case.
\end{proof}

We turn to the structure of continuous affine functions $T(A) \rightarrow \mathbb R$, which will be used repeatedly in this paper. The following result is well-known and is often attributed to Cuntz and Pedersen in \cite{CP79}, but the result does not explicitly appear in their paper.  A proof of the first part of the result can be found in \cite[Lemma~6.2]{KR14} or \cite[Proposition~2.1]{CGSTW}, for example.

\begin{proposition}[cf.\ {\cite[Proposition 2.7]{CP79}}]
	\label{prop:CP}
	Let $A$ be a unital $C^*$-algebra and let $f\colon T(A) \to \mathbb R$ be a continuous affine function.
	Then for any $\epsilon>0$, there is a self-adjoint element $a \in A$ such that
	\begin{equation} \|a\| \leq \|f\|_\infty + \epsilon \quad \text{and} \quad \tau(a) = f(\tau), \quad \tau \in T(A). \end{equation}
	Moreover, if $f$ is strictly positive, we may assume $a \geq 0$.
\end{proposition}

\begin{proof}
	We only prove the last sentence.  Assume that $f$ is strictly positive.  Since $T(A)$ is compact, there is a $\delta \in (0, 2\epsilon)$ such that $f(\tau) > \delta$ for all $\tau \in T(A)$.   Apply the first part of the proof to $f - \frac12 (\|f\|_\infty + \delta)$, which has norm at most $\frac12(\|f\|_\infty - \delta)$, to obtain a self-adjoint $b \in A$ such that
	\begin{equation}
		\|b\| \leq \frac12 \|f\|_\infty\qquad \text{and} \qquad  \tau(b) = f(\tau) - \frac12 (\|f\|_\infty + \delta)
	\end{equation}
	for all $\tau \in T(A)$.  Then set $a \coloneqq b + \frac12 (\|f\|_\infty+ \delta) 1_A$.
\end{proof}

The essential starting point for all of our classification results is the classification of projections in finite von Neumann algebras by traces (parts \ref{prop:TracesFiniteVNAS.3} and \ref{prop:TracesFiniteVNAS.4} in the proposition below), which goes back to Murray and von Neumann.  We review these results below; while these are well-known to experts, they are most often stated for factors -- particularly the existence of projections realising arbitrary continuous affine functions -- and in this paper we need to work with arbitrary finite von Neumann algebras.

\begin{proposition}\label{prop:TracesFiniteVNAS}
Let $\mathcal M$ be a finite von Neumann algebra.
\begin{enumerate}
\item Every trace on $\mathcal M$ factors uniquely through the centre-valued trace.\label{prop:TracesFiniteVNAS.1}\footnote{See \cite[Theorem V.2.6]{Tak79} or \cite[Theorem 8.2.8]{KadisonRingrose.2} for properties of the centre-valued trace.}
    \item The normal traces on $\mathcal M$ are dense in the traces on $\mathcal M$.\label{prop:TracesFiniteVNAS.2}
    \item Projections $p,q\in\mathcal M$ are unitarily equivalent if and only if $\tau(p)=\tau(q)$ for all $\tau\in T(\mathcal M)$; $p$ is Murray--von Neumann subequivalent to $q$ if and only if $\tau(p)\leq\tau(q)$ for all $\tau\in\mathcal T(\mathcal M)$.\label{prop:TracesFiniteVNAS.3}
    \item If $\mathcal M$ is type {\rm II}$_1$ and $f \colon T(\M) \rightarrow [0, 1]$ is a continuous affine function, then there is a projection $p \in \M$ such that $\tau(p) = f(\tau)$ for all $\tau \in T(\M)$.\label{prop:TracesFiniteVNAS.4}
\end{enumerate}
\end{proposition}

\begin{proof}
\ref{prop:TracesFiniteVNAS.1}.  This is a consequence of Dixmier's approximation theorem (see \cite[Proposition 8.3.10]{KadisonRingrose.2} or \cite[Theorem~III.2.5.7(iv)]{Bl06}).

\ref{prop:TracesFiniteVNAS.2}.  Since the centre-valued trace on $\M$ is normal (see \cite[Theorem~V.2.34]{Ta79}), the result follows from \ref{prop:TracesFiniteVNAS.1} and the density of the normal states in the states on the centre of $\mathcal M$.

\ref{prop:TracesFiniteVNAS.3}. The second part of the statement follows from Murray and von Neumann's comparison theorem for projections in von Neumann algebras, the properties of the centre-valued trace, and \ref{prop:TracesFiniteVNAS.1}; see \cite[Corollary V.2.8]{Tak79}, for example.  The first part follows from applying the second part both to $p$ and $q$ and to $1_{\mathcal M}-p$ and $1_{\mathcal M}-q$.

\ref{prop:TracesFiniteVNAS.4}. This proceeds in a very similar way to the proof that II$_1$ factors contain projections of arbitrary traces in $[0,1]$ (see \cite[Proposition 4.1.6]{PA}). As we have been unable to find a reference, we give the details for completeness. 

Let $P$ be the set of projections $p \in \mathcal M$ such that $\tau(p) \leq f(\tau)$ for all $\tau \in T(\mathcal M)$.  Then $P \neq \emptyset$ as $0 \in P$.  Also, if $(p_\lambda)$ is an increasing chain of projections in $\M$ with supremum $p \in \M$, then for any normal trace $\tau\in\mathcal M$, we have
	\begin{equation}
		\tau(p) = \lim_\lambda \tau(p_\lambda) \leq f(\tau).
	\end{equation}
 As $f$ is weak$^*$-continuous, part \ref{prop:TracesFiniteVNAS.2} gives $p \in P$.  By Zorn's Lemma, there is a maximal $p \in P$.  We will show $f(\tau) = \tau(p)$ for all $\tau \in T(\M)$.
	
	Suppose this is not the case.  By the Krein--Milman theorem, there exists $\tau_0 \in \partial_e T(\M)$ with $\tau_0(p) < f(\tau_0)$.
	Set $\epsilon \coloneqq \frac{1}{2}(f(\tau_0) - \tau_0(p))$. By \ref{prop:TracesFiniteVNAS.1}, there is an isomorphism \begin{equation}\label{eq:centre-vNa-spectrum}
		Z(\M) \overset\cong\longrightarrow C(\partial_e T(\M)) \colon a \mapsto \big(\tau \mapsto \tau(a)\big).
	\end{equation}
	In particular, $\partial_e T(\M)$ is totally disconnected.  Let $U \subseteq \partial_e T(\M)$ be a clopen neighbourhood of $\tau_0$ so that 
	\begin{equation}
		\tau(p) < f(\tau) - \epsilon, \quad \tau \in U,
	\end{equation}
    and let $z \in Z(\M)$ be the projection corresponding to the characteristic function of $U$ under the isomorphism in \eqref{eq:centre-vNa-spectrum}. 
	Since $\tau(1_\M-p) \geq \epsilon$ for all $\tau \in U$, it follows that 
	\begin{equation}
		\tau(z(1-p)) \geq \epsilon  
	\end{equation}
	for all $\tau \in \partial_e T(z\M)$, and hence for all $\tau \in T(z\M)$.

	Fix $d \geq 1$ with $1/d < \epsilon$. Since $\M$ is type II$_1$, so is the corner $z \M$. By the proof of \cite[Theorem~V.1.35]{Tak79} (which handles the case $d = 2$), there is a unital embedding $M_d \rightarrow z \M$, and in particular, there is a projection $e \in z\M$ so that $\tau(e) = 1/d$ for all $\tau \in T(z\M)$. By comparison of projections in the von Neumann algebra $z\M$, it follows that $e$ is Murray--von Neumann subequivalent to $z(1-p)$. 
	Hence after conjugating $e$ by a unitary in $z\M$, we may assume $e$ is orthogonal to $zp$. Since $e$ is also orthogonal to the central projection $z^\perp$, it follows that $e$ is orthogonal to $p$.
	
	Set $p' \coloneqq p + e$. Since $e$ is orthogonal to $p$, it follows that $p'$ is a projection.  For $\tau \in \partial_e X \setminus U$, we have $\tau(p') = \tau(p) \leq f(\tau)$, and we have 
    \begin{equation}
        \tau(p') = \tau(p) + \tfrac{1}{d} < f(\tau), \quad \tau \in U.
    \end{equation}  By the Krein--Milman theorem, it follows that $\tau(p') \leq f(\tau)$ for all $\tau \in T(\M)$, and so $p' \in P$.  This contradicts the maximality of $p$. 
\end{proof}

Given a $C^*$-algebra $A$ and $\tau \in T(A)$, let $\pi_\tau\colon A \to \mathcal B(\mathcal H_\tau)$ be the GNS representation associated to $\tau$. The induced von Neumann algebra $\pi_\tau(A)''$ has a faithful normal trace induced by $\tau$ and hence is finite. Therefore, $\pi_\tau(A)''$ is the direct sum of a finite von Neumann algebra of type I  and a von Neumann algebra of type II$_1$. Under the assumption that $A$ has no finite dimensional quotients, $\pi_\tau(A)''$ is of type II$_1$ -- this will be a common hypothesis throughout the paper.

The following well-known characterisation of the extreme points of $T(A)$ will be used frequently in the paper.  A version of the result for not necessarily bounded tracial weights is given in \cite[Theorem~6.7.3]{Di77} -- the result below is a special case.\footnote{Note that an extreme point of $T(A)$ is precisely a character of norm 1 in the sense of \cite[Definition~6.7.1]{Di77}.}

\begin{proposition}[{cf.\ \cite[Theorem 6.7.3]{Di77}}] \label{prop:factorial-extreme-point}
	If $A$ is a unital $C^*$-algebra and $\tau \in T(A)$, then $\tau$ is an extreme point of $T(A)$ if and only if $\pi_\tau(A)''$ is a factor.
\end{proposition}

The following lemma relates traces on the von Neumann algebra $\pi_\tau(A)''$ to traces on $A$.  For a $C^*$-algebra $A$ and a set of traces $X \subseteq T(A)$, we define
\begin{equation}\label{eq:GNS-X}
	\pi_X \coloneqq \bigoplus_{\tau \in X} \pi_\tau \colon A \rightarrow \mathcal{B}\Big( \bigoplus_{\tau \in X} \mathcal H_\tau \Big)
\end{equation}
to be the direct sum of the GNS representations associated with the traces in $X$.

\begin{lemma}\label{lem:GNSTraceFace}
	Suppose $A$ is a unital $C^*$-algebra and $X \subseteq T(A)$.  If $\tau \in T(\pi_X(A)'')$, then $\tau \circ \pi_X$ is in the closed face generated by $X$.
\end{lemma}

\begin{proof}
Let $F \subseteq T(A)$ be the weak$^*$-closed face generated by $X$. By the density of the normal traces on $\pi_X(A)''$ in the traces on $\pi_X(A)''$ (Proposition~\ref{prop:TracesFiniteVNAS}\ref{prop:TracesFiniteVNAS.2}), we may assume $\tau$ is normal.
	  By Theorem~\ref{thm:exposed}\ref{item:exposed2}, it suffices to show that if $f \colon T(A) \rightarrow [0, 1]$ is a continuous affine function vanishing on $X$, then $f(\tau \circ \pi_X) = 0$.  
	  Fix such an affine function $f$.  
	  By applying Proposition~\ref{prop:CP} to the strictly positive affine functions $f + \tfrac{1}{n}$ for each $n \in \mathbb{N}$, we get a bounded sequence $(a_n)_{n=1}^\infty \subseteq A_+$ such that
	\begin{equation}\label{eq:approx-point-eval2}
		\lim_{n\to\infty} \sup_{\sigma \in T(A)} | \sigma(a_n)-f(\sigma)| = 0.
	\end{equation}
	It follows that $\sigma(a_n) \rightarrow 0$ for all $\sigma \in X$.  
	Since $a_n \geq 0$ and the sequence $(a_n)_{n=1}^\infty$ is bounded, this implies $\pi_X(a_n) \rightarrow 0$ strongly in $\pi_X(A)''$.\footnote{Indeed, the topology on $\pi_X(A)''$ induced by the seminorms 
		\[\|a\|_{2, \sigma} \coloneqq \sigma(a^*a)^{1/2}, \qquad  a \in \pi_X(A)'',\ \sigma \in X, \] agrees with the strong topology on bounded sets -- see \cite[Proposition~III.2.2.19]{Bl06}, for example.}
	Since $\tau$ is normal, we have $\tau(\pi_X(a_n)) \rightarrow 0$, and hence using \eqref{eq:approx-point-eval2} with $\sigma=\tau\circ\pi_X$ shows that $f(\tau\circ \pi_X) = 0$.
\end{proof}

We shall need a non-commutative version of the Radon--Nikodym theorem for positive normal tracial functionals on von Neumann algebras. (More general non-commutative Radon-Nikodym theorems can be found in \cite{Dye52, Sakai65}.)  To set the stage, for positive normal functionals $\phi$ and $\psi$ on a von Neumann algebra $\mathcal M$, we write $\phi \leq \psi$ if $\phi(a) \leq \psi(a)$ for all $a \in \mathcal M_+$.  
Recall also that for a normal positive linear functional $\phi$ on a von Neumann algebra $\mathcal M$, the minimal projection $s(\phi) \in \mathcal M$ satisfying $\phi(s(\phi)as(\phi)) = \phi(a)$ for all $a \in \mathcal M$ is the \emph{support} of $\phi$.\footnote{Equivalently, $1-s(\phi)$ is the maximal projection $p$ with $\phi(p)=0.$}
Further, when $\phi$ is tracial, we have $s(\phi) \in Z(\mathcal M)$.\footnote{When $\phi$ is tracial, $\phi(u s(\phi) u^* a u s(\phi) u^*) = \phi(u^*au) = \phi(a)$
for any unitary $u \in \mathcal M$ and $a \in \mathcal M$. Hence, $u s(\phi)u^* \geq s(\phi)$. Replacing $u$ with $u^*$ shows the reverse inequality. Therefore, $s(\phi) \in \mathcal M \cap U(\mathcal M)' = Z(\mathcal M)$.}
For $a \in \mathcal M_+$, we write $s(a) \coloneqq \lim_{n \rightarrow \infty} a^{1/n}$ for the support projection of $a$, with the limit taken in the strong operator topology.  Note that $s(a) \in Z(\mathcal M)$ if $a \in Z(\mathcal M)$. 

\begin{theorem}[Radon--Nikodym Theorem]\label{thm:radon-nikodym}
    Let $\mathcal M$ be a von Neumann algebra and let $\sigma$ and $\tau$ be positive normal tracial functionals on $\mathcal M$.  If there is a constant  $C \geq 0$ such that $\sigma \leq C \tau$, there there is a unique $\frac{{\rm d} \sigma}{{\rm d} \tau} \in Z(\mathcal M)_+$ whose support is a subprojection of the support of $\tau$ and satisfies $\sigma(a) = \tau\big(\frac{{\rm d} \sigma}{{\rm d} \tau}a\big)$ for all $a \in \mathcal M$.
\end{theorem}

\begin{proof}
    After replacing $\mathcal M$ with $s(\tau)\mathcal M$ and $\sigma$ and $\tau$ with their restrictions to $s(\tau) \mathcal M$, we can assume $\tau$ is faithful.  Further, after replacing $\tau$ with $C\tau$, we may assume $\sigma \leq \tau$.  The result then follows from \cite[Theorem 3 in Part I, Chapter 6]{Di69EngTrans}. 
\end{proof}

\section{Tracially complete \texorpdfstring{$C^*$}{C*}-algebras}\label{sec:tracially-complete}

In this section, we introduce tracially complete $C^*$-algebras, which are the central objects of study in this paper.
We discuss the motivating examples, including the tracial completion of a $C^*$-algebra, and prove some basic approximation and extension lemmas.
We then construct inductive limits in the category of tracially complete $C^*$-algebras. 
In the final subsection, we examine the relationship between Ozawa's $W^*$-bundles (\cite{Oz13}) and tracially complete $C^*$-algebras.

\subsection{Definitions and basic properties}\label{sec:tracially-complete-defs}
The definition of tracially complete $C^*$-algebras is based on uniform 2-norms.

\begin{definition}[{cf.\ \cite{KR14, Sa12, TWW15, Oz13}, for example}]\label{def:trace-norm}
For a $C^*$-algebra $A$ and a set $X \subseteq T(A)$, define the \emph{uniform $2$-seminorm} on $A$ by
\begin{equation}\label{eq:UniformTraceNorm}
	\|a\|_{2,X} \coloneqq \sup_{\tau \in X} \sqrt{\tau(a^*a)}, \qquad a \in A.
\end{equation}
\end{definition}

The uniform 2-seminorm is indeed a seminorm, and the following inequality is easily verified: 
\begin{equation}                           \label{eq:SpecialHolderIneq}
	\|ab\|_{2,X} \leq \min\{\|a\|\|b\|_{2,X}, \|a\|_{2,X}\|b\|\}, \qquad a,b \in A.
\end{equation}
It is easily seen that $\|\cdot\|_{2,X}$ is a norm if and only if for every non-zero $a \in A_+$ there exists $\tau \in X$ with $\tau(a) > 0$. In this case, $X$ is said to be a \emph{faithful set of traces}, and we refer to $\|\cdot\|_{2,X}$ as the \emph{uniform $2$-norm} with respect to $X$.

The following properties of the uniform 2-norm are standard in the unique trace setting (i.e.\ when $X$ is a singleton).
\begin{proposition}\label{prop:ball-closed}
	Suppose $A$ is a $C^*$-algebra and $X \subseteq T(A)$ is a faithful set of traces.
	\begin{enumerate}
		\item\label{item:lsc} $\|\cdot\|$ is lower semi-continuous with respect to $\|\cdot\|_{2, X}$;
		\item\label{item:ball-closed} The unit ball of $A$ is $\|\cdot\|_{2, X}$-closed;
		\item\label{item:cone-closed} $A_+$ is $\|\cdot\|_{2, X}$-closed.
	\end{enumerate}
\end{proposition}
\begin{proof}
	Fix $(a_n)_{n=1}^\infty \subseteq A$ and $a \in A$ with $\|a_n - a\|_{2, X} \rightarrow 0$. 
	For $\tau \in X$, let $\pi_\tau \colon A \rightarrow \mathcal B(\mathcal H_\tau)$ denote the GNS representation of $\tau$ with associated cyclic vector $\xi_\tau \in \mathcal H_\tau$.  If $b \in A$, then 
	\begin{equation}\label{eqn:ball-closed-1}
\begin{split}
		\|(\pi_\tau(a_n) - \pi_\tau(a))\pi_\tau(b)\xi_\tau\| &= \|(a_n - a)b\|_{2,\tau} \\
		&\leq \|a_n-a\|_{2,X}\|b\|.
	\end{split}
\end{equation}	
	So $\pi_\tau(a_n)\eta$ converges to $\pi(a)\eta$ for all $\eta \in \pi_\tau(A)\xi_\tau$.  As $\pi_\tau(A)\xi_\tau$ is dense in $\mathcal H_\tau$,
	\begin{equation}
		\|\pi_\tau(a)\| \leq \liminf_{n \rightarrow \infty} \|\pi_\tau(a_n)\| \leq \liminf_{n \rightarrow \infty} \|a_n\|.
	\end{equation}
	 The product of the $\pi_\tau$ over $\tau \in X$ is faithful, and hence isometric, since $\|\cdot\|_{2, X}$ is a norm on $A$.  Therefore,
	\begin{equation}
		\|a\| = \sup_{\tau \in X} \|\pi_\tau(a)\| \leq \liminf_{n \rightarrow \infty} \|a_n\|.
	\end{equation}
This proves \ref{item:lsc}, and \ref{item:ball-closed} is an easy consequence of \ref{item:lsc}.

For \ref{item:cone-closed}, suppose now that $a_n \geq 0$ for all $n \in \N$.
Fix $\tau \in X$.
By \eqref{eqn:ball-closed-1}, we have 
	\begin{equation}
		\langle \pi_\tau(a)\pi_\tau(b)\xi_\tau, \pi_\tau(b)\xi_\tau \rangle = \lim_{n\to\infty} \langle \pi_\tau(a_n)\pi_\tau(b)\xi_\tau, \pi_\tau(b)\xi_\tau \rangle \geq 0,
	\end{equation}
	since $\pi_\tau(a_n) \geq 0$ for all $n \in \N$.
	The density of $\pi_\tau(A)\xi_\tau$ in $\mathcal H_\tau$ implies $\langle \pi_\tau(a)\eta,\eta \rangle \geq 0$ for all $\eta \in \mathcal H_\tau$, so $\pi_\tau(a) \geq 0$.
	As $X$ is a faithful set of traces, the product of the $\pi_\tau$ over $\tau \in X$ is faithful. 
	Since $\pi_\tau(a) \geq 0$ for all $\tau \in X$, it follows that $a \geq 0$.	
\end{proof}

We will restrict to the case when $X \subseteq T(A)$ is weak$^*$-compact since the weak$^*$-closure of $X$ defines the same uniform 2-norm as $X$.  The following observation will allow us to further restrict to the case when $X$ is convex.  For a set $X \subseteq T(A)$, let $\overline{\mathrm{co}}(X)$ denote the closed convex hull of $X$ in $A^*$.

\begin{proposition}\label{prop:co-2-norm}
	If $A$ is a $C^*$-algebra and $X \subseteq T(A)$ is compact, then $\overline{\mathrm{co}}(X) \subseteq T(A)$ and $\|a\|_{2, X} = \|a\|_{2, \overline{\mathrm{co}}(X)}$ for all $a \in A$. 
\end{proposition}

\begin{proof}

Every $\tau \in \overline{\mathrm{co}}(X)$ is a positive tracial functional with $\|\tau\| \leq 1$. To see the first claim, it remains to show that $\|\tau\| \geq 1$.  
	
	Let $(e_\lambda) \subseteq A$ be an approximate unit for $A$. Then $(\tau(e_\lambda))$ increases to 1 for all $\tau \in X$.  By Dini's Theorem,\footnote{Dini's Theorem is often stated for increasing sequences of continuous functions, but the standard proof is equally valid for increasing nets of continuous functions.} this convergence is uniform over $\tau$.  Fix $\epsilon > 0$ and let $\lambda$ be such that $\tau(e_\lambda) > 1 - \epsilon$ for all $\tau \in X$.  Then $\tau(e_\lambda) \geq 1 - \epsilon$ for all $\tau \in \overline{\mathrm{co}}(X)$.  Therefore, each $\tau \in \overline{\mathrm{co}}(X)$ has norm 1.
	
	Let $a \in A$. The set of all traces $\tau \in T(A)$ with $\tau(a^*a) \leq \|a\|_{2,X}^2$ is weak$^*$-closed, convex, and contains $X$. Therefore, $\|a\|_{2, \overline{\mathrm{co}}(X)} \leq \|a\|_{2, X}$. The reverse inequality is trivial as $X \subseteq \overline{\mathrm{co}}(X)$. 
\end{proof}
 
A tracial von Neumann algebra (i.e.\ a von Neumann algebra with a distinguished faithful normal trace) can be characterised abstractly as a pair $(\mathcal M, \tau)$ where $\mathcal M$ is a $C^*$-algebra and $\tau$ is a faithful trace on $\M$ such that the unit ball of $\mathcal M$ is $\|\cdot\|_{2, \tau}$-complete (this follows from \cite[Lemma A.3.3]{Si08}, for example).  A morphism from a tracial von Neumann algebra $(\mathcal M, \tau_\mathcal M)$ to another tracial von Neumann algebra $(\mathcal N, \tau_\mathcal N)$ is a (necessarily normal) $^*$-homomorphism $\phi \colon \mathcal M \rightarrow \mathcal N$ such that $\tau_\mathcal N \circ \phi = \tau_\mathcal M$.  Our definition of tracially complete $C^*$-algebras and their morphisms are modelled on these definitions.

\begin{definition}\label{def:TC}
	A \emph{tracially complete $C^*$-algebra} is a pair $(\M,X)$ where $\M$ is a $C^*$-algebra and $X \subseteq T(\M)$ is a compact convex set such that
	\begin{enumerate}
		\item $X$ is a faithful set of traces on $\M$, and
		\item the unit ball of $\M$ is $\|\cdot\|_{2,X}$-complete.
	\end{enumerate}
	Given two tracially complete $C^*$-algebras $(\M,X)$ and $(\mathcal N,Y)$, a \emph{morphism} $\phi\colon (\M,X) \rightarrow (\mathcal N,Y)$ between tracially complete $C^*$-algebras is a  $^*$-ho\-mo\-mor\-phism $\phi\colon \M \to \mathcal N$ such that $\phi^*(Y)\subseteq X$; i.e.\ $\tau\circ\phi\in X$ whenever $\tau\in Y$.
\end{definition}

Strictly speaking, $X$ can be the empty set, which forces $\M$ to be the zero $C^*$-algebra (which is unital).  Needless to say, this degenerate case is not of interest. Implicitly we always imagine $X$ to be non-empty.

A major source of examples of tracially complete $C^*$-algebras is obtained by completing a $C^*$-algebra with respect to a compact convex set of traces. 
In a special case, if $A$ is a $C^*$-algebra and $\tau$ is a trace on $A$, then the GNS completion $(\pi_\tau(A)'', \tau)$ is a tracial von Neumann algebra which coincides with the tracial completion of $A$ with respect to $\{\tau\}$.
Tracial completions will be discussed further in Section~\ref{sec:TraceNormCompletion} below.

Another significant source of examples are $W^*$-bundles, which were introduced by Ozawa in \cite{Oz13} and analysed further in \cite{Ev16,Ev18}.

\begin{definition}[{\cite[Section 5]{Oz13}}]
	\label{defn:W*bundle}
	Let $K$ be a compact Hausdorff space.\footnote{In~\cite{Oz13}, it is assumed $K$ is metrisable, but this is not needed here.  The non-metrisable case was also considered in~\cite{BBSTWW} where ultraproducts of $W^*$-bundles were developed.}  A \emph{$W^*$-bundle} over $K$ is a unital $C^*$-algebra $\M$ together with a unital embedding $C(K) \subseteq Z(\M)$ and a faithful conditional expectation $E\colon\M \to C(K)$ such that
	\begin{enumerate}
		\item $E$ is tracial in the sense that $E(ab) = E(ba)$ for all $a,b \in \M$, and
		\item the unit ball of $\M$ is complete with respect to the norm $\|\cdot\|_{2,E}$, defined by $\|a\|_{2,E}\coloneqq\|E(a^*a)\|^{1/2}$ for all $a \in \mathcal M$.
	\end{enumerate}
	The \emph{fibre} of $\M$ at a point $x \in K$ is $\pi_{\tau_x}(\M)$, where $\tau_x=\mathrm{ev}_x \circ E$; this is equal to $\pi_{\tau_x}(\M)''$ by~\cite[Theorem 11]{Oz13}.
\end{definition}

Every $W^*$-bundle gives rise to a tracially complete $C^*$-algebra.  We will give a partial converse of the next result (also essentially due to Ozawa) in Section \ref{sec:WStarBundles}.

\begin{proposition}\label{prop:WStarBundleToTC}
	Let $\M$ be a $W^*$-bundle over $K$ with conditional expectation $E\colon \M \rightarrow C(K)$.  Let $X$ be the set of traces on $\M$ of the form
	\begin{equation}\label{eqn:Traces}
		\tau(a) = \int_K E(a) \, \mathrm{d}\mu, \quad \quad a \in \M,
	\end{equation}
	for $\mu \in \mathrm{Prob}(K)$, the space of Radon probability measures on $K$.  Then $(\M, X)$ is a tracially complete $C^*$-algebra and $X$ is a Bauer simplex with extreme boundary $\partial_e X \cong K$.  
\end{proposition}

\begin{proof}
	There is a continuous affine embedding 
	\begin{equation}
		\mathrm{Prob}(K) \longrightarrow T(\mathcal M) \colon \mu \mapsto \int_K E(\,\cdot\,)\,\mathrm{d}\mu.
	\end{equation}
	If $X \subseteq T(\mathcal M)$ denotes the image of this map, then $X$ is compact and convex, and the map above restricts to a homeomorphism $K \rightarrow \partial_e X$.  Since $X$ is the closed convex hull of its extreme points, we have 
	\begin{equation}
		\|a\|_{2,X} = \|a\|_{2,\partial_e X} = \|E(a^*a)\|^{1/2} = \|a\|_{2, E}
	\end{equation} 
	for all $a \in \M$, using Proposition~\ref{prop:co-2-norm} for the first equality.  Since $\M$ is a $W^*$-bundle, it follows that $\|\cdot\|_{2,X}$ is a norm and the unit ball of $\M$ is $\|\cdot\|_{2,X}$-complete. 
\end{proof}

The most straightforward examples of $W^*$-bundles are the trivial bundles.

\begin{example}\label{ex:locally-trivial} (cf.\ \cite[Theorem 13 and the preceding paragraph]{Oz13})
	Let $K$ be a compact Hausdorff space and let $(\mathcal{N},\tau_{\mathcal N})$ be a tracial von Neumann algebra.
	Define $C_\sigma(K,\mathcal N)$ to be the set of $\|\cdot\|$-bounded $\|\cdot\|_{2,\tau_{\mathcal N}}$-continuous functions $K \to \mathcal N$.
	This is a unital $C^*$-algebra.
	
	Let $\mathrm{Prob}(K)$ be the set of Radon probability measures on $K$, and for $\mu \in \mathrm{Prob}(K)$, define $\tau_\mu \in T\big(C_\sigma(K,\mathcal N)\big)$ by
	\begin{align} 
    \tau_\mu(f) &\coloneqq \int_K \tau_{\mathcal N}(f(x))\,\mathrm{d}\mu(x), &&f \in C_\sigma(K,\mathcal N). 
	\intertext{Defining $X\coloneqq \{\tau_\mu: \mu \in \mathrm{Prob}(K)\}$, one finds that}
	\|f\|_{2,X} &= \max_{x \in K} \|f(x)\|_{2,\tau_{\mathcal N}}, &&f \in C_\sigma(K,\mathcal N). \end{align}
	Then $\big(C_\sigma(K,\mathcal N),X\big)$ is a tracially complete $C^*$-algebra called the
	\emph{trivial $W^*$-bundle} over $K$ with fibre $(\mathcal N, \tau_\mathcal N)$.
\end{example}

For any compact space $K$, we note that $\M\coloneqq C(K)$ together with $X \coloneqq T(\M) \cong \mathrm{Prob}(X)$ is a tracially complete $C^*$-algebra where $\|\cdot\|$ and $\|\cdot\|_{2,X}$ agree. More generally, if $M_n$ denotes the $C^*$-algebra of $n \times n$ matrices over $\mathbb C$, then $\M \coloneqq C(K, M_n)$, together with $X \coloneqq T(\M) \cong \mathrm{Prob}(X)$, is a tracially complete $C^*$-algebra where $\|\cdot\|$ and $\|\cdot\|_{2,X}$ are equivalent norms. These examples show that we cannot expect tracially complete $C^*$-algebras to behave like von Neumann algebras in general.

An important tool for investigating the structure of a tracially complete $C^*$-algebra $(\M,X)$ is the GNS representation $\pi_\tau\colon \M \rightarrow \mathcal B(\mathcal H_\tau)$ for $\tau \in X$. Each $\tau \in X$ induces a faithful trace on the von Neumann algebra $\pi_\tau(\M)''$. We view these GNS representations as giving rise to `fibres' of tracially complete $C^*$-algebras. We are intentionally somewhat vague as to the formal definition of the fibre at $\tau$: should it be $\pi_\tau(\M)$ or $\pi_\tau(\M)''$, and should we use all traces in $X$ or restrict to $\partial_e X$ (as in the $W^*$-bundle case)? 
For non-extreme traces, $\pi_\tau(\M)$ and $\pi_\tau(\M)''$ can differ, even when $\M$ is a $W^*$-bundle (i.e.\ $X$ is Bauer).
We return to this discussion with Question~\ref{Q:fibres}.

For many of our latter results in this paper, we will want to restrict to the case that the von Neumann algebras $\pi_\tau(\M)''$ are all type II$_1$.

\begin{definition}\label{def:II_1}
	A tracially complete $C^*$-algebra $(\M,X)$ is of \emph{type}~II$_1$ if $\pi_\tau(\M)''$ is a type II$_1$ von Neumann algebra for every $\tau \in X$.
\end{definition}  

We now turn to proving some basic properties of tracially complete $C^*$-algebras.
The following two results are both standard in the setting of tracial von Neumann algebras.

\begin{proposition}\label{prop:unital}
	Every tracially complete $C^*$-algebra is unital.
\end{proposition}

\begin{proof}
	Suppose $(\mathcal M, X)$ is a tracially complete $C^*$-algebra and let $(e_\lambda)$ be an approximate unit for $\mathcal M$.  Then $(\tau(e_\lambda)) \subseteq \mathbb R$ is an increasing net converging to 1 for all $\tau \in X$, and since $X$ is compact, Dini's Theorem implies this convergence is uniform in $\tau$.  As $(e_\lambda)$ is an increasing net of positive contractions,
	\begin{equation} \|e_\lambda - e_\mu\|_{2, X}^2 = \sup_{\tau \in X} \tau\big((e_\lambda - e_\mu)^2\big) \leq \sup_{\tau \in X} \tau(e_\lambda - e_\mu)
	\end{equation}
	whenever $\lambda \geq \mu$.
	Hence $(e_\lambda)$ is a $\|\cdot\|_{2, X}$-Cauchy net in the unit ball of $\M$.  Since $\M$ is tracially complete, $(e_\lambda)$ $\|\cdot\|_{2,X}$-converges to a positive contraction $e \in \M$ by Proposition~\ref{prop:ball-closed}\ref{item:cone-closed}.  For all $a \in \mathcal M$, working in the unitisation $\M^\dag$ of $\mathcal M$ and extending each $\tau \in X$ to $\M^\dag$,
	\begin{equation}
		\| (1_{\M^\dag} - e)a \|^2_{2, X} \leq \sup_{\tau \in X} \tau( 1_{\M^\dag}- e) \|a\|^2 = 0, 
	\end{equation}
	and hence $e$ is a unit for $\mathcal M$.
\end{proof}

\begin{proposition}
	Every morphism of tracially complete $C^*$-algebras is unital and contractive with respect to the uniform 2-norms.
\end{proposition}

\begin{proof}
	Suppose $\phi \colon (\mathcal M, X) \rightarrow (\mathcal N, Y)$ is a morphism between tracially complete $C^*$-algebras.  To see $\phi$ is contractive, fix $\tau \in T(Y)$ and $a \in \mathcal M$. Then $\tau \circ \phi \in X$ by the definition of a morphism, so
	\begin{equation}
		\tau\big(\phi(a)^* \phi(a)\big) = (\tau \circ \phi)(a^*a) \leq \sup\|a\|^2_{2, X}.
	\end{equation}
	Taking the supremum over $\tau$ yields $\|\phi(a)\|_{2, Y} \leq \|a\|_{2, X}$ for all $a \in \mathcal M$.
	
	To see $\phi$ is unital, first note that $\phi(1_\M) \leq 1_{\mathcal{N}}$.
	If $\tau \in Y$, then $\tau \circ \phi \in X$, so $\tau(\phi(1_\mathcal M)) = 1$.  
	Therefore, $\big\| 1_\mathcal N - \phi(1_\mathcal M)\big\|_{2, Y} = 0$, and hence $\phi(1_\mathcal M) = 1_\mathcal N$.
\end{proof}

Note that if $(\M, X)$ is a tracially complete $C^*$-algebra and $\mathcal N \subseteq \mathcal{M}$ is a $\|\cdot\|_{2, X}$-closed unital $C^*$-subalgebra,\footnote{Note that a $C^*$-subalgebra $\mathcal N \subseteq \mathcal M$ is $\|\cdot\|_{2, X}$-closed if and only if the unit ball of $\mathcal N$ is $\|\cdot\|_{2, X}$-closed in the unit ball of $\mathcal M$.  The forward direction is immediate, and the backward direction follows from Lemma~\ref{lem:UnitBallDensity} below.} we may form a tracially complete $C^*$-algebra $(\mathcal N, Y)$ where
\begin{equation}
	Y \coloneqq \{ \tau|_{\mathcal N} : \tau \in X \} \subseteq T(\mathcal N).
\end{equation}
We call $(\mathcal N,Y)$ a \emph{tracially complete $C^*$-subalgebra} of $(\mathcal M,X)$. This suggests the following notion of an embedding of tracially complete $C^*$-algebras.

\begin{definition}\label{def:embedding}
	Let $(\M, X)$ and $(\mathcal N, Y)$ be tracially complete $C^*$-algebras.  A morphism $\phi \colon (\M, X) \rightarrow (\mathcal{N}, Y)$ is called an \emph{embedding} if $\phi^*(Y) = X$.
\end{definition}

The next proposition justifies the terminology.

\begin{proposition}\label{prop:embedding-isometric}
	A morphism $\phi \colon (\M, X) \rightarrow (\mathcal N, Y)$ is an embedding if and only if $\|\phi(a)\|_{2, Y} = \|a\|_{2, X}$ for all $a \in \M$.  Further, if $\phi$ is an embedding, then $\phi$ is isometric in the operator norm. 
\end{proposition}

\begin{proof}
	If $\phi$ is an embedding, then for $a\in\mathcal M$, we have
 \begin{equation}
 \|\phi(a)\|_{2,Y}=\sup_{\tau\in Y}\tau(\phi(a^*)\phi(a))^{1/2}=\sup_{\sigma\in X}\sigma(a^*a)^{1/2}=\|a\|_{2,X}.
 \end{equation}
To show the converse, suppose $\phi^*(Y) \neq X$ and fix $\tau_0 \in X \setminus \phi^*(Y)$. As $\phi^*(Y)$ is weak$^*$-closed and convex, there exists a self-adjoint $a \in \M$ such that $\tau_0(a) > \sup_{\sigma \in \phi^*(Y)} \sigma(a)$ by the Hahn--Banach theorem.  Replacing $a$ by $a + \|a\| 1_\M$, we may assume $a \in \M_+$. We then have 
\begin{equation}
    \|a^{1/2}\|_{2, X} \geq \tau_0(a) > \|\phi(a^{1/2})\|_{2, Y},
\end{equation} 
so $\phi$ is not isometric in the uniform 2-norm.

When $\phi$ is an embedding, it is injective since $\|\cdot\|_{2,X}$ is a norm, and therefore, it is isometric in the operator norm.
\end{proof}

\subsection{Factoriality}\label{sec:fact}

Our definition of tracially complete $C^*$-algebra makes very few restrictions on the set $X$ (to allow for certain constructions -- free products in particular -- to remain in the category). However, in practise we care most about the subclass of \emph{factorial} tracially complete $C^*$-algebras, which we define below. This class includes the tracial completion of a unital $C^*$-algebras $A$ with respect to the trace simplex $T(A)$, which is the main example of interest (see Proposition~\ref{prop:tracial-completion}\ref{item:completion-factorial}).

\begin{definition}\label{def:factorial}
	A tracially complete $C^*$-algebra $(\M, X)$ is said to be \emph{factorial} if $X$ is a closed face of $T(\M)$.
\end{definition} 

When $X$ is a singleton, a tracially complete $C^*$-algebra is a tracial von Neumann algebra $(\M, \tau)$, and in this case, $\{\tau\}$ is a face of $T(\M)$ if and only if $\M$ is a factor (by Proposition~\ref{prop:factorial-extreme-point}). 
Further, the trivial $W^*$-bundle $C_\sigma(K, \mathcal{N})$ from Example~\ref{ex:locally-trivial} is factorial as a tracially complete $C^*$-algebra if and only if the fibre $\mathcal {N}$ is a factor.  An analogous result holds for non-trivial $W^*$-bundles (see Proposition~\ref{prop:WStarBundleToTC}).  More generally, we have the following result.

\begin{proposition}\label{prop:factorial}
	A tracially complete $C^*$-algebra $(\M, X)$ is factorial if and only if $\pi_\tau(\M)''$ is a factor for every extreme point $\tau$ of $X$.
\end{proposition}

\begin{proof}
	By Theorem~\ref{thm:choquet-faces}, $(\mathcal M, X)$ is factorial if and only if $\partial_e X \subseteq \partial_e T(\mathcal M)$.  Also, by Proposition~\ref{prop:factorial-extreme-point}, $\partial_e X \subseteq \partial_e T(\mathcal M)$ if and only if $\pi_\tau(\mathcal M)''$ is a factor for all $\tau \in \partial_e X$. 
\end{proof}

The following result shows the equivalence of the two versions of the question stated as the trace problem (Question \ref{Q:traces}).   The proof of this proposition illustrates the utility of factoriality: it allows for powerful separation arguments coming from the classical result that closed faces in Choquet simplices are relatively exposed.  The exposedness appears in the proof via an application of Theorem~\ref{thm:exposed}\ref{item:exposed2}.

\begin{proposition}\label{prop:face}
	If $(\M, X)$ is a tracially complete $C^*$-algebra, then every $\|\cdot\|_{2, X}$-continuous trace on $\M$ belongs to the closed face of $T(\M)$ generated by $X$.  In particular, if $(\M, X)$ is factorial, then $X$ is precisely the set of all $\|\cdot\|_{2, X}$-continuous traces on $\M$.
\end{proposition}

\begin{proof}
	Since $\M$ is a unital $C^*$-algebra, $T(\M)$ is a Choquet simplex (Theorem~\ref{thm:traces-choquet}). By Theorem~\ref{thm:exposed}, it suffices to show that if $\tau_0 \in T(\mathcal M)$ is $\|\cdot\|_{2, X}$-continuous and $f \colon T(\mathcal M) \rightarrow [0, 1)$ is a continuous affine function with $f|_X = 0$, then $f(\tau_0) = 0$.  Fix such a trace $\tau_0$ and function $f$.  By Proposition~\ref{prop:CP}, applied to $f + \frac1n$, there is a sequence of positive contractions $(a_n)_{n=1}^\infty \subseteq \M$ such that
	\begin{equation}
		\label{eq:approx-point-eval}
		\sup_{\tau \in T(\mathcal M)} |\tau(a_n) - f(\tau)| \rightarrow 0.
	\end{equation}
	Since $f(\tau) = 0$ for all $\tau \in X$, we have $\|a_n^{1/2}\|_{2, X} \rightarrow 0$.  As each $a_n$ is contractive, it follows from \eqref{eq:SpecialHolderIneq} that $\|a_n\|_{2, X} \rightarrow 0$.  Now, since $\tau_0$ is $\|\cdot\|_{2, X}$-continuous, we have $\tau_0(a_n) \rightarrow 0$, and hence \eqref{eq:approx-point-eval} implies $f(\tau_0) = 0$.
\end{proof}

Returning to our viewpoint that the GNS representations of tracially complete $C^*$-algebras should be viewed as giving rise to their fibres in some kind of affine bundle structure, the following question is natural. Ozawa's \cite[Theorem 11]{Oz13} gives a positive answer for all $W^*$-bundles (factorial or not).
\begin{question}\label{Q:fibres}
Let $(\M,X)$ be a factorial tracially complete $C^*$-algebra.  Is it the case that $\pi_\tau(\M)''=\pi_\tau(\M)$ for all $\tau\in\partial_eX$?
\end{question}

We end the subsection with a some examples.  Given a finite von Neumann algebra $\M$, instead of specifying a faithful trace $\tau$ (which will exist when $\M$ has a separable predual) and viewing $\M$ as a (generally non-factorial) tracially complete $C^*$-algebra over a singleton set, one can work with all the traces.  By \cite[Proposition~3.1.6]{Ev18}, $\|\cdot\|_{2, T(\M)}$ is a complete norm on the unit ball of $\M$, so that $\big(\M, T(\M)\big)$ is a tracially complete $C^*$-algebra which is evidently factorial.  This example will play a technical role in obtaining CPoU from property $\Gamma$ in Section \ref{sec:CPoU-proof}.

\begin{proposition}\label{prop:finitevnastc}
    Let $\M$ be a finite von Neumann algebra.  Then the pair $\big(\M,T(\M)\big)$ is a factorial tracially complete $C^*$-algebra.
\end{proposition}

It is worth noting that when $(\M,\tau)$ is a tracial von Neumann algebra, the tracially complete $C^*$-algebras $\big(\M,\{\tau\}\big)$ and $\big(\M,T(\M)\big)$ can behave very differently, despite having the same underlying $C^*$-algebra. See Proposition~\ref{prop:vNA-non-uniform-amenable-trace}, for example.\footnote{When $T(\M)$ is not finite-dimensional, it is a large space with zero-dimensional boundary, and $(\M,T(\M))$ might not be considered a particularly natural example of a tracially complete $C^*$-algebra.}

For later use, we note that matrix algebras over (factorial) tracially complete $C^*$-algebras are also (factorial) tracially complete $C^*$-algebras in a natural way.  Let $\mathrm{tr}_d$ denote the unique trace on the $C^*$-algebra $M_d$ of $d \times d$ matrices over $\mathbb C$.  For a $C^*$-algebra $A$ and a set $X \subseteq T(A)$, define
\begin{equation}
	X \otimes \{\mathrm{tr}_d\} \coloneqq \{ \tau \otimes \mathrm{tr}_d : \tau \in X \} \subseteq T(A \otimes M_d).
\end{equation}

\begin{proposition}\label{prop:matrix-algebras}
	If $(\M, X)$ is a tracially complete $C^*$-algebra and $d \in \mathbb N$, then $\big(\M \otimes M_d, X \otimes \{\mathrm{tr}_d\}\big)$ is a tracially complete $C^*$-algebra.  Further, $(\M, X)$ is factorial if and only if $\big(\M \otimes M_d, X \otimes \{\mathrm{tr}_d\}\big)$ is factorial.
\end{proposition}

\begin{proof}
	Since the map $T(\mathcal M) \rightarrow T(\M \otimes M_d)$ given by $\tau \mapsto \tau \otimes \mathrm{tr}_d$ is an affine homeomorphism, we have that $X \otimes \{\mathrm{tr}_d\}$ is a compact convex subset of $T(\M \otimes M_d)$ and is a face if and only if $X$ is a face in $T(\M)$.  Thus it suffices to show that $X \otimes \{\mathrm{tr}_d\}$ is a faithful set of traces on $\mathcal M \otimes M_d$ and the unit ball of $\M\otimes M_d$ is $\|\cdot\|_{2, X \otimes \{\mathrm{tr}_d\}}$-complete.  
 
 Given $a\in\mathcal M\otimes M_d$, write $a=\sum_{i,j=1}^da_{i,j}\otimes e_{i,j}$ for some $a_{i,j}\in\M$.  For $\tau\in X$, 
 \begin{equation}\label{eqn:matrix-algebras-1}
     (\tau\otimes\mathrm{tr}_d)(a^*a)=\frac{1}{d}\tau\Big(\sum_{i,j=1}^da_{i,j}^*a_{i,j}\Big).
 \end{equation}
If $a\neq 0$, then there exist $i$ and $j$ with $a_{i,j}^*a_{i,j}\neq 0$. As $X$ is a faithful set of traces, there exists $\tau\in X$ such that $\tau(a_{i,j}^*a_{i,j})\neq 0$. Therefore, \begin{equation}
    (\tau\otimes\mathrm{tr}_d)(a^*a)\geq \tau(a_{i,j}^*a_{i,j})>0,
\end{equation}
as required.

Since $X\times\{\mathrm{tr}_d\}$ is a faithful set of traces on $\M\otimes M_d$, the unit ball of $\M \otimes M_d$ is $\|\cdot\|_{2, X \otimes \{\mathrm{tr}_d\}}$-closed by Proposition~\ref{prop:ball-closed}\ref{item:ball-closed}.  Furthermore, 
	\begin{equation}\begin{split}
		\frac{1}{d}\max_{1 \leq i, j \leq d} \|a_{i, j}\|^2_{2, X} &\leq \| a \|^2_{2, X \otimes \{\mathrm{tr}_d\}}\leq \frac{1}{d}\sum_{i, j=1}^d \|a_{i, j}\|^2_{2, X}
	\end{split}\end{equation}
as an immediate consequence of \eqref{eqn:matrix-algebras-1}. The $\|\cdot\|_{2, X \otimes \{\mathrm{tr}_d\}}$-completeness of the unit ball of $\M\otimes M_d$ now follows from the $\|\cdot\|_{2, X}$-completeness of the unit ball of $\M$.
\end{proof}

\subsection{Tracial completions}\label{sec:TraceNormCompletion}

In this section, we recall Ozawa's tracial completion construction from \cite{Oz13} and its important features. This is both the motivation for our study of tracially complete $C^*$-algebras and an important tool in the theory that we develop.

\begin{definition}[{cf.~\cite{Oz13}}]\label{def:UTCompletion}
	Let $A$ be a $C^*$-algebra. For a compact convex set $X \subseteq T(A)$, the \emph{tracial completion} of $A$ with respect to $X$ is the $C^*$-algebra\footnote{It is not hard to see -- using \eqref{eq:SpecialHolderIneq} -- that the $\|\cdot\|$-bounded $\|\cdot\|_{2,X}$-Cauchy sequences in $A$ form a $C^*$-subalgebra of $\ell^\infty(A)$ and that the $\|\cdot\|$-bounded $\|\cdot\|_{2,X}$-null sequences form an ideal of this subalgebra.}
	\begin{equation}
		\completion{A}{X} \coloneqq  \frac{\{(a_n)_{n=1}^\infty \in \ell^\infty(A): (a_n)_{n=1}^\infty \text{ is }\|\cdot\|_{2,X}\text{-Cauchy}\}}{\{(a_n)_{n=1}^\infty \in \ell^\infty(A): (a_n)_{n=1}^\infty \text{ is }\|\cdot\|_{2,X}\text{-null}\}}.
	\end{equation}
\end{definition}

\begin{remark}\label{footnote:StrictClosure}
	In Ozawa's work, the notation $\completion{A}{\mathrm{u}}$ is used for the tracial completion with $X$ being understood from the context. Ozawa also notes that $\completion{A}{\mathrm u}$ can be viewed as the strict closure of the range of a certain representation of $A$ on a Hilbert module. This leads to the alternative terminology \emph{strict closure} for the uniform tracial completion and the alternative notation $\completion{A}{\mathrm{st}}$ used in~\cite{Oz13,BBSTWW}. Since it is important for us to keep track of the set $X$, we include it in the notation for tracial completions in this paper.
\end{remark}

We define $\alpha_X\colon A \to \completion{A}{X}$ to be the map given by sending $a \in A$ to the image of the constant sequence $(a,a,\dots)$. We also define a second norm on the tracial completion by $(a_n)_{n=1}^\infty \mapsto \lim_{n\rightarrow\infty} \|a_n\|_{2,X}$. It is immediate that this limit exists for any $\|\cdot\|_{2,X}$-Cauchy sequence $(a_n)_{n=1}^\infty \in \ell^\infty(A)$ and induces a well-defined norm on the quotient. Abusing notation slightly, we also write $\|\cdot\|_{2,X}$ for this second norm, the justification being that $\|\alpha_X(a)\|_{2,X} = \|a\|_{2,X}$ for all $a \in A$.\footnote{
	When $\|\cdot\|_{2,X}$ is a norm on $A$, $\alpha_X$ is an embedding, and it is natural to identify $A$ with $\alpha_X(A)$. If one desires, one can quotient $A$ by the ideal of $\|\cdot\|_{2,X}$-null elements first before performing the tracial completion (replacing $X$ with the set of induced traces on the quotient). In this way, one can reduce to the case that $\alpha_X$ is an embedding.
}

Tracial completions with respect to a single trace are nothing but the von Neumann algebra generated by the associated GNS representation.

\begin{proposition}\label{prop:tracialcompletiongns}
Let $A$ be a $C^*$-algebra and $\tau\in T(A)$.  Then, for $X\coloneqq\{\tau\}$, the tracial completion $\completion{A}{X}$ can be canonically identified with the von Neumann algebra $\pi_\tau(A)''$, where $\pi_\tau$ is the GNS representation associated to $\tau$; i.e.\ there is an isomorphism $\theta\colon \completion{A}{X}\to\pi_\tau(A)''$ such that $\theta\circ\alpha_X=\pi_\tau$.
\end{proposition}

\begin{proof}
As the unit ball of $\pi_\tau(A)''$ is $\|\cdot\|_{2,\tau}$-complete (see \cite[Lemma~A.3.3]{Si08}, for example), there is a well-defined $^*$-homomorphism
\begin{equation}
\tilde{\theta}\colon\{(a_n)_{n=1}^\infty \in \ell^\infty(A): (a_n)_{n=1}^\infty \text{ is }\|\cdot\|_{2,X}\text{-Cauchy}\}\to\pi_\tau(A)''
\end{equation}
given by $\tilde{\theta}((a_n))=\lim\pi_\tau(a_n)$, where the limit is taken in the norm $\|\cdot\|_{2,\tau}$.  The kernel of this map consists of the $\|\cdot\|_{2,X}$-null sequences, so the first isomorphism theorem gives an injective $^*$-homomorphism $\theta$ with $\theta \circ \alpha_X = \pi_\tau$.  As the unit ball of $\completion{A}{X}$ is complete in the 2-norm and $\theta$ is isometric with respect to the 2-norms, the image of the unit ball of $\theta$ is $\|\cdot\|_{2, \tau}$-closed and contains the unit ball of $\pi_\tau(A)$.  By Kaplansky's density theorem, the unit ball of $\pi_\tau(A)$ is $\|\cdot\|_{2, \tau}$-dense in the unit ball of $\pi_\tau(A)''$, and hence $\theta$ is surjective, proving that $\theta$ is an isomorphism.
\end{proof}

Very similarly, tracial completions with respect to finite dimensional simplices also come from GNS representations.
\begin{example}
    When $X$ is a finite dimensional simplex, the tracial completion is $\pi_\tau(A)''$, where $\tau$ is the average of the traces in $\partial_e X$, and so this tracial completion is again a von Neumann algebra.
\end{example}

For each trace $\tau \in X$, we define an induced trace $\widetilde{\tau}$ on $\completion{A}{X}$ by setting $\widetilde{\tau}(a)\coloneqq \lim_{n\rightarrow\infty}\tau(a_n)$ for any representative $\|\cdot\|_{2,X}$-Cauchy sequence $(a_n)_{n=1}^\infty \in \ell^\infty(A)$. The inequality $|\tau(a)| \leq \|a\|_{2,X}$, which holds for all $a \in A$, ensures that this limit exists and that $\widetilde{\tau}$ is well-defined on the tracial completion. The induced trace satisfies $\widetilde{\tau}(\alpha_X(a)) = \tau(a)$  for all $a \in A$ and is the unique $\|\cdot\|_{2,X}$-continuous trace with this property since $\alpha_X(A)$ is $\|\cdot\|_{2,X}$-dense in the tracial completion.\footnote{By construction, $\widetilde\tau$ is $\|\cdot\|_{2, X}$-continuous on $\|\cdot\|$-bounded subsets of $\completion{A}{X}$;  continuity on all of $\completion{A}{X}$ follows from Proposition~\ref{prop:tracial-completion}(i) (or from \eqref{eq:unbounded-continuity} in the proof).}

The following proposition establishes the basic properties of tracial completions.

\begin{proposition}\label{prop:tracial-completion}
	Let $A$ be a $C^*$-algebra and let $X$ be a compact convex subset of $T(A)$.
	Let $\widetilde{X} \subseteq T(\completion{A}{X})$ denote the set of all traces on $\completion{A}{X}$ that are induced by traces in $X$.
	\begin{enumerate}
		\item\label{item:completion-norm} With the notation in the paragraph preceding the proposition, the seminorm $\|\cdot\|_{2,\widetilde{X}}$ coincides with the norm $\|\cdot\|_{2,X}$ on $\completion{A}{X}$.
		\item\label{item:completion-trace} The map $X\rightarrow \widetilde{X}$ sending a trace to its induced trace is an affine homeomorphism, where $X$ and $\widetilde{X}$ are equipped with their respective weak$^*$ topologies.
		\item\label{item:completion-complete} The pair $(\completion{A}{X}, \widetilde{X})$ is a tracially complete $C^*$-algebra.
		\item\label{item:completion-factorial} The completion
		$(\completion{A}{X}, \widetilde{X})$ is factorial if and only if $X$ is a face in $T(A)$.
  \item\label{item:completion-subset} Let $Y\subset X$ be a compact convex set and let $\widetilde{Y}$ denote the image of $Y$ in $\tilde{X}$.  Then $(\completion{A}{Y},Y)$ and $(\completion{\completion{A}{X}}{\widetilde{Y}},\widetilde{Y})$ are isomorphic via an isomorphism $\theta$ such that
  \begin{equation}\label{eqn:completion-subset}
      \begin{tikzcd}          A\ar[r,"\alpha_Y"]\ar[d,swap,"\alpha_X"]&\completion{A}{Y}\ar[d,"\theta"]\\\completion{A}{X}\ar[r,swap,"\alpha_{\tilde{Y}}"]&\completion{\completion{A}{X}}{\widetilde{Y}}
      \end{tikzcd}
  \end{equation}
  commutes. 
  \item \label{item:completion-fibres} Given $\tau\in X$ with extension $\widetilde{\tau}$ to $\completion{A}{X}$, write $\pi_\tau\colon A\to \pi_\tau(A)''$ 
  and 
  $\pi_{\widetilde{\tau}}\colon \completion{A}{X}\to\pi_{\widetilde{\tau}}(\completion{A}{X})''$ for the respective GNS representations.  Then $\alpha_X\colon A\to \completion{A}{X}$ induces an isomorphism $\theta\colon\pi_\tau(A)''\to\pi_{\widetilde{\tau}}(\completion{A}{X})''$ such that
  \begin{equation}
      \begin{tikzcd}          A\ar[r,"\pi_\tau"]\ar[d,swap,"\alpha_X"]&\pi_\tau(A)''\ar[d,"\theta"]\\\completion{A}{X}\ar[r,swap,"\pi_{\widetilde{\tau}}"]&\pi_{\widetilde{\tau}}(\completion{A}{X})''
      \end{tikzcd}
  \end{equation}
  commutes.
	\end{enumerate}	
\end{proposition}

\begin{proof}
	\ref{item:completion-norm}: Let $a \in \completion{A}{X}$ be represented by a $\|\cdot\|_{2,X}$-Cauchy sequence $(a_n)_{n=1}^\infty \in \ell^\infty(A)$. Let $\tau \in X$. Then
	\begin{equation}\label{eq:unbounded-continuity}
		\|a\|_{2,\widetilde{\tau}}^2 = \lim_{n\rightarrow\infty} \tau(a_n^*a_n)
		\leq \lim_{n\rightarrow\infty} \|a_n\|_{2,X}^2\\
		= \|a\|_{2,X}^2.
	\end{equation}
	Hence $\|a\|_{2,\widetilde{X}} \leq \|a\|_{2,X}$ for all $a \in \completion{A}{X}$, and this implies $\|\cdot\|_{2, \widetilde{X}}$ is $\|\cdot\|_{2, X}$-continuous.  Since $\|\alpha_X(a)\|_{2,X}=\|a\|_{2,X} = \|\alpha_X (a)\|_{2, \widetilde{X}}$ for all $a \in A$ and  $\alpha_X(A)$ is $\|\cdot\|_{2, X}$-dense in $\completion{A}{X}$, we deduce $\|a\|_{2,X} = \|a\|_{2, \widetilde{X}}$ for all $a \in \completion{A}{X}$.
	
	\ref{item:completion-trace}: It is clear that the map $\tau \mapsto \widetilde{\tau}$ is affine and bijective.  As $X$ is compact and $\widetilde{X}$ is Hausdorff, it suffices to show continuity of this map.  Let $(\tau_\lambda)_\lambda$ be a net in $X$ that converges to $\tau \in X$ and let $\widetilde{\tau}_\lambda,\widetilde{\tau} \in \widetilde{X}$ denote the $\|\cdot\|_{2,X}$-continuous traces induced by $\tau_\lambda$ and $\tau$, respectively.  We must show that $\widetilde\tau_\lambda \to \widetilde\tau$.
	
	Let $a \in \completion{A}{X}$ and let $\epsilon>0$. Pick $b \in A$ such that $\|a-\alpha_X(b)\|_{2,X} < \epsilon/3$. Choose $\lambda_0$ such that for $\lambda \geq \lambda_0$, we have $\tau_{\lambda}(b)\approx_{\epsilon/3} \tau(b)$. Then for $\lambda \geq \lambda_0$,
	\begin{equation}
		\widetilde\tau(a) \approx_{\epsilon/3} \widetilde\tau(\alpha_X(b)) = \tau(b) \approx_{\epsilon/3} \tau_\lambda(b) = \widetilde\tau_\lambda(\alpha_X(b)) \approx_{\epsilon/3} \widetilde\tau_\lambda(a),
	\end{equation}
	showing continuity of the extension map.
	
	\ref{item:completion-complete}: It follows from \ref{item:completion-trace} that $\widetilde X$ is a compact convex subset of $T(\completion{A}{X})$.  By the construction of $\completion{A}{X}$, the seminorm $\|\cdot\|_{2, X}$ is a norm and the unit ball of $\completion{A}{X}$ is $\|\cdot\|_{2, X}$-complete, so by \ref{item:completion-norm}, the unit ball of $\completion{A}{X}$ is also $\|\cdot\|_{2, \widetilde{X}}$-complete.
	
	\ref{item:completion-factorial}: Suppose that $X$ is a face in $T(A)$.  To show that $\widetilde{X}$ is a face of $T(\completion{A}{X})$, suppose we can write $\tau \in \widetilde{X}$ as $\tau = \frac12(\tau_1+\tau_2)$ for some $\tau_1,\tau_2 \in T(\completion{A}{X})$. Since $(\alpha_X(e_\lambda))$ converges to the unit of $\completion{A}{X}$ for any approximate unit $(e_\lambda)$ of $A$, it follows that $\tau \circ \alpha_X$ is a trace on $A$. As $X$ is a face of $T(A)$, $\tau_1\circ\alpha_X, \tau_2\circ\alpha_X \in X$. Also, since $\tau_1,\tau_2 \leq 2\tau$ and $\tau$ is $\|\cdot\|_{2,\widetilde{X}}$-continuous, it follows that $\tau_1$ and $\tau_2$ are also $\|\cdot\|_{2, X}$-continuous. Thus $\tau_1$ and $\tau_2$ are the $\|\cdot\|_{2,\widetilde{X}}$-continuous extensions of $\tau_1 \circ \alpha_X,\tau_2 \circ \alpha_X$ respectively, so $\tau_1,\tau_2 \in \widetilde{X}$.
	
	Conversely, suppose $(\completion{A}{X}, \widetilde{X})$ is factorial and $\tau_1, \tau_2 \in T(A)$ satisfy $\tau \coloneqq \frac12(\tau_1 + \tau_2) \in \widetilde X$.  Then $\tau$ is $\|\cdot\|_{2, X}$-continuous, and hence so are $\tau_1$ and $\tau_2$ as they are dominated by $2 \tau$.  Now, $\tau$, $\tau_1$, and $\tau_2$ extend  to $\|\cdot\|_{2, X}$-continuous traces $\widetilde\tau$, $\widetilde\tau_1$, and $\widetilde{\tau}_2$ on $\completion{A}{X}$, respectively, using that the unit ball of $\alpha_X(A)$ is $\|\cdot\|_{2, \widetilde{X}}$-dense in the unit ball of $\completion{A}{X}$. Then $\widetilde\tau \in \widetilde{X}$ and $\widetilde\tau = \frac12 (\widetilde\tau_1 + \widetilde\tau_2)$.  Since $(\completion{A}{X}, \widetilde{X})$ is factorial, we have $\widetilde\tau_1$ and $\widetilde\tau_2$ are in $\widetilde{X}$, so $\tau_1$ and $\tau_2$ are in $X$.

 \ref{item:completion-subset}: The map $\alpha_X\colon A\to\completion{A}{X}$ is $\|\cdot\|_{2,Y}$-$\|\cdot\|_{2,\widetilde{Y}}$-isometric. Therefore, it sends $\|\cdot\|_{2,Y}$-Cauchy and $\|\cdot\|_{2,Y}$-null sequences in $A$ to $\|\cdot\|_{2,\widetilde{Y}}$-Cauchy and $\|\cdot\|_{2,\widetilde{Y}}$-null sequences in $\completion{A}{Y}$, respectively.  Accordingly, it induces $\theta$ making \eqref{eqn:completion-subset} commute.
 By commutativity of this diagram and the inequality $\|\cdot\|_{2, \widetilde Y} \leq \|\cdot\|_{2, X}$ on $\completion{A}{X}$, the image of the unit ball of $\completion{A}{Y}$ under $\theta$ is dense in the unit ball of $\completion{\completion{A}{X}}{\widetilde Y}$.  Since $\theta$ is isometric with respect to the uniform 2-norms and the unit ball of $\completion{A}{Y}$ is $\|\cdot\|_{2, Y}$-complete, $\theta$ is surjective.
 
\ref{item:completion-fibres}: This follows from \ref{item:completion-subset} using Proposition \ref{prop:tracialcompletiongns} to identify uniform tracial completions at $\tau$ and $\widetilde{\tau}$ with the von Neumann algebras generated by the associated GNS representations.

\end{proof}

\begin{notation}
	With Proposition~\ref{prop:tracial-completion} now established, we will typically identify $X$ with $\widetilde{X}$ henceforth. Furthermore, for a trace $\tau \in X$, we will write $\tau$ in place of $\widetilde\tau$ for the induced trace on $\completion{A}{X}$.
\end{notation}

Not surprisingly, tracial completions satisfy a universal property allowing for trace-preserving u.c.p.\ maps to be extended by continuity to tracial completions.

\begin{proposition}\label{prop:extend-by-continuity}
	Let $A$ be a $C^*$-algebra and let $X$ be a compact convex subset of $T(A)$.  If $(\mathcal N, Y)$ is a tracially complete $C^*$-algebra and $\theta \colon A \rightarrow \mathcal N$ is a u.c.p.\ map such that $\tau \circ \theta \in X$ for all $\tau \in Y$, then there is a unique u.c.p.\ map $\overline{\theta} \colon \completion{A}{X} \rightarrow \mathcal N$ of tracially complete $C^*$-algebras such that $\overline{\theta} \circ \alpha_X = \theta$ and $\tau \circ \overline{\theta} \in X$ for all $\tau \in Y$.  Further, if $\theta$ is a $^*$-homomorphism, then so is $\overline{\theta}$.
\end{proposition}

\begin{proof}
	For uniqueness, consider a u.c.p.\ map $\phi\colon \completion{A}{X}\to\mathcal N$ with $\tau\circ\phi\in X$ for all $\tau\in Y$.  Then for $\tau \in Y$ and $a \in \completion{A}{X}$,
   \begin{equation}\label{eq:trace-norm-contractive}
		\tau(\phi(a)^*\phi(a)) \leq \tau(\phi(a^*a)) \leq \|a\|_{2, X}^2,
	\end{equation}
 where the first inequality follows from the Schwarz inequality for u.c.p.\ maps (cf.\ \cite[Proposition 3.3]{Paulsen-cp-maps}) and the second inequality is a consequence of the hypothesis on traces. Accordingly, 
 $\phi$ is $\|\cdot\|_{2,X}$-$\|\cdot\|_{2,Y}$-continuous.  By the $\|\cdot\|_{2,X}$-density of $\alpha_A(A)$ in $\completion{A}{X}$, it follows that any extension of $\theta$ as in the proposition must be unique.

	We turn to the proof of existence.  Since $\theta^*(Y) \subseteq X$, the same computation used above shows $\|\theta(a)\|_{2, Y} \leq \|a\|_{2, X}$ for all $a \in A$.  Hence if $(a_n)_{n=1}^\infty \subseteq A$ is a $\|\cdot\|$-bounded $\|\cdot\|_{2, X}$-Cauchy sequence in $A$, then $(\theta(a_n))_{n=1}^\infty$ is a $\|\cdot\|$-bounded $\|\cdot\|_{2, Y}$-Cauchy sequence in $\mathcal N$.  Therefore, we obtain a well-defined map $\overline{\theta} \colon \completion{A}{X} \rightarrow \mathcal N$ by $\overline{\theta}(a) \coloneqq \lim_{n \rightarrow \infty} \theta(a_n)$, where $(a_n)_{n=1}^\infty \subseteq A$ is a $\|\cdot\|$-bounded $\|\cdot\|_{2, X}$-Cauchy sequence representing $a \in \completion{A}{X}$ and where the limit is taken in the norm $\|\cdot\|_{2, Y}$. By construction,  $\overline{\theta} \circ \alpha_X = \theta$. Using Propositions \ref{prop:ball-closed}\ref{item:cone-closed} and \ref{prop:matrix-algebras}, a density argument gives that $\overline{\theta}$ is a u.c.p.\ map.    Likewise when $\theta$ is a $^*$-homomorphism, a density argument, this time using continuity of multiplication in the uniform 2-norm on norm-bounded sets, shows that so too is $\bar{\theta}$.
	
	It remains to check that $\tau \circ \overline{\theta} \in X$ for all $\tau \in Y$.  Note that $\|\overline{\theta}(a)\|_{2, Y} \leq \|a\|_{2, X}$ for all $a \in \completion{A}{X}$.  Therefore, if $\tau \in Y$, then $\tau \circ \overline{\theta}$ is $\|\cdot\|_{2, X}$-continuous.  Now, since $\tau \circ \overline{\theta} \circ \alpha_X = \tau \circ \theta \in X$, we have $\tau \circ \overline{\theta} \in X$.
\end{proof}

\subsection{Dense subalgebras of tracially complete \texorpdfstring{$C^*$}{C*}-algebras}\label{sec:DenseSubalgberas}

We will frequently need a version of Kaplansky's density theorem for tracially complete $C^*$-algebras (Lemma~\ref{lem:UnitBallDensity} and Proposition~\ref{prop:UnitBallDensity} below).  This is obtained using the well-known matrix amplification trick to reduce to the self-adjoint case.  At the core is the following estimate.

\begin{lemma}\label{lem:Lipschitz}
	Let $A$ be a commutative $C^*$-algebra and $\tau \in T(A)$. Suppose $f\colon\mathbb{R} \rightarrow \mathbb{R}$ is Lipschitz continuous with constant $M > 0$. If $a,b \in A$ are self-adjoint, then $\|f(a) - f(b)\|_{2, \tau} \leq M\|a-b\|_{2,\tau}$.
\end{lemma}
\begin{proof}
	View $A \cong C_0(K)$ and $a,b \in C_0(K)$ for a locally compact Hausdorff space $K$. Since $f$ is $M$-Lipschitz continuous, we have 
    \begin{align}
		|f(a(x)) - f(b(x))| &\leq M|a(x) - b(x)|
	\intertext{for all $x \in K$. It follows that}
		|f(a)-f(b)|^2 &\leq M^2|a-b|^2
    \end{align}
    in $A$, and the claim follows by applying $\tau$.
\end{proof}

We isolate the following quantitative version of Kaplansky's density theorem for later use.

\begin{lemma}\label{lem:UnitBallDensity}
	Let $(\M,X)$ be a tracially complete $C^*$-algebra and let $A \subseteq \M$ be a $C^*$-subalgebra.  If $b \in \mathcal M$ is a contraction, $\epsilon > 0$, and there is an $a \in A$ with $\|a - b\|_{2, X} < \epsilon$, then there is a contraction $a' \in A$ with $\|a' - b\|_{2, X} < 3 \epsilon$.
\end{lemma}

\begin{proof}
	We first prove the result in the case where $b$ is self-adjoint.  
	By replacing $a$ with its real part, we may assume $a$ is self-adjoint as well.  
	Consider the function $f\colon\mathbb{R} \rightarrow \mathbb{R}$ given by
	\begin{equation} 
		f(t)\coloneqq \begin{cases} -1 & t < -1; \\ t, & t \in [-1,1]; \\ 1, & t>1. \end{cases}
	\end{equation}
   We will show $\|f(a) - b\|_{2, X} < 3\epsilon$, so that the result (in this case) holds by setting $a'\coloneqq f(a)$.

	Fix $\tau \in X$ and let $E \colon \pi_\tau(\M)'' \rightarrow \{\pi_\tau(a)\}''$ be the $\tau$-preserving conditional expectation, which is contractive with respect to each of the norms $\|\cdot\|$ and $\|\cdot\|_{2, \tau}$.\footnote{The expectation is given explicitly by $E(\pi_\tau(c)) \xi_\tau = P \pi_\tau(c)\xi_\tau$ for $c \in \mathcal M$, where $\xi_\tau \in \mathcal H_\tau$ is the canonical cyclic vector and $P$ is the projection of $\mathcal H_\tau$ onto  the closure of $\pi_\tau(a) \mathcal H_\tau$ -- see \cite[Lemma~1.5.11]{Br08}, for example, for more details.}
		Set $a_1 \coloneqq \pi_\tau(a)$, and note that $E(a_1) = a_1$.  
	Therefore, for $b_1 \coloneqq E(\pi_\tau(b))$, 
	\begin{align}
		\|a_1 - b_1\|_{2, \tau} &\leq \|\pi_\tau(a) - \pi_\tau(b)\|_{2, \tau} < \epsilon,
	\intertext{and Lemma \ref{lem:Lipschitz} applied to $\{\pi_\tau(a)\}''$ implies}
		\|f(a_1) - f(b_1)\|_{2,\tau} &\leq \|a_1 - b_1\|_{2, \tau} < \epsilon.
	\end{align}
	Since $b_1$ is a self-adjoint contraction, $f(b_1) = b_1$.

	Hence, in the norm $\|\cdot\|_{2, \tau}$, we have strict approximations
	\begin{equation} \pi_\tau(f(a)) = f(a_1) \approx_\epsilon f(b_1) =  b_1 \approx_\epsilon a_1 = \pi_\tau(a) \approx_\epsilon \pi_\tau(b). \end{equation}
	Therefore, $\|\pi_\tau(f(a)) - \pi_\tau(b)\|_{2, \tau} < 3\epsilon$.
	This completes the proof for the case when $b$ is self-adjoint.
	
	In the general case, view $\mathcal M \otimes M_2$ as a tracially complete $C^*$-algebra as in Proposition~\ref{prop:matrix-algebras}.  By the first part of the proof, there is a self-adjoint contraction $a'' = (a_{ij}'') \in A \otimes M_2$ such that
	\begin{equation}
		\Big\| a'' - \begin{pmatrix} 0 & b \\ b^* & 0 \end{pmatrix} \Big\|_{2, X \otimes \{\mathrm{tr}_2\}} < 3 \epsilon.
	\end{equation}
	Since $a''$ is self-adjoint, $a_{21}'' = a_{12}''^*$.
	For any $\tau \in X$, we compute that 
	\begin{equation}\begin{split} 
		\|a_{12}''-b\|_{2, \tau}^2 &\leq \frac{1}{2}\tau\big(|a_{11}''|^2 + 2|a_{12}''-b|^2 + |a_{22}''|^2\big) \\
		&= \Big\| a'' - \begin{pmatrix} 0 & b \\ b^* & 0 \end{pmatrix} \Big\|_{2, \tau \otimes \mathrm{tr}_2}^2 \\
		&< (3\epsilon)^2.
	\end{split}\end{equation}
	Hence, taking $a' \coloneqq a_{12}''$, we have $\|a' - b\|_{2, X} < 3 \epsilon$.
\end{proof}

The following version of Kaplansky's density theorem follows immediately from the previous lemma.

\begin{proposition}[{cf.~\cite[Theorem 4.3.3]{Mu90}}]
	\label{prop:UnitBallDensity}
	Let $(\M,X)$ be a tracially complete $C^*$-algebra and let $A \subseteq \M$ be a $\|\cdot\|_{2,X}$-dense $C^*$-subalgebra.
	Then the unit ball of $A$ is $\|\cdot\|_{2,X}$-dense in the unit ball of $\M$.
\end{proposition}

As an application, a tracially complete $C^*$-algebra is the tracial completion of any of its $\|\cdot\|_{2,X}$-dense subalgebras. 

\begin{corollary}\label{cor:dense-subalgebra}
	Suppose $(\M, X)$ is a tracially complete $C^*$-algebra and let $A \subseteq \M$ be a $\|\cdot\|_{2, X}$-dense $C^*$-subalgebra.
	\begin{enumerate}
		\item\label{item:dense-subalg1} $X_A \coloneqq \{\tau|_A : \tau \in X\} \subseteq T(A)$ is a compact convex set and the inclusion $A \hookrightarrow \M$ induces an isomorphism $(\completion{A}{X_A}, X_A) \rightarrow (\M, X)$.
		\item\label{item:dense-subalg2} If $(\M, X)$ is factorial, then $X_A$ is a closed face in $T(A)$.
        \item\label{item:dense-subalg3} If $(\mathcal N,Y)$ is another tracially complete $C^*$-algebra and $\phi\colon A\to \mathcal N$ is a $^*$-homomorphism with $\phi^*(Y)\subseteq X_A$, then $\phi$ has a unique extension to $\overline{\phi}\colon (\M,X)\to(\mathcal N,Y)$.
    \end{enumerate}
\end{corollary}

\begin{proof}
	\ref{item:dense-subalg1}: Proposition~\ref{prop:UnitBallDensity} implies the unit ball of $A$ is $\|\cdot\|_{2, X}$-dense in the unit ball of $\mathcal{M}$.  Therefore, if $\tau \in X$, then $\tau|_A$ has norm 1, and hence $X_A \subseteq T(A)$.  As $X$ is compact and convex so too is $X_A$.  By Proposition~\ref{prop:extend-by-continuity}, the inclusion $A \hookrightarrow \M$ extends to a morphism $\theta \colon (\completion{A}{X_A}, X_A) \rightarrow (\M, X)$ of tracially complete $C^*$-algebras.  
	
	Note that $\theta$ is an embedding in the sense of Definition~\ref{def:embedding}, and hence is isometric by Proposition~\ref{prop:embedding-isometric}.  Further, the range of $\theta$ contains the $\|\cdot\|_{2, X}$-dense subalgebra $A \subseteq \mathcal{M}$.  Since the unit ball of $\completion{A}{X_A}$ is $\|\cdot\|_{2,X_A}$-complete, Kaplansky's density theorem (Proposition~\ref{prop:UnitBallDensity}) implies $\theta$ maps the unit ball of $\completion{A}{X_A}$ onto the unit ball of $\mathcal M$, and in particular $\theta$ is surjective.  Finally, given $\rho\in X_A$, then there is $\tau\in X$ such that $\rho$ is the unique $\|\cdot\|_{2,X_A}$-continuous extension to $\completion{A}{X_A}$ of $\tau|_A$.  As $\tau\circ\theta$ is $\|\cdot\|_{2,X_A}$-continuous and $(\tau\circ\theta)|_A=\rho|_A$, we have $\tau\circ\theta=\rho$, and so $\rho\circ\theta^{-1}=\tau\in X$. In this way $\theta^{-1}\colon(\mathcal M,X)\to(\completion{A}{X_A}, X_A)$ is a morphism of tracially complete $C^*$-algebras.   
    
	\ref{item:dense-subalg2}: In light of \ref{item:dense-subalg1}, this follows from Proposition~\ref{prop:tracial-completion}\ref{item:completion-factorial}.

\ref{item:dense-subalg3}: This follows from \ref{item:dense-subalg1} and Proposition \ref{prop:extend-by-continuity}.
\end{proof}

\subsection{Constructions} There is a recipe for producing constructions on tracially complete $C^*$-algebras.
\begin{enumerate}
\item Perform the corresponding (spatial) $C^*$-construction on the underlying $C^*$-algebras.
\item Identify the suitable collection of traces on the newly constructed $C^*$-algebra
\item Take the tracial completion with respect to these traces.\label{constructions3}
\end{enumerate}

We illustrate this process here with direct sums, tensor products, and sequential inductive limits. Products and reduced products (including ultraproducts) are constructed in Section \ref{sec:reduced-product}.  In the case of direct sums, step \ref{constructions3} above is redundant as the $C^*$-direct sum is already tracially complete.

\begin{definition}\label{def:UTCdirect-sums}
Let $(\M,X)$ and $(\mathcal N,Y)$ be tracially complete $C^*$-algebras.  Let $X\oplus Y$ denote the set of traces of the form $t \tau_X+(1-t)\tau_Y$ for $\tau_X\in X$, $\tau_Y\in Y$ and $0\leq t \leq 1$.  The direct sum of $(\M,X)$ and $(\mathcal N,Y)$ is the tracially complete $C^*$-algebra $(\M\oplus \mathcal N,X\oplus Y)$. 
\end{definition}

Note that the tracially complete direct sum of two finite von Neumann algebras $(\M,T(\M))$ and $(\mathcal N,T(\mathcal N))$ is the finite von Neumann algebra $(\M\oplus \mathcal N,T(\mathcal M\oplus N))$. As the extreme traces on $X\oplus Y$ are precisely the union of the extreme traces on $X$ and the extreme traces on $Y$, the direct sum $(\M,X)\oplus (\mathcal N,Y)$ is factorial if and only if both $(\M,X)$ and $(\mathcal N,Y)$ are factorial.

Next up are tensor products.

\begin{definition}\label{def:UTCtensor-products}
Let $(\M,X)$ and $(\mathcal{N},Y)$ be tracially complete $C^*$-algebras.  Let $\M \otimes \mathcal N$ denote the minimal $C^*$-tensor product of $\M$ and $\mathcal N$ and let $X \otimes Y \subseteq T(\M \otimes \mathcal N)$ be the closed convex hull of the traces $\sigma \otimes \rho$ for $\sigma \in X$ and $\rho \in Y$.  Define $\M \barotimes \mathcal{N}$ to be the tracial completion of $\M \otimes \mathcal N$ with respect to $X \otimes Y$. The tensor product $(\M,X) \barotimes (\mathcal{N},Y)$ is the pair $(\M \barotimes \mathcal N, X \otimes Y)$.
\end{definition}

\begin{proposition}
If $(\M, X)$ and $(\mathcal N, Y)$ are factorial tracially complete $C^*$-algebras, then $(\M, X) \barotimes (\mathcal N, Y)$ is also factorial.
\end{proposition}

\begin{proof}
Let $F \subseteq T(\M \otimes \mathcal N)$ be the closed face generated by $X \otimes Y$.  By Proposition~\ref{prop:tracial-completion}\ref{item:completion-factorial}, it is enough to show that $F = X \otimes Y$.  In fact, by the Krein--Milman theorem, it is enough to show $\partial_e F \subseteq X \otimes Y$.  To this end, let $\tau_0 \in \partial_e F$ be given.  Since $F$ is a closed face in $T(\M \otimes \mathcal N)$, we have $\tau_0 \in \partial_e T(\M \otimes \mathcal N)$.  By \cite[Proposition~3.5]{BBSTWW}, there are traces $\sigma_0 \in T(\M)$ and $\rho_0 \in T(\mathcal N)$ such that $\tau_0 = \sigma_0 \otimes \rho_0$.  Now, it is enough to show $\sigma_0 \in X$ and $\rho_0 \in Y$ -- we will only show the former as the latter follows by symmetry.

By Theorem~\ref{thm:exposed}\ref{item:exposed2}, to show $\sigma_0 \in X$, it is enough to show $f(\sigma_0) = 0$ for every continuous affine function $f \colon T(\M) \rightarrow [0, 1]$ with $f|_X = 0$.  Let $f$ be such a function and apply Proposition~\ref{prop:CP} to produce a self-adjoint $a \in \M$ with $\sigma(a) = f(\sigma)$ for all $\sigma \in T(\M)$.  Define
\begin{equation} 
G \coloneqq \{ \tau \in T(\M \otimes \mathcal N) : \tau(a \otimes 1_\mathcal N) = 0 \}. 
\end{equation}
For any $\sigma \in T(\M)$ and $\rho\in T(\mathcal{N})$, we have $(\sigma \otimes \rho)(a \otimes 1_\mathcal N) = f(\sigma) \geq 0$.
By \cite[Proposition~3.5]{BBSTWW} and the Krein--Milman theorem, 
\begin{equation}
    T(\M \otimes \mathcal N) = \overline{\mathrm{co}}\,\big\{\sigma \otimes \rho: \sigma \in T(\M),\ \rho \in T(\mathcal{N})\big\}.
\end{equation}
Hence $\tau(a \otimes 1_\mathcal N) \geq 0$ for all $\tau \in T(\M \otimes \mathcal N)$.  It follows that $G$ is a closed face in $T(\M \otimes \mathcal N)$.  For all $\sigma \in X$ and $\rho\in Y$, 
\begin{equation}
    (\sigma\otimes\rho)(a\otimes 1_{\mathcal N}) = \sigma(a) = f(\sigma) = 0,
\end{equation} 
and hence $X \otimes Y \subseteq G$.  Then, as $G$ is a closed face containing $X \otimes Y$, we have $F \subseteq G$, and so $\tau_0 \in G$. Since $\tau_0 =\sigma_0\otimes\rho_0$, this shows that $f(\sigma_0) = \sigma_0(a) = 0$.
\end{proof}

Finally, we construct sequential inductive limits in the category of tracially complete $C^*$-algebras.

\begin{definition}\label{UTC:Limit}
	Let 
	\begin{equation}\label{eq:inductive}
		\cdots \rightarrow (\M_n,X_n) \xrightarrow{\phi_n^{n+1}} (\M_{n+1},X_{n+1}) \xrightarrow{\phi_{n+1}^{n+2}} (\M_{n+2}, X_{n+2}) \rightarrow \cdots
	\end{equation}
	be a sequential inductive system of tracially complete $C^*$-algebras.
	Form the $C^*$-inductive limit $A \coloneqq \varinjlim (\M_n,\phi_n^{n+1})$ and let $\hat\phi_n^\infty\colon\M_n \rightarrow A$ be the canonical unital $^*$-homomorphism. The inductive system \eqref{eq:inductive} induces a projective system
\begin{equation}
	\cdots \longleftarrow X_n \xleftarrow{(\phi_n^{n+1})^*} X_{n+1} \xleftarrow{(\phi_{n+1}^{n+2})^*} X_{n+2} \longleftarrow \cdots 
\end{equation}
of compact convex sets.  Set
	\begin{equation}\label{eq:ind-lim-traces}
		X \coloneqq \{ \tau \in T(A) : \tau \circ \hat\phi_n^\infty \in X_n \text{ for all } n \geq 1 \} \subseteq T(A)
	\end{equation}
	and note that $X$ is a compact convex subset of $T(A)$.  Further, the maps $(\hat\phi_n^\infty)^* \colon X \rightarrow X_n$ induce an affine homeomorphism
\begin{equation}\label{eq:proj-lim-traces}
	 X\stackrel{\cong}{\longrightarrow}\varprojlim \big(X_n, (\phi_n^{n+1})^*\big)
\end{equation}
(and in particular, $X$ is non-empty whenever all of the $X_n$ are non-empty; see \cite[Theorem~VIII.3.5]{ES52}, for example). Define $\M\coloneqq\completion{A}{X}$ and
	\begin{equation} \varinjlim \big((\M_n,X_n),\phi_n^{n+1}\big) \coloneqq (\M,X)\end{equation}
	Finally, write $\phi_n^\infty \colon \M_n \rightarrow \M$ for the unital $^*$-homomorphism obtained by composing $\hat\phi_n^\infty$ with the canonical map $\alpha_X \colon A \to \M$.  
\end{definition}

We now show the definition of inductive limit in Definition~\ref{UTC:Limit} satisfies the required universal property and hence is an inductive limit in the category of tracially complete $C^*$-algebras.
\begin{proposition}
	\label{prop:UTCInductiveLimit}
	Let $\big((\M_n,X_n)\big)_{n=1}^\infty$ be a sequence of tracially complete $C^*$-algebras and $\big(\phi_n^{n+1}\colon (\M_n, X_n) \to (\M_{n+1}, X_{n+1})\big)_{n=1}^\infty$ be a sequence of morphisms. Set
	\begin{equation} (\M,X) \coloneqq \varinjlim \big((\M_n,X_n),\phi_n^{n+1}\big) \end{equation}
	as in Definition~\ref{UTC:Limit} and write $\phi_n^\infty \colon \M_n \rightarrow \M$ for the canonical $^*$-homo\-mor\-phisms as above.
	Then $(\M,X)$ is a tracially complete $C^*$-algebra, $\phi_n^\infty$ is a morphism for $n \geq 1$, and together, they define the inductive limit of $\big((\M_n,X_n),\phi_n^{n+1}\big)_{n=1}^\infty$ in the category of tracially complete $C^*$-algebras.
	If each $X_n$ is metrisable, then so is $X$.  Further, if each $(\M_n, X_n)$ is factorial, then so is $(\M, X)$.
\end{proposition}

\begin{proof}
	Define $A$, $X$, and $\hat\phi_n^\infty$ as in Definition~\ref{UTC:Limit}, so that $(\M, X)=(\completion{A}{X}, X)$. By Proposition~\ref{prop:tracial-completion}\ref{item:completion-complete}, $(\M,X)$ is a tracially complete $C^*$-algebra.  By construction, $(\hat\phi_n^\infty)^*(X) \subseteq X_n$, so $(\phi_n^\infty)^*(X)\subseteq X_n$ and each $\phi_n^\infty$ is a morphism.  As sequential projective limits of metrisable spaces are metrisable, $X$ is metrisable if each $X_n$ is metrisable.  If each $(\M_n, X_n)$ is factorial, then $X_n$ is a face in $T(\M_n)$ for all $n \geq 1$.  It follows from \eqref{eq:ind-lim-traces} that $X$ is a face in $T(A)$. Proposition~\ref{prop:tracial-completion}\ref{item:completion-factorial} then implies that $(\M, X)$ is factorial.
	
	To show that $(\M,X)$ is the inductive limit of $\big((\M_n,X_n),\phi_n^{n+1}\big)_{n=1}^\infty$, suppose we have a tracially complete $C^*$-algebra $(\mathcal N,Y)$ and morphisms 
    \begin{equation}
        \psi_n \colon(\M_n,X_n) \to (\mathcal{N},Y), \qquad n \in \mathbb{N},
    \end{equation}
    such that $\psi_{n+1} \circ \phi_n^{n+1} = \psi_n$.  Then, working in the category of unital $C^*$-algebras, we obtain a unique unital $^*$-homomorphism $\hat\psi \colon A \to \mathcal{N}$ such that $\psi_n = \hat{\psi} \circ \hat\phi_n^\infty$ for each $n\in\mathbb N$.  We have
	\begin{equation} \tau \circ \hat\psi \circ \hat\phi_n^\infty = \tau \circ \psi_n \in X_n, \quad \tau \in Y,\ n \geq 1,
    \end{equation}
	so by \eqref{eq:ind-lim-traces}, $\tau \circ \hat\psi \in X$.  This shows $\hat\psi^*(Y) \subseteq X$.  Accordingly, 
	there is a unique morphism $\psi\colon (\M,X)\rightarrow(\mathcal N,Y)$ with $\hat\psi=\psi\circ \alpha_X$ by Proposition~\ref{prop:extend-by-continuity}.
	Moreover
	\begin{equation}
        \psi_n=\hat{\psi}\circ\hat{\phi}_n^\infty=\psi\circ\alpha_X \circ\hat\phi_n^\infty=\psi\circ\phi_n^\infty,
	\end{equation}
	as required.  Uniqueness of $\psi$ follows from uniqueness in Proposition~\ref{prop:extend-by-continuity} as any morphism $\psi\colon (\M,X) \to (\mathcal{N},Y)$ with $\psi\circ\phi_n = \psi_n$ for all $n$ satisfies $\psi\circ\alpha_X = \hat\psi$.
\end{proof}

For each metrisable Choquet simplex $X$, we now use an inductive limit construction to produce the concrete model $(\mathcal R_X, X)$ of a tracially complete $C^*$-algebra covered by Theorem~\ref{InformalClassification} as discussed in the overview of results.  This is achieved by mimicking construction of a simple AF algebra $A$ for which $T(A)=X$, found in~\cite{Bl80,Go77}, and can be deduced from these results by considering the tracial completion of such an AF algebras.  The details are slightly easier in the tracially complete setting since $K_0(\mathcal{R}) \cong \mathbb{R}$ and one need not worry about simplicity.

\begin{example}\label{Ex:ConcreteModels}
	Let $X$ be a metrisable Choquet simplex. By~\cite[Theorem~11.6]{Go86}, we can write $X$ as the inverse limit of a system of finite dimensional simplices
	\begin{equation} X_1 \stackrel{\alpha_2^{1}}\longleftarrow X_2 \stackrel{\alpha_3^{2}}\longleftarrow \cdots, \end{equation}
	where the connecting maps are continuous and affine. We will construct an inductive system of tracially complete $C^*$-algebras realising this data.
	
	Set $\M_n \coloneqq \mathcal R^{\oplus \, \partial_e X_n}$,
	where $\mathcal R$ is the hyperfinite II$_1$ factor. Note that $T(\M_n)$ can be canonically identified with $X_n$ and $\big(\M_n,T(\M_n)\big)$ is tracially complete and factorial.  Next, we can choose a $^*$-homomorphism $\phi_n^{n+1}\colon \M_n \to \M_{n+1}$ that induces the map $\alpha_{n+1}^{n}$.  To do this explicitly, write $\partial_e X_{n+1} = \{x_1,\dots,x_k\}$ and $\partial_e X_n = \{y_1,\dots,y_l\}$.
	For each $i=1,\dots,k$, write $\alpha_{n+1}^{n}(x_i)$ as the convex combination $\sum_{j=1}^l t_{i,j} y_j$.
	Fixing $i$, find a partition of unity $p_{i,1},\dots,p_{i,l} \in \mathcal R$ of projections such that $\tau(p_{i,j}) = t_{i,j}$ for each $j$.
	For each $i$ and $j$, choose any unital $^*$-homomorphism $\psi_{i,j}\colon \mathcal R \to p_{i,j}\mathcal R p_{i,j}$ (which exists since $\mathcal R$ has full fundamental group, a result which goes back to Murray and von Neumann in \cite{MvN43}).
	Then define $\phi_n^{n+1}\colon \mathcal{R}^{\oplus l} \to \mathcal{R}^{\oplus k}$ by
	\begin{equation} \phi_n^{n+1}(a_1\oplus \cdots \oplus a_l) \coloneqq \Big(\sum_{j=1}^l \psi_{1,j}(a_j),\dots,\sum_{j=1}^l \psi_{k,j}(a_j)\Big), 
	\end{equation}
	and note that $(\phi_n^{n+1})^* = \alpha_{n+1}^n$.
	  
	Define $(\R_X, X')\coloneqq \varinjlim \big((\M_n,X_n),\phi_n^{n+1}\big)$.  As $(\phi_n^{n+1})^* = \alpha_{n+1}^n$, \eqref{eq:proj-lim-traces} provides an isomorphism $X \rightarrow X'$, so after identifying $X$ with $X'$ via this isomorphism, $(\R_X, X)$ is the desired factorial tracially complete $C^*$-algebra.
\end{example}

A priori, the construction outlined above depends not only on the Choquet simplex $X$ but also on the choice of inverse limit and the choices made when defining connecting maps. However, it will follow from Theorem~\ref{thm:classification} that the tracially complete $C^*$-algebra $(\R_X,X)$ depends only on $X$.  Moreover, as $X$ varies over all metrisable Choquet simplices, these will provide models for the classifiable tracially complete $C^*$-algebras (see Theorem \ref{thm:hyperfinite}, which contains Theorem \ref{InformalClassification}).

\begin{remark}
    The infinite tensor product $\bigotimes_{n=1}^\infty (\M_n,X_n)$ of a countable family $(\M_n,X_n)$ of tracially complete $C^*$-algebras can now be constructed as the inductive limit of the finite tensor products $\bigotimes_{n=1}^N(\M_n,X_n)$ with the obvious connecting maps.
\end{remark}

\subsection{\texorpdfstring{$W^*$}{W*}-bundles}\label{sec:WStarBundles}

In Proposition~\ref{prop:WStarBundleToTC}, we showed that Ozawa's $W^*$-bundles can be viewed as tracially complete $C^*$-algebras.  Here we show that, by a reformulation of a theorem due to Ozawa (\cite{Oz13}), any factorial tracially complete $C^*$-algebra $(\mathcal M, X)$ with $X$ a Bauer simplex can be given the structure of a $W^*$-bundle whose fibres are factors.  

Recall from Section~\ref{sec:simplices} that if $X$ is a Choquet simplex then for every $\tau \in X$ there is a \emph{unique} determined Radon probability measure $\mu_X$ on $X$ supported on $\partial X$ with barycentre $\tau$. As discussed in Section~\ref{sec:simplices}, the meaning of supported on $\partial X$ is slightly subtle in general as $\partial X$ may not be a Borel set; however, when $X$ is a Bauer simplex, this is not an issue and there is an affine homeomorphism
\begin{equation}
\label{eq:mu_tau-def}
	X \overset\cong\longrightarrow \mathrm{Prob}(\partial_e X) 
\end{equation}
given by $\tau \mapsto \mu_\tau$ where $\mu_\tau$ is the Choquet measure of $\tau$ restricted to $\partial_e X$.

\begin{theorem}[{cf.\ \cite[Theorem 3]{Oz13}}]\label{thm:TCtoWStarBundle}
	If $(\mathcal M, X)$ is a factorial tracially complete $C^*$-algebra such that $X$ is a Bauer simplex with $K \coloneqq \partial_e X$, then there is a unique embedding $C(K) \subseteq Z(\mathcal M)$ such that
	\begin{equation}\label{eq:central-embedding}
		\tau(fa) = \int_K f(\sigma)\sigma(a) \, \mathrm{d} \mu_\tau(\sigma), \qquad f \in C(K),\ a \in \mathcal M,\ \tau \in X.
	\end{equation}
 	 Further, the map $E \colon \mathcal M \rightarrow C(K)$ given by $E(a)(\tau) \coloneqq \tau(a)$ is a conditional expectation endowing $\mathcal M$ with the structure of a $W^*$-bundle.
\end{theorem}

Ozawa's version of Theorem~\ref{thm:TCtoWStarBundle} has an additional assumption that $X$ is metrisable.\footnote{Ozawa's work predates the formalism of tracially complete $C^*$-algebras. In \cite[Theorem 3]{Oz13}, take $A\coloneqq \mathcal{M}$ and $S\coloneqq X$ so that $\mathcal M$ coincides with the strict closure of $A$ in \cite[Theorem 3]{Oz13}. The condition that $S$ is a Bauer simplex ensures that $\theta(C(\partial_e S))$ lands in the strict closure and so in $\mathcal Z(\mathcal M)$ (as explained in the paragraphs before \cite[Theorem 3]{Oz13}) recapturing the statement given here for metrisable $X$. The factorality hypothesis of Theorem \ref{thm:TCtoWStarBundle} appears in Ozawa's work by requiring $S$ to be a closed face of $T(A)$.}  With some minor technical variations of Ozawa's proof, the result extends to the non-metrisable case.  For the convenience of the reader, we devote most of the rest of this section to giving a complete proof and include some of the details omitted in \cite{Oz13}.   Some of the preliminary lemmas will be stated in greater generality than Theorem~\ref{thm:TCtoWStarBundle} as they require no more work and might be useful in other situations. However we emphasise that the proof is essentially that of \cite{Oz13}.

Until the end of the proof of Theorem~\ref{thm:TCtoWStarBundle} we assume $\mathcal M$ is a $C^*$-algebra and $X \subseteq T(\mathcal M)$ is a compact face -- note that this holds when $(\mathcal M, X)$ is a factorial tracially complete $C^*$-algebra.  

For $\tau \in X$, let $\pi_\tau \colon \mathcal M \rightarrow \mathcal B(\mathcal H_\tau)$ denote the GNS representation, with canonical cyclic vector $\xi_\tau \in \mathcal H_\tau$.  Let $\hat \tau$ denote the faithful normal trace on $\pi_\tau(\mathcal M)''$ given by 
\begin{equation}
\label{eq:hattau-def}
\hat \tau(a) \coloneqq \langle a \xi_\tau, \xi_\tau \rangle,\quad a\in \pi_\tau(\mathcal M)'',
\end{equation}
so $\hat \tau \circ \pi_\tau = \tau$.  Our first goal is to compute the centre of $\pi_\tau(\mathcal M)''$.

\begin{lemma}[{cf.\ \cite[Lemma~10]{Oz13}}]\label{lem:oz-local}
    Suppose $\mathcal M$ is a $C^*$-algebra, $X \subseteq T(\mathcal M)$ is a compact face, and $\tau \in X$. Let $\mu_\tau \in \mathrm{Prob}(X)$ be the Choquet measure for $\tau$ supported on $\partial X$, and let $\hat \tau$ denote the faithful normal trace on $\pi_\tau(\mathcal M)''$ as in \eqref{eq:hattau-def}.  Then there is a normal isomorphism
    $\hat\theta_\tau\colon L^\infty(X, \mu_\tau) \rightarrow Z(\pi_\tau(\M)'')$ such that
    \begin{equation}\label{oz-1}
        \hat\tau\big(\hat\theta_\tau(f)\pi_\tau(a)\big) = \int_{X} f(\sigma)\sigma(a) \,{\rm d}\mu_\tau(\sigma), \quad a \in \mathcal M,\ f\in L^\infty(X,\mu_\tau).
    \end{equation}
\end{lemma}

\begin{proof}
    It suffices to construct a unital positive linear invertible map $\hat\theta_\tau$ such that $\hat\theta_\tau^{-1}$ is positive and \eqref{oz-1} holds.  Indeed, positive linear maps between commutative $C^*$-algebras are completely positive, and it is a standard consequence of the Schwarz inequality that a complete order isomorphism between $C^*$-algebras is multiplicative.  Finally, isomorphisms between von Neumann algebras are automatically normal (see \cite[Corollary 1 in Part I, Chapter 4, Section 3]{Di69EngTrans}).

    Let $f \in L^\infty(X, \mu_\tau)_+$.  
    We may define a positive tracial functional on $\M$ by setting $\tau_f(a)$ equal to the right-hand side of \eqref{oz-1}.  We have
    \begin{equation}
  \tau_f(a)\leq \|f\|_\infty\int_X\sigma(a) \,{\rm d}\mu_\tau(\sigma)=\|f\|_\infty\tau(a),\quad a\in \mathcal M_+,
    \end{equation} so that $\tau_f \leq \|f\|_{\infty}\tau$.  It follows that $\tau_f$ extends to a positive normal tracial functional $\hat\tau_f$ on $\pi_\tau(\mathcal M)''$ with  $\hat\tau_f\leq\|f\|_\infty\hat{\tau}$.  Define \begin{equation}
        \hat\theta_\tau(f) \coloneqq \frac{{\rm d} \hat\tau_f}{{\rm d} \hat\tau} \in Z(\pi_\tau(\mathcal M)'')_+
    \end{equation}
 using the Radon--Nikodym result recalled as Theorem~\ref{thm:radon-nikodym} and note that \eqref{oz-1} holds by construction.  Since $\hat\theta_\tau(f)$ is uniquely determined by $f$, the linearity of the right-hand side of \eqref{oz-1} ensures that $\hat\theta_\tau$ preserves addition and multiplication by non-negative scalars. Therefore, we can extend $\hat\theta_\tau$ to a $\mathbb R$-linear map $L^\infty(X, \mu_\tau)_\mathrm{sa} \rightarrow Z(\M)_\mathrm{sa}$ by setting 
    \begin{equation}
        \hat\theta_\tau(f) \coloneqq \hat\theta_\tau(f_+)-\hat\theta_\tau(f_-) \in Z(\pi_\tau(\M)''),
    \end{equation}
    where $f_+ \coloneqq \max(f,0)$ and $f_- \coloneqq \max(-f,0)$, and then to a linear map $L^\infty(X, \mu_\tau) \rightarrow Z(\M)$.
    Hence $\hat\theta_\tau$ is a well-defined positive linear map, and it is obviously unital.

    We now prove injectivity. 
    Suppose $f \in L^\infty(X, \mu_\tau)$ and $\hat\theta_\tau(f) = 0$. Without loss of generality, assume $f$ is real-valued.
    Write $f = f_+ - f_-$, where $f_+ \coloneqq \max(f,0)$ and $f_- \coloneqq \max(-f,0)$. By applying \eqref{oz-1} to an approximate unit in $\mathcal M$, we have $\int_X f \, \rm{d} \mu_\tau$ = 0. 
    Set 
    \begin{equation}
        \alpha \coloneqq  \int_X f_+ \, {\rm d}\mu_\tau = \int_X f_- \, {\rm d}\mu_\tau.
    \end{equation}
    If $\alpha = 0$, then $f = 0$ in $L^\infty(X, \mu_\tau)$, as required. Suppose $\alpha \neq 0$.    Then $\alpha^{-1} f_+ \mu_\tau$ and $\alpha^{-1} f_- \mu_\tau$ are Radon probability measures on $X$ supported on $\partial X$ (as $\mu_\tau$ is supported on $\partial X$) giving rise to the traces
     \begin{align}
    \tau_+(a)&\coloneqq\frac{1}{\alpha}\int_X f_+(\sigma)\sigma(a)\,{\rm d}\mu_\tau, && a\in \mathcal M 
    \shortintertext{and} \tau_-(a)&\coloneqq \frac{1}{\alpha}\int_X f_-(\sigma)\sigma(a)\,{\rm d}\mu_\tau, &&a\in\mathcal M.
    \end{align}
    As $\hat{\theta}_\tau(f)=0$, the traces $\tau_+$ and $\tau_-$ are equal. Since $X$ is a compact face in $T(\mathcal M)$, $X$ is a Choquet simplex by Theorem~\ref{thm:traces-choquet}. Accordingly, $\tau_+=\tau_-$ has a unique representing measure, so that $\alpha^{-1} f_+ \mu_\tau = \alpha^{-1} f_- \mu_\tau$. Then, $f_+ =  f_-$ holds $\mu_\tau$-a.e, so $f = 0 $ in  $L^\infty(X, \mu_\tau)$.

    Finally, we prove that $\hat\theta_\tau$ is surjective with a positive inverse. Suppose $z \in Z(\pi_\tau(\M)'')_+ \setminus\{0\}$.  
    Then $\tau_z(a) \coloneqq \hat\tau(z\pi_\tau(a)) / \hat\tau(z)$ defines a trace on $\mathcal M$ with $\tau_z \leq \|z\| \hat{\tau}(z)^{-1}\tau$.
    Since $\tau \in X$ and $X$ is a face in $T(\M)$, we have $\tau_z \in X$. As $X$ is a Choquet simplex, there is a unique Radon probability measure $\mu_{\tau_z}$ supported on $\partial X$ representing $\tau_z$. Since $\tau_z \leq \|z\|\hat{\tau}(z)^{-1}\tau$, the classical Radon--Nikodym theorem implies that $\mu_{\tau_z} = f \mu_\tau$ for some $f \in L^\infty(X, \mu_\tau)_+$.  For $a \in \mathcal M$, we have
    \begin{equation}
    \begin{array}{rcl}
        \hat\tau\big(\hat\theta_\tau(f)\pi_\tau(a)\big) &\stackrel{\eqref{oz-1}}{=}& \int_X f(\sigma) \sigma(a) \, {\rm d}\mu_\tau(\sigma) \\
        &=& \int_X \sigma(a) \, {\rm d}\mu_{\tau_z}(\sigma) \\
        &=& \tau_z(a) \\
        &=& \frac{\hat\tau\big(z\pi_\tau(a)\big)}{\hat\tau(z)}.
    \end{array}
    \end{equation}
    Hence $\hat\theta_\tau(\hat{\tau}(z)f) = z$, which shows $\hat\theta_\tau$ is surjective.  Since $\hat\tau(z)f \geq 0$, this also shows $\hat\theta_\tau^{-1}$ is positive.
\end{proof}

Our next goal is to combine the maps $\hat\theta_\tau$ for $\tau \in X$ into a single map $\theta$ from the algebra $B(X)$ of bounded Borel functions $X \rightarrow \mathbb C$ to a suitable von Neumann algebra.  As in \eqref{eq:GNS-X}, define
\begin{equation}
    \pi_X \coloneqq \bigoplus_{\tau \in X} \pi_\tau \colon \mathcal M \rightarrow \mathcal B\Big(\bigoplus_{\tau \in X} \mathcal H_\tau \Big)
\end{equation}
to be the product of the GNS representations of $\mathcal M$ over $\tau \in X$.  For each $\tau \in X$, compressing to the summand $\mathcal H_\tau$ produces a normal representation $\bar \pi_\tau \colon \pi_X(\mathcal M)'' \rightarrow \mathcal B(\mathcal H_\tau)$.  Then $\bar\pi_\tau \circ \pi_X = \pi_\tau$.  Also, for $\tau \in X$
\begin{equation}\label{eq:bartau-def}
    \bar\tau(a) \coloneqq \langle \bar\pi_\tau(a) \xi_\tau, \xi_\tau \rangle, \quad a \in \pi_X(\mathcal M)'',
\end{equation}
defines a normal trace on $\pi_X(\mathcal M)''$ with $\bar \tau \circ \pi_X = \tau$, where $\xi_\tau$ is the canonical cyclic vector for $\pi_\tau$ (that is, $\bar\tau = \hat\tau\circ\bar\pi_\tau$, using $\hat\tau$ from \eqref{eq:hattau-def}).

Identifying $\pi_\tau(\mathcal M)''$ with a direct summand of $\pi_X(\mathcal M)''$ and viewing $L^\infty(X, \mu_\tau)$ as a quotient of $B(X)$, the maps $\hat\theta_\tau$ induce maps $\theta_\tau \colon B(X) \rightarrow \pi_X(\mathcal M)''$.  The next lemma shows these maps can be glued together to form a single map $\theta \colon B(X) \rightarrow Z(\pi_X(\mathcal M)'')$. Note that in the proof, if $X$ is small enough that there is a $\tau\in X$ which induces a faithful trace $\bar{\tau}$ on $\pi_X(\M)''$, then the limit argument is not needed and one can just work with the map $\theta_\tau$ (from \eqref{defthetatau-oz2}).

\begin{lemma}[{cf.\ \cite[Theorem~3]{Oz13}}]\label{lem:oz-global}
    Let $\mathcal M$ be a $C^*$-algebra and let $X \subseteq T(\mathcal M)$ be a compact face. For each $\tau \in X$,  let $\mu_\tau \in \mathrm{Prob}(X)$ be the Choquet measure for $\tau$ supported on $\partial X$, and let $\bar \tau$ denote the induced normal trace on $\pi_X(\mathcal M)''$ as in \eqref{eq:bartau-def}. Then there is a $^*$-homomorphism $\theta \colon B(X) \rightarrow Z(\pi_X(\mathcal M)'')$ such that
    \begin{equation}\label{oz-2}
        \bar\tau\big(\theta(f)\pi_X(a)\big) = \int_X f(\sigma) \sigma(a) \,{\rm d}\mu_\tau(\sigma)
    \end{equation}
    for all $\tau \in X$, $f \in B(X)$, and $a \in \mathcal M$.
\end{lemma}

\begin{proof}
    For notational convenience, set $\mathcal N \coloneqq \pi_X(\mathcal M)''$.  For each $\tau \in X$, compressing $\mathcal N$ to the direct summand $\mathcal H_\tau$ provides a normal representation $\bar\pi_\tau \colon \mathcal N \rightarrow \mathcal B(\mathcal H_\tau)$.  
    
    Let $p_\tau \in Z(\mathcal N)$ be the support of $\bar\pi_\tau$ -- i.e.\ the projection $1 - p_\tau$ is the unit of $\ker(\bar\pi_\tau)$.  Note that $p_\tau$ is also the support projection of $\bar{\tau}$.

    For each $\tau \in X$, the representation $\bar\pi_\tau$ (co)restricts to an isomorphism $p_\tau \mathcal N \rightarrow \pi_\tau(\mathcal M)''$.  Hence the map $\hat\theta_\tau$ of Lemma~\ref{lem:oz-local} induces a $^*$-homomorphism $\theta_\tau \colon B(X) \rightarrow \mathcal N$ such that 
    \begin{equation}\label{defthetatau-oz2}
        \theta_\tau(1_{B(X)}) = p_\tau \qquad \text{and}  \qquad \bar\pi_\tau(\theta_\tau(f)) = \hat\theta_\tau(f)
    \end{equation} 
    for $f \in B(X)$, where the $f$ on the right-hand side of the second equation denotes the class of $f$ in $L^\infty(X, \mu_\tau)$.  Then
    \begin{equation}\label{oz-neweq}
        \bar\tau\big(\theta_\tau(f)\pi_X(a)) = \int_X f(\sigma) \sigma(a) \, {\rm d}\mu_\tau(\sigma), \quad f \in B(X),\ a \in \mathcal M.
    \end{equation}
    We will define $\theta$ as a pointwise limit of the maps $\theta_\tau$.

    Define a partial order on $X$ by $\tau_1 \preceq \tau_2$ if there is a constant $C \geq 0$ such that $\tau_1\leq C \tau_2$.  Then $\preceq$ is directed since for any $\tau_1, \tau_2 \in X$, the element $\frac12(\tau_1 + \tau_2) \in X$ is a common upper bound.  If $\tau_1 \preceq \tau_2$, then $\bar\tau_2(p_{\tau_1}) > 0$, 
    and hence $p_{\tau_1} \leq p_{\tau_2}$.  Therefore, $(p_\tau)_{\tau \in X}$ is an increasing net of central projections in $Z(\mathcal N)$.  Further, if $p \coloneqq \sup_{\tau \in X} p_\tau \in Z(\mathcal N)$, then $\bar\tau(p_\tau) = 1$ for all $\tau \in X$.  
    As $\{\bar \tau : \tau \in X\}$ is a faithful set of traces on $\mathcal N$, it follows that $p = 1_\mathcal N$, so the $p_\tau$ converge strongly to $1_\mathcal N$.

    Suppose $\tau_1, \tau_2 \in X$ with $\tau_1 \preceq \tau_2$.  We work to show $\theta_{\tau_1} = p_{\tau_1} \theta_{\tau_2}$.  As both $^*$-homomorphisms send the unit to $p_{\tau_1}$, it suffices to show 
    \begin{equation}\label{oz-neweq3}
        \bar\tau_1\big(\theta_{\tau_1}(f)\pi_X(a)\big) = \bar\tau_1\big(\theta_{\tau_2}(f)\pi_X(a)\big), \quad f \in B(X),\ a \in \mathcal M.
    \end{equation}
    To this end, first note that $\tau_1 \preceq \tau_2$ implies $\mu_{\tau_1}$ is bounded by a constant multiple of $\mu_{\tau_2}$ as measures on $X$,\footnote{Say $\tau_1\leq K\tau_2$ with $K>1$, and set $\sigma\coloneqq (K\tau_2-\tau_1)/(K-1)$. Since $X$ is a face and $\tfrac{1}{K}\tau_1 + \left(1-\tfrac{1}{K}\right)\sigma = \tau_2$, we have $\sigma \in X$. Then $\mu_{\tau_1} \leq \mu_{\tau_1} + (K-1)\mu_\sigma = K\mu_{\tau_2}$.} and so is absolutely continuous with respect to $\mu_{\tau_2}$. Let $\frac{{\rm d} \mu_{\tau_1}}{{\rm d} \mu_{\tau_2}} \in B(X)$ denote any representative of the (classical) Radon--Nikodym derivative in $L^\infty(X, \mu_\tau)$.  Note that for all $a \in \mathcal M$,
    \begin{equation}\label{oz-compatible}
    \begin{array}{rcl}
        \bar\tau_2\Big(\theta_{\tau_2}\Big(\frac{{\rm d} \mu_{\tau_1}}{{\rm d} \mu_{\tau_2}}\Big)\pi_X(a)\Big) &\stackrel{\eqref{oz-neweq}}=& 
        \int_X \frac{{\rm d} \mu_{\tau_1}}{{\rm d} \mu_{\tau_2}}(\sigma) \sigma(a) \,{\rm d}\mu_{\tau_2}(\sigma)\\
        &=& \int_X \sigma(a) \,{\rm d} \mu_{\tau_1}(\sigma) \\
        &=& \bar\tau_1(\pi_X(a)).
    \end{array}
    \end{equation}
    Further, $\theta_{\tau_2}\big(\frac{{\rm d} \mu_{\tau_1}}{{\rm d} \mu_{\tau_2}}\big)$ has support contained in $p_{\tau_2} = \theta_{\tau_2}(1_{B(X)})$.
    Therefore, by the uniqueness in the non-commutative Radon--Nikodym theorem (Theorem~\ref{thm:radon-nikodym}), 
\begin{equation}\label{oz-neweq2}
    \frac{{\rm d} \bar \tau_1}{{\rm d} \bar \tau_2} = \theta_{\tau_2}\left(\frac{{\rm d} \mu_{\tau_1}}{{\rm d} \mu_{\tau_2}}\right).
    \end{equation} 
    Now, for all $f \in B(X)$ and $a \in \mathcal M$,
    \begin{equation}
    \begin{array}{rcl}
        \bar\tau_1\big(p_{\tau_1}\theta_{\tau_2}(f)\pi_X(a)\big) &=&\bar\tau_1\big(\theta_{\tau_2}(f)\pi_X(a)\big)\\
        &=&\bar\tau_2\Big(\frac{{\rm d} \bar\tau_1}{{\rm d} \bar \tau_2} \theta_{\tau_2}(f)\pi_X(a)\Big) \\
        &\stackrel{\eqref{oz-neweq2}}{=}& \bar\tau_2\Big(\theta_{\tau_2}\Big(\frac{{\rm d} \mu_{\tau_1}}{{\rm d} \mu_{\tau_2}}\Big)\theta_{\tau_2}(f)\pi_X(a)\Big) \\
        &=& \bar\tau_2\Big(\theta_{\tau_2}\Big(\frac{{\rm d} \mu_{\tau_1}}{{\rm d} \mu_{\tau_2}} f\Big)\pi_X(a)\Big) \\
        &\stackrel{\eqref{oz-neweq}}{=}& \int_X \frac{{\rm d} \mu_{\tau_1}}{{\rm d} \mu_{\tau_2}}(\sigma) f(\sigma) \sigma(a) \, {\rm d} \mu_{\tau_2}(\sigma) \\
        &=& \int_X f(\sigma) \sigma(a) \, {\rm d}\mu_{\tau_1}(\sigma) \\
        &\stackrel{\eqref{oz-neweq}}{=}& \bar\tau_1\big( \theta_{\tau_1}(f)\pi_X(a)\big),
    \end{array}
    \end{equation}
    which verifies \eqref{oz-neweq3} and hence shows $p_{\tau_1} \theta_{\tau_2} = \theta_{\tau_1}$.

    For each $\tau_0 \in X$, $f \in B(X)$, and $a \in \mathcal M$, the net $\big(\bar\tau_0(\theta_\tau(f)\pi_X(a))\big)_{\tau \in X}$ is eventually constant -- in fact, by the previous paragraph, the net is constant as soon as $\tau_0 \preceq \tau$.  Therefore, for each $\tau_0 \in X$, $f \in B(X)$, and $b \in \mathcal N$, the net $\big(\bar\tau_0(\theta_\tau(f)b)\big)_{\tau \in X}$ converges, and hence the net $\big(\theta_\tau(f)\big)_{\tau \in X}$ in $Z(\mathcal N)$ converges strongly.  Let $\theta(f) \in Z(\mathcal N)$ denote the strong limit.  As multiplication on $\mathcal N$ is strongly continuous on $\|\cdot\|$-bounded sets and the adjoint on $\mathcal N$ is strongly continuous (since $\mathcal N$ is finite), the map $\theta$ is a $^*$-homomorphism.  Further, it is unital as $\theta_\tau(1_{B(X)}) = p_\tau \rightarrow 1_{\mathcal N}$.  By construction, \eqref{oz-2} holds.
\end{proof}

Suppose now that $X$ is a Bauer simplex.  In this case, $\partial_e X$ is compact and hence Borel.  Further, if $\mu$ is a Radon probability measure on $X$ supported on $\partial_e X$, then $\mu(X \setminus \partial_e X) = 0$.\footnote{Using the inner regularity of $\mu$, it is enough to show that every compact set $K \subseteq X \setminus \partial_e X$ satisfies $\mu(K) = 0$.  
Given such a set $K$, by Urysohn's lemma, there is a positive contraction $f \in C(K)$ vanishing on $\partial_e X$ with $f(\tau) = 1$ for all $\tau \in K$.  Then $E \coloneqq f^{-1}(1)$ is a Baire measurable subset of $X$ containing $K$ with $E \cap \partial_e X = \emptyset$.  As $\mu$ is supported on $\partial_e X$, we have $\mu(E) = 0$, and hence $\mu(K) = 0$.}  This allows us to take the integrals in \eqref{oz-1} and \eqref{oz-2} over $\partial_e X$ in place of $X$.  Further, in Lemma~\ref{lem:oz-local} we may replace $L^\infty(X, \mu_\tau)$ with $L^\infty(\partial_e X, \mu_\tau)$, and in Lemma~\ref{lem:oz-global}, we may replace $B(X)$ with $B(\partial_e X)$.  In particular, restricting the map $\theta$ in Lemma~\ref{lem:oz-global} to $C(\partial_e X) \subseteq B(\partial_e X)$ provides the following result, which is what will be used in the proof of Theorem~\ref{thm:TCtoWStarBundle}.

\begin{corollary}\label{cor:oz-continuous}
    Let $\mathcal M$ be a $C^*$-algebra, let $X \subseteq T(\mathcal M)$ be a compact face, and for each $\tau \in X$, let $\bar \tau$ denote the induced normal trace on $\pi_X(\mathcal M)''$.  If $X$ is a Bauer simplex, then there is an embedding $\theta \colon C(\partial_e X) \rightarrow Z(\pi_X(\mathcal M)'')$ such that 
    \begin{equation}\label{oz-3}
        \bar\tau\big(\theta(f)\pi_X(a)\big) = \int_{\partial_e X} f(\sigma) \sigma(a) \,{\rm d}\mu_\tau(\sigma)
    \end{equation}
    for all $\tau \in X$, $f \in C(\partial_e X)$, and $a \in \mathcal M$.
\end{corollary}

\begin{proof}
    As mentioned above the corollary, the existence of $\theta$ follows by first noting that the map $\theta$ in Lemma~\ref{lem:oz-global} factors through $B(\partial_e X)$ as $\mu_\tau(X \setminus \partial_e X) = 0$ for all $\tau \in X$ and then restricting the induced map from $B(\partial_e X)$ to $C(\partial_e X)$.  
    
    The remaining claim is that $\theta$ is injective.  To this end, fix $f \in \ker(\theta)$ and note that by taking $a$ to be an approximate unit in \eqref{oz-3} and limiting, we have $\int_{\partial_e X} f(\sigma) \, {\rm d} \mu_\tau(\sigma) = 0$ for all $\tau \in X$.  In particular, restricting to $\tau \in \partial_e X$ implies $f(\tau) = 0$ for all $\tau \in \partial_e X$, and hence $f = 0$.
\end{proof}

We can now give the proof of Theorem~\ref{thm:TCtoWStarBundle}.  Note that factoriality ensures that $X$ is a face in $T(\M)$, allowing the use of Choquet theory in Lemma \ref{lem:oz-local} for both injectivity and surjectivity of the local maps $\hat{\theta_\tau}\colon L^\infty(X,\mu_\tau)\to Z(\pi_\tau(\M)'')$.

\begin{proof}[Proof of Theorem~\ref{thm:TCtoWStarBundle}]
	We first show the embedding of $C(K)$ is unique, assuming it exists.  Let $\phi, \psi \colon C(K) \rightarrow Z(\mathcal M)$ be two embeddings satisfying \eqref{eq:central-embedding}.  Fix $f \in C(K)$.  If $\tau \in K$, then $\mu_\tau$ is the point mass at $\tau$, so \eqref{eq:central-embedding} gives
	\begin{equation}
		\tau(\phi(f) \psi(f)) = f(\tau) \tau(\psi(f)) = f(\tau)^2.
	\end{equation}
	Similarly, $\tau(\phi(f)^2) = \tau(\psi(f)^2) = f(\tau)^2$.  Now, if $f \in C(K)$ is self-adjoint, we have
	\begin{equation}
		\|\phi(f) - \psi(f)\|_{2, \tau}^2 = \tau\big((\phi(f) - \psi(f))^2\big) = 0
	\end{equation}
	for all $\tau \in \partial_e X$, and hence for all $\tau \in X$.  Therefore, $\phi(f) = \psi(f)$ for all self-adjoint $f \in C(K)$, and hence for all $f \in C(K)$.

We now turn to existence of the embedding $C(\partial_e X) \rightarrow Z(\mathcal M)$.
Let $\mathcal N \coloneqq \pi_X(\mathcal M)''$.  For $\tau \in X$, let $\bar\tau \in T(\mathcal N)$ denote the canonical extension of $\tau$ to a normal trace on $\mathcal N$ as defined in \eqref{eq:bartau-def}, and let $\bar X \coloneqq \{\bar \tau : \tau \in X \} \subseteq T(\mathcal N)$.  Then $\|\pi_X(a)\|_{2, \bar X} = \|a\|_{2, X}$ for all $a \in \mathcal M$.  Let $\theta \colon C(\partial_e X) \rightarrow Z(\mathcal N)$ be the embedding given by Corollary~\ref{cor:oz-continuous} (which applies as $(\M,X)$ is factorial). Since $X$ is a faithful set of traces on $\M$, our objective is to show that $\theta$ takes values in $\pi_X(\M)$, so that it gives rise to the required embedding.  To this end, fix $f \in C(\partial_e X)$ and $\epsilon > 0$. As $(\M,X)$ is tracially complete, it will suffice to show there exists $a \in \mathcal M$ with $\|a\| \leq \|f\|$ and $\|\pi_X(a) - \theta(f)\|_{2, \bar X} < \epsilon$.

As $\pi_X(\mathcal M)$ is strongly dense in $\mathcal N$, Kaplansky's density theorem provides a net $(a_\lambda) \subseteq \mathcal M$ with $\pi_X(a_\lambda) \rightarrow \theta(f)$ strongly and $\|a_\lambda\| \leq \|f\|$ for all $\lambda$.  As each $\bar \tau \in \bar X$ is normal, we also have $\|\pi_X(a_\lambda) - \theta(f)\|_{2, \bar \tau}^2 \rightarrow 0$.  Define $h_\lambda \colon X \rightarrow \mathbb R$ by 
\begin{equation}
    h_\lambda(\tau) = \|\pi_X(a_\lambda) - \theta(f)\|_{2, \bar\tau}^2, \quad \tau \in X. 
\end{equation}
Then $h_\lambda$ is affine.  Further, for all $\tau \in X$, we have
\begin{equation}\label{oz-neweq4}
\begin{alignedat}{2}
    h_\lambda(\tau) &= &\, &\bar\tau\big(\pi_X(a_\lambda)^*\pi_X(a_\lambda)\big) - 2 \mathrm{Re} \,\bar\tau\big(\theta(f)^* \pi_X(a_\lambda)\big) \\ & &&+ \bar\tau\big(\theta(f)^*\theta(f)\big) \\
    &= &&\tau(a_\lambda^*a_\lambda) - 2 \mathrm{Re} \,\int_{\partial_e X} \overline{f(\sigma)} \sigma(a_\lambda) \, {\rm d}\mu_\tau(\sigma) \\ & &&+ \int_{\partial_e X} |f(\sigma)|^2 \, {\rm d}\mu_{\tau}(\sigma),
\end{alignedat}
\end{equation}
and hence each $h_\lambda$ is continuous.\footnote{The integrals in \eqref{oz-neweq4} are continuous in $\tau$ because, as $X$ is Bauer, any continuous function $g\in C(\partial_e X)$ extends to a continuous affine function $\hat{g} \in \mathrm{Aff}_\C(X)$ by \cite[Theorem II.4.1]{Alf71} and $\hat{g}(\tau) = \int_{\partial X} g(\sigma)\, {\rm d}\mu_\tau(\sigma)$ by the definition of $\mu_\tau$.}
Note that $h_\lambda(\tau) \rightarrow 0$ for all $\tau \in X$ and hence Proposition~\ref{prop:pointwise-to-uniform}\ref{prop:pointwise-to-uniform2} implies $h_\lambda \rightarrow 0$ weakly in $\mathrm{Aff}(X)$.

By the Hahn--Banach theorem, there are indices $\lambda_1, \ldots, \lambda_n$ and real numbers $t_1, \ldots t_n \geq 0$ summing to 1 such that
\begin{equation}
    \sum_{i=1}^n t_i \|\pi_X(a_{\lambda_i}) - \theta(f)\|_{2, \bar\tau}^2 < \epsilon^2, \quad \tau \in X.
\end{equation}
Let $a \coloneqq \sum_{i=1}^n t_i a_{\lambda_i} \in \mathcal M$ and note that $\|a\| \leq \|f\|$.  Further, for all $\tau \in X$,
\begin{equation}
\begin{split}
    \| \pi_X(a) - \theta(f) \|_{2, \bar\tau} &= \Big\| \sum_{i=1}^n t_i \big(\pi_X(a_{\lambda_i}) - \theta(f)\big) \Big\|_{2, \bar\tau} \\
    &\leq \sum_{i=1}^n t_i^{1/2} \cdot t_i^{1/2} \bar\| \pi_X(a_{\lambda_i}) - \theta(f)\|_{2, \bar\tau} \\
    &\leq\Big(\sum_{i=1}^n t_i \|\pi_X(a_{\lambda_i}) - \theta(f) \|_{2, \bar\tau}^2 \Big)^{1/2} \\
    &< \epsilon,
\end{split}
\end{equation}
where the inequality on the third line follows from the Cauchy--Schwarz inequality and the fact that the $t_i$ sum to 1.  This shows the image of $\theta$ is contained in $\pi_X(\mathcal M)$.

Finally, the map $E\colon \M\to C(K)$ given by $E(a)(\tau)\coloneqq\tau(a)$, for $a\in\M$ and $\tau\in X$ is necessarily a tracial conditional expectation onto $C(K)$ such that $\|\cdot\|_{2, X} = \|\cdot\|_{2, E}$.  In particular, the unit ball of $\mathcal M$ is $\|\cdot\|_{2, E}$-complete, and $(\M,E)$ is a $W^*$-bundle over $\partial_eX$.
\end{proof}

As the following example demonstrates, Theorem~\ref{thm:TCtoWStarBundle} can fail without factoriality even for the one-dimensional simplex. The salient point is that in a $W^*$-bundle the fibres are fairly independent of each other, whereas in a non-factorial tracially complete algebra $(\M,X)$, the GNS representation $\pi_\tau$ with respect to the extreme traces $\tau \in \partial_eX$ may not be.

\begin{example}\label{eg:nonbundle}
Let $\M\coloneqq \R \oplus \R \oplus \R$ and let $\tau_i \in T(\M)$ be given by $\tau_i(a_1, a_2, a_3) \coloneqq \tau_\R(a_i)$ for $i = 1, 2, 3$.  Define $\sigma_1, \sigma_2 \in T(\mathcal M)$ by
\begin{equation}
	\sigma_1 \coloneqq \frac12 (\tau_1 + \tau_2) \qquad \text{and} \qquad \sigma_2 \coloneqq \frac12(\tau_1 + \tau_3)
\end{equation}

Let $Y\coloneqq\mathrm{co}\{\sigma_1, \sigma_2\} \subseteq T(\M)$.  Then $(\M, Y)$ is a tracially complete $C^*$-algebra and $Y$ is a Bauer simplex.  Note that $(\M,Y)$ is not factorial -- in fact, neither extreme point of $Y$ is an extreme point of $T(\M)$.  Let $K \coloneqq \partial_e Y = \{\sigma_1,\sigma_2 \}$.

We will show there is no embedding of $C(K) \cong \mathbb C^2$ into $Z(\M) \cong \mathbb C^3$ satisfying \eqref{eq:central-embedding}.  Note that any embedding satisfying \eqref{eq:central-embedding} is unital.  Let $f_1 \in C(K)$ be given by $f_1(\sigma_1) \coloneqq 1$ and $f_1(\sigma_2) \coloneqq 0$ and define $f_2 \coloneqq 1_{C(K)} - f_1 \in C(K)$.   Exploiting the symmetry of $\tau_2$ and $\tau_3$, it suffices to show that the three unital embeddings $j_1, j_2, j_3 \colon C(K) \rightarrow Z(\M)$ determined by
\begin{equation}
	j_1(f_1) = (1_\mathcal R, 0, 0),\ j_2(f_1) = (1_\mathcal R, 1_\mathcal R, 0),\ \text{and}\ j_3(f_1) = (0, 1_\mathcal R, 0)
\end{equation}
fail \eqref{eq:central-embedding}.

In the first case, for $a \coloneqq (0, 1_\mathcal R, 0)$, $\sigma_1(j_1(f_1) a) = 0$, but 
\begin{equation}
	\int_K f_1(\sigma)\sigma(a) \, \mathrm{d}\mu_{\sigma_1} = f_1(\sigma_1)\sigma_1(a) = \frac12.
\end{equation}
In the second case, for $a \coloneqq (1_\mathcal R, 0, 0)$, $\sigma_2(j_2(f_2)a) = 0$, but 
\begin{equation}
 	\int_K f_2(\sigma)\sigma(a) \, \mathrm{d}\mu_{\sigma_2} = f_2(\sigma_2)\sigma_2(a) = \frac12.
\end{equation}
In the third case, for $a \coloneqq (1_\R, 0, 0)$, $\sigma_1(j_3(f_1)a) = 0$, but 
\begin{equation}
	\int_K f_1(\sigma)\sigma(a) \, \mathrm{d}\mu_{\sigma_1} = f_1(\sigma_1)\sigma_1(a) = \frac12.
\end{equation}
So \eqref{eq:central-embedding} fails in all three cases.
\end{example}

The following question is natural in the light of Example~\ref{eg:nonbundle} and the factoriality requirement of Theorem~\ref{thm:TCtoWStarBundle}.
The subtlety in the following is that a general $W^*$-bundle need not have factorial fibres.

\begin{question}
    Characterise when a tracially complete $C^*$-algebra $(\mathcal M,X)$ (assuming $X\subseteq T(\mathcal M)$ is a Bauer simplex) is a $W^*$-bundle over $\partial_eX$.
\end{question}

\section{Amenability for tracially complete \texorpdfstring{$C^*$}{C*}-algebras}\label{sec:amenable}

In this section, we show how to combine fibrewise amenability to obtain a global uniform 2-norm completely positive approximation property. The main result is Theorem~\ref{thm:amenable}, characterising morphisms into tracially complete $C^*$-algebras which are \emph{tracially nuclear} -- the appropriate notion of amenability which feeds into classification. Theorem~\ref{thm:introamenable} will follow as a special case of Theorem~\ref{thm:amenable}.

\subsection{Definition and basic properties}

Our notion of amenability for tracially complete $C^*$-algebras is given by the following version of the completely positive approximation property.

\begin{definition}\label{def:tracially-nuclear}
	Let $A$ be a $C^*$-algebra and let $(\mathcal N, Y)$ be a tracially complete $C^*$-algebra.  We say that a c.p.\ map $\theta \colon A \rightarrow \mathcal N$ is \emph{tracially nuclear} if there are nets of finite dimensional $C^*$-algebras $F_\lambda$ and c.p.\ maps
	\begin{equation}\label{eqn:tracially-nuclear}
		A \overset{\psi_\lambda}{\longrightarrow} F_\lambda \overset{\phi_\lambda}{\longrightarrow} \mathcal N
	\end{equation}
	such that
	\begin{equation}
		\lim_\lambda \| \phi_\lambda(\psi_\lambda(a)) - \theta(a) \|_{2, Y} = 0, \quad a \in A.
	\end{equation}
	Further, we say $(\mathcal N, Y)$ is \emph{{}amenable} if $\mathrm{id}_\mathcal N$ is tracially nuclear.
\end{definition}

As usual, we may restrict to sequences in Definition~\ref{def:tracially-nuclear} when $A$ is separable.\footnote{This requires having a uniform bound on the norms of $\phi_\lambda$ and $\psi_\lambda$, which is always possible by Proposition~\ref{prop:bounded-cpap}.}  

Using the canonical inclusion and projection maps, it immediate that {}amenability passes to direct sums.
\begin{proposition}\label{prop:semidiscretenessdirectsums}
    For tracially complete $C^*$-algebras $(\M,X)$ and $(\mathcal N,Y)$ the direct sum $(\M\oplus\mathcal N,X\oplus Y)$ is {}amenable if and only if both $(\M,X)$ and $(\mathcal N,Y)$ are {}amenable.
\end{proposition}

Recall that if $\mathcal N$ is a von Neumann algebra, then a c.p.\ map $\theta\colon A\to \mathcal N$ is \emph{weakly nuclear} if there are c.p.\ maps as in \eqref{eqn:tracially-nuclear} such that for all $a\in \mathcal N$, $\phi_\lambda(\psi_\lambda(a))\rightarrow a$ in the weak$^*$ topology on $\mathcal N$.  When $\tau$ is a faithful normal trace, then $(\mathcal N,\{\tau\})$ is a tracially complete $C^*$-algebra tracial nuclearity and weak nuclearity agree (and likewise, semidiscreteness of $\mathcal N$ agrees with {}amenability of $(\mathcal N,\{\tau\})$).  We defer the proof to Proposition~\ref{prop:tracialnuc=weaknuc} as we first need some prerequisite results.

The following standard lemma allows us to reduce problems about tracial nuclearity to the unital case. We use $A^\dagger$ for the forced unitisation of $A$; i.e.\ when $A$ is unital, we add a new unit to form $A^\dagger\cong A\oplus \mathbb C$.

\begin{lemma}\label{lem:amenable-unitise}
	Suppose $A$ is a $C^*$-algebra, $(\mathcal N, Y)$ is a tracially complete $C^*$-algebra, and $\theta \colon A \rightarrow \mathcal N$ is a c.p.c.\ map. Then $\theta$ is tracially nuclear if and only if the unitisation $\theta^\dag \colon A^\dag \rightarrow \mathcal N$ is also tracially nuclear.
\end{lemma}

\begin{proof}
First note that $\theta$ factorises as the inclusion $A \to A^\dag$ followed by $\theta^\dag$, so that tracial nuclearity of $\theta^\dag$ implies that of $\theta$.

Conversely, suppose that $\theta$ is tracially nuclear.
	Let $(e_\lambda)$ be an approximate unit for $A$ consisting of positive contractions and define $\theta_\lambda \colon A^\dag \rightarrow \mathcal N$ by
	\begin{equation}
		\theta_\lambda(a + \alpha 1_{A^\dag}) \coloneqq \theta(e_\lambda a e_\lambda) + \alpha 1_\mathcal N.
	\end{equation}
	Then $\theta_\lambda$ is unital.  Further, $\theta_\lambda$ is tracially nuclear as it is the sum of the tracially nuclear maps
	\begin{align}
		a + \alpha 1_{A^\dag} &\mapsto \theta(e_\lambda (a + \alpha 1_{A^\dag}) e_\lambda)
	\intertext{and}
		a + \alpha 1_{A^\dag} &\mapsto \alpha (1_\mathcal N - \theta(e_\lambda^2)).
	\end{align}
	Since each $\theta_\lambda$ is unital and $\|\theta_\lambda(a) - \theta(a)\| \rightarrow 0$ for all $a \in A$, we have $\|\theta_\lambda(a) - \theta^\dag(a)\| \rightarrow 0$ for all $a \in A^\dag$.  As each $\theta_\lambda$ is tracially nuclear, so is $\theta$.
\end{proof}

As with the analogous versions of the completely positive approximation property for both $C^*$-algebras and von Neumann algebras, we can arrange for a uniform bound on the norms of the $\phi_\lambda$ and $\psi_\lambda$ in Definition~\ref{def:tracially-nuclear}.  The proof here follows the proof of the von Neumann algebraic version in \cite[Proposition~3.8.2]{Br08}, taking care to avoid the use of Borel functional calculus, which, in general, does not exist in tracially complete $C^*$-algebras.

\begin{proposition}\label{prop:bounded-cpap}
	Suppose that $A$ is a $C^*$-algebra, $(\mathcal N, Y)$ is a tracially complete $C^*$-algebra, and $\theta \colon A \rightarrow \mathcal N$ is a tracially nuclear map.  Then there are nets
	\begin{equation}\label{eq:bounded-cpap-maps}
		A \overset{\psi_\lambda}{\longrightarrow} F_\lambda \overset{\phi_\lambda}{\longrightarrow}  \mathcal N
	\end{equation}
	as in Definition~\ref{def:tracially-nuclear} with $\|\psi_\lambda\| \leq \|\theta\|$ and $\phi_\lambda(1_{F_\lambda}) = 1_\mathcal N$.  Further, if $A$ and $\theta$ are unital, we may arrange each $\psi_\lambda$ to be unital.
\end{proposition}

\begin{proof}
Rescaling $\theta$, we may assume that $\|\theta\| \leq 1$.  
Then by adding a unit to $A$ and using Lemma~\ref{lem:amenable-unitise}, we may assume $A$ and $\theta$ are unital.  We will construct u.c.p.\ maps $\phi_\lambda$ and $\psi_\lambda$ as in \eqref{eq:bounded-cpap-maps} approximately factorising $\theta$.

Since $\theta$ is assumed to be tracially nuclear, there are nets of finite dimensional $C^*$-algebras $F_\lambda$ and c.p.\ maps 
	\begin{equation}
		A \overset{\psi_\lambda''}{\longrightarrow} F_\lambda \overset{\phi_\lambda'''}{\longrightarrow} \mathcal N
	\end{equation}
	such that for all $a \in A$,
	\begin{equation}
		\lim_\lambda \| \phi'''_\lambda(\psi''_\lambda(a)) - \theta(a) \|_{2, Y} = 0, \quad a \in A.
	\end{equation}

	By \cite[Lemma~2.2.5]{Br08}, there is a u.c.p.\ map $\psi_\lambda' \colon A \rightarrow F_\lambda$ such that
	\begin{equation}
		\psi_\lambda''(a) = \psi_\lambda''(1_A)^{1/2} \psi_\lambda'(a) \psi_\lambda''(1_A)^{1/2}, \qquad a \in A.
	\end{equation}
	Define $\phi_\lambda'' \colon F_\lambda \rightarrow \mathcal N$ by
	\begin{equation}
		\phi_\lambda''(x) \coloneqq \phi_\lambda'''\big( \psi_\lambda''(1_A)^{1/2} x \psi_\lambda''(1_A)^{1/2}\big), \qquad x \in F_\lambda,
	\end{equation}
	and note that $\phi_\lambda'' \circ \psi_\lambda' = \phi_\lambda''' \circ \psi_\lambda''$.  Therefore, we have
	\begin{equation}\label{eq:cpap-proof1}
		\lim_\lambda \| \phi_\lambda''(\psi_\lambda'(a)) - \theta(a) \|_{2, Y} = 0, \qquad a \in A.
	\end{equation}	

	Define a continuous function  $f \colon \mathbb R \rightarrow [0,1]$ by
\begin{equation}
		f(t) \coloneqq \begin{cases} 1, & t \leq 1; \\ 2 - t, & 1 < t < 2; \\ 0, & t \geq 2. \end{cases}
	\end{equation}
	Let $b_\lambda \coloneqq f(\phi_\lambda''(1_{F_\lambda}))$ and define $\phi_\lambda' \colon F_\lambda \rightarrow \mathcal N$ by
	\begin{equation}
		\phi_\lambda'(x) \coloneqq b_\lambda \phi''_\lambda(x) b_\lambda, \qquad x \in F_\lambda.
	\end{equation}
	By elementary calculus, $0 \leq t f(t)^2 \leq 1$ for all $t \in [0,\infty)$. Therefore, $\|\phi_\lambda'(1_{F_\lambda})\| \leq 1$, and $\phi_\lambda'$ is a c.p.c.\ map.
	Also, for all $t \in \mathbb R$, we have $0 \leq 1 - f(t) \leq |t - 1|$, and so
	\begin{equation}\label{eq:cpap-proof2}
		0 \leq 1_\mathcal N - b_\lambda \leq |\phi''_\lambda(1_{F_\lambda}) - 1_\mathcal N|.
	\end{equation}
	Since $\theta$ and $\psi_\lambda'$ are unital, \eqref{eq:cpap-proof1} and \eqref{eq:cpap-proof2} imply
	\begin{equation}\label{eq:cpap-proof6}
		\lim_\lambda \|1_\mathcal N - b_\lambda\|_{2, Y} = 0.
	\end{equation}
	Fix $a \in A$. Since the elements $b_\lambda$ are contractions, we have 
	\begin{equation}
		\| \phi_\lambda'(\psi_\lambda'(a)) - b_\lambda \theta(a)b_\lambda  \|_{2, Y} 
		\leq \|\phi''_\lambda(\psi'_\lambda(a)) - \theta(a)\|_{2,Y},	
	\end{equation}
	so $\lim_\lambda \| \phi'_\lambda(\psi'_\lambda(a)) - b_\lambda \theta(a)b_\lambda  \|_{2, Y} = 0$ by \eqref{eq:cpap-proof1}. Moreover, as multiplication in $\mathcal{N}$ is $\|\cdot\|_{2,Y}$-continuous on $\|\cdot\|$-bounded sets, we have 
	\begin{equation}	
	\lim_\lambda \| b_\lambda \theta(a)b_\lambda - \theta(a) \|_{2, Y} = 0, \qquad a \in A,
	\end{equation}
	by \eqref{eq:cpap-proof6}. Therefore, 
	\begin{equation}\label{eq:cpap-proof7}
		\lim_\lambda \|\phi_\lambda'(\psi_\lambda'(a)) - \theta(a)\|_{2, Y} = 0, \qquad a \in A.
	\end{equation}
	
	While $\psi'_\lambda$ was arranged to be unital, the map $\phi'_\lambda$ need not be.  We modify the maps once more to correct this.  Let $\rho$ be a state on $A$ and define $\psi_\lambda \colon A \rightarrow F_\lambda \oplus \mathbb C$ by $\psi_\lambda(a) \coloneqq (\psi'_\lambda(a), \rho(a))$.  Then $\psi_\lambda$ is u.c.p.\ as both $\psi'_\lambda$ and $\rho$ are.  Also, define $\phi_\lambda \colon F_\lambda \oplus \mathbb C \rightarrow \mathcal N$ by
	\begin{equation}
		\phi_\lambda(x, \alpha) \coloneqq \phi_\lambda'(x) + \alpha(1_\mathcal N - \phi_\lambda'(1_{F_\lambda}))
	\end{equation}
	Then $\phi_\lambda$ is u.c.p.\ as $\phi'_\lambda$ is c.p.c.  By \eqref{eq:cpap-proof7}, to see that 
	\begin{equation}
		\lim_\lambda \|\phi_\lambda(\psi_\lambda(a)) - \theta(a) \|_{2, Y} = 0, \qquad a \in A,
	\end{equation} 
	it suffices to show that $\lim_\lambda \|1_\mathcal N -  \phi_\lambda(1_{F_\lambda}) \|_{2, Y} = 0$.  However, this follows immediately from taking $a \coloneqq 1_A$ in \eqref{eq:cpap-proof7} using that the maps $\psi_\lambda'$ and $\theta$ are both unital.
\end{proof}

We now return to the promised connection between tracial and weak nuclearity.

\begin{proposition}\label{prop:tracialnuc=weaknuc}
Let $A$ be a $C^*$-algebra and $(\mathcal N,\tau)$ be a tracial von Neumann algebra.  A c.p.\ map $\theta\colon A\to (\mathcal N,\{\tau\})$ is tracially nuclear if and only if it is weakly nuclear. In particular, $(\mathcal N,\tau)$ is semidiscrete as a tracial von Neumann algebra if and only if $(\mathcal N,\{\tau\})$ is {}amenable as a tracially complete $C^*$-algebra.
\end{proposition}
\begin{proof}
The key observation is that the strong operator topology on the operator norm unit ball of $\mathcal N$ coincides with the topology induced by $\|\cdot\|_{2, \tau}$ (see \cite[Proposition III.2.2.17]{Bl06}, for example).  From here, assume that $\theta$ is weakly nuclear.  
By scaling $\theta$, we may assume $\|\theta\| \leq 1$.  Then, as in the proof of Lemma~\ref{lem:amenable-unitise}, the unitisation $\theta^\dag \colon A^\dag \rightarrow \mathcal N$ 
is weakly nuclear.  Hence we may assume that $A$ and $\theta$ are unital.  By \cite[Proposition~3.8.2]{Br08}, there are nets of finite dimensional $C^*$-algebras and c.p.c.\ maps
\begin{equation}
    A \overset{\psi_\lambda}{\longrightarrow} F_\lambda \overset{\phi_\lambda}{\longrightarrow} \mathcal N
\end{equation}
such that $\phi_\lambda(\psi_\lambda(a)) \rightarrow \theta(a)$ weak$^*$ for all $a \in A$.  Since the weak$^*$ and weak operator topologies agree on $\|\cdot\|$-bounded sets, we also have $\phi_\lambda(\psi_\lambda(a)) \rightarrow \theta(a)$ in the weak operator topology for all  $a \in A$.  Since the set of c.p.c.\ maps $A \rightarrow \mathcal N$ 
which admit c.p.c.\ factorisations through a finite dimensional $C^*$-algebra form a convex set (\cite[Lemma~2.3.6 and Remark~2.3.7]{Br08}), the Hahn--Banach theorem, in the form of \cite[Lemma~3.8.1]{Br08}, shows that there is a net of c.p.c.\ maps $\theta_\lambda \colon A \rightarrow \mathcal N$ such that 
$\theta_\lambda(a) \rightarrow \theta(a)$ in the strong operator topology for all $a \in A$, and each $\theta_\lambda$ admits a c.p.c.\ factorisation through a finite dimensional $C^*$-algebra.  
As each $\theta_\lambda$ is contractive, this implies $\|\theta(a) - \theta_\lambda(a)\|_{2, \tau} \rightarrow 0$ for all $a \in A$, and hence $\theta$ is tracially nuclear.

Conversely, suppose that $\theta$ is tracially nuclear. Again, after scaling $\theta$, we may assume that $\theta$ is contractive.  By Proposition~\ref{prop:bounded-cpap}, there are nets of finite dimensional $C^*$-algebras and c.p.c.\ maps
\begin{equation}
    A \overset{\psi_\lambda}{\longrightarrow} F_\lambda \overset{\phi_\lambda}{\longrightarrow} \mathcal N
\end{equation}
such that $\|\theta(a) - \phi_\lambda(\psi_\lambda(a))\|_{2, \tau} \rightarrow 0$ for all $a \in A$.  Then $\phi_\lambda(\psi_\lambda(a)) \rightarrow a$ in the strong operator topology and hence in the weak operator topology.  So the convergence holds weak$^*$, and this implies that $\theta$ is weakly nuclear.
\end{proof}

It is known that if $A$ is an exact $C^*$-algebra and $\mathcal N$ is a von Neumann algebra, then any weakly nuclear map $A \rightarrow \mathcal N$ is nuclear -- see \cite[Remark~3.4]{HKW12} (as observed there, this characterises exactness of $A$).  This suggests the following question, which we have been unable to answer even with the additional assumptions that $(\mathcal N, Y)$ is factorial and has CPoU.  The difficulty is that the passage from weak nuclearity to nuclearity involves taking point-weak$^*$-limit points.\footnote{The proof in \cite[Remark~3.4]{HKW12} uses Kirchberg's $\mathcal O_2$-embedding theorem, but this can be avoided using the following slight variation of their proof.  Let $\theta \colon A \rightarrow \mathcal N$ be weakly nuclear.  Fix nets of finite dimensional $C^*$-algebras $F_\lambda$ and c.p.c.\ maps $\psi_\lambda \colon A  \rightarrow F_\lambda$ and $\phi_\lambda \colon F_\lambda \rightarrow \mathcal N$ such that $\phi_\lambda \circ \psi_\lambda$ converges to $\theta$ point-weak$^*$.  Let $\pi \colon A \rightarrow \mathcal B(\mathcal H)$ be a faithful representation of $A$ and use Arveson's extension theorem to find c.p.c.\ maps $\widetilde\psi_\lambda \colon \mathcal B(\mathcal H) \rightarrow F_\lambda$ with $\widetilde \psi_\lambda \circ \pi = \psi_\lambda$.  Let $\rho \colon \mathcal B(\mathcal H) \rightarrow \mathcal N$ be a point-weak$^*$ limit point of $\phi \circ \widetilde\psi_\lambda$ and note that $\theta = \rho \circ \pi$.  As $A$ is exact, $\pi$ is nuclear, and hence so is $\theta$.}  We do not have the required compactness for this in the setting of tracially complete $C^*$-algebras.

\begin{question}\label{Q:exact-nuclear}
	Suppose $A$ is an exact $C^*$-algebra and $(\mathcal N, Y)$ is a tracially complete $C^*$-algebra.  Is every tracially nuclear map $A \rightarrow \mathcal N$ nuclear?
\end{question}

\subsection{Fibrewise {}amenability}

We work towards showing that {}amenability of tracially complete $C^*$-algebras can be detected in its tracial von Neumann algebra completions.  The basic strategy is to use a convexity argument similar to the one in the fundamental result from \cite{Effros-Lance77,Choi-Effros76, Choi-Effros77} stating that if $A$ is a $C^*$-algebra such that $A^{**}$ is semidiscrete, then $A$ is nuclear (see \cite[Proposition~2.3.8]{Br08}, for example).  

We set up the convexity argument in a somewhat general setting for later use in Section~\ref{sec:CPoU-proof}.  Both the statement and proof are abstracted from an argument in the proof of Ozawa's~\cite[Theorem~3]{Oz13} (cf.\ the proof of Theorem~\ref{thm:TCtoWStarBundle} above) computing the centre of tracial completions of $C^*$-algebras.  A similar Hahn--Banach argument appears in the proof of (8)$\Rightarrow$(1) in \cite[Theorem~2.5]{GGKNV}.

\begin{lemma}\label{lem:affine-selection}
	Let $\mathcal C$ be a convex subset of a real vector space and let $(\M, X)$  be a tracially complete $C^*$-algebra. Suppose we are given $m \in \mathbb N$ and a finite collection $f_1,\dots,f_m \colon \mathcal C \rightarrow \M$ of affine functions such that for all $\tau \in X$ and $\epsilon > 0$, there is $c_\tau \in \mathcal C$ such that
	\begin{align}
		\max_{1 \leq i \leq m} \|f_i(c_\tau)\|_{2, \tau} &< \epsilon.
	\intertext{Then for all $\epsilon > 0$, there is $c \in \mathcal C$ such that} 
		 \max_{1 \leq i \leq m}  \|f_i(c)\|_{2, X} &< \epsilon.
	\end{align}
\end{lemma}

\begin{proof}
	Let $\Lambda$ be the directed set of pairs $(\mathcal T, \epsilon)$ where $\mathcal T \subseteq X$ is a non-empty finite set and $\epsilon > 0$, equipped  with the ordering $(\mathcal T_1, \epsilon_1) \leq (\mathcal T_2, \epsilon_2)$ if and only if $\mathcal T_1 \subseteq \mathcal T_2$ and $\epsilon_1 \geq \epsilon_2$.  
	
	Fix $\lambda \coloneqq (\mathcal T_\lambda, \epsilon_\lambda) \in \Lambda$ and let $\sigma_\lambda \coloneqq |\mathcal T_\lambda|^{-1} \sum_{\tau\in \mathcal T_\lambda}\tau\in X$ denote the average of the traces in $\mathcal T_\lambda$.  By assumption, there is $c_\lambda \in \mathcal C$ such that 
	\begin{align}
		|\mathcal T_\lambda|^{1/2} \max_{1 \leq i \leq m} \|f_i(c_\lambda)\|_{2, \sigma_\lambda} &<  \epsilon_\lambda.
	\intertext{It follows from the definition of $\sigma_\lambda$ that}
		\max_{\tau \in \mathcal T_\lambda} \max_{1 \leq i \leq m} \|f_i(c_\lambda)\|_{2, \tau} &< \epsilon_\lambda \label{eq:net-pointwise-small}.
	\end{align}
	
	For $\lambda \in \Lambda$ and $i=1,\ldots,m$, define $h_{i,\lambda}\colon X\to \mathbb R$ by
	\begin{equation}\label{eq:hilambdaDef}
		h_{i, \lambda}(\tau)\coloneqq\|f_i(c_\lambda)\|_{2, \tau}^2
	\end{equation}
	and note that $h_{i, \lambda} \in {\rm Aff}(X)$.  For all $i=1,
	\ldots,m$, $h_{i, \lambda} \rightarrow 0$ pointwise on $X$ by \eqref{eq:net-pointwise-small}, and hence Proposition~\ref{prop:pointwise-to-uniform}\ref{prop:pointwise-to-uniform2} implies $h_{i, \lambda} \rightarrow 0$ weakly.  View $h_\lambda \coloneqq (h_{i,\lambda}, \ldots, h_{m,\lambda})$ as a net in the Banach space ${\rm Aff}(X)^{\oplus m}$ and note that $h_\lambda \rightarrow 0$ weakly.  
	
	Fix $\epsilon > 0$.  By the Hahn--Banach theorem, there are $l\in\mathbb N$, $\lambda_1,\dots,\lambda_l\in\Lambda$, and real numbers  $t_1, \ldots,  t_l \geq 0$ such that $\sum_{k=1}^l t_k = 1$ and
	\begin{equation}\label{eq:affine-selection1}
		\max_{1 \leq i \leq m} \Big\|\sum_{k=1}^l t_k h_{i,\lambda_k}\Big\|_\infty < \epsilon^2.
	\end{equation}
	Define $c \coloneqq \sum_{k=1}^l t_k c_{\lambda_k} \in \mathcal C$.  For $i=1,\ldots,m$, using the triangle inequality and the Cauchy--Schwarz inequality, we have
	\begin{equation}\begin{split}
		\|f_i(c)\|_{2,X} 
		&= \Big\|\sum_{k=1}^l t_k f_i(c_{\lambda_k})\Big\|_{2, X}  \\
		&\leq \sum_{k=1}^l t_k \|f_i(c_{\lambda_k})\|_{2, X}  \\
		&= \sum_{k=1}^l \big(t_k^{1/2})\big( t_k^{1/2} \|f_i(c_{\lambda_k})\|_{2, X}\big)  \\
		&\leq \Big(\sum_{k=1}^l t_k\Big)^{1/2} \Big( \sum_{k=1}^l t_k \|f_i(c_{\lambda_k})\|_{2, X}^2 \Big)^{1/2}  \\
		&= \Big( \sum_{k=1}^l t_k \|f_i(c_{\lambda_k})\|_{2, X}^2 \Big)^{1/2} \label{eq:convexity-bound}
	\end{split}\end{equation}
	Then for $i=1,
	\ldots,m$ and $\tau \in X$, we have
	\begin{equation}
		\|f_i(c)\|_{2, \tau}^2 \stackrel{\eqref{eq:hilambdaDef},\eqref{eq:convexity-bound}}{\leq} \sum_{k=1}^l t_k h_{i,\lambda_k}(\tau) \stackrel{\eqref{eq:affine-selection1}}{<} \epsilon^2.
	\end{equation}
	Hence $\|f_i(c)\|_{2, X} < \epsilon$ for all $i=1,\ldots,m$.
\end{proof}

The following lemma is standard.  For example, it follows from the proof of \cite[Lemma~1.1]{HKW12} by quoting the Choi--Effros lifting theorem in place of the projectivity of order zero maps in the last paragraph.

\begin{lemma}\label{lem:finite-rank-cp}
	If $F$ and $B$ are $C^*$-algebras with $F$ finite dimensional, $\pi \colon B \rightarrow \mathcal B(H)$ is a $^*$-homomorphism, and $\phi \colon F \rightarrow \mathcal \pi(B)''$ is a c.p.\ map, then there is a net of c.p.\ maps $\phi_\lambda \colon F \rightarrow B$ such that $\|\phi_\lambda\| \leq \|\phi\|$ for all $\lambda$ and
	\begin{equation}
		\phi(b) = \text{\rm strong$^*$-}\lim_\lambda \pi(\phi_\lambda(b)), \qquad b \in B.	
	\end{equation}
	Further, if $\phi$ is unital, we may arrange for each $\phi_\lambda$ to be unital.
\end{lemma}

We are now in a position to give the `fibrewise' characterisation of tracially nuclear $^*$-homomorphisms (Theorem \ref{thm:amenable}) from which Theorem \ref{thm:introamenable} follows immediately (by taking $A\coloneqq \mathcal M$ and $\theta\coloneqq \id_{\mathcal M}$).  Note that the following theorem also implies the first part of Theorem~\ref{IntroThmClassMap} from the overview.

The equivalence of \ref{amen2} and \ref{amen3} in Theorem \ref{thm:amenable} is essentially obtained by Brown in \cite[Theorem 3.2.2]{Bro06} (which is a variation on Connes' theorem), so the main implication we need to show is \ref{amen2} implies \ref{amen1}.  This will follow from the Hahn--Banach argument of Lemma~\ref{lem:affine-selection}.

\begin{theorem}\label{thm:amenable}
   Let $A$ be a $C^*$-algebra, let $(\M, X)$ be a tracially complete $C^*$-algebra, and let $\theta \colon A \rightarrow \mathcal M$ be a c.p.\ map.   The following are equivalent:
\begin{enumerate}
	\item\label{amen1} $\theta$ is tracially nuclear;
	\item\label{amen2} for all $\tau \in X$, $\pi_\tau \circ \theta \colon A \rightarrow \pi_\tau(\M)''$ is weakly nuclear.
 \end{enumerate}
 If $\theta$ is a $^*$-homomorphism then these are also equivalent to
 \begin{enumerate}
 \setcounter{enumi}{2}
	\item\label{amen3} for every $\tau \in X$ with $\tau \circ \theta \neq 0$, the trace $\| \tau \circ \theta\|^{-1} \cdot \tau \circ \theta \in T(A)$ is uniformly amenable in the sense of \cite[Definition~3.2.1]{Bro06}.\footnote{Brown implicitly only defines uniformly amenable for traces on separable unital $C^*$-algebras in \cite[Definition 3.2.1]{Br08}: $\tau$ is uniformly amenable if there is a sequence of u.c.p.\ maps $\phi_n\colon A \to M_{k_n}$ such that $\|\phi_n(ab)-\phi_n(a)\phi_n(b)\|_{2,\mathrm{tr}_{k_n}} \to 0$ and $\|\mathrm{tr}_{k_n}\circ \phi_n-\tau\| \to 0$ in the norm topology on the dual of $A$ (note the use of u.c.p.\ maps forces $A$ to be unital, and the use of sequences is only reasonable if $A$ is separable). We extend this definition to the non-separable case by allowing nets in place of sequences, and to the non-unital case, by saying that $\tau \in T(A)$ is uniformly amenable if its unitisation, $\tau^\dag \in T(A^\dag)$, is uniformly amenable.}
\end{enumerate}
\end{theorem}

\begin{proof}
	To see \ref{amen1} implies \ref{amen2}, note that if $\theta$ is tracially nuclear, then for all $\tau \in X$, $\pi_\tau \circ \theta$ is tracially nuclear as a map into the tracial von Neumann algebra $(\pi_\tau(\mathcal M)'', \tau)$.  By Proposition~\ref{prop:tracialnuc=weaknuc}, for maps into tracial von Neumann algebras, weak and tracial nuclearity are equivalent, so \ref{amen2} follows.
	
	When $\theta$ is a $^*$-homomorphism, the equivalence of \ref{amen2} and \ref{amen3} can be reduced to the case that $A$ and $\theta$ are unital by adding a unit to $A$ and using Lemma~\ref{lem:amenable-unitise}. For $\tau \in X$, since there is a normal trace-preserving conditional expectation $\pi_\tau(\M)'' \rightarrow \pi_\tau(\theta(A))''$ (see \cite[Lemma 1.5.10]{Br08}, for example), $\pi_{\tau \circ \theta} = \pi_\tau \circ \theta$ is weakly nuclear when viewed as a map into $\pi_\tau(\M)''$ if and only if it is weakly nuclear when viewed as a map into $\pi_{\tau \circ \theta}(A)'' = \pi_\tau(\theta(A))''$.  The equivalence of (ii) and (iii) then follows from the equivalence of (1) and (6) in  \cite[Theorem~3.2.2]{Bro06}.
	
	It remains to show that \ref{amen2} implies \ref{amen1}.  Assume that $\theta$ is contractive. Fix a finite set $\mathcal F \subseteq A$ and $\epsilon > 0$.  Let $\mathcal C$ denote the set of all c.p.\ maps $A \rightarrow \M$ which factor by c.p.\ maps through a finite dimensional $C^*$-algebra and note that $\mathcal C$ is convex (cf.~\cite[Lemma~2.3.6]{Br08}).  We will apply Lemma~\ref{lem:affine-selection} to the affine functions
	\begin{equation} 
		f_a \colon \mathcal C \rightarrow \M \colon \eta \mapsto \eta(a) - \theta(a), \qquad a \in \mathcal F.
	\end{equation}
	
	For $\tau \in X$, $\pi_\tau \circ \theta$ is weakly nuclear by~\ref{amen2}, so there are a finite dimensional $C^*$-algebra $F_\tau$ and  c.p.c.\ maps
	\begin{equation}
		A \overset{\psi_\tau}{\longrightarrow} F_\tau \overset{\bar\phi_\tau}{\longrightarrow} \pi_\tau(\M)''
	\end{equation}
	such that
	\begin{alignat}{2}
		\| \bar\phi_\tau(\psi_\tau(a)) - \pi_\tau(\theta(a)) \|_{2, \tau} &< \epsilon, & \hspace{4ex} &a \in \mathcal F.
	\intertext{By Lemma~\ref{lem:finite-rank-cp}, there is then a c.p.c.\ map $\phi_\tau \colon F_\tau \rightarrow \mathcal M$ such that} 
		\|\phi_\tau(\psi_\tau(a)) - \theta(a) \|_{2, \tau} &< \epsilon, & \hspace{4ex} &a \in \mathcal F.
	\end{alignat}
	Then $\eta_\tau \coloneqq \phi_\tau \circ \psi_\tau \in \mathcal C$ and $\|f_a(\eta_\tau)\|_{2, \tau} < \epsilon$ for all $a \in \mathcal F$.  By Lemma~\ref{lem:affine-selection}, there is $\eta \in \mathcal C$ such that $\|f_a(\eta)\|_{2, X} < \epsilon$ for all $a \in \mathcal F$.  Unpacking notation, $\eta$ is a c.p.\ map factoring through a finite dimensional $C^*$-algebra, and
	\begin{equation}
		\|\eta(a) - \theta(a)\|_{2, X} < \epsilon, \qquad a \in \mathcal F,
	\end{equation}
	so $\theta$ is tracially nuclear.
\end{proof}

Combining the previous result with Connes' theorem gives a tracially complete analogue of the fact that a von Neumann algebra completion of a nuclear $C^*$-algebra is semidiscrete.  Let $T_{\rm am}(A) \subseteq T(A)$ denote the set of amenable traces on $A$.

\begin{corollary}\label{cor:semidiscrete-completion}
	If $A$ is an exact $C^*$-algebra with $T_{\rm am}(A)$ compact,\footnote{Note that $T_{\rm am}(A)$ is closed in $T(A)$ by \cite[Proposition~3.5.1]{Bro06}, so the compactness is automatic if $A$ is unital.} then the tracial completion of $A$ with respect to $T_{\mathrm{am}}(A)$ is an {}amenable factorial tracially complete $C^*$-algebra.  In particular, if $A$ is a nuclear $C^*$-algebra with $T(A)$ compact, then the tracial completion of $A$ with respect to $T(A)$ is an {}amenable factorial tracially complete $C^*$-algebra.
\end{corollary}

\begin{proof}
	By \cite[Lemma~3.4]{Kirchberg94}, $T_{\rm am}(A)$ is a closed face in $T(A)$.  Therefore, the tracial completion of $A$ with respect to $T_{\mathrm{am}}(A)$ is a factorial tracially complete $C^*$-algebra by Proposition~\ref{prop:tracial-completion}\ref{item:completion-factorial}.

	As $A$ is exact, $A$ is locally reflexive by \cite[Remark~(11)]{Kirchberg95}, and hence by \cite[Theorem~4.3.3]{Bro06}, all amenable traces on $A$ are uniformly amenable.  By the equivalence of (1) and (5) in \cite[Theorem~3.2.2]{Bro06}, it follows that $\pi_\tau(A)''$ is semidiscrete for $\tau \in T_{\rm am}(A)$.  Using Proposition~\ref{prop:tracial-completion}\ref{item:completion-fibres}, we also have that $\pi_\tau(\completion{A}{T_{\rm am}(A)})''$ is semidiscrete for all $\tau \in T_{\rm am}(A)$.
    By Theorem~\ref{thm:introamenable}, this implies that $\big(\completion{A}{T_{\rm am}(A)},T_{\rm am}(A)\big)$ is {}amenable.
    This proves the first sentence of the theorem.  The second sentence follows since $T(A) = T_{\rm am}(A)$ when $A$ is nuclear (see \cite[Theorem~4.2.1]{Bro06}, for example).
\end{proof}

Without exactness, it need not be the case that all amenable traces are uniformly amenable. Indeed, given a sequence $(k_n)_{n=1}^\infty$ of natural numbers converging to $\infty$, let $A\coloneqq \prod M_{k_n}$.  For a free ultrafilter $\omega$ on $\mathbb N$, the trace $\tau_\omega((x_n))\coloneqq \lim_{n\to\omega}\mathrm{tr}_{k_n}(x_n)$ is an amenable trace that is not uniformly amenable.\footnote{This is because the tracial von Neumann ultraproduct $\prod^\omega M_{n_k}$ is not amenable as it contains a copy of $L(F_2)$, by \cite[Theorem~3.8]{Voi91}.\label{fn:NonAmenableUltraprod}} This observation leads to the following characterisation of those finite von Neumann algebras which are {}amenable as tracially complete $C^*$-algebras over their trace space.

\begin{proposition}\label{prop:vNA-non-uniform-amenable-trace}
Let $\mathcal M$ be a semidiscrete finite von Neumann algebra.  Then $\big(\mathcal M,T(\M)\big)$ is {}amenable as a tracially complete $C^*$-algebra if and only if its type {\rm II}$_1$ summand has only finitely many extremal traces and it has no type ${\rm I}_n$ summand for all sufficiently large $n$.
\end{proposition}
\begin{proof}
Suppose $\M$ satisfies the stated condition. Then the type I part $\M_{\mathrm I}$ of $\M$ is a finite direct sum of matrices over abelian $C^*$-algebras (see \cite[Theorem V.1.27]{Tak79}, for example) and hence nuclear, so it has the completely positive approximation property in norm.  Hence $\big(\M_{\mathrm I},T(\M_{\mathrm I})\big)$ is {}amenable as a tracially complete $C^*$-algebra.  Let $e_1,\dots,e_m$ denote the minimal central projections of the type II$_1$ part $\M_{\rm II}$ of $\M$  so that each $e_i \M =e_i \M_{\mathrm{II}}$ is a semidiscrete factor, whence $\big(e_i \M ,T(e_i \M )\big)$ is an {}amenable tracially complete $C^*$-algebra (by Proposition \ref{prop:tracialnuc=weaknuc}).  Since a finite direct sum of {}amenable tracially complete $C^*$-algebras is {}amenable (Proposition \ref{prop:semidiscretenessdirectsums}), it follows that $\big(\M,T(\M)\big)$ is {}amenable.

Conversely, supposing the condition does not hold, we can find  sequences $(n_k)_{k=1}^\infty$ converging to $\infty$ and orthogonal central projections $(e_k)_{k=1}^\infty$ in $\M$ such that there are unital embeddings $M_{n_k}\to e_k\M$.\footnote{This uses the fact that for any type II$_1$ von Neumann algebra $\mathcal N$, and any $n\in\mathbb N$, one can find a unital embedding $M_n\to\mathcal N$. This goes back to Murray and von Neumann. The proof for $n=2$ from \cite[Proposition V.1.35]{Tak79} can be readily modified to cover general $n$.}  In this way, we have an embedding of the infinite product $C^*$-algebra $\prod M_{n_k}$ into $\M$.  For each $k$, let $\tau_k$ be a trace on $\M$ with $\tau_k(e_k)=1$, and for a free ultrafilter $\omega$ on $\N$, let $\tau_\omega\coloneqq \lim_{k\to\omega}\tau_k$, which exists by the compactness of the trace space of $\prod M_{n_k}$.  Then $\tau_\omega$ restricts to the trace $(x_n)\mapsto \lim_{n\to\omega}\mathrm{tr}_{n_k}(x_n)$ on $\prod M_{n_k}$.  This restricted trace is not uniformly amenable (see Footnote~\ref{fn:NonAmenableUltraprod}). As uniform amenability of traces passes to subalgebras (this is immediate from the approximation form of the definition in \cite[Definition 3.2.1]{Bro06}), $\tau_\omega$ is not uniformly amenable, and hence $\big(\M,T(\M)\big)$ is not {}amenable.
\end{proof}

In particular, the  tracially complete $C^*$-algebra $\big(\ell^\infty(\R), T(\ell^\infty(\R))\big)$ is not {}amenable even though $\ell^\infty(\R)$ is semidiscrete as a von Neumann algebra.

We end this section by applying the characterisation of tracial nuclearity to show that it can be tested on $C^*$-subalgebras which are dense in the uniform 2-norm. A naive argument (showing that the witnesses of tracial nuclearity for $\phi|_A$ extend to witnesses of tracial nuclearity for $\phi$) is not possible as one cannot a priori assume uniform $\|\cdot\|_{2,X}$-boundedness of the c.p.c.\ approximations in the definition of tracial nuclearity.

\begin{lemma}\label{lem:semidiscrete-dense}
	Suppose $(\mathcal M, X)$ and $(\mathcal N, Y)$ are tracially complete $C^*$-alge\-bras and $\phi \colon (\mathcal M , X) \rightarrow (\mathcal N, Y)$ is a morphism.  If $A \subseteq \mathcal \M$ is a $\|\cdot\|_{2, X}$-dense $C^*$-subalgebra, then $\phi$ is tracially nuclear if and only if $\phi|_A$ is tracially nuclear.
\end{lemma}

\begin{proof}
	For all $\tau \in X$, the inclusion $A \hookrightarrow \M$ induces an isomorphism $\pi_\tau(A)'' \rightarrow \pi_\tau(\M)''$ (by combining Corollary \ref{cor:dense-subalgebra}\ref{item:dense-subalg1} and Proposition \ref{prop:tracial-completion}\ref{item:completion-fibres}, for example).  Since a trace on a $C^*$-algebra is uniformly amenable if and only if it generates a semidiscrete von Neumann algebra (see \cite[Theorem~3.2.2(2)$\Leftrightarrow$(3)]{Bro06}), we have that for each $\tau \in Y$, $\tau \circ \phi$ is uniformly amenable if and only if $\tau \circ \phi|_A$ is uniformly amenable.  The result follows from Theorem~\ref{thm:amenable}.
\end{proof}

\section{Reduced products and central sequences}
\label{sec:InductiveLimitsETC} 

In this section, we define the reduced product $\prod^\omega(\M_n,X_n)$ associated with a sequence of tracially complete $C^*$-algebras $\big((\M_n,X_n)\big)_{n=1}^\infty$ and a free filter $\omega$ on $\N$. Using this language, and motivated by analogous conditions in von Neumann algebra theory, we introduce the \mbox{McDuff} property in Section~\ref{sec:McDuff} and property $\Gamma$ in Section~\ref{sec:Gamma}.

\subsection{Reduced products}\label{sec:reduced-product}

Reduced products provide an algebraic setting for manipulating properties involving approximations.  The most common constructions are ultrapowers with respect to a free ultrafilter on the natural numbers $\N$ and sequence algebras, consisting of the algebra of bounded sequences modulo the ideal of $c_0$-sequences.  

For many basic applications, ultrapowers and sequence algebras can be used interchangeably, but each has its technical advantages.  In settings where traces are considered, ultrapowers are often more natural as the resulting set of limit traces is already convex; on the other hand, working with sequence algebras allows for the reparameterisation argument of \cite[Theorem~4.3]{Gabe18} (see Theorem~\ref{thm:reparameterisation}), which will be used in our classification result in Section \ref{sec:Classification}.  In order to allow for both constructions simultaneously, we will work with reduced products defined with respect to a free filter.

For the remainder of the paper, $\omega$ will denote a free filter on the natural numbers $\N$.  We recall that a filter $\omega$ on the natural numbers is free if and only if it contains all cofinite sets (see \cite[Appendix~A]{Br08}, for example, for a general discussion of filters).  The following selection theorem of Kirchberg will be used frequently.  It is most often stated for ultrafilters, but the result (and proof) is equally valid for general filters.  For the readers convenience, we include the details.

\begin{lemma}[Kirchberg's $\epsilon$-test, {cf.\ \cite[Lemma A.1]{Kir06}}]
	\label{lem:EpsTest}
	Let $\omega$ be a free filter on $\mathbb N$,
	let $(X_n)_{n=1}^\infty$ be a sequence of non-empty sets, and for $k,n\in\mathbb N$, let $f^{(k)}_n\colon X_n\rightarrow [0,\infty]$ be a function.  Define functions $f^{(k)}\colon \prod_{n=1}^\infty X_n \to [0,\infty]$ by
	\begin{equation}
		\label{eq:EpsTest1}
		f^{(k)}(x_1,x_2,\dots) \coloneqq \limsup_{n\to\omega} f^{(k)}_n(x_n). 
	\end{equation}
	If for every $\epsilon>0$ and $k_0\in\mathbb N$, there exists $x\in \prod_{n=1}^\infty X_n$ such that $f^{(k)}(x)<\epsilon$ for all $k=1,\dots,k_0$, then there exists $y\in \prod_{n=1}^\infty X_n$ such that $f^{(k)}(y)=0$ for all $k\in\mathbb N$.
\end{lemma}

\begin{proof}
	For each $r \in \mathbb N$, there exists $x^{(r)}= (x_1^{(r)},x_2^{(r)},\dots) \in \prod_{n=1}^\infty X_n$ such that $f^{(k)}(x^{(r)}) < \frac1r$ for $k=1,\dots,r$. By \eqref{eq:EpsTest1}, there exists $I_r \in \omega$ such that $f_n^{(k)}(x_n^{(r)}) < \frac1r$ for all $n \in I_r$ and $k=1,\dots,r$. As $\omega$ is a free filter we may assume that $I_r\subseteq \{r,r+1,\dots\}$.
	
	For each $n\in\mathbb N$, if $n$ lies in $\bigcup_{r=1}^\infty I_r$, then let $r_n\in\mathbb N$ be maximal such that $n\in I_{r_n}$ (noting that $n\notin I_r$ for $r>n$) and set $y_n\coloneqq x_n^{(r_n)}$. Otherwise, define $y_n\in X_n$ arbitrarily.  Fix $k,r\in\mathbb N$ with $k\leq r$. Then for all $n\in I_r$, it follows that $r_n\geq r$, and hence
	\begin{equation} f^{(k)}_n(y_n) < \frac1{r_n}\leq\frac{1}{r}. \end{equation}
	Thus 
    \begin{equation}
        f^{(k)}(y) = \limsup_{n\to\omega} f_n(y_n) \leq \sup_{n \in I_r} f_n(y_n) \leq \frac1r.
    \end{equation}
	Since this holds for all $r\geq k$, we obtain $f^{(k)}(y)=0$, as required.
\end{proof} 

We now formally define reduced products.
\begin{definition}\label{def:reduced-product}
For a sequence $\big((\M_n,X_n)\big)_{n=1}^\infty$ of tracially complete $C^*$-algebras, define a $C^*$-algebra\footnote{
Here, $\prod_{n=1}^\infty \M_n$ denotes the $\ell^\infty$-product.  Using \eqref{eq:SpecialHolderIneq}, it is easy to see that \begin{equation*}
    \big\{(a_n)_{n=1}^\infty: \lim_{n\to\omega} \|a_n\|_{2,X_n} = 0\big\}
\end{equation*} 
is an ideal of $\prod_{n=1}^\infty \M_n$, so this quotient is a $C^*$-algebra.}
\begin{equation}
\prod^\omega \M_n \coloneqq  \prod_{n=1}^\infty \M_n \big/ \big\{(a_n)_{n=1}^\infty: \lim_{n\to\omega} \|a_n\|_{2,X_n} = 0\big\}.
\end{equation}
For every sequence of traces $(\tau_n)_{n=1}^\infty \in \prod_{n=1}^\infty X_n$ and every ultrafilter $\omega'$ containing $\omega$, there is a trace defined on $\prod^\omega \M_n$ by $a \mapsto \lim_{n \rightarrow \omega'} \tau_n(a_n)$, where $(a_n)_{n=1}^\infty \in \prod_{n=1}^\infty \M_n$ is any sequence representing $a$ -- such traces are called \emph{limit traces}. Let $\sum^\omega X_n$ be the closed convex hull of the set of limit traces on $\prod^\omega \M_n$.

Then the pair
\begin{equation}
 \prod^\omega (\M_n,X_n) \coloneqq \Big(\prod^\omega \M_n, \sum^\omega X_n\Big)
\end{equation}
is called the \emph{reduced product} of the sequence $\big((\M_n, X_n)\big)_{n=1}^\infty$ with respect to $\omega$ (and the \emph{ultraproduct} when $\omega$ is an ultrafilter).  In the case when $\omega$ is the Fr\'echet filter, we write the reduced product as
\begin{equation}
\prod^\infty (\M_n, X_n) \coloneqq \Big( \prod^\infty \M_n, \sum^\infty X_n \Big).
\end{equation}
\end{definition}

Our first goal is to prove that the reduced product of a sequence of tracially complete $C^*$-algebras (with respect to a given free filter $\omega$) is itself a tracially complete $C^*$-algebra. Before doing that, we isolate a useful lemma that will be used frequently in our analysis of reduced products.  

\begin{lemma}\label{lem:2norm-reducedproduct}
Let $\big((\M_n, X_n)\big)_{n=1}^\infty$ be a sequence of tracially complete $C^*$-algebras. If $a \in \prod^\omega (\M_n, X_n)$ is represented by the sequence $(a_n)_{n=1}^\infty \in \prod_{n=1}^\infty \M_n$, then
\begin{equation}
	\|a\|_{2,\sum^{\omega} X_n} = \limsup_{n\to\omega} \|a_n\|_{2,X_n}.
\end{equation}
\end{lemma}
\begin{proof}
Given a sequence of traces $(\tau_n)_{n=1}^\infty$ and an ultrafilter $\omega' \supseteq \omega$, let $\tau$ be the associated limit trace.
Since $\|a_n\|_{2,\tau_n} \leq \|a_n\|_{2,X_n}$ for all $n \in \N$, we have 
\begin{equation}
	\|a\|_{2,\tau} = \lim_{n\to\omega'}\|a_n\|_{2,\tau_n} \leq \lim_{n\to\omega'}\|a_n\|_{2,X_n} \leq \limsup_{n \to\omega} \|a_n\|_{2,X_n}.
\end{equation}
Hence $\|a\|_{2, \sum^\omega X_n} \leq \limsup_{n \to\omega} \|a_n\|_{2,X_n}$. 

Conversely, let $\omega' \supseteq \omega$ be an ultrafilter with 
\begin{equation}
\lim_{n\to\omega'} \|a_n\|_{2,X_n} = \limsup_{n\to\omega} \|a_n\|_{2,X_n}.
\end{equation}
For every $n \in \N$, there exists $\tau_n \in X_n$ with $\|a_n\|_{2,\tau_n} > \|a_n\|_{2,X_n} - 2^{-n}$. Let $\tau$ be the limit trace corresponding to the sequence $(\tau_n)_{n=1}^\infty$ and the ultrafilter $\omega'$. Then
\begin{equation}
\|a\|_{2, \sum^\omega X_n} \geq \|a\|_{2,\tau} = \lim_{n \to\omega'} \|a_n\|_{2,\tau_n} \geq \limsup_{n\to\omega}\|a_n\|_{2,X_n}.\qedhere
\end{equation}
\end{proof}

We now prove that a reduced product of a sequence of tracially complete $C^*$-algebras is a tracially complete $C^*$-algebra.

\begin{proposition}[{cf.\ \cite[Lemma~1.6]{CETWW}}]\label{prop:reduced-product}
Let $\big((\M_n, X_n)\big)_{n=1}^\infty$ be a sequence of tracially complete $C^*$-algebras. Then $\prod^\omega (\M_n, X_n)$ is a tracially complete $C^*$-algebra.
\end{proposition} 

\begin{proof}
By Lemma \ref{lem:2norm-reducedproduct} and the definition of $\prod^\omega \M_n$, it is clear that $\|\cdot\|_{2, \sum^\omega X_n}$ is a norm. It remains to show that the $\|\cdot\|$-closed unit ball is $\|\cdot\|_{2, \sum^\omega X_n}$-complete.

Let $(a^{(k)})_{k=1}^\infty$ be a $\|\cdot\|_{2,\sum^\omega X_n}$-Cauchy sequence in the unit ball of $\prod^\omega\M_n$, and for each $k\in \N$, fix a sequence $(a_n^{(k)})_{n=1}^\infty$ of contractions which represents $a^{(k)}$. Set
\begin{equation}
  \epsilon^{(k)}\coloneqq \sup\big\{\|a^{(l)}-a^{(l')}\|_{2,\sum^\omega X_n}:l,l'\geq k\big\},
\end{equation}
noting that $\epsilon^{(k)}\rightarrow 0$ as $k\rightarrow\infty$ since $(a^{(k)})_{k=1}^\infty$ is $\|\cdot\|_{2,\sum^\omega X_n}$-Cauchy. Define functions $f^{(k)}_n\colon \{ b \in \M : \|b \| \leq 1 \} \rightarrow [0,\infty]$ by
\begin{equation}
\label{eq:fknDef}
  f^{(k)}_n(b) \coloneqq \max\big\{\|b-a^{(k)}_n\|_{2,X_n}-\epsilon^{(k)},0\big\}.
\end{equation}
For $k_0\in \N$ and $k=1,\dots,k_0$, since $\|a^{(k_0)}-a^{(k)}\|_{2,\sum^\omega X_n}\leq \epsilon^{(k)}$, we have $\limsup_{n\rightarrow\omega}f^{(k)}_n(a_n^{(k_0)})=0$ by Lemma \ref{lem:2norm-reducedproduct}.  Therefore, Kirchberg's $\epsilon$-test (Lemma~\ref{lem:EpsTest}) gives a sequence $(a_n)_{n=1}^\infty$ of contractions representing an element $a$ in the unit ball of $\prod^\omega \M_n$ such that $\limsup_{n\rightarrow\omega}f_n^{(k)}(a_n)=0$ for all $k\in \N$. By Lemma \ref{lem:2norm-reducedproduct} and \eqref{eq:fknDef}, this means that
\begin{equation}
  \|a-a^{(k)}\|_{2,\sum^\omega X_n}\leq\epsilon^{(k)}\rightarrow 0
\end{equation}
as $k\rightarrow\infty$. Hence the unit ball of $\prod^\omega\M_n$ is $\|\cdot\|_{2, \sum^\omega X_n}$-complete.
\end{proof}

A particularly relevant case of a reduced product is when $\big( (\M_n, X_n) \big)_{n=1}^\infty$ is a constant sequence; i.e.\ for some tracially complete $C^*$-algebra $(\M, X)$, we have $(\M_n, X_n) = (\M, X)$ for all $n \in \N$. In this case, we write
\begin{equation}
  (\M^\omega, X^\omega) \coloneqq \prod^\omega (\M_n, X_n)
\end{equation}
and call $(\M^\omega, X^\omega)$ the \emph{reduced power} of $(\M, X)$ with respect to the free filter $\omega$, or the \emph{ultrapower} if $\omega$ is an ultrafilter. Again, when $\omega$ is the Fr\'echet filter, we write $(\M^\infty, X^\infty)$ in place of $(\M^\omega, X^\omega)$.

Note that there is a natural embedding $\iota_{(\M, X)} \colon (\M, X) \rightarrow (\M^\omega, X^\omega)$ of tracially complete $C^*$-algebras given by identifying $\M$ with constant sequences in $\M^\omega$. Typically, the map $\iota_{(\M, X)}$ will be suppressed, and we will view $\M$ as a subalgebra of $\M^\omega$ -- this is the case, for example, when considering the central sequence algebra $\M^\omega \cap \M'$.

For a $C^*$-algebra $A$ with $T(A)$ compact, we write $A^\omega$ for the reduced power of $A$ in the norm $\|\cdot\|_{2, T(A)}$, which is defined in a way analogous to Definition~\ref{def:reduced-product}, and we write $T_\omega(A)$ for the limit traces on $A^\omega$ induced by $T(A)$.  These uniform tracial reduced powers of $C^*$-algebras appear explicitly in connection with the Toms--Winter conjecture in \cite{CETWW} (working with ultrafilters $\omega$) and in the abstract approach to classification in \cite{CGSTW} (working with the Fr\'echet filter).\footnote{Various  related constructions appeared earlier.  A suitable quotient of the norm ultrapower, which is isomorphic to the tracial ultrapower, appeared in \cite{KR14, TWW15}, and all these ideas have their spiritual origins in Matui and Sato's work \cite{MS12,MS14}.  Uniform tracial ultrapowers of $W^*$-bundles also were used in \cite{BBSTWW}.}
In our formalism, the pair $\big(A^\omega, \overline{\rm co}(T_\omega(A))\big)$ is a tracially complete $C^*$-algebra, and in fact, if $(\M, X)$ is the tracial completion of $A$ with respect to $T(A)$, then the canonical map $\alpha_X \colon A \rightarrow \M$ induces an isomorphism
\begin{equation}
  \big(A^\omega, \overline{\rm co}(T_\omega(A))\big) \rightarrow (\M^\omega, X^\omega),
\end{equation}
defined on representative sequences by $(a_n)_{n=1}^\infty  \mapsto (\alpha_X(a_n))_{n=1}^\infty$.
This follows from a more general result on the compatibility of tracial completions: Proposition~\ref{prop:CStarVsTCreducedProducts} below.  We first introduce some more notation.

\begin{definition}
Let $(A_n)_{n=1}^\infty$ be a sequence of $C^*$-algebras and let $X_n \subseteq T(A_n)$ be a compact convex set for each $n \geq 1$. We write
\begin{equation}
\prod^\omega A_n \coloneqq  \prod_{n=1}^\infty A_n \big/ \big\{(a_n)_{n=1}^\infty: \lim_{n\to\omega} \|a_n\|_{2,X_n} = 0\big\}.
\end{equation}
Let $\sum^\omega X_n$ be the closed convex hull of the limit traces on $\prod^\omega A_n$ defined by sequences $(\tau_n)_{n=1}^\infty \in \prod_{n=1}^\infty X_n$.  In the case of a constant sequence, say $A_n = A$ and $X_n = X$, we write $A^\omega = \prod^\omega A_n$ and $X^\omega = \sum^\omega X_n$. 
\end{definition}

\begin{proposition}\label{prop:CStarVsTCreducedProducts}
Let $(A_n)_{n=1}^\infty$ be a sequence of $C^*$-algebras, let $X_n \subseteq T(A_n)$ be a compact convex set for each $n \geq 1$, let $(\M_n, X_n)$ be the tracial completion of $A_n$ with respect to $X_n$ as in Definition~\ref{def:UTCompletion}, and let $\alpha_n \colon A_n \rightarrow \M_n$ be the canonical map for $n \geq 1$.  Then $\big(\prod^\omega A_n, \sum^\omega X_n\big)$ is a tracially complete $C^*$-algebra and there is an isomorphism of tracially complete $C^*$-algebras
\begin{equation}
  \prod^\omega \alpha_n \colon \Big( \prod^\omega A_n, \sum^\omega X_n \Big) 
  \overset\cong\longrightarrow \Big( \prod^\omega \M_n, \sum^\omega X_n \Big),
\end{equation}
defined at the level of representative sequences by $(a_n)_{n=1}^\infty  \mapsto (\alpha_n(a_n))_{n=1}^\infty$.
\end{proposition}
\begin{proof}
It is easy to see that $\alpha \coloneqq \prod^\omega \alpha_n$ is a $^*$-homomorphism and is isometric with respect to the uniform 2-norms, so it suffices to show that $\alpha$ is surjective.  Fix $b \in \prod^\omega \mathcal M_n$ and represent $b$ by a bounded sequence $(b_n)_{n=1}^\infty$.  By the construction of the tracial completion $(\M_n, X_n)$, there is, for each $n \geq 1$, $a_n \in A_n$ such that $\| \alpha_n(a_n) - b_n\|_{2, X_n} < \frac1n$ and $\|a_n\| \leq \|b_n\|$.  Then the sequence $(a_n)_{n=1}^\infty$ defines an element $a \in \prod^\omega A_n$ such that $\alpha(a) = b$.
\end{proof}

In the case of reduced powers, an important additional observation is that the isomorphism defined in Proposition~\ref{prop:CStarVsTCreducedProducts} is well-behaved on the central sequence algebras. 

\begin{proposition}\label{rem:CStarVsTCreducedProducts}
Let $A$ be a $C^*$-algebra, let $X \subseteq T(A)$ be compact and convex, let $(A^\omega, X^\omega)$ be the uniform tracial reduced power, let $\iota:A \rightarrow A^\omega$ be given by constant sequences,
let $(\M, X)$ be the tracial completion of $A$ with respect to $X$, and let $\alpha_X:A \rightarrow \M$ be the canonical map. 
The isomorphism $\alpha^\omega:(A^\omega, X^\omega) \rightarrow (\M^\omega,  X^\omega)$ defined at the level of representative sequences by $(a_n)_{n=1}^\infty  \mapsto (\alpha_X(a_n))_{n=1}^\infty$ satisfies 
\begin{equation}
	\alpha^\omega(A^\omega \cap \iota(A)') = \M^\omega \cap \M'.
\end{equation}
More generally, for any $\|\cdot\|_{2,X}$-separable subset $S \subseteq \M$, there is a $\|\cdot\|$-separable subset $S_0 \subseteq A$ such that $\alpha(A^\omega \cap \iota(S_0)') \subseteq \M^\omega \cap S'$.
\end{proposition}
\begin{proof} 
By Proposition \ref{prop:CStarVsTCreducedProducts}, $\alpha^\omega$ is an isomorphism. 
Since 
\begin{equation}
 \alpha^\omega(\iota(A)) = \alpha_X(A) \subseteq \mathcal M^\omega,
\end{equation}
we have 
\begin{equation}
	\alpha^\omega(A^\omega \cap \iota(A)') = \M^\omega \cap \alpha_X(A)'.
\end{equation}

By Lemma \ref{lem:2norm-reducedproduct}, $\iota_{(\M, X)}$ is an isometry for the respective uniform 2-norms. Hence $\alpha_X(A)$ is  $\|\cdot\|_{2,X^\omega}$-dense in $\M$. 
In a tracially complete $C^*$-algebra, left and right multiplication by a fixed element are continuous with respect to the uniform 2-norm by \eqref{eq:SpecialHolderIneq}. Therefore, any element of $\M^\omega$ that commutes with $\alpha_X(A)$ must also commute with $\M$. Hence  $\M^\omega \cap \alpha_X(A)' = \M^\omega \cap \M'$.

Let $S \subseteq \M$ be $\|\cdot\|_{2,X}$-separable. Since $\alpha_X(A)$ is $\|\cdot\|_{2,X^\omega}$-dense in $\M$, there is a countable subset $S_0 \subseteq A$ such that $\overline{\alpha_X(S_0)}^{\|\cdot\|_{2,X}} \supseteq S$. Then 
\begin{equation}
    \alpha^\omega(A^\omega \cap \iota(S_0)') = \M^\omega \cap \alpha_X(S_0)' = \M^\omega \cap \big(\overline{\alpha_X(S_0)}^{\|\cdot\|_{2,X}}\big)' \subseteq \M^\omega \cap S',
\end{equation}
as claimed.
\end{proof}

As with operator norm reduced products (see \cite[Lemma~3.9.5]{Br08}), matrix amplifications commute with reduced products in the sense of the following theorem.  The proof is essentially the same as in the $C^*$-algebra setting.

\begin{proposition}\label{prop:matrix-ultrapower}
	For any sequence $\big((\M_n, X_n)\big)_{n=1}^\infty$ of tracially complete $C^*$-algebras with reduced product $(\M, X)$ and $d \in \mathbb N$, there is a natural isomorphism
	\begin{equation}\label{eq:matrix-ultrapower}
		\big(\M \otimes M_d, X \otimes \{\mathrm{tr_d}\}\big) \overset\cong\longrightarrow  \prod^\omega \big( \M_n \otimes M_d, X_n \otimes \{\mathrm{tr_d}\}\big).
	\end{equation}
	defined on representing sequences by $(a_n)_{n=1}^\infty \otimes b \mapsto (a_n \otimes b)_{n=1}^\infty$.
\end{proposition}

\begin{proof}
	The natural map 
	\begin{equation}
		\phi \colon \Big(\prod_{n=1}^\infty \M_n \Big)\otimes M_d \longrightarrow \prod_{n=1}^\infty (\M_n \otimes M_d)
	\end{equation}
	is an isomorphism of $C^*$-algebras (see the proof of \cite[Lemma~3.9.4]{Br08}, for example).
   Now suppose $(\tau_n)_{n=1}^\infty \in \prod_{n=1}^\infty T(\M_n)$ and $\omega'$ is an ultrafilter containing $\omega$, and let $\tau$ denote the trace on $\prod_{n=1}^\infty \M_n$ given by
	\begin{equation}
		\tau(a) \coloneqq \lim_{n \rightarrow \omega'} \tau_n(a_n), \qquad a =(a_n)_{n=1}^\infty \in \mathcal \prod_{n=1}^\infty \M_n.
	\end{equation}
	Then for $a =(a_n)_{n=1}^\infty \in \prod_{n=1}^\infty (\M_n \otimes M_d)$, we have
	\begin{equation}
		\lim_{n \rightarrow \omega'} (\tau_n \otimes \mathrm{tr}_d)(a_n) = (\tau \otimes \mathrm{tr}_d)(\phi^{-1}(a)).
	\end{equation}
	Therefore, $\phi$ descends to an isomorphism as in \eqref{eq:matrix-ultrapower}.
\end{proof}

We now turn to the question of traces on reduced products.  It is well known that the tracial ultrapower of II$_1$ factor is also a II$_1$ factor.  We have been unable to answer the general question for tracially complete $C^*$-algebras.  Under the additional hypothesis of CPoU (see Section~\ref{sec:CPoU}), a positive answer to the following question is given in Theorem \ref{thm:no-silly-traces}.  In particular, the answer is affirmative in the presence of property $\Gamma$ (see Section~\ref{sec:Gamma}) by Theorem~\ref{introthmgammaimpliescpou}.

\begin{question}
	If $\big((\M_n, X_n)\big)_{n=1}^\infty$ is a sequence of type II$_1$ factorial tracially complete $C^*$-algebras, is the reduced power $\big(\prod^\omega \M_n, \sum^\omega X_n\big)$ also factorial?
\end{question}

\begin{remark}\label{VaccaroRemark}
In an earlier draft of this paper, we asked this question without the type II$_1$ assumption.  A counterexample in the type I setting was given by Vaccaro in \cite{Va23}.  It is based on a family of $C^*$-algebras introduced in \cite{PP70}: $\M_n$ is the continuous  sections of a bundle over the complex projective space $\mathbb C P^n$ with fibre $M_2$, and $X_n = T(\M_n)$.
\end{remark}
 
We end this subsection with a reparameterisation theorem which will allow us to prove existence results for morphisms given both existence and uniqueness results for approximate morphisms.  This should be regarded as an abstract version of the standard intertwining arguments commonly used in $C^*$-algebra theory.  An operator norm version of this result appears in \cite[Theorem~4.3]{Gabe18}, which, in turn, is a sequential version of \cite[Proposition~1.3.7]{Phillips00}, attributed to Kirchberg.  

If $(\mathcal N, Y)$ is a tracially complete $C^*$-algebra and $r \colon \mathbb N \rightarrow \mathbb N$ is a function such that $\lim_{n \rightarrow \infty} r(n) = \infty$, then there is an induced endomorphism $r^*$ of $(\mathcal N^\infty, Y^\infty)$ given on representing sequences by
\begin{equation}
  r^*\big((b_n)_{n=1}^\infty) \coloneqq (b_{r(n)})_{n=1}^\infty, \qquad (b_n)_{n=1}^\infty \in \ell^\infty(\mathcal N).
\end{equation}
Equivalently, viewing $\ell^\infty(\mathcal N)$ as bounded functions $\mathbb N \rightarrow \mathcal N$, $r^*$ is the map induced by
\begin{equation}
	\ell^\infty(\mathcal N) \rightarrow \ell^\infty(\mathcal N) \colon f \mapsto f \circ r.
\end{equation}
Note that the map $r^*$ would typically not be well-defined if $(\mathcal N^\infty, Y^\infty)$ were replaced with an ultrapower -- this is the main reason for working with general reduced products of tracially complete $C^*$-algebras.

In applications of the following theorem, the metric space $S$ will typically be either a separable $C^*$-algebra with the operator norm or a tracially complete $C^*$-algebra which is separable in its uniform 2-norm.

\begin{theorem}[Intertwining via reparameterisation]\label{thm:reparameterisation}
 Let $(\mathcal N, Y)$ be a tracially complete $C^*$-algebra, let $S$ be a separable metric space, and let $\phi \colon S \rightarrow \mathcal N^\infty$ be a $\|\cdot\|_{2, Y^\infty}$-continuous function. Suppose also that every unitary in $\mathcal N^\infty$ lifts to a unitary in $\ell^\infty(\mathcal N)$.  If for every function $r \colon \mathbb N \rightarrow \mathbb N$ with $\lim_{n \rightarrow \infty} r(n) = \infty$, we have $r^* \circ \phi$ is approximately unitarily equivalent to $\phi$, then there is a $\|\cdot\|_{2, Y^\infty}$-continuous function $\psi \colon S \rightarrow (\mathcal N, Y)$ such that $\iota_{(\mathcal{N}, Y)} \circ \psi$ is unitarily equivalent to $\phi$.
\end{theorem}

The proof is a minor modification of the $C^*$-algebra version obtained in \cite[Theorem~4.3]{Gabe18}, where the reduced power is taken in the operator norm.  In the operator norm setting, the condition on unitaries is automatic since every approximate unitary is close to a genuine unitary.  We do not know if the condition on unitaries in $(\mathcal N^\infty, Y^\infty)$ is necessary.  In all of our applications of Theorem~\ref{thm:reparameterisation}, $(\mathcal N, Y)$ will satisfy CPoU (see Section~\ref{sec:CPoU}), and hence the condition on unitaries will follow from Corollary~\ref{cor:unitary-stable-relation}.

\begin{proof}[Proof of Theorem~\ref{thm:reparameterisation}]
	Define a sequence of functions $\phi_n\colon S \rightarrow Y$ by choosing a $\|\cdot\|$-preserving lift $(\phi_n(b))_{n=1}^\infty \in \ell^\infty(\mathcal{N})$ of $\phi(b) \in \mathcal{N}^\infty$ for every $b \in S$. 
	We claim that for every finite subset $\mathcal F \subseteq S$, every $\epsilon > 0$, and every $m \in \N$, there exists $k \geq m$ such that for every $n \geq k$, there exists a unitary $u \in \mathcal{N}$ such that
	\begin{equation}
		\max_{a \in F}\|u^*\phi_n(a)u - \phi_k(a)\|_{2,Y} < \epsilon,\quad a\in\mathcal F.
	\end{equation}

Suppose for a contradiction that the claim is false. Then there exist a finite set $\mathcal F_0 \subseteq S$, $\epsilon_0 > 0$, $m_0 \in \N$, and a sequence of natural numbers $(n_k)_{k=m_0}^\infty$ with $n_k \geq k$ such that for all unitaries $u \in \mathcal{N}$ and $k\geq m_0$,
	\begin{equation}\label{eqn:contra-1}
		\max_{a \in \mathcal F_0}\|u^*\phi_{n_k}(a)u - \phi_k(a)\|_{2,Y} \geq \epsilon_0.
	\end{equation} 
Let $r:\N \rightarrow \N$ be given by $r(k) \coloneqq n_k$ for $k \geq m_0$ and define $r(k)$ arbitrarily for $k < m_0$. Then $\lim_{k\rightarrow\infty} r(k) = \infty$. By our hypothesis, $\phi$ and $r^* \circ \phi$ are approximately unitarily equivalent. Therefore, there exists a unitary $u \in \mathcal{N}^\infty$ such that
\begin{equation}\label{eqn:contra-2}
	\|u^*(r^* \circ\phi)(a)u - \phi(a)\|_{2,Y^\infty} < \epsilon_0
\end{equation}
for all $a \in \mathcal F_0$. By our hypothesis, we may lift $u$ to a sequence of unitaries $(u_k)_{k=1}^\infty$ in $\mathcal{N}$. By Lemma \ref{lem:2norm-reducedproduct}, we have 
\begin{equation}\begin{split}\label{eqn:contra-3}
	\limsup_{k\rightarrow\infty}\|u_k^*\phi_{n_k}(a)u_k - \phi_k(a)\|_{2,Y} < \epsilon_0
\end{split}\end{equation}
for all $a \in \mathcal F_0$. 
Then \eqref{eqn:contra-3} contradicts \eqref{eqn:contra-1} for some sufficiently large $k \in \N$. This proves the claim. We now use this to construct $\psi$. 

Let $(\mathcal F_i)_{i=1}^\infty$ be an increasing sequence of finite subsets whose union is dense in $S$. Applying the claim recursively, we obtain an increasing sequence of natural numbers $(k_n)_{n=1}^\infty$ and a sequence of unitaries $(u_n)_{n=1}^\infty$ in $\mathcal{N}$ such that
\begin{equation}
	\|u_n^*\phi_{k_n}(a)u_n - \phi_{k_{n-1}}(a)\|_{2,Y} < 2^{-n}
\end{equation} 
for all $a \in \mathcal F_n$. Set $v_n\coloneqq u_nu_{n-1}\cdots u_1$ (and put $v_0 \coloneqq 1_{\mathcal N}$). Then, since $v_n = u_nv_{n-1}$ and $\|v_{n-1}\| \leq 1$, we have
\begin{equation}\begin{split}
	\|v_n^*\phi_{k_n}(a)v_n - v_{n-1}^*\phi_{k_{n-1}}(a)v_{n-1}\|_{2,Y} < 2^{-n}
\end{split}\end{equation}
for all $n \in \N$ and $a \in \mathcal F_n$. 

By construction, for every $a \in \bigcup_{i=1}^\infty \mathcal F_i$, the sequence $(v_n^*\phi_{k_n}(a)v_n)_{n=1}^\infty$ is $\|\cdot\|_{2,Y}$-Cauchy. Indeed, if $a \in \mathcal F_{i_0}$, then for $n > m > i_0$, we  have
\begin{equation}
	\|v_m^*\phi_{k_m}(a)v_m - v_n^*\phi_{k_n}(a)v_n\|_{2,Y} < \sum_{i=n+1}^m 2^{-i} < 2^{-n}.
\end{equation}

Let $b \in S$ and $\epsilon > 0$. Since $\bigcup_{i=1}^\infty \mathcal F_i$ is dense in $S$ and $\phi\colon S \rightarrow (\mathcal{N}^\infty, Y^\infty)$ is $\|\cdot\|_{2,Y^\infty}$-continuous, there exist $i_0 \in \N$ and $a \in F_{i_0}$ with \begin{equation}
    \|\phi(b) - \phi(a)\|_{2,Y^\infty} < \epsilon.
\end{equation}
Hence, by Lemma \ref{lem:2norm-reducedproduct}, we have
\begin{equation}
	\limsup_{k\to\infty}\|\phi_k(b) - \phi_k(a)\|_{2,Y} < \epsilon.
\end{equation}
Choose $N_1 \in \N$ with $\|\phi_k(b) - \phi_k(a)\|_{2,Y} < \epsilon$ for all $k > N_1$, and choose $N_2 \in \N$ with $\|v_m^*\phi_{k_m}(a)v_m -v_n^*\phi_{k_n}(a)v_n\|_{2,Y} < \epsilon$ for all $n,m > N_2$. Then a simple $3\epsilon$-argument gives that
\begin{equation}
	\|v_m^*\phi_{k_m}(b)v_m - v_n^*\phi_{k_n}(b)v_n\|_{2,Y} < 3\epsilon
\end{equation}
whenever $n,m > \max(N_1,N_2)$. Hence $(v_n^*\phi_{k_n}(b)v_n)_{n=1}^\infty$ is $\|\cdot\|_{2,Y}$-Cauchy for all $b \in S$. Moreover, we have 
\begin{equation}
    \|v_n^*\phi_{k_n}(b)v_n\| = \|\phi_n(b)\| \leq  \|\phi(b)\|, \quad n \in \mathbb N.
\end{equation}
Since $(\mathcal{N},Y)$ is a tracially complete $C^*$-algebra,  $(v_n^*\phi_{k_n}(b)v_n)_{n=1}^\infty$ converges in the uniform 2-norm for all $b \in S$.
Hence we may define $\psi\colon S \rightarrow (\mathcal{N},Y)$ by $\psi(b) \coloneqq \lim_{n\to\infty} v_n^*\phi_{k_n}(b)v_n$. 
  
By construction, $\iota_{(\mathcal{N}, Y)} \circ \psi$ is unitarily equivalent to $r^* \circ \phi$, where $r:\N \rightarrow \N$ is given by $r(n) \coloneqq k_n$, via the unitary represented by the sequence $(v_{k_n})_{n=1}^\infty$. Since $r^* \circ \phi$ is unitarily equivalent to $\phi$ by hypothesis (using Kirchberg's $\epsilon$-test to replace approximate unitary equivalence with unitary equivalence), we conclude that $\iota_{(\mathcal{N}, Y)} \circ \psi$ is unitarily equivalent to $\phi$. 
\end{proof}

\subsection{The McDuff property}\label{sec:McDuff}

Central sequences in II$_1$ factors have been studied beginning with the foundational work of Murray and von Neumann (\cite{MvN43}) when they used them to distinguish the hyperfinite II$_1$ factor from any free group factor. The systematic study of central sequences was later instigated by two breakthrough results of McDuff: the existence of infinitely many (and in fact, uncountably many) non-isomorphic II$_1$ factors (\cite{Du69a,Du69b}) and, of particular relevance to us, her characterisation those II$_1$ factors $\M$ with separable predual which tensorially absorb the hyperfinite II$_1$ factor (i.e.\ $\M\cong\M\barotimes\R$) as those $\mathcal M$ with  non-abelian central sequence algebras.  Motivated by this last result, a II$_1$ factor $\M$ is said to have the \emph{McDuff property} if there are approximately central unital embeddings of matrix algebras into $\M$.\footnote{From a ultrapower viewpoint, McDuff's work shows that the central sequence algebra $\mathcal M^\omega\cap \mathcal M'$ is non-abelian if and only if it is type II$_1$ (and so admits unital embeddings of all matrix algebras).  For II$_1$ factors with non-separable predual, the equivalence of tensorial absorption of $\mathcal R$ and the McDuff property no longer holds (\cite[Theorem~1.5]{FHKT}). Experience has shown that the formulation in terms of approximately central matrix embeddings is the correct way to extend the McDuff property to the non-separable predual situation (where one should of course work with ultrapowers over uncountable sets).  Indeed, the ultrapower $\mathcal M^\omega$ of a McDuff II$_1$ factor $\M$ with separable predual has the McDuff property but is not stable under tensoring by $\mathcal R$ -- see \cite[Footnote~65]{CGSTW}, which uses \cite{Gh15}.}

The McDuff property of II$_1$ factors has been of considerable interest to $C^*$-algebraists working in the classification programme because of its relation to $\Z$-stability at both conceptual and technical levels (\cite{TW07,MS12,KR14,TWW15,Sa12,CETW,CGSTW}).

In this section, we generalise the McDuff property to the setting of tracially complete $C^*$-algebras and show that the McDuff property is equivalent to $\R$-stability in the separable case (Theorem~\ref{thm:McDuff-tensor}). We begin with a formal definition of the McDuff property before establishing some useful technical reformulations and permanence properties.

\begin{definition}[{cf.\ \cite[Definition~4.2]{CETW}}]\label{def:UTCmcDuff}
Let $(\M,X)$ be a tracially complete $C^*$-algebra.  We say that $(\M,X)$ is \emph{McDuff} if for any $\|\cdot\|_{2, X}$-separable set $S \subseteq \M$ and $k \geq 1$, there is a unital embedding $M_k \rightarrow \M^\omega \cap S'$.\footnote{Of course, when $\M$ is $\|\cdot\|_{2, X}$-separable, we can just take $S=\M$.}
\end{definition}

The \emph{uniform McDuff property} was defined for a separable $C^*$-algebra $A$ with $T(A)$ compact in \cite[Definition 4.2]{CETW} as the existence of unital embeddings $M_k \rightarrow A^\omega \cap A'$ for all $k \geq 1$. By Proposition~\ref{rem:CStarVsTCreducedProducts}, this is consistent with our definition.

\begin{proposition}
    Let $A$ be a separable $C^*$-algebra with $T(A)$ compact.
    Then $A$ is uniformly McDuff in the sense of \cite{CETW} if and only if its tracial completion with respect to $T(A)$ is McDuff in the sense of Definition~\ref{def:UTCmcDuff}.
\end{proposition}

We now establish some equivalent reformulations of the McDuff property. Note that the equivalence of \ref{item:McDuff} and \ref{item:McDuff-approx} in the following theorem shows that the McDuff property is independent of choice of free filter $\omega$. The result is standard in the setting of $\mathrm{II}_1$ factors, and the same techniques work in this context.

\begin{proposition}[{cf.\ \cite[Proposition~3.11]{BBSTWW}}]\label{prop:McDuff}
Let $(\M,X)$ be a tracially complete $C^*$-algebra. The following are equivalent:
\begin{enumerate}[wide, labelwidth=0pt, labelindent=0pt]
  \item\label{item:McDuff} $(\M, X)$ is McDuff.
   \item\label{item:McDuff-approx} Given a finite set $\mathcal F \subseteq \M$ and $\epsilon > 0$, there is a contraction $v \in \M$ such that
   \begin{equation}
	  \max_{a \in \mathcal F} \|[v, a]\|_{2, X} < \epsilon,\ \|v^*v + vv^* -1_\mathcal M\|_{2, X} < \epsilon,\ \text{and}\ \|v^2\|_{2, X} < \epsilon.
   \end{equation}

   \item\label{item:v-reduced-product} For each $\|\cdot\|_{2, X^\omega}$-separable subset $S \subseteq \mathcal{M}^\omega$, there is a contraction $v \in \M^\omega \cap S'$ such $v^2 = 0$ and $vv^* + vv^* = 1_{\M^\omega}$.
   
   \item\label{item:approx-central-matrix} For each $\|\cdot\|_{2, X^\omega}$-separable subset $S \subseteq \mathcal{M}^\omega$, there is $k \geq 2$ and a unital embedding $M_k \rightarrow \M^\omega \cap S'$.
  
  \item\label{item:approx-central-R}For each $\|\cdot\|_{2, X^\omega}$-separable subset $S \subseteq \mathcal{M}^\omega$, there is a unital embedding $\R \rightarrow \M^\omega \cap S'$.
\end{enumerate}
\end{proposition}

\begin{proof} \ref{item:McDuff}$\Rightarrow$\ref{item:McDuff-approx}:  Given a finite set $\mathcal F \subseteq \M$, fix a unital embedding $\phi\colon M_2 \rightarrow \M^\omega \cap \mathcal F'$.  Let $(v_k)_{k=1}^\infty \subseteq \M$ be a sequence of contractions representing $\phi(e_{1, 2})$ and set $v \coloneqq v_k$ for some suitable index $k$.

\ref{item:McDuff-approx}$\Rightarrow$\ref{item:v-reduced-product}:
Let $\{s^{(i)}:i\in\N\}$ be a countable dense set in $S$. For each $i\in\N$, let $(s^{(i)}_n)_{n=1}^\infty$ be a sequence representing $s^{(i)}$.
For each $n \in \N$, set $\epsilon_n \coloneqq 2^{-n}$ and $\mathcal{F}_n \coloneqq \{ s^{(i)}_n: i \leq n\}$. Let $v_n \in \M$ be given as in \ref{item:McDuff-approx} with $(\mathcal{F}_n, \epsilon_n)$ in place of $(\mathcal F, \epsilon)$.  Then the sequence $(v_n)_{n=1}^\infty$ induces a contraction $v \in \M^\omega \cap S'$ with $v^2 = 0$ and $v^*v + vv^* = 1_{\M^\omega}$.

\ref{item:v-reduced-product}$\Rightarrow$\ref{item:approx-central-matrix}:
Let $v \in \M^\omega \cap S'$ be as in \ref{item:v-reduced-product}. 
Since $v$ is a contraction and $v^2 = 0$, we have $v^*v$ and $vv^*$ are orthogonal positive contractions.  As $v^*v + vv^* = 1_\mathcal M$, $v^*v$ and $vv^*$ are projections, and in particular, $v$ is a partial isometry.  It follows that the $C^*$-subalgebra generated by $v$ is spanned by $1_\mathcal M$, $v$, $v^*$, and $v^*v$.  As this subalgebra is non-commutative and has dimension at most four, it is isomorphic to $M_2$.  This verifies \ref{item:approx-central-matrix} with $k = 2$.

\ref{item:approx-central-matrix}$\Rightarrow$\ref{item:approx-central-R}: 
Fix the subset $S$ from \ref{item:approx-central-R}. Using \ref{item:approx-central-matrix}, let $k_1\geq 2$ and let $\phi_1 \colon M_{k_1} \rightarrow \M^\omega \cap S'$ be a unital embedding.  Let $S_1 \coloneqq C^*(\phi_1(M_{k_1}) \cup S)$ and use \ref{item:approx-central-matrix} again to produce an integer $k_2 \geq 2$ and a unital embedding $\phi_2 \colon M_{k_2} \rightarrow \M^\omega \cap S_1'$.  Continuing inductively, there are integers $k_n \geq 2$ and unital embeddings $\phi_n \colon M_{k_n} \rightarrow \M^\omega \cap S'$ with commuting ranges.  The $\phi_n$ induce a unital embedding
\begin{equation}
	 \R \cong \overline{\bigotimes}_{n=1}^{\,\infty} (M_{k_n}, \mathrm{tr}_{k_n}) \rightarrow \M^\omega \cap S'.
\end{equation}

\ref{item:approx-central-R}$\Rightarrow$\ref{item:McDuff}: This follows as there is a unital embedding $M_k \rightarrow \R$ for all $k \geq 1$.
\end{proof}

\begin{remark}\label{rmk:ultrapower-commutant}
Other equivalent properties to McDuffness can be given by strengthening Definition~\ref{def:UTCmcDuff} to allow $S$ to be any $\|\cdot\|_{2, X^\omega}$-separable subset of $\M^\omega$ or weakening any of \ref{item:v-reduced-product}, \ref{item:approx-central-matrix} or \ref{item:approx-central-R} to only have $S \subseteq \M$.
In all cases, the argument (for the non-trivial direction) is by reindexing.

For example, to go from the weakening of \ref{item:approx-central-matrix} to \ref{item:approx-central-matrix}, given a $\|\cdot\|_{2, X}$-separable subset $S\subseteq \M$, let $T\subseteq \M$ be a countable set consisting of the entries of sequences lifting a countable dense subset of $S$.
Then, given a unital embedding $\phi \colon M_k \to \M^\omega \cap T'$, we may lift $\phi$ to a sequence of c.p.c.\ maps $\phi_n:M_k \to \M$ which are asymptotically multiplicative, unital, and commute with $T$.
An appropriate reindexing $(\phi_{m_n})_{n=1}^\infty$ will provide a unital embedding $\phi\colon M_k \to \M^\omega \cap S'$.
\end{remark}

As a corollary of the equivalence between \ref{item:McDuff} and \ref{item:McDuff-approx} above, we obtain the following permanence properties.

\begin{corollary}\label{cor:McDuff-indlim-ultraproducts}
Inductive limits and reduced products of McDuff tracially complete $C^*$-algebras are McDuff.
\end{corollary}

The following result gives a large supply of examples of McDuff tracially complete $C^*$-algebras.  The result extends \cite[Proposition~2.3]{CETWW}, and the proof is very similar. 

\begin{proposition}[{cf.\ \cite[Proposition~2.3]{CETWW}}]\label{prop:z-stable-mcduff}
If $A$ is a separable $\mathcal Z$-stable $C^*$-algebra and $X \subseteq T(A)$ is a compact convex set, then $(\completion{A}{X}, X)$ is a McDuff tracially complete $C^*$-algebra.
\end{proposition}
\begin{proof}
Let $\omega$ be a free ultrafilter on $\mathbb N$, let $A_\omega \coloneqq \ell^\infty(A) / c_\omega(A)$ be the operator norm ultrapower of $A$, and let $\M \coloneqq \completion{A}{X}$.  By \cite[Proposition~4.4(4)]{Kir06}, there is a unital embedding $\phi \colon \mathcal Z \rightarrow (A_\omega \cap A')/A^\perp$, where 
\begin{equation}
    A^\perp \coloneqq \{a \in A_\omega : ab = ba = 0 \text{ for all } b \in A \}.
\end{equation}
The natural map $\alpha_X \colon A \rightarrow \M$ induces a $^*$-homomorphism $q \colon A_\omega \rightarrow \M^\omega$.  Since $\alpha_X$ maps the unit ball of $A$ onto a $\|\cdot\|_{2, X}$-dense subset of the unit ball of $\M$, we have $q$ is surjective and $q(A_\omega \cap A') \subseteq \M^\omega \cap \M'$.

Following the proof of \cite[Lemma~1.10]{CETWW}, we show $q(A^\perp) = 0$.  Assume $b \in A^\perp$ with $0 \leq b \leq 1$.  Let $\tau \in X^\omega \subseteq T(\mathcal M^\omega)$ be given and note that
\begin{equation}\label{eq:z-stable-mcduff-1}
  \tau \circ q|_A = \tau|_{\mathcal M} \circ \alpha_X.
\end{equation}
Combining Corollary~\ref{cor:dense-subalgebra}\ref{item:dense-subalg1} with \eqref{eq:z-stable-mcduff-1} shows that $\tau \circ q|_A$ has norm 1.  Fix $\epsilon > 0$ and a positive contraction $e\in A$ with $\tau(q(e)) \geq 1 - \epsilon$.  Since $b \in A^\perp$ and $e \in A$ are orthogonal positive contractions, their sum is also a positive contraction.  Therefore,
\begin{equation}\label{eq:z-stable-mcduff-2}
  0 \leq \tau(q(b)) = \tau(q(b + e)) - \tau(q(e)) \leq 1 - (1 - \epsilon) = \epsilon.
\end{equation}
Since $\epsilon > 0$ was arbitrary, $\tau(q(b)) = 0$ for all $\tau \in X^\omega$, and since $q(b) \geq 0$, this implies $q(b) = 0$.  Hence $q(A^\perp) = 0$.

Let $\bar{q} \colon (A_\omega \cap A') / A^\perp \rightarrow \M^\omega \cap \M'$ be the $^*$-homomorphism determined by $q$.  Then $\bar{q} \circ \phi \colon \mathcal Z \rightarrow \M^\omega \cap \M'$ is a unital $^*$-homomorphism.  As $\M^\omega \cap \M'$ is a $\|\cdot\|_{2, X^\omega}$-closed, unital $C^*$-subalgebra of $\M^\omega$, we may view $\M^\omega \cap \M'$ as a tracially complete $C^*$-algebra as in the comments preceding Definition~\ref{def:embedding}.  Since $\mathcal Z$ has a unique trace $\tau$ and $\pi_\tau(\mathcal Z)'' \cong \R$, Proposition~\ref{prop:extend-by-continuity} allows us to extend $\bar q \circ \phi$ to a unital embedding $\R \rightarrow \M^\omega \cap \M'$. Since $M_k$ embeds unitally in $\R$ for all $k \in \N$, we see that $\M$ has the McDuff property. 
\end{proof}

The following result gives a tensor product characterisation of McDuff tracially complete $C^*$-algebras in the separable setting analogous to McDuff's original result for II$_1$ factors from \cite{Du70}. Our proof is an adaptation of the argument found in \cite[Proposition~3.11]{BBSTWW} in the setting of $W^*$-bundles. These follow the framework of the analogous results for absorption of a strongly self-absorbing $C^*$-algebra (see \cite[Theorem~2.2]{TW07} or \cite[Theorem~7.2.2]{Rordam-Book}), which are powered by an Elliott intertwining argument.

\begin{theorem}\label{thm:McDuff-tensor}
Let $(\M, X)$ be a $\|\cdot\|_{2,X}$-separable tracially complete $C^*$-algebra.  Then $(\M, X)$ is McDuff if and only if $(\M,X) \barotimes (\R,\tr_\R) \cong (\M,X)$.
\end{theorem}

\begin{proof}
Since  $\M$ is $\|\cdot\|_{2, X}$-separable and McDuff, there exists a unital embedding $\phi\colon \R \rightarrow \M^\omega \cap \M'$.
Take a set-theoretic lift of $\phi$ to a sequence $(\phi_{n}\colon \R \to \M)_{n=1}^\infty$  of maps.
We define commuting unital embeddings 
\begin{equation}
    \alpha,\beta\colon \R \rightarrow (\M \barotimes \R)^\omega \cap (\M \otimes 1_{\R})'
\end{equation}
at the level of representative sequences by 
\begin{equation}
    \alpha(x)\coloneqq (1_{\M} \otimes x)_{n=1}^\infty \quad \text{and} \quad \beta(x) \coloneqq (\phi_n(x) \otimes 1_{\R} )_{n=1}^\infty,
\end{equation} respectively. 

Let $M_{2^\infty}$ be the UHF algebra with supernatural number $2^\infty$, which we view as a $\|\cdot\|_{2,\tr_\R}$-dense subalgebra of $\R$.
As $M_{2^\infty}$ is nuclear, there exists an embedding 
\begin{equation}
    \gamma\colon  \M \otimes M_{2^\infty}  \otimes M_{2^\infty}  \longrightarrow (\M \barotimes \R)^\omega
\end{equation}
such that 
\begin{equation}
    \gamma(a \otimes b \otimes c) = (a \otimes 1_\R)\alpha(b)\beta(c)= (a \otimes b)\beta(c)
\end{equation}
for all $a \in \M$ and $b,c \in M_{2^\infty}$.
Since  $M_{2^\infty}$ has a unique trace, it follows that $\gamma$ is trace-preserving\footnote{To ease notation, set $A \coloneqq \mathcal M \otimes  M_{2^\infty}$ and $(\mathcal N, Y) = (\mathcal M, X) \bar\otimes (\mathcal R, \mathrm{tr}_\mathcal R)$, and view $A$ as a $\|\cdot\|_{2, Y}$-dense subalgebra of $\mathcal N$.  Then $\gamma$ is a map $A \otimes M_{2^\infty} \rightarrow (\mathcal N \otimes \mathcal R)^\omega$ given by $\gamma(a \otimes c) = a \beta(c)$.  It suffices to show $\gamma^*(\tau) = \tau|_A \otimes \mathrm{tr}_{2^\infty}$ for all $\tau \in Y^\omega$.  Fix $\tau \in Y^\omega$ and positive $a \in A$.  If $\tau(a) = 0$, then $\tau(a \beta(c)) = 0 = \tau(a) \mathrm{tr}_{2^\infty}(c)$ by the Cauchy--Schwarz inequality.  If $\tau(a) \neq 0$, then $c \mapsto \tau(\gamma(a \otimes c)) / \tau(a)$ is a trace on $M_{2^\infty}$ and hence is the unique trace $\mathrm{tr}_{2^\infty}$.  Therefore, for all $a \in A$ and $c \in M_{2^\infty}$, we have $\tau(\gamma(a \otimes c)) = \tau(a) \mathrm{tr}_{2^\infty}(c) = (\tau \otimes \mathrm{tr}_{M_{2^\infty}})(a\otimes c)$.}
and, accordingly,  extends to an embedding of the tracial completion 
\begin{equation}
    \bar{\gamma}\colon \M \barotimes \R \barotimes \R \longrightarrow (\M \barotimes \R)^\omega
\end{equation}
with 
\begin{equation}
    \bar{\gamma}(a \otimes b \otimes c) = (a \otimes 1_\R)\alpha(b)\beta(c) = (a \otimes b)\beta(c)
\end{equation}
for all  $a \in \M$ and $b,c \in \R$.

Recall that $\R$ has an approximately inner half-flip in the sense that there exists unitaries $(\bar{u}_m)_{m=1}^\infty$ in $\R \barotimes \R$ such that for all $b \in \R$,
\begin{equation}
	\lim_{m\to\infty} \|\bar{u}_m^*(b \otimes 1_\R) \bar{u}_m - 1_\R \otimes b\|_{2, \tr_{\R} \otimes \tr_{\R}} = 0.
\end{equation}
Moreover, since $\R \barotimes \R$ is a von Neumann algebra, each unitary $\bar{u}_m$ is of the form $\exp(2\pi i \bar{h}_m)$ for some self-adjoint $\bar{h}_m \in \R\barotimes\R$.

Set $u_m \coloneqq \bar{\gamma}(1_{\M} \otimes \bar{u}_m) \in (\M \barotimes \R)^\omega \cap (\M \otimes 1_\R)'$ (it is in this commutant because $\bar\gamma(\M \otimes 1_\R \otimes 1_\R) = \M \otimes 1_\R$).
Then the two embeddings $\M \barotimes \R \rightarrow (\M \barotimes \R)^\omega$ given by 
\begin{equation}
    a \otimes b \mapsto \bar{\gamma}(a \otimes b \otimes 1_\R)=a\otimes b\ \ \text{and}\ \  a \otimes b \mapsto \bar{\gamma}(a \otimes 1_\R \otimes b)=(a\otimes 1)\beta(b)
\end{equation}
are approximately unitary equivalent in $\|\cdot\|_{2,(X \otimes \{\tr_\R\})^\omega}$ via the sequence $(u_m)_{m=1}^\infty$.
Since $\beta(b)\in (\M\otimes 1_{\R})^\omega$, it follows that
\begin{equation}\label{eqn:inf-formula}
	\lim_{m\to\infty} \inf\Big\{ \|u_m^* x u_m - z\|_{2,(X \otimes \{\tr_\R\})^\omega}: \begin{matrix}z \in (\M \otimes 1_\R)^\omega \\ \|z\| \leq 1\end{matrix}\Big\} = 0
\end{equation}
for all contractions $x \in \M \otimes \R$. We also have $u_m = \exp(2\pi i h_m)$, where $h_m = \bar{\gamma}(1_{\M} \otimes \bar{h}_m)$. Since we can lift the self adjoint $h_m \in (\M \barotimes \R)^\omega$ to a sequence of self-adjoint elements $(h_{m,n})_{n=1}^\infty$ in $\ell^\infty(\M \barotimes \R)$, we can find a sequence of unitaries $(u_{m,n})_{n=1}^\infty$ in $\M \otimes \R$ representing $u_m$. 

We are now ready to construct an isomorphism $\psi\colon \M \rightarrow \M \otimes \R$. 
Let $(x_k)_{k=1}^\infty$ be a $\|\cdot\|_{2,X}$-dense sequence in the unit ball of $\M$ and $(y_k)_{k=1}^\infty$ be a $\|\cdot\|_{2,X \otimes \{\tr_\R\}}$-dense sequence in the unit ball of $\M \otimes \R$. 
We shall iteratively produce unitaries $w_k$ in $\M \otimes \R$ and contractions $z_k^{(j)} \in \M$ for $k\in\mathbb N$ and $1 \leq j \leq k$.
The construction begins with $w_0 \coloneqq 1$. Fix $k \geq 1$ and suppose that $w_s, z^{(r)}_s \in \M$ for $1 \leq r \leq s < k$ have already been constructed.
By \eqref{eqn:inf-formula}, there exists contractions $\bar{z}^{(1)},\ldots,\bar{z}^{(k)} \in  (\M \otimes 1_\R)^\omega$ and $m_k \in \N$ such that
\begin{equation}\begin{split}
	\|u_{m_k}^* w_{k-1}^* y_j w_{k-1}u_{m_k} - \bar{z}_k^{(j)} \otimes 1_\R\|_{2,(X \otimes \{\tr_\R\})^\omega} &< \frac{1}{k}.
\end{split}\end{equation}
for $j=1,\ldots,k-1$.
Let $(u_{m_k,n})_{n=1}^\infty$ be a sequence of unitaries in $\M \otimes \R$ representing $u_{m_k}$ and let $(z^{(j)}_{k,n})_{n=1}^\infty$ be sequences of contractions representing each $\bar{z}_k^{(j)}$. 
Since $u_{m_k} \in (\M \barotimes \R)^\omega \cap (\M \otimes 1_\R)'$, we can choose $n_k \in \N$ such that
\begin{align}\
	\|u_{m_k,n_k}^* w_{k-1}^* y_j w_{k-1}u_{m_k,n_k} - z^{(j)}_{k,n_k} \otimes 1_\R\|_{2,X \otimes \{\tr_\R\}} &< \frac{1}{k},
	\label{eqn:inductive-McDuff1}\\
	\|[u_{m_k,n_k},x_j\otimes 1_\R]\|_{2,X \otimes \{\tr_\R\}} &< 2^{-k}, \label{eqn:inductive-McDuff2}
\end{align}
for all $1 \leq j \leq k$, and 
\begin{equation}\begin{split}\label{eqn:inductive-McDuff3}
	\|[u_{m_k,n_k},z^{(r)}_{s}\otimes 1_\R]\|_{2,X \otimes \{\tr_\R\}} &< 2^{-k},
\end{split}\end{equation}
for all $1\leq r \leq s < k$. 
Set $w_k \coloneqq w_{k-1}u_{m_k,n_k}$ and set $z^{(j)}_k \coloneqq z^{(j)}_{k,n_k}$ for $j=1,\ldots,k$.

For $j \in \N$, the sequence $(w_k (x_j \otimes 1_\R) w_k^*)_{k=1}^\infty$ is $\|\cdot\|_{2,X \otimes \{\tr_\R\}}$-Cauchy by \eqref{eqn:inductive-McDuff2} as $\sum_{k=1}^\infty 2^{-k}$ converges. Since $(x_j)_{j=1}^\infty$ is  a $\|\cdot\|_{2,X}$-dense sequence in the unit ball of $\M$, a $3\epsilon$-argument gives that  $(w_k (a \otimes 1_\R) w_k^*)_{n=1}^\infty$ is $\|\cdot\|_{2,X \otimes \{\tr_\R\}}$-Cauchy for all $a \in \M$. As $(w_k (a \otimes 1_\R) w_k^*)_{n=1}^\infty$ is $\|\cdot\|$-bounded, we may define a map $\psi\colon \M \rightarrow \M \otimes \R$ by $\psi(a) \coloneqq \lim_{k\to\infty} w_k (a \otimes 1_\R) w_k^*$. Since the $w_k$ are unitaries, we see that $\psi$ is an injective $^*$-homomorphism and we have $(\tau \otimes \tr_{\R}) \circ \psi = \tau$ for all $\tau \in X$. Surjectivity follows from the $\|\cdot\|_{2,X \otimes \{\tr_\R\}}$-density of $(y_j)_{j=1}^\infty$ in the unit ball of $\M$, because for $k > j$, we have
\begin{equation}\begin{array}{rcl}
	\|y_j - \psi(z^{(j)}_k)\|_{2,X \otimes \{\tr_\R\}} \!\! &\stackrel{\eqref{eqn:inductive-McDuff3}}{\leq}& \!\displaystyle\sum_{r=k+1}^\infty 2^{-r} + \|w_k(z^{(j)}_{k,n_k} \otimes 1_{\R})w_k^*\|  \\
	&\stackrel{\eqref{eqn:inductive-McDuff1}}{<}& \!\displaystyle\sum_{r=k+1}^\infty 2^{-r} + \frac{1}{k},
\end{array}\end{equation} 
which converges to zero as $k \to \infty$. Indeed, as the unit ball of $\M$ is \mbox{$\|\cdot\|_{2,X}$}-complete and $\psi$ preserves the uniform 2-norm, its image must be \mbox{$\|\cdot\|_{2,X \otimes \{\tr_\R\}}$}-closed.
\end{proof}

\subsection{Property \texorpdfstring{$\Gamma$}{Gamma}}\label{sec:Gamma}

In its original formulation, a II$_1$ factor $\mathcal M$ has property $\Gamma$ if there is an approximately central net of trace-zero unitaries in $\mathcal M$. Since $\R \cong \overline{\bigotimes}_{n=1}^\infty(\R, \tr_\R)$, all McDuff II$_1$ factors have property $\Gamma$.
On the other hand, Murray and von Neumann's  $14\epsilon$-argument shows that the factors associated to free groups do not have property $\Gamma$ (\cite[Lemma 6.2.1]{MvN43}).

Dixmier extended the work of Murray and von Neumann, proving that property $\Gamma$ was equivalent to the existence of approximately central projections that sum to the unit and equally divide the trace (\cite{Di69}).  It is through this reformulation that most structural consequence of property $\Gamma$ are obtained for II$_1$ factors (see \cite{Chr01,CPSS03,GP98,Pi01}), and so it was the basis for the definition of \emph{uniform property $\Gamma$} for $C^*$-algebras introduced in \cite{CETWW} and studied further in \cite{CETW}.  

Here, we define property $\Gamma$ for tracially complete $C^*$-algebras. It is an immediate consequence of Proposition~\ref{rem:CStarVsTCreducedProducts} and a reindexing argument (see Remark \ref{rmk:gamma-ultrapower-commutant} below) that a $C^*$-algebra $A$ with $T(A)$ compact has uniform property $\Gamma$ in the sense of \cite[Definition 2.1]{CETWW} if and only if its tracial completion with respect to $T(A)$ has property $\Gamma$ as defined here (see Proposition~\ref{prop:Gamma-Agrees}).

\begin{definition}[{cf.\ \cite[Definition~2.1]{CETWW}}]\label{def:UTCgamma}
Let $(\M,X)$ be a factorial tracially complete $C^*$-algebra.  We say that $(\M,X)$ has \emph{property $\Gamma$} if for any $\|\cdot\|_{2, X}$-separable subset $S \subseteq \M$ and any $k \in \N$ there exist projections $p_1,\ldots,p_k \in \M^\omega \cap S'$ summing to $1_{\M^\omega}$ such that
\begin{equation}\label{eq:tracially-divides}
  \tau(ap_i) = \frac1k\tau(a),\qquad  a \in S,\ \tau\in X^\omega,\ i=1,\dots,k.
\end{equation}
\end{definition}

\begin{proposition}\label{prop:Gamma-Agrees}
    Let $A$ be a separable $C^*$-algebra with $T(A)$ compact.
    Then $A$ has uniform property $\Gamma$ as in \cite[Definition~2.1]{CETWW} if and only if its tracial completion with respect to $T(A)$ has property $\Gamma$ in the sense of Definition~\ref{def:UTCgamma}.
\end{proposition}

\begin{remark}
We have chosen to restrict our definition of property $\Gamma$ to the case of factorial tracially complete $C^*$-algebras. Although Definition~\ref{def:UTCgamma} makes sense for non-factorial tracially complete $C^*$-algebras, in the absence of results, we are not confident that this would be the appropriate definition outside of the factorial setting.
\end{remark}

The following simple observation provides many examples of factorial tracially complete $C^*$-algebras with property $\Gamma$.  Combining this with Proposition~\ref{prop:z-stable-mcduff} shows that the tracial completion of a separable $\mathcal Z$-stable $C^*$-algebra with respect to a compact face of traces has  property~$\Gamma$ (cf.\ \cite[Proposition 2.3]{CETWW}).

\begin{proposition}\label{prop:McDuff-implies-Gamma}
If $(\M, X)$ is a factorial McDuff tracially complete $C^*$-algebra, then $(\M, X)$ satisfies property $\Gamma$.
\end{proposition}

\begin{proof}
Let $S \subseteq \M$ be a $\|\cdot\|_{2, X}$-separable set and fix $k \geq 1$.  Fix a unital embedding $\phi \colon M_k \rightarrow \M^\omega \cap S'$.  If $a \in S$, then the function $\tau(a\phi(\,\cdot\,))$ is a tracial functional on $M_k$.  Hence $\tau(a\phi(e_{ii})) = \tau(a\phi(e_{11}))$ for all $i = 1, \ldots, k$, and since $\phi$ is unital and $\sum_{i=1}^k e_{ii} = 1_{M_k}$, we have $\tau(a \phi(e_{ii})) = \frac1k \tau(a)$ for all $a \in S$.  Thus the projections $p_i \coloneqq \phi(e_{ii})$, $i = 1, \ldots, k$, satisfy the conditions in Definition~\ref{def:UTCgamma}.
\end{proof}

The following proposition records several equivalent formulation of property $\Gamma$.

\begin{proposition}[{cf.\ \cite[Proposition~2.3]{CETW}}]\label{prop:Gamma}
For a factorial tracially complete $C^*$-algebra $(\M, X)$, the following are equivalent.
\begin{enumerate}[wide, labelwidth=0pt, labelindent=0pt]
	\item\label{item:Gamma} $(\M, X)$ satisfies property $\Gamma$.
	
		\item\label{item:Gamma-approx} For every finite set $\mathcal F \subseteq \M$ and $\epsilon > 0$, there is a self-adjoint contraction $p \in \M$ such that for all $a \in \mathcal F$ and $\tau \in X$,
	\begin{equation}
		\|[p, a]\|_{2, X} < \epsilon,\
		\|p - p^2\|_{2, X} < \epsilon,\ \text{and}\ \
		\big| \tau(ap) - \frac12 \tau(a)\big| < \epsilon.
	\end{equation}
	
	\item\label{item:approx-central-projection} There is  $c \in (0, 1)$ such that for every $\|\cdot\|_{2, X}$-separable set $S \subseteq \M^\omega$, there is a projection $p \in \M^\omega \cap S'$ such that $\tau(ap) = c \tau(a)$ for all $a \in S$ and $\tau \in X^\omega$.
 
\item\label{item:embeddingC2} There exists a faithful trace $\sigma \in T(\C^2)$ such that for every $\|\cdot\|_{2, X}$-separable set $S \subseteq \M^\omega$, there is a unital embedding $\phi\colon \C^2 \rightarrow \M^\omega \cap S'$ such that $\tau(a\phi(x)) = \tau(a)\sigma(x)$ for all $a \in S$, $x \in \C^2$, and $\tau \in X^\omega$.

  	\item\label{item:approx-central-interval} For every $\|\cdot\|_{2, X}$-separable set $S \subseteq \M^\omega$, there is a $^*$-homomorphism $\phi \colon L^\infty[0, 1] \rightarrow \M^\omega \cap S'$ such that 
	\begin{equation}
		\tau(a \phi(f)) = \tau(a) \int_0^1 f(t) dt
	\end{equation}
	for all $a \in S$, $f \in L^\infty[0, 1]$, and $\tau \in X^\omega$.
\end{enumerate}
\end{proposition}

\begin{proof}
\ref{item:Gamma}$\Rightarrow$\ref{item:Gamma-approx}: Apply the definition of property $\Gamma$ with $k \coloneqq 2$ and $S \coloneqq \mathcal F$ and take a suitable element of a representative sequence for one of the projections witnessing property $\Gamma$ to give \ref{item:Gamma-approx}.

\ref{item:Gamma-approx}$\Rightarrow$\ref{item:approx-central-projection}: 
Let $\{s^{(i)}:i\in\N\}$ be a countable dense set in $S$. For each $i\in\N$, let $(s^{(i)}_n)_{n=1}^\infty$ be a sequence representing $s^{(i)}$.
For each $n \in \N$, set $\epsilon_n \coloneqq 2^{-n}$ and $\mathcal{F}_n \coloneqq \{ s^{(i)}_n: i \leq n\}$. Let $p_n \in \M$ be given as in \ref{item:Gamma-approx} with $(\mathcal{F}_n, \epsilon_n)$ in place of $(\mathcal F, \epsilon)$.  Then the sequence $(p_n)_{n=1}^\infty$ induces a projection $p \in \M^\omega \cap S'$ such that $\tau(ap) = \frac{1}{2} \tau(a)$ for all $a \in S$ and $\tau \in X^\omega$.

\ref{item:approx-central-projection}$\Rightarrow$\ref{item:embeddingC2}:
Let $e_1, e_2 \in \mathbb C^2$ be the standard basis vectors and let $p \in \M^\omega \cap S'$ be the projection constructed in \ref{item:approx-central-projection}.
Define $\phi\colon \C^2 \rightarrow \M^\omega \cap S'$ by $\phi(e_1)\coloneqq p$ and $\phi(e_2) \coloneqq  1_{\mathcal M^\omega} -p$.
Define $\sigma(x_1,x_2)\coloneqq cx_1 + (1-c)x_2$. Then $\tau(a\phi(x)) = \tau(a)\sigma(x)$ for all $a \in S$, $x \in \C^2$, and $\tau \in X^\omega$.

\ref{item:embeddingC2}$\Rightarrow$\ref{item:approx-central-interval}:  
Let $\phi_1 \colon \C^2 \rightarrow \M^\omega \cap S'$ be a unital embedding such that $\tau(a\phi_1(x)) = \tau(a)\sigma(x)$ for all $a \in S$, $x \in \C^2$, and $\tau \in X^\omega$. Then define $S_1 \coloneqq  C^*(\phi_1(\C^2) \cup S)$ and use \ref{item:approx-central-matrix} again to produce a unital embedding $\phi_2 \colon \C^2 \rightarrow \M^\omega \cap S_1'$ such that $\tau(a\phi_2(x)) = \tau(a)\sigma(x)$ for all $a \in S_1$, $x \in \C^2$, and $\tau \in X^\omega$.

Continuing inductively, there are unital embeddings $\phi_n \colon \C^2 \rightarrow \M^\omega \cap S'$ with commuting ranges such that the induced embedding
\begin{equation}
	 \phi \colon \bigotimes_{n=1}^\infty \C^2 \rightarrow \M^\omega \cap S'
\end{equation}
satisfies $\tau(a\phi(x)) = \tau(a)\sigma^{\otimes\infty}(x)$ for all $a \in S$, $x \in \bigotimes_{n=1}^\infty\C^2$, and $\tau \in X^\omega$.

Since $\M^\omega \cap S'$ is tracially complete, $\phi$ extends to a $^*$-homomorphism
\begin{equation}
  \bar\phi \colon \pi_{\sigma^{\otimes\infty}}\Big(\bigotimes_{n=1}^\infty \mathbb C^2\Big)'' \rightarrow \M^\omega \cap S'.
\end{equation}
Writing $\bar{\sigma}$ for the trace on $\pi_{\sigma^{\otimes\infty}} \Big(\bigotimes_{n=1}^\infty \C^2 \Big)''$ induced by $\sigma^{\otimes\infty}$, we have $\tau(a\bar\phi(x)) = \tau(a) \bar{\sigma}(x)$ for all $a \in S$, $x \in \pi_{\sigma^{\otimes \infty}}\big(\bigotimes_{n=1}^\infty \mathbb C^2\big)''$, and $\tau \in X^\omega$.  The uniqueness of the standard probability space (\cite[Corollary A.11]{Tak79}) then provides an isomorphism 
\begin{equation}
  \pi_{\sigma^{\otimes \infty}}\Big(\bigotimes_{n=1}^\infty \mathbb C^2\Big)'' \cong L^\infty[0, 1]
\end{equation}
which carries $\bar{\sigma}$ to the trace on $L^\infty[0, 1]$ given by integration with respect to the Lebesgue measure.

\ref{item:approx-central-interval}$\Rightarrow$\ref{item:Gamma}: For $k \geq 1$ and $1 \leq i \leq k$, set $p_i \coloneqq  \phi(\chi_{[(i-1)/k, i/k)})$.
\end{proof}

\begin{remark}\label{rmk:gamma-ultrapower-commutant}
Using reindexing, as in Remark~\ref{rmk:ultrapower-commutant}, we can see that Definition~\ref{def:UTCmcDuff} is equivalent to an \emph{a priori} stronger definition where $S$ is allowed to be any $\|\cdot\|_{2, X^\omega}$-separable subset of $\M^\omega$. Similarly, each of \ref{item:approx-central-projection}, \ref{item:embeddingC2} and \ref{item:approx-central-interval} is equivalent to an \emph{a priori} weaker statement with $S \subseteq \M$.
\end{remark}

The local characterisation of property $\Gamma$ given in Proposition~\ref{prop:Gamma}\ref{item:Gamma-approx} provides the following permanence property.  Recall that factoriality is preserved by inductive limits by Proposition~\ref{prop:UTCInductiveLimit}.

\begin{proposition}\label{prop:Gamma-indlim}
Property $\Gamma$ is preserved by inductive limits of factorial tracial complete $C^*$-algebras.
\end{proposition}

\begin{remark}\label{rem:GammaUltrapower}
It is also true that the reduced product of a sequence of factorial tracially complete $C^*$-algebras with property $\Gamma$ is both factorial and has property $\Gamma$, but we do not yet have the machinery to prove the factorial part of this claim.\footnote{Recall that in general it is not true that a reduced product of factorial tracially complete $C^*$-algebras is factorial (see Remark~\ref{VaccaroRemark}).} Once the factorial issue is sorted, it is clear that the condition in Proposition~\ref{prop:Gamma}\ref{item:Gamma-approx} is preserved by reduced products.  See Corollary~\ref{cor:factorial-ultrapower}\ref{cor:factorial-ultrapower:Gamma}.
\end{remark}

The definition of property $\Gamma$ requires that one can tracially divide all elements of a tracially complete $C^*$-algebra in an approximately central fashion, whereas for a II$_1$ factor $\M$, it suffices just to divide the unit (i.e.\ \eqref{eq:tracially-divides} only needs to hold for $a=1_{\M^\omega}$).  In \cite[Proposition 3.2]{CETW}, several of the present authors observed that such a result is also true for separable $C^*$-algebras with Bauer trace simplices.  The same holds in the tracially complete setting with a very similar proof.

\begin{proposition}[{cf.\ \cite[Corollary 3.2]{CETW}}] \label{prop:BauerDivideUnit}
Let $(\M,X)$ be a factorial tracially complete $C^*$-algebra such that $X$ is a Bauer simplex.  Suppose that for any $\|\cdot\|_{2, X}$-separable subset $S\subseteq \M$ and $k\in \N$, there exists pairwise orthogonal projections $p_1,\dots,p_k\in\M^\omega\cap S'$ with $\tau(p_i)=1/k$ for all $i=1,\dots,k$. Then $(\M,X)$ has property $\Gamma$.
\end{proposition}

\begin{proof}
First assume that $\M$ is $\|\cdot\|_{2,X}$-separable, so we can take $\mathcal S\coloneqq\M$ in the statement.
Also, let $\omega$ be a free ultrafilter.
Using Remark~\ref{rmk:gamma-ultrapower-commutant}, it will suffice to show that we automatically have 
\begin{equation}
\label{eq:BauerDivideUnit1}
    \tau(ap_i) = \frac1k\tau(a),\qquad a\in \mathcal M,\ \tau \in X^\omega,\ i = 1, \ldots k.
\end{equation}
By \cite[Proposition~3.9]{BBSTWW}, $\M^\omega$ is a $W^*$-bundle over $(\partial_e X)^\omega$ (in the notation defined before \cite[Proposition~3.9]{BBSTWW}), where $(\partial_e X)^\omega$ is the closure of the limit traces coming from a sequence of traces in $\partial_e X$.
Therefore, it suffices to prove \eqref{eq:BauerDivideUnit1} where $\tau$ is such a limit trace. 
In this case, since $\partial_e X$ is compact, $\tau|_{\M} \in \partial_e X$.
Define $\sigma\colon \M \to \C$ by $\sigma(x)\coloneqq \tau(xp_i)$.
This is a positive tracial functional such that $\sigma(1_{\M})=\frac1k$; since $\tau$ is an extreme point in $T(\M)$ and $\sigma\leq\tau$, we have $\sigma=\frac1k\tau$, which amounts to \eqref{eq:BauerDivideUnit1}.

In the non-separable case, we can use the fact that factoriality is separably inheritable (Theorem~\ref{thm:sep}\ref{thm:sep-factorial}) to see that for any $\|\cdot\|_{2,X}$-separable subset $S$ of $\M$, there exists a $\|\cdot\|_{2, X}$-separable factorial tracially complete subalgebra $\M_0$ of $\M$ which contains $S$ and satisfies the hypotheses of the proposition and, therefore, has property $\Gamma$.
Consequently, $\M$ has property $\Gamma$ as the projections witnessing property $\Gamma$ for $\mathcal M_0$ also witness property $\Gamma$ for $\mathcal M$.
\end{proof}

In the setting of $W^*$-bundles, abstracting the results in \cite{TWW15,Sa12,KR14}, an amenable factorial type II$_1$ tracially complete $C^*$-algebra $(\M,X)$ with $\partial_eX$ compact and finite dimensional turns out to have property $\Gamma$. This will be generalized in the forthcoming work by the third- and fifth-named authors (\cite{EvingtonSchafhauser}), to show that a factorial $W^*$-bundle over a finite dimensional space has property $\Gamma$ if and only if each fibre has property $\Gamma$ as a von Neumann algebra.\footnote{Recall that by Ozawa's \cite[Theorem 11]{Oz13}, there is no ambiguity of the meaning of a fibre $\pi_\tau(\M)$ of a $W^*$-bundle at an extreme trace $\tau$: these are automatically von Neumann algebras. Accordingly in Proposition \ref{prop:zerodimgamma}, the double commutants are superfluous.}  We show here that this holds in the zero-dimensional setting.

\begin{proposition}\label{prop:zerodimgamma}
    If $(\mathcal M, X)$ is a factorial tracially complete $C^*$-algebra such that $\partial_eX$ is compact and totally disconnected and $\pi_\tau(\mathcal M)''$ has property $\Gamma$ for each $\tau \in \partial_eX$, then $(\mathcal M, X)$ has property $\Gamma$.  In particular, if $(\mathcal M, X)$ is an amenable type $\mathrm{II}_1$ factorial tracially complete $C^*$-algebra such that $\partial_e X$ is compact and totally disconnected, then $(\mathcal M, X)$ has property $\Gamma$.
\end{proposition}

\begin{proof}
   Note that the second statement follows from the first.  Indeed, if $(\mathcal M, X)$ is as in the second statement, then for all $\tau \in \partial_e X$, $\pi_\tau(\mathcal M)''$ is semidiscrete by Theorem~\ref{thm:introamenable} and is a II$_1$ factor by Proposition~\ref{prop:factorial}.  It then follows from \cite[Corollary 2.2]{Co76} that $\pi_\tau(\mathcal M)''$ has property $\Gamma$.

   We now prove the first statement.  As $\partial_e X$ is compact, $\mathcal M$ can be viewed as a $W^*$-bundle over $K \coloneqq \partial_e X$ by Theorem~\ref{thm:TCtoWStarBundle}.  Let $C(K) \subseteq Z(\mathcal M)$ be the corresponding inclusion.  Fix a finite set $\mathcal F \subset \mathcal M$, $\epsilon > 0$, and $k \in \mathbb N$.  By Proposition~\ref{prop:BauerDivideUnit}, it suffices to show that there are positive contractions $p_1, \ldots, p_k \in \mathcal M$ such that
   \begin{enumerate}
       \item\label{zdb1} $\|[p_i, a]\|_{2, X} < \epsilon$ for all $a \in \mathcal F$ and $i = 1, \ldots, k$,
       \item\label{zdb2} $\|p_i - p_i^2\|_{2, X} < \epsilon$ for all $i, \ldots, k$,
       \item\label{zdb3} $\|p_i p_{i'}\|_{2, X} < \epsilon$ for $i, i' = 1,\ldots,k$ with $i \neq i'$, and
       \item\label{zdb4} $|\tau(p_i) - 1/k| < \epsilon$ for all $\tau \in X$ and $i = 1, \ldots, k$.
   \end{enumerate}

   For each $\tau \in K$, as the II$_1$ factor $\pi_\tau(\mathcal M)''$ satisfies property $\Gamma$, there are mutually orthogonal projections $\bar p_{1, \tau}, \ldots, \bar p_{k, ,\tau} \in \pi_\tau(\mathcal M)''$ with trace $1/k$ such that
   \begin{equation}
       \|[\bar p_{i, \tau}, \pi_\tau(a)]\|_{2, \tau} < \epsilon, \qquad a \in \mathcal F,\ i = 1, \ldots, k.
   \end{equation}
   Use Kaplansky's density theorem to obtain a $\|\cdot\|_{2,\tau}$ approximation of each $\bar p_{\tau, i}$ by a positive contraction in $\mathcal M$.
   Then, combining this with the compactness of $K$, there is a finite open cover $U_1, \ldots, U_n$ of $K$ and positive contractions $p_{i, j} \in \mathcal M$ for $i =1, \ldots, k$ and $j = 1, \ldots, n$ such that 
    \begin{enumerate}
    \renewcommand{\labelenumi}{\theenumi}
    \renewcommand{\theenumi}{\rm{(\roman{enumi}$'$)}}
    
       \item\label{lzdb1} $\|[p_{i, j}, a]\|_{2, U_j} < \epsilon$ for all $a \in \mathcal F$, $i = 1, \ldots, k$, and $j = 1, \ldots, n$,
       \item\label{lzdb2} $\|p_{i, j} - p_{i, j}^2\|_{2, U_j} < \epsilon$ for all $i, \ldots, k$ and $j= 1, \ldots, n$,
       \item\label{lzdb3} $\|p_{i,j} p_{i',j}\|_{2, U_j} < \epsilon$ for $i, i' = 1,\ldots,k$ with $i \neq i'$ and $j = 1, \ldots, n$, and
       \item\label{lzdb4} $|\tau(p_{i,j}) - 1/k| < \epsilon$ for all $\tau \in U_j$ and $i = 1, \ldots, k$.
   \end{enumerate}

   As $K$ is totally disconnected, after refining the open cover, we may assume that the sets $U_1, \ldots, U_n$ form a clopen partition of $K$.  Then the elements
   \begin{equation}
        p_i \coloneqq \sum_{j=1}^n \chi_{U_j} p_{i, j} \in \mathcal M, \qquad j = 1, \ldots n,
   \end{equation}
   are positive contractions satisfying \ref{zdb1}--\ref{zdb4}.
\end{proof}

It remains open whether the approximately central division of the unit as in Proposition \ref{prop:BauerDivideUnit} implies property $\Gamma$ outside the setting of Bauer simplices (\cite[Question 3.5]{CETW}). This question is just as valid in the tracially complete framework.
\begin{question}
    Does Proposition \ref{prop:BauerDivideUnit} hold for all factorial tracially complete $C^*$-algebras (i.e.\ without assuming that $X$ is a Bauer simplex)?
\end{question}

Immediately following \cite[Proposition 5.10]{CETW}, the second-, third-, and last two named authors asserted that all the Villadsen algebras of the first type -- namely the examples from \cite{Vi98}, which were used by Toms to give striking counterexamples to Elliott's original classification conjecture (\cite{To08}) -- have property $\Gamma$. These algebras are so-called diagonal AH algebras (i.e.\ inductive limits of homogeneous algebras of a particularly nice form), and they prove a fertile testing ground for comparing regularity properties of very different natures, with the analysis in \cite{TW09} leading to an early version of the Toms--Winter conjecture (\cite[Remark 3.5]{TW09}) and its positive solution for this class of algebras.  Uniform tracial completions of diagonal AH algebras satisfy the central divisibility-of-the-unit hypothesis of Proposition \ref{prop:BauerDivideUnit} (this is \cite[Proposition 5.10]{CETW}), and the assertion that Villadsen algebras of the first type have property $\Gamma$ was based on \cite[Section 8]{TW09} (where it is suggested, based on computations in \cite[Theorem 4.1]{To09}, that these have Bauer simplices of traces).\footnote{In \cite{CETW}, the second-, third-, and last two named authors incorrectly referenced the non-existent \cite[Theorem~4.1]{To08} rather than \cite[Theorem 4.1]{To09}.} However, recent results of Elliott, Li, and Niu (\cite[Theorem 4.5]{ELN:preprint}) show that, in fact, the trace space of Villadsen algebras of the first type is the Poulsen simplex  (whenever the algebra is not $\mathcal Z$-stable) -- maximally far from being Bauer!  Accordingly, the argument in \cite{CETW} is not valid, and it does not seem straightforward to determine whether the Villadsen type I algebras have property $\Gamma$.

\begin{question}\label{Q:VA}
Do the uniform tracial completions of the (non-$\Z$-stable) Villadsen algebras of the first type have property $\Gamma$?
\end{question}

The connecting maps in the inductive limit construction of a Villadsen type I algebra involve a combination of coordinate projections and point evaluations, the latter of which ensure simplicity of the inductive limit.  In the case of a non-$\Z$-stable Villadsen type I algebra $A$, the point evaluations are relatively sparse (see \cite[Theorem 3.4 and Lemma 5.1]{TW09}). As such, they do not affect the structure of the tracial completion of $A$.  For this reason, the heart of Question \ref{Q:VA} lies in determining whether
$\Gamma$ holds for the tracially complete inductive limit of the trivial $W^*$-bundles
\begin{equation}
    C([0,1]) \to C([0,1]^2,M_2) \to C([0,1]^4, M_4) \to \cdots, 
\end{equation}
where each stage is equipped with all its traces, and the connecting maps are given as the direct sum of the two coordinate projections.

\section{Complemented partitions of unity}\label{sec:CPoU}

This section formally introduces the notion of CPoU for factorial tracially complete $C^*$-algebras as discussed in Section~\ref{sec:intro:ltg2} and proves Theorem~\ref{introthmgammaimpliescpou} (property $\Gamma$ implies CPoU).  The definition of CPoU is given in Section~\ref{sec:cpou-def}.  Some initial examples of tracially complete $C^*$-algebras with CPoU are set out in Section~\ref{sec:cpou-examples}, including the special case of finite von Neumann algebras which will be used to prove Theorem~\ref{introthmgammaimpliescpou}, with permanence properties discussed in Section~\ref{sec:cpou-permanence}.  Section~\ref{sec:CPoU-proof} contains the proof of Theorem~\ref{introthmgammaimpliescpou}.

\subsection{Formal definition and its reformulations}\label{sec:cpou-def}

The following definition of CPoU extends the one introduced in \cite{CETWW} from the $C^*$-setting to factorial tracial complete $C^*$-algebras.  More precisely, a separable $C^*$-algebra  $A$ with $T(A)$ compact has CPoU in the sense of \cite{CETWW} if and only if its tracial completion with respect to $T(A)$ has CPoU in the sense of Definition~\ref{dfn:utcCPOU} below; this is immediate from the definition and Proposition~\ref{rem:CStarVsTCreducedProducts}, which identifies the uniform 2-norm central sequence algebra of $A$ with that of its uniform tracial completion.\footnote{In \cite{CETWW}, CPoU was only defined in the separable setting, but we will not make this restriction here.}  Recall our convention that (unless specified otherwise) $\omega$ is a free filter on $\mathbb N$. It will follow from the local characterisation in Proposition~\ref{CPoU:Ultra:FiniteSet} that Definition~\ref{dfn:utcCPOU} does not depend on the choice of $\omega$.

\begin{definition}[{cf.\ \cite[Definition~3.1]{CETWW}}]\label{dfn:utcCPOU}
	Let $(\M,X)$ be a factorial tracially complete $C^*$-algebra. We say that $(\M,X)$ has \emph{complemented partitions of unity} (CPoU) if for any $\|\cdot\|_{2,X}$-separable subset $S \subseteq \M$, any family $a_1,\dots,a_k$ of positive elements in $\M$, and any scalar
	\begin{equation}\label{eq:CPoUTraceIneq1}
		\delta>\sup_{\tau \in X} \min_{1 \leq i \leq k}\tau(a_i),
	\end{equation}
	there exist orthogonal projections $p_1,\dots,p_k\in \M^\omega\cap S'$ summing to $1_{\M^\omega}$ such that
	\begin{equation}\label{eq:CPoUTraceIneq2}
		\tau(a_ip_i)\leq \delta\tau(p_i),\qquad \tau\in X^\omega,\ i=1,\dots,k.
	\end{equation}
\end{definition}

We note that when $\M$ is $\|\cdot\|_{2,X}$-separable, it suffices to take $S \coloneqq \M$. 
As with property $\Gamma$, we have have chosen to restrict the definition to the setting of factorial tracially complete $C^*$-algebras since it is not clear if this definition is appropriate outside of the factorial setting.

The following proposition provides equivalent formulations of CPoU; \ref{item:CPOU-local} avoids the language of reduced products and \ref{item:CPOU++} is useful in many applications.

\begin{proposition}\label{CPoU:Ultra:FiniteSet}
	Let $(\M, X)$ be a factorial tracially complete $C^*$-algebra. Then the following are equivalent:
	\begin{enumerate}[wide, labelwidth=0pt, labelindent=0pt]
		\item\label{item:CPOU} $(\M,X)$ has CPoU;
		\item\label{item:CPOU-local} 
		for every finite set $\mathcal F \subseteq \M$, $\epsilon>0$, $a_1,\dots,a_k \in \M_+$, and 
	\begin{equation}
		\delta>\sup_{\tau\in X}\min_{1 \leq i \leq k}\tau(a_i),
	\end{equation}
	there exist orthogonal positive contractions $e_1,\dots,e_k \in \M$ such that
	\begin{equation}\begin{aligned}\label{eqn:CPoU-local-form}
		\Big\|\sum_{j=1}^ke_j - 1_\M \Big\|_{2,X} &< \epsilon, && \\
		\max_{x \in \mathcal F} \|[e_i,x]\|_{2,X} &<\epsilon,   && i=1,\dots,k, \text{ and} \\
		\tau(a_ie_i) - \delta\tau(e_i) &< \epsilon,  && \tau \in X,\ i=1,\dots,k;
	\end{aligned}\end{equation}
	\item\label{item:CPOU++} 
		for every $\|\cdot\|_{2,X}$-separable subset $S \subseteq \M^\omega$, $a_1,\dots,a_k \in (\M^\omega)_+$, and
	\begin{equation}\label{eq:CPoU++TraceIneq1}
		\delta > \sup_{\tau \in X^\omega} \min_{1 \leq i \leq k}\tau(a_i),
	\end{equation}
	there exist orthogonal projections $p_1,\dots,p_k\in \M^\omega\cap S'$ summing to $1_{\M^\omega}$ such that
	\begin{equation}\label{eq:CPoU++TraceIneq2}
		\tau(a_ip_i)\leq \delta\tau(p_i),\qquad \tau\in X^\omega,\ i=1,\dots,k.
	\end{equation}
	\end{enumerate}
\end{proposition}

\begin{proof}
	\ref{item:CPOU}$\Rightarrow$\ref{item:CPOU-local}:
	Let $\epsilon > 0$, let $a_1,\dots,a_k \in \M_+$, let $\mathcal{F} \subseteq \M$ be a finite subset, and  let $\delta > 0$ be such that \eqref{eq:CPoUTraceIneq1} holds. 
	Since $(\M, X)$ has CPoU, we obtain orthogonal projections $p_1,\ldots,p_k \in \M^\omega \cap \mathcal F'$ summing to $1_{\mathcal M^\omega}$ such that \eqref{eq:CPoUTraceIneq2} holds. We can lift $p_1,\ldots,p_k \in \M^\omega$ to orthogonal positive contractions $e^{(1)},\ldots,e^{(k)} \in \ell^\infty(\M)$ by \cite[Proposition~2.6]{AP77} (this is a special case of projectivity of cones over finite dimensional $C^*$-algebras; see \cite[Corollary~3.8]{LP98}). 
	Write $(e^{(i)}_j)_{j=1}^\infty$ for a sequence of positive contractions representing $e^{(i)}$ for $i=1,\ldots,k$. 
	Note that for each $j \in \N$, $e^{(1)}_j,\ldots,e^{(k)}_j$ are orthogonal positive contractions. 
	Since the $p_i$ sum to $1_{\M^\omega}$ and commute with $\mathcal{F}$, we have
	\begin{equation}
		\limsup_{j\to \omega} \big\|e^{(1)}_j + \cdots +e^{(k)}_j - 1_{\M}\big\|_{2,X}= 0, 
	\end{equation}
	and
	\begin{equation}
		\limsup_{j\to\omega} \|[e^{(i)}_j,x]\|_{2,X} = 0, \quad x \in \mathcal{F},\ i=1,\ldots,k.
	\end{equation}
	Moreover, we must have
	\begin{equation}
		\limsup_{j\to\omega}\sup_{\tau \in X} \big(\tau\big(a_ie^{(i)}_j\big) - \delta\tau\big(e^{(i)}_j\big)\big) \leq 0, \quad  i = 1, \ldots, k,
	\end{equation}
	as otherwise we could choose a sequence of traces $(\tau_j)_{j=1}^\infty$ in $X$ and an ultrafilter $\omega' \supseteq \omega$ such that the corresponding limit trace $\tau$ does not satisfy $\tau(a_ip_i) \leq \delta \tau(p_i)$.
	Therefore, taking $e_i \coloneqq e^{(i)}_j$ for a suitable choice of $j$, the orthogonal positive contractions $e_1,\ldots,e_k$ will satisfy
	\eqref{eqn:CPoU-local-form}.
	
\ref{item:CPOU-local}$\Rightarrow$\ref{item:CPOU++}:
Let $S \subseteq \M^\omega$ be a $\|\cdot\|_{2,X^\omega}$-separable subset and fix a countable dense subset $\{s^{(i)}:i\in\N\}$ of $S$. For each $i\in\N$, let $(s^{(i)}_j)_{j=1}^\infty$ be a sequence representing $s^{(i)}$.
Let $a_1,\ldots,a_k \in \M^\omega$ be positive contractions in $\M^\omega$. Let $(a^{(i)}_j)_{j=1}^\infty$ be a sequence of positive contractions representing $a_i$ for $i=1,\ldots,k$.
Let $\delta > 0$ satisfy \eqref{eq:CPoU++TraceIneq1}.
Then we have
\begin{equation}
	\limsup_{j\to\omega} \sup_{\tau \in X} \min_{1\leq i \leq k} \tau\big(a^{(i)}_j\big) < \delta.
\end{equation}
Therefore, there exists a set $J \in \omega$ such that for all $j \in J$, we have
\begin{equation}
	\sup_{\tau \in X} \min_{1\leq i \leq k} \tau\big(a^{(i)}_j\big) < \delta.
\end{equation}
For each $j \in \N$, set $\epsilon_j \coloneqq 2^{-j}$ and $\mathcal{F}_j \coloneqq \{ s^{(i)}_j: i \leq j\}$. Applying \ref{item:CPOU-local} for each $j \in J$, we obtain orthogonal positive contractions $e^{(1)}_j,\ldots,e^{(k)}_j$ such that
\begin{equation}\begin{aligned}\label{eqns:local-j}
		\Big\|\sum_{i=1}^k e^{(i)}_j - 1_\M \Big\|_{2,X} &< \epsilon_j, &&    \\
		\max_{x \in \mathcal F_j} \big\|\big[e^{(i)}_j,x\big]\big\|_{2,X} &<\epsilon_j, && i=1,\dots,k,\ x\in \mathcal F_j,\text{ and}  \\
		\tau\big(a^{(i)}_je^{(i)}_j\big) - \delta\tau\big(e^{(i)}_j\big) &< \epsilon_j, && \tau \in X,\ i=1,\dots,k.
\end{aligned}\end{equation}
For $j \not\in J$, we may choose $e^{(1)}_j,\ldots,e^{(k)}_j$ arbitrarily. Let $p_i \in \M^\omega$ be the element represented by $\big(e^{(i)}_j\big)_{j=1}^\infty$. Then \eqref{eqns:local-j} ensures that the $p_i$ sum to $1_{\M^\omega}$, commute with $S$, and satisfy \eqref{eq:CPoU++TraceIneq2}. Since $p_1,\ldots,p_k$ are orthogonal positive contractions summing to the identity, they are in fact projections.

\ref{item:CPOU++}$\Rightarrow$\ref{item:CPOU}: If $S \subseteq \M$ is $\|\cdot\|_{2,X}$-separable, then viewed as a subset of $\M^\omega$, it is $\|\cdot\|_{2,X^\omega}$-separable. If $a_1,\ldots,a_k \in \M_+$, we have \begin{equation}
\sup_{\tau \in X^\omega}\min_{1\leq i \leq k} \tau(a_i) = \sup_{\tau \in X}\min_{1\leq i \leq k} \tau(a_i).
\end{equation} 
Therefore, \ref{item:CPOU} is a special case of \ref{item:CPOU++}.
\end{proof}

\subsection{First examples and non-examples}\label{sec:cpou-examples}

The main source of examples of factorial tracially complete $C^*$-algebras with CPoU come from Theorem~\ref{introthmgammaimpliescpou}, which, as a special case, implies the uniform tracial completion of a separable $\mathcal Z$-stable $C^*$-algebra with respect to a compact face of traces will have CPoU.  In this subsection, we shall investigate some more elementary examples of tracially complete $C^*$-algebras $(\M,X)$ with CPoU in the case when $X$ is a Bauer simplex.  We will also prove a special case of Theorem~\ref{introthmgammaimpliescpou}, showing that  $\Gamma$ implies CPoU in the case of $W^*$-bundles (Proposition~\ref{prop:Gamma-bundle}), which is easier and more conceptual than the general result taken up in Section~\ref{sec:CPoU-proof}. We also provide examples of type II$_1$ factorial tracially complete $C^*$-algebras without CPoU in Example~\ref{eg:free-group-bundle}.

In the setting of $W^*$-bundles with totally disconnected base space, a strong form a CPoU holds: one can take the partition of unity to be central as opposed to approximately central. Moreover, as will be important later, we can additionally work under any given positive element $q \in \M$ that is nonzero on all traces (this technicality will be used in Section~\ref{sec:CPoU-proof}).

\begin{proposition}
	\label{prop:ZeroDimBoundary}
	Let $(\M,X)$ be a factorial tracially complete $C^*$-algebra such that $\partial_e X$ is compact and totally disconnected. Let $q \in \M_+$ with $\tau(q) > 0$ for all $\tau \in X$. Let $a_1, \ldots, a_k \in \M_+$ with
	\begin{equation}
		 \sup_{\tau \in X} \min_{1 \leq i \leq k} \tau(a_iq) < \delta \tau(q).
	\end{equation}
	Then there are projections $p_1, \ldots, p_k \in Z(\M)$ summing to $1_\mathcal M$ such that $\tau(a_i p_i q) \leq \delta \tau(p_i q)$ for all $\tau \in X$ and $i = 1,\ldots, k$.  In particular, $(\M,X)$ has CPoU.
\end{proposition}

\begin{proof}
	It suffices to prove the first part to see CPoU for $(\M,X)$, because we may set $q\coloneqq 1_{\mathcal M}$ and, for any $\|\cdot\|_{2, X}$-separable subset $S \subseteq \M$, we have $Z(\M) \subseteq \M^\omega \cap S'$.
		
	For $i=1,\ldots,k$, define
	\begin{equation} 
		U_i\coloneqq \{\tau \in \partial_e X: \tau(a_iq) < \delta \tau(q) \}, 
	\end{equation}
	so that $U_1,\dots,U_k$ is an open cover of $\partial_e X$.  By total disconnectedness, we may partition $\partial_e X$ into clopen sets $V_1,\dots,V_k$ such that $V_i \subseteq U_i$ for $i=1,\dots,k$.
	
	As $\M$ is factorial, $X$ is a closed face of $T(\M)$ and so is a Choquet simplex in its own right.  It is therefore Bauer by the assumption that  $\partial_e X$ is compact.  By Theorem~\ref{thm:TCtoWStarBundle}, there is a natural embedding $C(\partial_e X) \subseteq Z(\M)$. Let us define 
	\begin{equation}
		p_i\coloneqq \chi_{V_i} \in C(\partial_e X) \subseteq Z(\M),
	\end{equation}
 so that the defining feature of the embedding $C(\partial_e X) \subseteq Z(\M)$ in \eqref{eq:central-embedding} gives $\tau(p_ib)=\chi_{V_i}(\tau)\tau(b)$ for $\tau\in\partial_eX$ and $b\in\M$. Accordingly, for $\tau\in \partial_eX$, we have $\tau(a_ip_iq)=\tau(a_iqp_i)\leq\delta\tau(qp_i)$ as both sides of the inequality vanish for $\tau\in\partial_eX\setminus V_i$, and we have $\tau(a_iq) < \delta \tau(q)$ for $\tau \in V_i$. 
 It then follows from the Krein--Milman theorem that $\tau(a_i p_iq) \leq \delta \tau(p_iq)$ for all $i = 1, \ldots, k$ and $\tau \in X$.
\end{proof}

A noteworthy special case of Proposition \ref{prop:ZeroDimBoundary} is when $\M$ is a finite dimensional $C^*$-algebra and $X = T(\M)$. This is easily seen to be a tracially complete $C^*$-algebra, as the norms $\|\cdot\|_{2,X}$ and $\|\cdot\|$ are equivalent. Although this example is fairly trivial, it played a notable role in the proof of \cite[Theorem 3.8]{CETWW}, which could be viewed as bootstrapping CPoU from the finite dimensional setting to the nuclear setting with the aid of uniform property $\Gamma$. The proof of Theorem \ref{introthmgammaimpliescpou} will similarly make use of the following special case of Proposition \ref{prop:ZeroDimBoundary}.

\begin{corollary}\label{cor:vNa-CPoU}
	If $\M$ is a finite von Neumann algebra,
	then the factorial tracially complete $C^*$-algebra $\big(\M, T(\M)\big)$ has CPoU.  In fact, if $q \in \M$ with $\tau(q) > 0$ for all $\tau \in T(\M)$, if and $a_1, \ldots, a_k \in \M_+$ satisfy
	\begin{equation}
		 \sup_{\tau \in X} \min_{1 \leq i \leq k} \tau(a_iq) < \delta \tau(q),
	\end{equation}
	then there are projections $p_1, \ldots, p_k \in Z(\M)$ summing to $1_\mathcal M$ such that $\tau(a_i p_iq) \leq \delta \tau(p_iq)$ for all $\tau \in T(\M)$ and $i = 1,\ldots, k$.
\end{corollary}
\begin{proof}	
As all traces on $\M$ factor through the centre-valued trace $\M \rightarrow Z(\M)$ (Proposition \ref{prop:TracesFiniteVNAS}\ref{prop:TracesFiniteVNAS.1}), composition with the centre-valued trace produces a homeomorphism from the Gelfand spectrum of $Z(\M)$ to $\partial_e T(\M)$.  
In particular, $\partial_e T(\M)$ is compact and totally disconnected (in fact hyperstonean), as $Z(\M)$ is a von Neumann algebra, and so the result follows from Proposition~\ref{prop:ZeroDimBoundary}.
\end{proof}

In the next proposition, we consider what it means for a trivial $W^*$-bundle to have CPoU. (See Example~\ref{ex:locally-trivial} for the definition of trivial $W^*$-bundles.)  
The heuristic idea is that for a $W^*$-bundle over a space that is not totally disconnected, the approximately central projections in Definition~\ref{dfn:utcCPOU} cannot come from the centre of the algebra, and hence any constructions of approximately central projections must make use of central sequences in the fibres.  In particular, if there are no central sequences in the fibres, then the $W^*$-bundle must not have CPoU.  

\begin{proposition}\label{prop:trivial-bundle-CPoU}
	Let $(\M, X)$ be the trivial $W^*$-bundle over a compact Hausdorff space $K$ with fibre a finite factor $\mathcal N$.  If $(\M, X)$ has CPoU, then either $K$ is totally disconnected or $\mathcal N$ satisfies property $\Gamma$.
\end{proposition}

\begin{proof}
	Suppose $(\M, X)$ satisfies CPoU and $K$ is not totally disconnected.  Let $\tau_{\mathcal N}$ be the trace on $\mathcal N$ and let $S \subseteq \mathcal N$ be a $\|\cdot\|_{2, \tau_\mathcal N}$-separable subset.  We work to show that there is a projection $q \in \mathcal{N}^\omega \cap S'$ with trace $1/2$.  A standard von Neumann algebra result will therefore show that $\mathcal N$ satisfies property $\Gamma$.  This is also a special case of \ref{item:Gamma}$\Rightarrow$\ref{item:approx-central-projection} of Proposition~\ref{prop:Gamma}, for example.
	
	Fix a point $x \in K$ and an open neighbourhood $U \subseteq K$ of $x$ such that there is no clopen neighbourhood of $x$ contained in $U$.  Let $a_1 \in C(K) \subseteq \M$ be such that $0 \leq a_1 \leq 1_{\mathcal M}$, $a_1(x) = 1$, and $\mathrm{supp}(a_1) \subseteq U$, and set $a_2 \coloneqq 1_{\mathcal M}- a_1$.  Then
	\begin{equation}
		\sup_{\tau \in X} \min \{\tau(a_1), \tau(a_2) \} \leq \frac12 < \frac23.
	\end{equation}
	Let $\omega$ be a free ultrafilter on $\mathbb N$ and view $\mathcal N \subseteq \mathcal M$ as constant functions.  Since $(\M, X)$ has CPoU, there are projections $p_1, p_2 \in \M^\omega \cap S'$ such that $p_1 + p_2 = 1_{\mathcal M^\omega}$ and
	\begin{equation}\label{eq:cpou-trivial-bundle-1}
		\tau(a_i p_i) \leq \frac23 \tau(p_i), \qquad \ i=1, 2,\ \tau \in X^\omega.
	\end{equation}
	
	For $i = 1,2$, let $(p_{i, n})_{n=1}^\infty \subseteq \M$ be a sequence of positive contractions representing $p_i$.  If $\tau \in X^\omega$ is the limit trace determined by the constant sequence $(\tau_{\mathcal N} \circ \mathrm{ev}_x)_{n=1}^\infty \subseteq X$.  Then $\tau(a_1) = 1$, and hence $\tau(a_2) = 0$.  Since $a_2\geq 0$, the Cauchy--Schwarz inequality implies  $\tau(a_2 p_1) = 0$.  So $\tau(a_1 p_1) = \tau(p_1)$.  Now, \eqref{eq:cpou-trivial-bundle-1} implies $\tau(p_1) = 0$.  Therefore,
	\begin{equation}\label{eq:cpou-trivial-bundle-2}
		\lim_{n \rightarrow \omega} \tau_{\mathcal N}(p_{1, n}(x)) = 0.
	\end{equation}
	Similarly, consider a sequence $(y_n)_{n=1}^\infty \subseteq K \setminus U$, and let $\tau \in X^\omega$ be the limit trace defined by the sequence $(\tau_{\mathcal N} \circ \mathrm{ev}_{y_n})_{n=1}^\infty \subseteq X$.  Then $\tau(a_2) = 1$, and as above, this shows $\tau(p_2) = 0$.  This implies $\tau(p_1) = 1$.  As this holds for all sequences $(y_n)_{n=1}^\infty \subseteq K \setminus U$, it follows that
	\begin{equation}\label{eq:cpou-trivial-bundle-3} 
		\lim_{n \rightarrow \omega} \inf_{y \in K \setminus U} \tau_{\mathcal N}(p_{1, n}(y)) = 1.
	\end{equation}
	
	Define 
	\begin{equation}
		N \coloneqq \Big\{n \in \mathbb N : \tau_{\mathcal N}(p_{1,n}(x)) < \frac12, \inf_{y \in K \setminus U} \tau(p_{1, n}(y)) > \frac12 \Big\},
	\end{equation}
	and note that $N \in \omega$ by \eqref{eq:cpou-trivial-bundle-2} and \eqref{eq:cpou-trivial-bundle-3}.  For each $n \in N$, there is a $y_n \in K$ such that $\tau_{\mathcal N}(p_{1, n}(y)) = 1/2$ since otherwise
	\begin{equation}
		\Big\{ y \in K : \tau_{\mathcal N}(p_{1 ,n}(y)) < \frac12 \Big\} \subseteq K
	\end{equation}
	is a clopen neighbourhood of $x$ contained in $U$.  For $n \in N$, set $q_n \coloneqq p_{1, n}(y_n) \in \mathcal N$ and for $n \in \mathbb N \setminus N$, set $q_n \coloneqq 0 \in \mathcal N$.  Then the sequence $(q_n)_{n=1}^\infty$ defines a projection $q \in \mathcal N^\omega \cap S'$ with trace $1/2$.
\end{proof}

Proposition~\ref{prop:trivial-bundle-CPoU} provides us with explicit examples of tracially complete $C^*$-algebras that do not have CPoU.

\begin{example}\label{eg:free-group-bundle}
	For any $n \geq 2$, the trivial $W^*$-bundle $C_\sigma([0,1], L(\mathbb{F}_n))$ does not have CPoU.  Indeed, the II$_1$ factors $L(\mathbb{F}_n)$ do not have property~$\Gamma$ (\cite{MvN43}; see also \cite[Theorem A.7.2]{Si08}), and hence the result follows from Proposition~\ref{prop:trivial-bundle-CPoU}.
\end{example}

The converse of Proposition~\ref{prop:trivial-bundle-CPoU} is also true.  If $(\M, X)$ is a trivial $W^*$-bundle whose fibre is a $\mathrm{II}_1$ factor with property $\Gamma$, then it is easy to show $(\M, X)$ has property $\Gamma$, and then Theorem~\ref{introthmgammaimpliescpou} will show that $(\M, X)$ has CPoU.  In fact, in the W$^*$-bundle case we can give a more direct proof that property $\Gamma$ implies CPoU.
The following result is a special case of Theorem \ref{introthmgammaimpliescpou}, which we choose to include for its conceptual value. Our subsequent proof of Theorem \ref{introthmgammaimpliescpou} will not depend on Proposition \ref{prop:Gamma-bundle}.

\begin{proposition}\label{prop:Gamma-bundle}
	Let $(\M,X)$ be a factorial tracially complete $C^*$-algebra coming from a $W^*$-bundle over the space $K \coloneqq \partial_e X$.
	If $(\M,X)$ has property $\Gamma$, then $(\M,X)$ has CPoU. 
\end{proposition}

\begin{proof}
	Fix an ultrafilter $\omega$ on $\N$. 
	By the Krein--Milman theorem and Proposition~\ref{prop:co-2-norm}, we have that $\|a\|_{2,X} = \|a\|_{2,K}$ for all $a \in \M$. Moreover, writing $K^\omega$ for the weak$^*$-closure of the set of limit traces coming from a sequence of traces in $K$, we have $\|a\|_{2,X^\omega} =  \|a\|_{2,K^\omega}$ for all $a \in \M^\omega$. 

	Let $a_1,\ldots,a_k \in \M_+$ and $\delta > 0$ satisfy
	\begin{equation}
		\sup_{\tau \in K} \min_{1\leq i\leq k} \tau(a_i) < \delta.
	\end{equation}
	
	Let $S \subseteq \M$ be a $\|\cdot\|_{2,X}$-separable subset, which we may assume contains $a_1,\ldots,a_k$. 
	Let $U_i \coloneqq \{\tau \in K: \tau(a_i) < \delta \}$ for $i=1,\ldots,k$. Then $\{U_1,\ldots,U_k\}$ is an open cover of $K$. Let $g_1,\ldots,g_k\colon K \rightarrow  [0,1]$ be a continuous partition of unity subordinate to this open cover. Set $h_0 \coloneqq 0$ and $h_i \coloneqq \sum_{j=1}^{i} g_i$ for $i=1,\ldots,k$.
	
	Since $\M$ has property $\Gamma$, by \ref{item:Gamma}$\Rightarrow$\ref{item:approx-central-interval} of  Proposition~\ref{prop:Gamma}, there exists a $^*$-homomorphism $\phi \colon L^\infty[0, 1] \rightarrow \M^\omega \cap S'$ such that 
	\begin{equation}
		\tau(a \phi(f)) = \tau(a)\tr_{\mathrm{Leb}}(f)
	\end{equation}
	for all $a \in S$, $f \in L^\infty[0, 1]$ and $\tau \in X^\omega$, where $\tr_{\mathrm{Leb}}$ is integration with respect to the Lebesgue measure.
	Since $\M$ is a $W^*$-bundle, there is a canonical copy of $C(K)$ in $Z(\M)$. 
	As the image of $\phi$ commutes with $C(K) \subseteq  Z(\M^\omega)$, we have an induced $^*$-homomorphism $\psi\colon C(K) \otimes L^\infty[0, 1] \rightarrow \M^\omega \cap S'$.  Let $\tau$ be a limit trace given by a sequence of traces $(\tau_n)_{n=1}^\infty $ in $K$. Then, since every $\tau_n \in K$ restricts to a point evaluation on $C(K)$, we see that  
	\begin{equation}\begin{split}
		\tau(a\psi(g \otimes f)) &= \lim_{n\to\omega} \tau_n(a)g(\tau_n) \tr_{\mathrm{Leb}}(f)\\
		&= \lim_{n\to\omega} \tau_n(a) \tr_{\mathrm{Leb}}(g(\tau_n)f)
	\end{split}\end{equation} 
	for all $a \in S$, $g \in C(K)$, and $f \in L^\infty[0, 1]$. 
 
    Identifying $C(K) \otimes L^\infty[0, 1]$ with $C(K, L^\infty[0, 1])$, we have 
	\begin{equation}\label{eqn:trace-on-C(K,L-inf)}
		\tau(a\psi(F)) =\lim_{n\to\omega} \tau_n(a)\tr_{\mathrm{Leb}}(F(\tau_n))
	\end{equation} 
	for all $a \in S$ and $F \in C(K, L^\infty[0, 1])$.  In particular, we have \begin{equation}\|\psi(F)\|_{2,K^\omega} = \sup_{\tau \in K} \|F(\tau)\|_{2,\tau_{\mathrm{Leb}}}. \end{equation} It follows that $\psi$ extends to an embedding \begin{equation}
    C_\sigma\big(K, (L^\infty[0, 1],\tau_{\mathrm{Leb}})\big) \rightarrow \M^\omega \cap S',\end{equation}
    which we also denote by $\psi$, of the trivial $W^*$-bundle with fibre $L^\infty[0, 1]$ such that \eqref{eqn:trace-on-C(K,L-inf)} holds for all $a \in S$ and $F \in C_\sigma\big(K, (L^\infty[0, 1],\tau_{\mathrm{Leb}})\big)$. 
	
Write $\chi_{[h_{i-1},h_i]}\colon K \to L^\infty([0,1])$ for the function $\tau \mapsto \chi_{[h_{i-1}(\tau),h_i(\tau)]}$. This function is in $C_\sigma(K,(L^\infty[0,1],\tau_{\mathrm{Leb}}))$, and so we can define \begin{equation}
    p_i \coloneqq \psi(\chi_{[h_{i-1},h_i]}), \qquad i = 1, \dots, k.
\end{equation}
Then $p_1,\ldots ,p_k$ are orthogonal projections in $\M^\omega \cap S'$ summing to $1_{\M^\omega}$. Let $\tau$ be a limit trace given by a sequence of traces $(\tau_n)_{n=1}^\infty $ in $K$. Then
	\begin{equation}\begin{split}\label{eqn:trace-comp}
		\tau(a_ip_i) &= \lim_{n\to\omega} \tau_n(a_i)\tr_{\mathrm{Leb}}(\chi_{[h_{i-1},h_i]}(\tau_n))\\
		&= \lim_{n\to\omega} \tau_n(a_i)g_i(\tau_n)\\
		&\leq \lim_{n\to\omega}\delta g_i(\tau_n)\\
		&= \delta \tau(p_i),
	\end{split}\end{equation} 
where the inequality in the third line holds as for any $\tau_n \in K$, either $\tau_n(a_i) < \delta$ or $g_i(\tau_n) = 0$ because $g_1,\ldots g_k$ is a partition of unity subordinate to $\{U_1,\ldots,U_k\}$. An application of the Krein--Milman theorem shows that \eqref{eqn:trace-comp} holds for all $\tau \in X^\omega$. Therefore, $\M$ has CPoU.
\end{proof}

\subsection{Permanence properties of CPoU}\label{sec:cpou-permanence}

We will show that CPoU is preserved by inductive limits and matrix amplifications.  CPoU is also preserved under reduced products of tracially complete $C^*$-algebras, but the proof will be deferred to Section~\ref{sec:app-CPoU} (see Remark~\ref{rem:reduced-products-CPoU}).  We begin by showing that CPoU passes to ``quotients'' in the sense that restricting to a closed face of traces and completing in the corresponding uniform 2-norm preserves CPoU (see Corollary~\ref{cor:CPoU-quotient}).  We isolate the following lemma, which will go into the proof of Proposition \ref{prop:CPoU-quotient}, as well as being recycled in Appendix \ref{sec:sep}.

\begin{lemma}\label{lem:separable-face}
	If $\mathcal M$ is a unital $C^*$-algebra, $Y \subseteq \mathcal T(\mathcal M)$ is a closed face, and $S \subseteq \mathcal M$ is $\|\cdot\|$-separable, then there is a separable unital $C^*$-algebra $A \subseteq \mathcal M$ containing $S$ such that 
	\begin{equation}\label{eq:separable-face}
		Y_A \coloneqq \{ \tau|_A : \tau \in Y \} \subseteq T(A)
	\end{equation}
	is a closed face in $T(A)$.
\end{lemma}

\begin{proof}
	We will construct a sequence $(A_n)_{n=1}^\infty$ of separable unital $C^*$-sub\-algebras of $\M$ containing $S$ and self-adjoint elements $a_n \in A_{n+1}$ such that for each $n \geq 1$,
	\begin{enumerate}
		\item\label{item:sep-face-increase} $A_n \subseteq A_{n+1}$,
		\item\label{item:sep-face-extend} every trace on $A_n$ extends to a trace on $\M$,
		\item\label{item:sep-face-positive} $\tau(a_n) \geq 0$ for all $\tau \in T(\M)$, and
		\item\label{item:sep-face-detect} if $Y_n \coloneqq \{\tau|_{A_n} : \tau \in Y \}$, then for all $\tau \in T(\M)$, we have $\tau(a_n) = 0$ if and only if $\tau|_{A_n} \in Y_n$.
	\end{enumerate}
	
	The construction is by induction on $n$.  Using \cite[Lemma~9]{Oz13}, there is a separable unital $C^*$-algebra $A_1 \subseteq \M$ which contains $S$ such that every trace on $A_1$ extends to a trace on $\M$.  With $Y_1$ as in \ref{item:sep-face-detect}, note that $Y_1$ is closed\footnote{Given a net $(\tau_i|_A)$ in $Y_1$ arising from a net of traces $\tau_i\in T(\M)$ with $\tau_i|_A\to\sigma\in T(A)$, it follows that any weak$^*$-cluster point $\tau\in T(\M)$ of the $\tau_i$ is an extension of $\sigma$, so $\sigma\in Y_1$.\label{foot:closed}} and $T(A_1)$ is metrisable, and hence $T(A_1) \setminus Y_1$ is an $F_\sigma$-set.  By the continuity of the restriction map $T(\M) \rightarrow T(A_1)$, we have
	\begin{equation}
		U_1 \coloneqq \{ \tau \in T(\M) : \tau|_{A_1} \notin Y_1 \}
	\end{equation}
	is an $F_\sigma$-set in $T(\M)$ disjoint from $Y$.  By Theorem~\ref{thm:exposed}, there is a continuous affine function $f \colon T(\M) \rightarrow [0, 1]$ such that $f(\tau) = 0$ for all $\tau \in Y$ and $f(\tau) > 0$ for all $\tau \in U_1$.  By Proposition~\ref{prop:CP}, there is a self-adjoint $a_1 \in \M$ with $f(\tau) = \tau(a_1)$ for all $\tau \in T(\M)$.  Now repeat the argument with $A_1 \cup \{a_1\}$ in place of $S$ to obtain $A_2$ and $a_2$, and continue in this fashion.
	
Let $A\subseteq \M$ be the $\|\cdot\|$-closure of the union of the $A_n$.  
Note that every trace $\tau \in T(A)$ extends to a trace on $\M$ -- indeed, for each $n \geq 1$, $\tau|_{A_n}$ extends to a trace $\widetilde\tau_n \in T(\M)$, and then any weak$^*$-limit point of the traces $\widetilde\tau_n$ will extend $\tau$.  
Now suppose $\sigma, \rho \in T(A)$ and  $\tau \coloneqq \frac{1}{2}(\sigma + \rho) \in Y_A$.  
Since $\sigma$ and $\rho$ extend to traces on $\M$, \ref{item:sep-face-positive} implies $\sigma(a_n), \rho(a_n) \geq 0$ for all $n \geq 1$.  
Since $\tau \in Y_A$, we have $\tau|_{A_n} \in Y_n$ for all $n \geq 1$, and hence $\tau(a_n) = 0$ for all $n \geq 1$.  
Therefore, $\sigma(a_n)= 0$ for all $n \geq 1$.  
Using again that $\sigma$ extends to a trace on $\M$, it follows from \ref{item:sep-face-detect} that $\sigma|_{A_n} \in Y_n$ for all $n \geq 1$.  
Let $\widetilde\sigma_n \in Y$ be a trace with $\widetilde\sigma_n|_{A_n} = \sigma|_{A_n}$.  
Let $\widetilde\sigma$ be a weak$^*$-limit point of the $\widetilde\sigma_n$; then $\widetilde\sigma \in Y$ and $\widetilde\sigma|_{A} = \sigma$.  
So $\sigma \in Y_A$.  Similarly, $\rho \in Y_A$, and hence $Y_A$ is a face in $T(A)$. The argument of Footnote \ref{foot:closed} shows $Y_A$ is closed in $T(A)$.
\end{proof}

\begin{proposition}\label{prop:CPoU-quotient}
	If $(\mathcal M, X)$ is a factorial tracially complete $C^*$-algebra with CPoU and $Y \subseteq X$ is a closed face, then $\big(\completion{\mathcal M}{Y}, Y\big)$ is factorial and satisfies CPoU.
\end{proposition}

\begin{proof}
	Note that $(\mathcal N, Y) \coloneqq (\completion{\mathcal M}{Y}, Y)$ is factorial by Proposition~\ref{prop:tracial-completion}\ref{item:completion-factorial}. Define $\phi \colon \mathcal (\mathcal M, X) \rightarrow (\mathcal N, Y)$ to be the canonical map.  Since the range of $\phi$ is $\|\cdot\|_{2, Y}$-dense in $\mathcal N$, it suffices to show that for every $\delta > 0$, $a_1, \ldots, a_n \in \mathcal M$, and $S \subseteq \mathcal M$ such that $S$ is $\|\cdot\|$-separable and
	\begin{equation} \sup_{\tau \in Y} \min_{1 \leq i \leq n} \tau(\phi(a_i)) < \delta,
	\end{equation}
	there are projections $p_1, \ldots, p_n \in \mathcal M^\omega \cap S'$ such that 
	\begin{equation}\label{eq:CPoU-condition}
		\sum_{j=1}^n \phi^\omega(p_j) = 1_\mathcal {N^\omega} \qquad \text{and} \qquad \tau(\phi^\omega(a_i p_i)) < \delta \tau(\phi^\omega(p_i))
	\end{equation}
	for all $i = 1, \ldots, n$ and $\tau \in Y^\omega$.
	
	Given $\delta$, $S$ and  $a_1,\ldots,a_n$ as above, define
	\begin{equation}\label{eq:CPoU-quotient}
		C \coloneqq \{ \tau \in T(\mathcal M) : \min_{1 \leq i \leq n} \tau(a_i) \geq \delta \} \subseteq T(\mathcal M).
	\end{equation}
	Then $C$ is a compact convex subset of $T(\mathcal M)\subseteq \mathcal M^*$ disjoint from the compact convex set $Y \subseteq T(\mathcal M)$.  
	By the Hahn--Banach theorem, there is $\alpha > 0$ and self-adjoint $b \in \M$ such that $\tau(b) > \alpha$ for all $\tau \in Y$ and $\tau(b) < \alpha$ for all $\tau \in C$. Replacing $b$ with $b + (\|b\|+1)1_\M$ and $\alpha$ with $\alpha + \|b\|+1$, we may assume that $b \in \M_+$.  Finally, by scaling $b$, there exists $a_0 \in \M_+$ such that $\tau(a_0) > \delta$ for all $\tau \in Y$ and $\tau(a_0) < \delta$ for all $\tau \in C$.
	
	If $\tau \in X$ such that $\tau(a_i) \geq \delta$ for all $i = 1, \ldots, n$, then $\tau \in C$, and hence $\tau(a_0) = f(\tau) < \delta$.  Therefore, 
	\begin{equation}
		\sup_{\tau \in X} \min_{0 \leq i \leq n} \tau(a_i) < \delta.
	\end{equation}
	By Lemma~\ref{lem:separable-face}, there is a separable unital $C^*$-algebra $A \subseteq \mathcal M$ containing $S$ and $a_0$ such that the set $Y_A \subseteq T(A)$ defined in \eqref{eq:separable-face} is a face in $T(A)$.  Now, since $(\mathcal M, X)$ has CPoU, there are projections $p_0, p_1, \ldots, p_n \in \mathcal M^\omega \cap A'$ such that
	\begin{equation}\label{eq:cpou-extension}
		\sum_{j=0}^n p_j = 1_\mathcal {M^\omega} \qquad \text{and} \qquad \tau(a_i p_i) \leq \delta \tau(p_i)
	\end{equation}
	for all $i = 0, \ldots, n$ and $\tau \in X^\omega$.
	
	To verify \eqref{eq:CPoU-condition}, it suffices to show $\phi^\omega(p_0) = 0$.  Suppose this fails.  Then there is a trace $\tau \in Y^\omega$ with $\tau(\phi^\omega(p_0)) \neq 0$.  Define $\sigma_0 \colon A \rightarrow \mathbb C$ by
	\begin{equation}
		\sigma_0(a) \coloneqq \frac{\tau(\phi^\omega(a p_0))}{\tau(\phi^\omega(p_0))}, \qquad a \in A.
	\end{equation}
	Then $\sigma_0$ is a trace on $A$ that is dominated by a multiple of $\tau \circ \phi|_A \in Y_A$, and since $Y_A$ is a face in $T(A)$, we have $\sigma_0 \in Y_A$.  By the definition of $Y_A$, there is a trace $\sigma \in Y$ such that $\sigma|_A = \sigma_0$.  Then $\sigma(a_0) > \delta$ by the choice of $a_0$.  But this contradicts \eqref{eq:cpou-extension} with $i= 0$ and with $\tau \circ \phi^\omega$ in place of $\tau$.  Therefore, $\phi^\omega(p_0) = 0$, as required.
\end{proof}

As recalled at the beginning of Section~\ref{sec:cpou-def}, a $C^*$-algebra $A$ with $T(A)$ compact has CPoU in the sense of \cite{CETWW} if and only if, in our notation, the uniform tracial completion of $A$ with respect to $T(A)$ has CPoU as a tracially complete $C^*$-algebra.  As a consequence of Proposition~\ref{prop:CPoU-quotient}, CPoU passes to quotients of $C^*$-algebras.

\begin{corollary}\label{cor:CPoU-quotient}
	Let $A$ be a $C^*$-algebra with $T(A)$ compact and let $I \unlhd A$ be an ideal.  Then $T(A/I)$ is compact.  Further, if $A$ satisfies CPoU in the sense of \cite{CETWW}, then so does $A/I$.
\end{corollary}

\begin{proof}
We may identify $T(A/I)$ with the closed face $Y$ of $T(A)$ consisting of traces on $A$ vanishing on $I,$ so $T(A/I)$ is  compact.  Apply Proposition~\ref{prop:CPoU-quotient} after making this identification -- note that we are implicitly using Proposition~\ref{prop:tracial-completion}\ref{item:completion-subset} to identify the uniform tracial completion of $\completion{A}{T(A)}$ with respect to $T(A/I)$ with the uniform tracial completion of $A/I$ with respect to $T(A/I)$.
\end{proof}

Now we show that CPoU is preserved under direct limits.  While the corresponding result for the McDuff property and property $\Gamma$ followed easily from the approximate characterisations of these properties, this permanence property for CPoU will take more work.  The new difficulty arises from the supremum over traces in \eqref{eq:CPoUTraceIneq1} since traces on the finite terms of an inductive system will not generally extend to traces on the limit.  We overcome this challenge with the following lemma.

\begin{lemma}\label{lem:Technical}
	Let $(\M, X) \coloneqq \varinjlim\, \big((\M_n,X_n),\phi_n^{n+1}\big)$ be an inductive limit of tracially complete $C^*$-algebras and write the inductive limit morphisms as $\phi_n^\infty\colon (\M_n,X_n) \rightarrow (\M, X)$.  Suppose $k, m_0 \geq 1$, $\delta > 0$, and $b_1,\ldots,b_k \in \M_{m_0}$ are positive with
	\begin{equation}\label{eqn:Hypothesis}
		\sup_{\tau \in X} \min_{1 \leq i \leq k} \tau(\phi_{m_0}^\infty(b_i)) < \delta.
	\end{equation}
	Then there exists $m \geq m_0$ such that
	\begin{equation}\label{eqn:Conclusion1}
		\sup_{\tau \in X_m} \min_{1 \leq i \leq k} \tau(\phi_{m_0}^m(b_i)) < \delta,
	\end{equation}
	where $\phi_{m_0}^m \coloneqq \phi_{m-1}^{m} \circ \cdots \circ \phi_{m_0}^{m_0+1}$.
\end{lemma}

\begin{proof}
	Suppose the conclusion of the lemma does not hold.  For each $m \geq m_0$, since $X_m$ is compact, the failure of \eqref{eqn:Conclusion1} implies there exists $\tau_m \in X_m$ such that
	\begin{equation}
		\min_{1 \leq i \leq k} \tau_m (\phi_{m_0}^m(b_i)) \geq \delta \label{eqn:Criminal}.
	\end{equation}
	
	Fix a free ultrafilter $\omega$ on $\mathbb N$, and for every $m \geq m_0$, define 
	\begin{equation}
		\sigma_m \coloneqq \lim_{n \rightarrow \omega} \tau_n \circ \phi_m^n \in X_m,
	\end{equation} 
	which exists as $X_m$ is compact.  If $n\geq m_2 \geq m_1 \geq m_0$, then $\tau_n \circ \phi_{m_1}^n = \tau_n \circ \phi_{m_2}^n \circ \phi_{m_1}^{m_2}$ for all $n \geq m_2$. Since $\omega$ is a free ultrafilter, we obtain $\sigma_{m_1} = \sigma_{m_2} \circ \phi_{m_1}^{m_2}$. Therefore, the sequence $(\sigma_m)_{m=m_0}^\infty$ induces a trace $\sigma \in X$ such that $\sigma_m = \sigma \circ \phi^{\infty}_m$ for all $m \geq m_0$.
	
	By \eqref{eqn:Criminal}, $\min_{1 \leq i \leq k} \sigma_{m_0}(b_i) \geq \delta$. Hence, $\min_{1\leq i\leq k}\sigma(\phi_{m_0}^\infty(b_i)) \geq \delta$, contrary to \eqref{eqn:Hypothesis}. This completes the proof.
\end{proof}

With the above technical lemma in hand, we now proceed with the proof that CPoU passes to inductive limits.

\begin{proposition}\label{prop:InductiveLimitCPoU}
	Let $(\M, X) \coloneqq \varinjlim\, \big((\M_n,X_n),\phi_n^{n+1}\big)$ be an inductive limit of factorial tracially complete $C^*$-algebras.
	If $(\M_n, X_n)$ has CPoU for all $n \geq 1$, then $(\M, X)$ has CPoU as well.\footnote{Note that $(\M, X)$ is factorial by Proposition~\ref{prop:UTCInductiveLimit}.}
\end{proposition}

\begin{proof}
	Write the inductive limit morphisms as $\phi_n^\infty\colon(\M_n,X_n) \rightarrow (\M, X)$.
    We will use \ref{item:CPOU-local}$\Rightarrow$\ref{item:CPOU} of Proposition~\ref{CPoU:Ultra:FiniteSet}. As $\bigcup_{m=1}^\infty\phi_m^\infty(\M_m)$ is $\|\cdot\|_{2,X}$-dense in $\M$, we may assume the finite set $\mathcal F=\{x_1,\dots,x_l\}$ in the statement of Proposition~\ref{CPoU:Ultra:FiniteSet}\ref{item:CPOU-local} is contained in $\bigcup_{m=1}^\infty\phi_m^\infty(\M_m)$.  By rescaling, we may assume that each $x_j$ is a contraction. Fix $\epsilon>0$ and consider a family $a_1,\dots,a_k$ of positive elements in $\M$ and a scalar
	\begin{equation}
		\delta>\sup_{\tau\in X}\min_{1 \leq i \leq k}\tau(a_i).
	\end{equation}
	
	By the $\|\cdot\|_{2, X}$-density of $\bigcup_{m=1}^\infty\phi_m^\infty(\M_m)$ in $\M$, there are an integer $m_0 \geq 1$ and positive $b_1, \ldots, b_k \in \M_{m_0}$ such that
	\begin{equation}\label{eq:InductiveLimitCPoU1}
		\|a_i - \phi_{m_0}^\infty(b_i)\|_{2,X}<\epsilon/3.
	\end{equation}
	Since $\tau \circ \phi_{m_0}^\infty \in X_{m_0}$ for all $\tau \in X$, this implies
	\begin{equation}
		\sup_{\tau\in X}\min_{1 \leq i \leq k} \tau(\phi_{m_0}^\infty(b_i)) < \delta+\epsilon/3.
	\end{equation}
	Accordingly, by Lemma \ref{lem:Technical}, we can find $m \geq m_0$ large enough so that
	\begin{equation}\label{eqn:Conclusion}
		\sup_{\tau \in X_m} \min_{1 \leq i \leq k} \tau(\phi_{m_0}^m(b_i)) < \delta+\epsilon/3,
	\end{equation}
	and enlarging $m$ if necessary, we may find $y_1, \ldots, y_l \in \M_m$ such that $\phi_m^\infty(y_j) = x_j$ for $1 \leq j \leq l$.
	
	Now we apply CPoU in $\M_m$ using Proposition~\ref{CPoU:Ultra:FiniteSet}\ref{item:CPOU-local} to obtain pairwise orthogonal positive contractions $f_1,\dots,f_k\in\M_m$ such that
	\begin{equation}\begin{aligned}\label{eq:InductiveLimitCPoU2}
		\|[f_i,y_j]\|_{2,X_m} &<\epsilon,&&i=1,\dots,k,\ j=1\dots,l,  \\
		\tau(f_1+\cdots+f_k) &> 1-\epsilon, && \tau \in X_m, \text{ and}  \\
		\tau(\phi_{m_0}^m(b_i)f_i) &< (\delta+\epsilon/3)\tau(f_i) + \epsilon/3  \\
		&\leq \delta\tau(f_i)+2\epsilon/3, && \tau \in X_m,\ i=1,\dots,k.
	\end{aligned}\end{equation}
	
	Set $e_i\coloneqq\phi_m^\infty(f_i)\in \mathcal N$.  These are pairwise orthogonal positive contractions. By the $\|\cdot\|_{2,X_m}$-$\|\cdot\|_{2,X}$-contractivity of $\phi_m^\infty$, we have
	\begin{equation}
		\|[e_i,x_j]\|_{2,X}<\epsilon, \qquad i =1, \ldots, k, \ j=1, \ldots, l.
	\end{equation}
	As $(\phi_m^\infty)^*(X)\subseteq X_m$,
	\begin{equation}
		\tau(e_1+\dots+e_k)>1-\epsilon,\qquad \tau\in X,
	\end{equation}
	and
	\begin{equation}\begin{array}{rcl}
		\tau(a_ie_i)&\stackrel{\eqref{eq:InductiveLimitCPoU1}}\leq& \tau(\phi_{m_0}^\infty(b_i)e_i)+\epsilon/3  \\
		&=& (\phi_m^\infty)^*(\tau)(\phi_{m_0}^m(b_i)f_i)+\epsilon/3  \\
		&\stackrel{\eqref{eq:InductiveLimitCPoU2}}<&\delta(\phi_m^\infty)^*(\tau)(f_i)+\epsilon=\delta\tau(e_i)+\epsilon,\qquad \tau\in X.
	\end{array}\end{equation}
	Therefore, $(\mathcal M, X)$ has CPoU.
\end{proof}

\begin{remark}[cf.\ Remark~\ref{rem:GammaUltrapower}]\label{rem:reduced-products-CPoU}
	We will later show that a reduced product of factorial tracially complete $C^*$-algebras with CPoU is factorial and has CPoU (Corollary~\ref{cor:factorial-ultrapower}\ref{cor:factorial-ultrapower:CPoU}).  Once the reduced product is shown to be factorial, the reduced product will have CPoU since the local characterisation of CPoU in Proposition~\ref{CPoU:Ultra:FiniteSet}\ref{item:CPOU-local} is easily seen to be preserved by reduced products.
\end{remark}

The property CPoU also passes to matrix algebras.

\begin{proposition}\label{prop:matrix-CPoU}
	If $(\M, X)$ is a factorial tracially complete $C^*$-algebra with CPoU, then so is $(\M \otimes M_d, X \otimes \{\mathrm{tr}_d\})$ for all $d \in \mathbb N$.
\end{proposition}

\begin{proof}
	First note that $(\M \otimes M_d, X \otimes \{\mathrm{tr_d}\})$ is a factorial tracially complete $C^*$-algebra by Proposition~\ref{prop:matrix-algebras}.  Now fix a $\|\cdot\|_{2, X \otimes \{\mathrm{tr}_d\}}$-separable set $S \subseteq \M \otimes M_d$ and let $a_1, \ldots, a_k \in (\M \otimes M_d)_+$ and $\delta > 0$ be such that
	\begin{equation}
		\sup_{\tau \in X} \min_{1 \leq i \leq k} (\tau \otimes \mathrm{tr}_d)(a_i) < \delta.
	\end{equation}
	Let $T \subseteq \M$ be the set of entries of elements of $S$ and note that $T$ is $\|\cdot\|_{2, X}$-separable.  Applying CPoU to the elements $(\mathrm{id}_\M \otimes \mathrm{tr}_d)(a_i) \in \M_+$, there are projections $q_1, \ldots, q_k \in \M^\omega \cap T'$ such that
	\begin{equation}
		\sum_{j=1}^k q_j = 1_\M \quad \text{and} \quad \tau((\mathrm{id}_\M \otimes \mathrm{tr}_d)(a_i)q_i) \leq \delta\tau(q_i)
	\end{equation}
	for all $i = 1, \dots, k$ and $\tau \in X$.  Then $p_i \coloneqq q_i \otimes 1_{M_d}$, $1 \leq i \leq k$, are projections in $(\M \otimes M_d)^\omega \cap T'$ which verify CPoU.\footnote{Note that we are using Proposition~\ref{prop:matrix-ultrapower} to identify $\M^\omega \otimes M_d$ and $(\M \otimes M_d)^\omega$.}
\end{proof}

\subsection{Property \texorpdfstring{$\Gamma$}{Gamma} implies CPoU}\label{sec:CPoU-proof}

We now turn to the proof of Theorem~\ref{introthmgammaimpliescpou}. 
In the separable setting, the proof reduces to the strategy outlined towards the end of Section~\ref{sec:intro:ltg2}, which is modelled on the proof of \cite[Theorem~3.8]{CETWW}.  

We begin with the key technical lemma.
\begin{lemma}\label{lemma:weak-cpou}
    Let $(\M,X)$ be a factorial tracially complete $C^*$-algebra. 
    Let $\delta > 0$.
    Let $q \in \M_+$ satisfy $\tau(q) > 0$ for all $\tau \in X$.
    Let $\epsilon > 0$ and let $\F \subseteq \M$ be a finite subset.
    Suppose $a_1,\ldots,a_k \in \M_+$ satisfy 
    \begin{equation}
		\min_{1 \leq i \leq k} \tau(a_i q) < \delta \tau(q)
    \end{equation}
    for all $\tau \in X$.
    Then there exist positive contractions $e_1,\ldots,e_k \in \M_+$ such that
    \begin{alignat}{2}
		\left\|\sum_{i=1}^k e_i - 1_\M\right\|_{2,X} &< \epsilon,  && \label{eq:lemma-weak-cpou-statement-1}\\
        \|[e_i,x]\|_{2,X} &<\epsilon, &&x \in \F, \, i = 1, \ldots, k,\ \text{and}  \label{eq:lemma-weak-cpou-statement-2}\\
		\tau(a_i e_i q) - \delta \tau(e_i q) &< \epsilon, & \hspace{4ex}& \tau \in X,\ i = 1, \ldots, k. \label{eq:lemma-weak-cpou-statement-3}
    \end{alignat}
\end{lemma}
\begin{proof}
    Let $\epsilon_0 > 0$.
     Fix $\tau \in X$ for now. We work in the von Neumann algebra $\mathcal N \coloneqq \pi_\tau(\M)''$. For brevity, denote $\bar{q} \coloneqq \pi_\tau(q)$, $\bar{x}\coloneqq\pi_\tau(x)$ for $x \in \F$, and $\bar{a}_i \coloneqq \pi_\tau(a_i)$ for $i=1,\ldots,k$.
     For each $\sigma \in T(\mathcal N)$, we have $\sigma \circ \pi_\tau \in X$ by Lemma~\ref{lem:GNSTraceFace}.  
     Therefore, $\sigma(\bar{q}) > 0$ for all $\sigma \in T(\mathcal N)$ and
	\begin{equation}
		\sup_{\sigma \in T(\mathcal N)} \min_{1 \leq i \leq k} \sigma(\bar{a}_i\bar{q}) < \delta \sigma(\bar{q}).
	\end{equation}
	By Corollary~\ref{cor:vNa-CPoU}, there are projections $\bar e_1, \ldots, \bar e_k \in Z(\mathcal N)$ summing to $1_{\mathcal N}$ such that $\sigma(\bar{a}_i \bar{e_i}\bar{q}) \leq \delta \sigma(\bar{e_i}\bar{q})$ for all $i = 1, \ldots, k$ and $\sigma\in T(\mathcal N)$.  Note also that $[\bar{e}_i, \bar{x}]=0$ for all $x \in\F$ and $i=1,\ldots,k$. 
    For $i = 1, \ldots, k$, define 
    \begin{equation}
		\bar b_i \coloneqq \mathrm{tr}_{\mathcal N}(\delta \bar{e}_i\bar{q}-\bar{a}_i \bar{e}_i\bar{q}) \in Z(\mathcal N),
    \end{equation}
    where $\mathrm{tr}_{\mathcal N}$ is the centre-valued trace on $\mathcal N$.     Then $\bar b_i \in \mathcal N_+$ and 
    \begin{equation}
		\sigma(\bar{a}_i \bar{e}_i\bar{q}) + \sigma(\bar{b}_i) = \delta \sigma(\bar{e}_i\bar{q}), \qquad  \sigma \in T(\mathcal N).
    \end{equation}
	Let $\bar{c}_i \coloneqq \bar{a}_i \bar{e}_i\bar{q} - \delta\bar{e}_i\bar{q} + \bar{b}_i\in\mathcal N$.  Then $\bar c_i$ vanishes on all traces, so $\bar c_i$ is a sum of commutators (see \cite[Th{\'e}or{\`e}me~2.3]{FH80}).\footnote{A standard Hahn--Banach argument implies that since $\bar c_i$ vanishes on all traces, $\bar c_i$ is in the $\|\cdot\|$-closed span of the commutators.  After adjusting the bounds, this weaker result is sufficient to run the proof without quoting \cite[Th{\'e}or{\`e}me~2.3]{FH80}}  
 
    Using Kaplansky's density theorem, followed by lifting positive contractions to positive contractions, positive elements to positive elements, and commutators to commutators, we get positive contractions $e_1^\tau,\ldots,e_k^\tau \in \M_+$, positive elements $b_1^\tau,\ldots,b_k^\tau \in \M_+$, and elements $c_1^\tau,\ldots,c_k^\tau \in \M$ that are finite sums of commutators such that
       \begin{equation}
       \begin{array}{rlrl}
        \left\|\sum_{i=1}^k e_i^\tau - 1_\M\right\|_{2,\tau} &< \epsilon_0, \quad &&\\
        \|[e_i^\tau,x]\|_{2,\tau} &<\epsilon_0, && x \in \F, \, i = 1, \ldots, k, \\
		\|a_i e_i^\tau q - \delta e_i^\tau q + b_i^\tau-c_i^\tau\|_{2,\tau} &< \epsilon_0,  &&i = 1, \ldots, k.
       \end{array}
    \end{equation}

We apply Lemma~\ref{lem:affine-selection} with $\mathcal C$ equal to the set of all $(3k)$-tuples \begin{equation}
 (e_1,\dots,e_k,b_1,\dots,b_k,c_1,\dots,c_k)\in \mathcal M^{3k}   
\end{equation} such that $e_1,\dots,e_k$ are positive contractions, $b_1,\dots,b_k$ are positive, and $c_1,\dots,c_k$ are sums of commutators (this set is evidently convex), using the affine functions given by $(e_1,\dots,c_k)\mapsto \sum_{i=1}^k e_i-1_{\mathcal M}$, $(e_1,\dots,c_k) \mapsto [x,e_i]$ for $x\in \mathcal F$ and $i=1,\dots,k$, and $(e_1,\dots,c_k) \mapsto a_i e_i q - \delta e_iq + b_i-c_i$ for all $i = 1, \ldots, k$.  The pointwise result of the previous paragraph provides the hypothesis of the lemma. Therefore, for any $\epsilon>0$ there are positive contractions $e_1, \ldots e_k \in \M$, positive elements $b_1, \ldots, b_k \in \M$, and elements $c_1, \ldots, c_k \in \M$ in the span of the commutators in $\M$ such that 
		\begin{equation} \begin{array}{rlrl}
		\Big\|\sum_{j=1}^k e_j - 1_\mathcal M \Big\|_{2, X} &< \epsilon,   \\
        \|[x, e_i]\|_{2, X} &< \epsilon, \quad x \in \F, \, i = 1, \ldots, k,\ \text{and}  \\
		\|a_i e_i q - \delta e_iq + b_i-c_i\|_{2, X} &< \epsilon,  \quad i = 1, \ldots, k. \label{eq:cpou-proof}
	\end{array}\end{equation}
        Thus, \eqref{eq:lemma-weak-cpou-statement-1} and \eqref{eq:lemma-weak-cpou-statement-2} hold.
	For any $\tau\in X$, we have $\tau(c_i)=0$ (as $c_i$ is in the span of the commutators), so that equation \eqref{eq:cpou-proof} implies
	\begin{equation}
		\tau(a_i e_i q - \delta e_iq + b_i-c_i)=\tau(a_i e_i q - \delta e_iq + b_i) < \epsilon.
	\end{equation}
	Rearranging, this gives
	\begin{equation}
		\tau(a_i e_iq) < \delta \tau(e_iq) - \tau(b_i) + \epsilon \leq \delta \tau(e_iq) + \epsilon,
	\end{equation}
	since $b_i$ is positive.  Thus, \eqref{eq:lemma-weak-cpou-statement-3} holds. 
\end{proof}

The following result is the weak form of CPoU discussed in the first step of the outline.  This was shown to hold for uniform tracial completions of nuclear $C^*$-algebras with compact trace simplex in \cite[Lemma~3.6]{CETWW} -- we extend the result to all factorial tracially complete $C^*$-algebras. 

\begin{theorem}\label{thm:weak-CPoU}
    Let $(\M, X)$ is a factorial tracially complete $C^*$-algebra. 
    Let $q \in \M^\omega \cap \M'$ be a positive contraction with $\tau(q) > 0$ for all $\tau \in X^\omega$.
    Let $S \subseteq \M^\omega$ be a $\|\cdot\|_{2, X^\omega}$-separable set.
    Let $\delta > 0$ and suppose $a_1,\ldots,a_k \in \M_+$ satisfy
    \begin{equation}
        \sup_{\tau \in X} \min_{1 \leq i \leq k} \tau(a_i) < \delta.
    \end{equation}
    Then there exist positive contractions $e_1, \ldots, e_k \in \M^\omega \cap S'$ summing to $1_{\M^\omega}$ with $\tau(a_i e_iq) \leq \delta \tau(e_iq)$ for all $\tau \in X^\omega$ and $i = 1, \ldots, k$.
\end{theorem}	

\begin{proof}
    Let $\tau \in X^\omega$.
    Define $\tau_q \in \M^*$ by $x \mapsto \tfrac{1}{\tau(q)}\tau(xq)$ for all $x \in \M$. 
    Since $q \in \M^\omega \cap \M'$, $\tau_q$ is a trace on $\M$. 
    The restriction of $\tau$ to $\M$ is in $X$ by compactness.
    Using this, and since $\tau_q \leq \tau(q)^{-1}\tau|_\M$, it follows that $\tau_q \in X$ because $X$ is a face.
    Since $\tau(a_iq) = \tau_q(a_i)\tau(q)$ for $i=1,\ldots,k$, we get that
    \begin{equation}
        \min_{1 \leq i \leq k} \tau(a_iq) < \delta \tau(q) \label{eqn:localised-cpou-hyp}.
    \end{equation}
    As $\tau \in X^\omega$ was arbitrary, \eqref{eqn:localised-cpou-hyp} holds for all $\tau \in X^\omega$.

    We claim that we can represent $q$ by a sequence of positive contractions $(q_n)_{n=1}^\infty$ such that
    \begin{align}
        \min_{1 \leq i \leq k} \tau(a_iq_n) &< \delta \tau(q_n), \quad \quad \mbox{ and }\label{eqn:rep-1}\\
        \tau(q_n) &> 0, \label{eqn:rep-2}
    \end{align}
    for all $n \in \N$ and all $\tau \in X$. 
    
    To see this, start with any representative sequence of positive contractions $(q_n)_{n=1}^\infty$ for $q$. 
    Let $I \subseteq \N$ be the set of all $n \in \N$ such that \eqref{eqn:rep-1} holds for all $\tau \in X$.
    Suppose for a contradiction $I \not\in \omega$. 
    Then every set in $\omega$ must intersect $\N \setminus I$, so there is a ultrafilter $\omega' \supseteq \omega$ such that $\N \setminus I \in \omega'$. 
    For each $n \in \N \setminus I$ we can find a trace $\tau_n \in X$ such that \eqref{eqn:rep-1} fails. Choosing $\tau_n \in X$ arbitrarily for $n\in I$, the sequence $(\tau_n)_{n=1}^\infty$ and the ultrafilter $\omega'$ define a limit trace such that \eqref{eqn:localised-cpou-hyp} fails. So we must have $I \in \omega$. Similarly, since $\tau(q)$ is bounded below (by compactness of $X^\omega$) there must be a set $J \in \omega$ such that \eqref{eqn:rep-2} holds for all $n \in J$ and $\tau \in X$. We now redefine $q_n$ to be $1_{\M}$ for each $n \not \in I \cap J$. Since $I \cap J \in \omega$, the redefined sequence still represents $q$.

    Since $S$ is  $\|\cdot\|_{2, X^\omega}$-separable, there is a dense sequence $(s_j)_{j=1}^\infty$.
    For each $j \in \N$, choose a representative sequence $(s_j^{(n)})_{n=1}^\infty$ for $s_j$. 
    Let $n \in \N$. Set $\F_n = \{s_1^{(n)}, \ldots,s_n^{(n)}\}$.
    By Lemma \ref{lemma:weak-cpou}, there exist positive contractions $e_1^{(n)},\ldots,e_k^{(n)} \in \M$ such that
     \begin{equation}\begin{array}{rlrl}
		\left\|\sum_{i=1}^k e_i^{(n)} - 1_\M\right\|_{2,X} &< 1/n,  && \\
        \|[e_i^{(n)},x]\|_{2,X} &< 1/n, &&x \in \F_n, \, i = 1, \ldots, k,\ \text{and}  \\
		\tau(a_i e_i^{(n)} q) - \delta \tau(e_i^{(n)} q) &< 1/n, & \hspace{4ex}& \tau \in X,\ i = 1, \ldots, k. 
    \end{array}\end{equation}
    Define $e_i \in \M^\omega$ by the representative sequence $(e_i^{(n)})_{n=1}^\infty$ for $i=1,\ldots,k$.
    Then $e_i \in \M^\omega \cap S'$ for $i=1,\ldots,k$. Moreover, $e_1, \ldots, e_k$ are positive contractions summing to $1_{\M^\omega}$ and $\tau(a_ie_iq) \leq \tau(e_iq)$ for all $\tau \in X^\omega.$
\end{proof}

The rest of the proof of Theorem \ref{introthmgammaimpliescpou} that property $\Gamma$ implies CPoU proceeds as in \cite{CETWW}.  The following proposition  is a tracially complete version of the ``tracial projectionisation'' result in \cite[Lemma~2.4]{CETWW}, and the proof is essentially identical.  For the sake of completeness, we provide a sketch.

\begin{proposition}[Projectionisation]\label{prop:projectionisation}
	Suppose $(\M, X)$ is a factorial tracially complete $C^*$-algebra with property $\Gamma$.  If $S \subseteq \M^\omega$ is a $\|\cdot\|_{2, X^\omega}$-separable subset and $e \in \M^\omega \cap S'$ is a positive contraction, then there is a projection $p \in \M^\omega \cap S'$ commuting with $e$ such that $\tau(ae) = \tau(ap)$ for all $a \in S$ and $\tau \in X^\omega$.
\end{proposition}

\begin{proof}
	Let $k \geq 1$.
 Using property $\Gamma$, we may fix projections $r_1, \ldots, r_k \in \M^\omega \cap S' \cap \{e\}'$ partitioning the unit and satisfying
	\begin{equation}
		\tau(ar_i) = \frac1k \tau(a), \qquad \tau \in X^\omega,\ a \in C^*(S \cup \{e\}),\ i = 1, \ldots, k.
	\end{equation}
	For $i = 1, \ldots, k$, consider the continuous function $f_i \colon [0, 1] \rightarrow \mathbb R$ given by
	\begin{equation}
		f_i(t) \coloneqq
		\begin{cases}
			0, & 0 \leq t \leq (i - 1)/k; \\
			kt - i + 1, & (i - 1)/k \leq t \leq i/k; \\
			1, & i/k \leq t \leq 1.
		\end{cases}
	\end{equation}
	Then set $q \coloneqq \sum_{i=1}^k f_i(e) r_i \in \M^\omega \cap S' \cap \{e\}'$, which is a positive contraction.  A computation (as in \cite[Equations (2.9) and (2.10)]{CETWW})\footnote{\cite[Equation (2.7)]{CETWW} is still correct if the right side is changed to $\frac14 1_{C([0,1])}$, and using this, the right side of \cite[Equation (2.10)]{CETWW} can be improved to $1/(4n)$.} shows
	\begin{equation}
		\tau(ae) = \tau(aq) \quad\text{and}\quad \tau(q - q^2) < \frac{1}{4k},\quad \tau \in X^\omega,\ a\in S.
	\end{equation}

 The result follows from Kirchberg's $\epsilon$-test (Lemma \ref{lem:EpsTest}).
\end{proof}

Applying the above projectionisation to each of the $e_i$ from the conclusion of Theorem~\ref{thm:weak-CPoU} (with $q=1$) will produce projections which have the correct behaviour on traces, but fail to be orthogonal. As in \cite{CETWW}, this is addressed via another application of property $\Gamma$ with a construction deemed ``orthogonalisation'' in \cite{CETWW}.  The orthogonalised projections will no longer sum to the unit, but they will sum to a projection of constant trace $\frac1k$.  From here, the final stage in the proof of Theorem \ref{introthmgammaimpliescpou} -- which we restate below for the convenience of the reader -- is to use a maximality argument by repeating the construction in the complementary corner -- this is the reason for the projection $q$ in Theorem~\ref{thm:weak-CPoU}. (A geometric series argument could be used instead of the maximality argument for this last step.)  

\gammacpou*

\begin{proof}
The results of Appendix~\ref{sec:sep} reduce the theorem to the case when $\M$ is $\|\cdot\|_{2, X}$-separable.  
Namely, Theorem~\ref{thm:sep} shows that both factoriality plus property $\Gamma$ and factoriality plus CPoU are strongly separably inheritable properties;
therefore, to show one implies the other, it suffices to do so in the $\|\cdot\|_{2,X}$-separable case.
	
For the rest of the proof, we assume $\M$ is $\|\cdot\|_{2, X}$-separable. Fix a free filter $\omega$ on $\N$.   
Suppose $a_1 \ldots, a_k \in \M_+$ and $\delta > 0$ are given as in the definition of CPoU, i.e.\ they satisfy
\begin{equation}\label{eq:cpou-proof-setup}
    \sup_{\tau \in X} \min_{1 \leq i \leq k} \tau(a_i) < \delta.
\end{equation}
Let $I \subseteq [0, 1]$ denote the set of $\alpha \in [0, 1]$ such that for all $\|\cdot\|_{2, X^\omega}$-separable sets $S \subseteq \mathcal M^\omega$, there are orthogonal projections $p_1, \ldots, p_k \in \M^\omega \cap S'$ such that 
\begin{equation}
    \sum_{j=1}^k \tau(p_j) = \alpha \qquad \text{and} \qquad \tau(a_i p_i) \leq \delta \tau(p_i)
\end{equation}
for all $\tau \in X^\omega$ and $i = 1, \ldots, k$.  
Clearly $I \neq \emptyset$ as $0 \in I$, and $I$ is closed by Kirchberg's $\epsilon$-test.  
Hence $I$ contains a maximal element $\alpha$.  
It suffices to show $\alpha = 1$ since this forces $\sum_{j=1}^k p_j = 1_\mathcal {M^\omega}$. 
We will assume $\alpha < 1$ and show $\alpha + \frac1k(1 - \alpha) \in I$, which contradicts the maximality of $\alpha$.
	
Assume $\alpha <1$ and let $S \subseteq \M^\omega$ be a $\|\cdot\|_{2, X^\omega}$-separable set with $\M \subseteq S$.  
By the assumption on $\alpha$, there are mutually orthogonal projections $p_1' \ldots, p_k' \in \M^\omega \cap S'$ such that
\begin{equation}
    \sum_{j=1}^k \tau(p_j') = \alpha \qquad \text{and} \qquad \tau(a_i p_i') \leq \delta \tau(p_i')
\end{equation}
for all $\tau \in X^\omega$ and $i = 1, \ldots, k$.  
Define 
\begin{equation}
    q \coloneqq 1_{\M^\omega} - \sum_{j=1}^k p_j' \in \M^\omega \cap S' \subseteq \M^\omega \cap \M'
\end{equation}  
Note $q$ is a projection and $\tau(q) = 1 - \alpha > 0$ for all $\tau \in X^\omega$.
	
Theorem~\ref{thm:weak-CPoU} provides positive contractions $e_1, \ldots, e_k \in \M^\omega \cap (S \cup \{q\})'$ such that
\begin{align}
    \sum_{i=1}^k e_i &= 1_{\M^\omega},\label{Proof1.4:eqn-sum-to-1}\\
    \tau(a_i e_iq) &\leq \delta \tau(e_iq),\qquad \tau\in X^\omega,\ i=1,\ldots,k.\label{Proof1.4:neweq1}
\end{align}
    
We apply Proposition~\ref{prop:projectionisation} to the set $S\cup\{q\}$ and each $e_iq\in \mathcal M\cap (S\cup\{q\})'$ to obtain projections $q_1, \ldots, q_k \in \M^\omega \cap (S\cup\{q\})'$ so that $q_i$ additionally commutes with $e_iq$ and satisfies 
\begin{equation}\label{eqn:application-projectionisation}
    \tau(b e_iq) = \tau(bq_i),\qquad \tau\in X^\omega,\ b\in C^*(S\cup\{q\}),\ i=1,\ldots,k.
\end{equation} 
Let $T \subseteq \M^\omega$ be the $\|\cdot\|_{2,X^\omega}$-separable $C^*$-subalgebra generated by $S$, $q$, and the elements $a_i$, $e_i$, $q_i$ for $i=1,\ldots,k$.
Using property $\Gamma$, let $r_1, \ldots, r_k \in \M^\omega \cap T'$ be projections summing to $1_{\mathcal M^\omega}$ such that 
\begin{equation}\label{Proof1.4:neweq2}
   \tau(b r_i) = \frac1k \tau(b),\qquad \tau\in X^\omega,\ b\in T,\ i=1,\ldots,k. 
\end{equation}
	
Define $p_i'' \coloneqq q_i r_i$ for $i = 1, \ldots k$.  
These are orthogonal projections since $r_1,\dots,r_k$ are orthogonal projections which commute with all the $q_i$'s.
We have
\begin{equation}
    \sum_{j=1}^k \tau(p_i'') \stackrel{\eqref{Proof1.4:neweq2}}{=} \frac1k \sum_{j=1}^k \tau(q_i) \stackrel{\eqref{eqn:application-projectionisation}}{=} \frac1k \sum_{j=1}^k  \tau(e_iq) \stackrel{\eqref{Proof1.4:eqn-sum-to-1}}{=} \frac1k\tau(q) = \frac1k (1 - \alpha)
\end{equation}
for all $\tau \in X^\omega$.  
Further, for all $\tau \in X^\omega$ and $i = 1, \ldots, k$, we have
\begin{equation}\begin{array}{rcl}
    \tau(a_i p_i'') &\stackrel{\eqref{Proof1.4:neweq2}}{=}& \frac1k \tau(a_i q_i) \\
    &\stackrel{\eqref{eqn:application-projectionisation}}{=}& \frac1k \tau(a_i e_iq) \\
    &\stackrel{\eqref{Proof1.4:neweq1}}{\leq}& \frac{\delta}{k} \tau(e_iq) \\
    &\stackrel{\eqref{eqn:application-projectionisation}}{=}& \frac{\delta}{k} \tau(q_i)  \\
    &\stackrel{\eqref{Proof1.4:neweq2}}{=}&  \delta\tau(p_i'').
    \end{array}
\end{equation}

Since $q$ and $q_i$ are commuting projections, $q_iq$ is projection satisfying $q_iq \leq q_i$ and $q_iq \leq q$.
Moreover, $\tau(q_iq) = \tau(qq_i) = \tau (qe_iq) = \tau(e_iq) = \tau(q_i)$ for all $\tau \in X^\omega$ by \eqref{eqn:application-projectionisation}.
Hence, $q_i = q_iq \leq q$.  
Therefore, $p_i'' \leq q_i \leq q$ for all $i = 1, \ldots, k$.  
Set $p_i \coloneqq p_i' + p_i'' \in \M^\omega \cap S'$.
Since $p_1',\dots,p_k',p_1'',\dots,p_k''$ are mutually orthogonal projections, the same is true of $p_1,\dots,p_k$. 
Further, the projections $p_i$ witness $\alpha + \frac1k (1 - \alpha) \in I$, which contradicts the maximality of $\alpha$.  
\end{proof}

\section{Applications of CPoU}\label{sec:app-CPoU}

In this section we demonstrate the power of CPoU: it enables us to prove many von Neumann algebraic-type structural properties for a reduced product of a sequence of factorial tracially complete $C^*$-algebras with CPoU.

First, in Section~\ref{sec:app-CPoU-traces}, we show that a reduced product of factorial tracially complete $C^*$-algebras with CPoU is factorial, by solving the trace problem for such reduced products. Along the way, we show that such reduced products have real rank zero and comparison of projections with respect to limit traces.  We then analyse the unitary group of factorial tracially complete $C^*$-algebras with CPoU in Section~\ref{sec:app-CPoU-unitaries} and show that a uniform 2-norm dense set of unitaries are exponentials.  This also provides a stability result for unitaries showing that in the uniform 2-norm, every approximate unitary is close to a unitary, a result needed for the one-sided intertwining argument.  Finally, Section~\ref{sec:app-CPoU-projections} strengthens the comparison property obtained in Section~\ref{sec:app-CPoU-traces} mentioned above, and classifies projections up to unitary equivalence via traces, in factorial tracially complete $C^*$-algebras with CPoU, obtaining Theorems~\ref{IntroThmCtsProj} and~\ref{introthm-classprojections2}.

\subsection{Traces on reduced products}\label{sec:app-CPoU-traces}

The following lemma gives an approximate version of the existence of spectral projections in factorial tracially complete $C^*$-algebras with CPoU, and the proof is a typical application of CPoU.  The idea is the following: if $\M$ is a von Neumann algebra and $x \in \M$ is self-adjoint, then there is a projection $q \in \M$ such that $qx = x_+$; indeed, one may take $q$ to be the spectral projection of $x$ corresponding to the interval $[0, \infty)$. Then CPoU provides a method for transferring this result from tracial von Neumann algebra completions of a tracially complete $C^*$-algebra $(\M, X)$ to produce an analogous result in $\M$ up to a small $\|\cdot\|_{2, X}$-error.

\begin{lemma}\label{lem:approx-spectral-proj}
	Suppose $(\M, X)$ is a factorial tracially complete $C^*$-algebra with CPoU.  If $x \in \M$ is self-adjoint and $\epsilon > 0$, then there is a positive contraction $q \in \M$ such that
	\begin{equation}
		\|q - q^2\|_{2, X} < \epsilon \qquad \text{and} \qquad \|qx - x_+\|_{2, X} < \epsilon.
	\end{equation}
\end{lemma}

\begin{proof}
	For each $\tau \in X$, let $\bar q_\tau \in \pi_\tau(\M)''$ be the spectral projection of $\pi_\tau(a)$ corresponding to the interval $[0, \infty)$ so that $\bar q_\tau \pi_\tau(x) = \pi_\tau(x_+)$.  By Kaplansky's density theorem, there is a positive contraction $q_\tau \in \M$ such that
	\begin{equation}
		\|q_\tau - q_\tau^2\|_{2, \tau} < \frac{\epsilon}{\sqrt3} \qquad \text{and} \qquad \| q_\tau x - x_+ \|_{2, \tau} < \frac{\epsilon}{\sqrt3}.
	\end{equation}
	For all $\tau \in X$, define 
	\begin{equation}
		a_\tau \coloneqq |q_\tau - q_\tau^2|^2 +  |q_\tau x - x_+|^2 \in \M_+
	\end{equation} 
	and note that $\tau(a_\tau) < 2\epsilon^2/3$.  For $\tau \in X$, let
	\begin{equation}
		U_\tau \coloneqq \Big\{\sigma \in X : \sigma(a_\tau) < \frac{2\epsilon^2}{3}\Big\},
	\end{equation}
	which is an open neighbourhood of $\tau$.  As $X$ is compact, there are traces $\tau_1, \ldots, \tau_k \in X$ such that $(U_{\tau_i})_{i=1}^k$ is an open cover of $X$.  Therefore,
	\begin{equation}
		\sup_{\tau \in X} \min_{1 \leq i \leq k} \tau(a_{\tau_i}) < \frac{2\epsilon^2}{3}.
	\end{equation}
	
	Set $S \coloneqq \{q_{\tau_i}: i=1,\ldots,k\} \cup \{x\}$.	As $(\M, X)$ has CPoU, there are pairwise orthogonal projections $p_1, \ldots, p_k \in \M^\omega \cap S'$ such that
	\begin{equation}
		\sum_{i=1}^k p_i = 1_{\M^\omega} \qquad \text{and} \qquad \tau(a_{\tau_i} p_i) \leq \frac{2\epsilon^2}{3}\tau(p_i)
	\end{equation}
	for all $1 \leq i \leq k$ and $\tau \in X^\omega$. Let $q \coloneqq \sum_{i=1}^k q_{\tau_i} p_i \in \M^\omega_+$. Since $p_1, \ldots, p_k$ are mutually orthogonal projections summing to the identity which commute with $\{q_{\tau_i}: i=1,\ldots,k\} \cup \{x\}$, we have
	\begin{equation}\begin{split}
		|q - q^2|^2 + |qx - x_+|^2 &= \sum_{i=1}^k |q_{\tau_i}^2 - q_{\tau_i}|^2p_i + \sum_{i=1}^k |q_{\tau_i}x - x_+|^2p_i \\
		&= \sum_{i=1}^k a_{\tau_i}p_i.
	\end{split}\end{equation}	
	Now we compute
	\begin{equation}\begin{split}
		\|q - q^2\|_{2, \tau}^2 + \|qx - x_+\|_{2, \tau}^2 &= \sum_{i=1}^k \tau(a_{\tau_i} p_i) \leq \frac{2\epsilon^2}{3} \sum_{i=1}^k \tau(p_i) < \epsilon^2
	\end{split}\end{equation} 
	for all $\tau \in X^\omega$.  If $(q_n)_{n=1}^\infty$ is a sequence of positive contractions in $\M$ representing $q$, then for $\omega$-many $n$,
	\begin{equation}
		\|q_n - q_n^2\|_{2, X} < \epsilon \qquad \text{and} \qquad \| q_nx - x_+\|_{2, X} < \epsilon. \qedhere
	\end{equation}
\end{proof}

We will prove a stronger version of the following result in Corollary~\ref{cor:stable-rank-one} once we show reduced products of factorial tracially complete $C^*$-algebras with CPoU are factorial and have CPoU.

\begin{proposition}\label{prop:real-rank-zero}
	Let $\big((\M_n, X_n)\big)_{n=1}^\infty$ be a sequence of factorial tracially complete $C^*$-algebras with CPoU. Then $\prod^\omega \M_n$ has real rank zero.
\end{proposition}

\begin{proof}
	Define $(\M, X) \coloneqq \big(\prod^\omega \M_n, \sum^\omega X_n\big)$ and let $x \in \M$ be self-adjoint.  Let $(x_n)_{n=1}^\infty \in \prod_{n=1}^\infty \M_n$ be a self-adjoint element lifting $x$ and then use Lemma~\ref{lem:approx-spectral-proj} to produce a positive contraction $(q_n)_{n=1}^\infty \in \prod_{n=1}^\infty \M_n$ such that for all $n \geq 1$,
	\begin{equation}
		\|q_n - q_n^2\|_{2, X_n} < \frac1n \qquad \text{and} \qquad \|q_n x_n - (x_n)_+\|_{2, X_n} < \frac1n.
	\end{equation}
	Let $q \in \M^\omega$ denote the image of $(q_n)_{n=1}^\infty$.  Then we have
	\begin{equation}
		\|q - q^2\|_{2, X} = 0 \qquad \text{and} \qquad \|qx - x_+ \|_{2, X} = 0.
	\end{equation}
	Hence $q$ is a projection with $q x = x_+$, which also implies that $x$ and $q$ commute.
	
	Fix $\epsilon > 0$ and note that 
	\begin{equation}
		y \coloneqq x +\epsilon(2q - 1_\M) = (x_+ + \epsilon)q - (x_- + \epsilon)q^\perp \in \M
	\end{equation}
	is self-adjoint and invertible with $\|x - y\| \leq \epsilon$. Hence the invertible self-adjoint elements in $\M$ form a dense subset of the self-adjoint elements of $\M$.  Therefore, $\M$ has real rank zero.
\end{proof}

Next we consider strict comparison of projections in tracially complete $C^*$-algebras.  Just as Murray--von Neumann subequivalence of projections in finite von Neumann algebras is determined by traces,  we will show in Theorem~\ref{thm:comparison} that a factorial tracial complete $C^*$-algebra $(\M, X)$ with CPoU has comparison of projections with respect to $X$, i.e.\ that a projection $p\in\M$ is Murray--von Neumann subequivalent to a projection $q\in\M$ if and only if $\tau(p) \leq \tau(q)$ for all $\tau \in X$. Before doing this, we need the following approximate comparison result, which will lead to comparison of projections in reduced products.

\begin{lemma}\label{lem:approximate-comparison}
	Suppose $(\M, X)$ is a factorial tracially complete $C^*$-algebra with CPoU, $\epsilon > 0$, and $p, q \in \M$ are positive contractions such that
	\begin{equation}\label{eq:almost-projection}
		\tau(p - p^2) < \epsilon, \quad \tau(q - q^2) < \epsilon, \quad \text{and} \quad \tau(p) < \tau(q) + \epsilon
	\end{equation}
	for all $\tau \in X$.  Then there is a contraction $v \in \M$ such that
	\begin{equation}\label{eq:approx-subequivalence}
		\| v^*q v - p \|_{2, X} < \big(2 \sqrt{2} + \sqrt{5}\big)\sqrt{\epsilon}.
	\end{equation}
\end{lemma}

\begin{proof}
	Using that $X$ is compact, there is $\epsilon' \in (0, \epsilon)$ with 
	\begin{equation}\label{eq:almost-projection-prime}
		\tau(p - p^2) < \epsilon', \quad \tau(q - q^2) < \epsilon', \quad \text{and} \quad \tau(p) < \tau(q) + \epsilon'
	\end{equation}
	for all $\tau \in X$.	We first show that for each $\tau \in X$, there is a partial isometry $\bar v_\tau \in \pi_\tau(\M)''$ such that
	\begin{equation}\label{eq:vN-approx-subequiv}
		\| \bar v_\tau^* \pi_\tau(q) \bar v_\tau - \pi_\tau(p) \|_{2, \tau} < \delta \coloneqq \big(2 \sqrt{2} + \sqrt{5}\big)\sqrt{\epsilon'}.
	\end{equation}
	Fix $\tau \in X$ and let $p_\tau, q_\tau \in \pi_\tau(\M)''$ denote the spectral projections of $\pi_\tau(p)$ and $\pi_\tau(q)$ corresponding to the interval $[1/2, 1]$.  Then
	\begin{equation}\begin{split}\label{eq:close-projection}
|\pi_\tau(p)-p_\tau|^2\leq	|\pi_\tau(p) - p_\tau| &\leq 2(\pi_\tau(p) - \pi_\tau(p)^2) \qquad \text{and}  \\
		|\pi_\tau(q)-q_\tau|^2\leq|\pi_\tau(q) - q_\tau| &\leq 2(\pi_\tau(q) - \pi_\tau(q)^2),
	\end{split}\end{equation}
 so \eqref{eq:almost-projection-prime} gives
 \begin{equation}\label{eq:lem:approx-comparison-neweq1}
\|\pi_\tau(p)-p_\tau\|_{2,\tau} < \sqrt{2\epsilon'}\text{ and }\|\pi_\tau(q)-q_\tau\|_{2,\tau} < \sqrt{2\epsilon'}.
 \end{equation}
	By Lemma~\ref{lem:GNSTraceFace}, we have $\sigma \circ \pi_\tau \in X$ for all $\sigma \in T(\pi_\tau(\M)'')$, and so  \eqref{eq:almost-projection-prime} and \eqref{eq:close-projection} imply that 
	\begin{equation}\label{eq:almost-smaller-trace}
		\sigma(p_\tau) < \sigma(q_\tau) + 5\epsilon', \qquad \sigma \in T(\pi_\tau(\M)'').
	\end{equation}
	
	By the general comparison theorem for projections in von Neumann algebras (see for example \cite[Theorem V.1.8]{Tak79}), there exists a central projection $z_\tau \in \pi_\tau(\M)''$ with $z_\tau p_\tau \precsim z_\tau q_\tau$ and $z_\tau^\perp q_\tau \precsim z_\tau^\perp p_\tau$. Fix a partial isometry $\bar v_\tau \in \pi_\tau(\M)''$ such that
	\begin{align}
		\label{eqn:underZ} z_\tau \bar v_\tau^* \bar v_\tau &= z_\tau p_\tau \quad \text{and} \quad z_\tau \bar v_\tau \bar v_\tau^* \leq z_\tau q_\tau, \\
		\label{eqn:underZperp}	z_\tau^\perp \bar v_\tau^* \bar v_\tau &\leq z_\tau^\perp p_\tau \quad \text{and} \quad 	z_\tau^\perp \bar v_\tau \bar v_\tau^* = z_\tau^\perp q_\tau.
	\end{align}

	If we are in the special case where $z_\tau = 1_\M$, then \eqref{eqn:underZ} implies that $p_\tau = \bar v_\tau^* q_\tau \bar v_\tau$; hence the required inequality \eqref{eq:vN-approx-subequiv} is now a consequence of \eqref{eq:close-projection} and \eqref{eq:almost-projection-prime}.
	Suppose now that $z_\tau \neq 1_\M$. Then $\tau(z_\tau^\perp) \neq 0$, so we may define $\sigma \in T(\pi_\tau(\M)'')$ by $\sigma(a) \coloneqq \tau(z_\tau^\perp)^{-1} \tau(z_\tau^\perp a)$.  Then \eqref{eqn:underZ} implies that $z_\tau p_\tau = z_\tau \bar v_\tau^* q_\tau \bar v_\tau$, and \eqref{eqn:underZperp} implies that $z_\tau^\perp p_\tau - z_\tau^\perp \bar v_\tau^* \bar v_\tau$ is a projection. 
	Therefore, we have
	\begin{equation}\begin{split}\label{eqn:big-computation}
		\|p_\tau -  \bar v_\tau^* q_\tau \bar v_\tau \|_{2, \tau}^2
		&= \|z_\tau(p_\tau - \bar v_\tau^* q_\tau \bar v_\tau) + z_\tau^\perp(p_\tau - \bar v_\tau^* q_\tau \bar v_\tau)\|_{2, \tau}^2  \\
		&= \|z_\tau^\perp p_\tau - z_\tau^\perp \bar v_\tau^* q_\tau \bar v_\tau\|_{2, \tau}^2  \\
		&= \tau(z_\tau^\perp p_\tau - z_\tau^\perp \bar v_\tau^* \bar v_\tau)  \\
		&= \tau(z_\tau^\perp p_\tau) - \tau(z_\tau^\perp \bar q_\tau \bar v_\tau \bar v_\tau^*)  \\
		&= \tau(z_\tau^\perp) \big(\sigma(p_\tau) - \sigma(q_\tau)\big)  \\
		&< 5 \epsilon',
	\end{split}\end{equation}
	where the final estimate uses  \eqref{eq:almost-smaller-trace}. 
	Combining \eqref{eqn:big-computation} with \eqref{eq:lem:approx-comparison-neweq1} proves \eqref{eq:vN-approx-subequiv}.
	
	Now we use CPoU to obtain \eqref{eq:approx-subequivalence} from \eqref{eq:vN-approx-subequiv}.  By Kaplansky's density theorem, there is a contraction $v_\tau \in \M$ such that
	\begin{equation}
		\| v_\tau^*q v_\tau - p \|_{2, \tau} < \delta.
	\end{equation}
	Define $a_\tau \coloneqq |v_\tau^* q v_\tau - p|^2$ and use this to define the open neighbourhood
	\begin{equation}
		U_\tau \coloneqq \{ \sigma \in X : \sigma(a_\tau) < \delta^2 \}
	\end{equation}
	of $\tau$ in $X$.  As $X$ is compact, there are $\tau_1, \ldots, \tau_k \in X$ such that $(U_{\tau_i})_{i=1}^k$ covers $X$.  Therefore,
	\begin{equation}
		\sup_{\tau \in X} \min_{1 \leq i \leq k} \tau(a_{\tau_i}) < \delta^2.
	\end{equation}
	
	Set $S \coloneqq \{v_{\tau_i}: i=1,\ldots,k\} \cup \{p,q\}$. As $(\M, X)$ has CPoU, there are mutually orthogonal projections $e_1, \ldots, e_k \in \M^\omega \cap S'$ such that
	\begin{equation}
		\sum_{i=1}^k e_i = 1_{\M^\omega} \qquad \text{and} \qquad \tau(a_{\tau_i} e_i) \leq \delta^2\tau(e_i)
	\end{equation}
	for all $\tau \in X^\omega$ and $i= 1, \ldots, k$.  Define $v \coloneqq \sum_{i=1}^k e_i v_{\tau_i} \in \M^\omega$. 
	Since $e_1, \ldots, e_k$ are mutually orthogonal projections summing to the identity which commute with $\{v_{\tau_i}: i=1,\ldots,k\} \cup \{p,q\}$, we have
	\begin{equation}\begin{split}
		|v^*qv - p|^2 &= \sum_{i=1}^k |v_{\tau_i}^* q v_{\tau_i} - p|^2e_i = \sum_{i=1}^k a_{\tau_i}e_i. 
	\end{split}\end{equation}
	Hence, for each $\tau \in X^\omega$,
	\begin{equation}
		\|v^* q v - p \|_{2, \tau}^2
		= \tau(|v^*qv - p|^2)
		= \sum_{i=1}^k \tau(a_{\tau_i} e_i)
		\leq \delta^2 \sum_{i=1}^k \tau(e_i) = \delta^2.
	\end{equation}
	Finally, if $(v_n)_{n=1}^\infty \subseteq \M$ is a sequence representing $v$, then
	\begin{equation}
		\lim_{n \rightarrow \omega} \|v_n^* q v_n - p\|_{2, X} = \|v^*q v - p\|_{2, X^\omega} \leq \delta < (2 \sqrt{2} + \sqrt{5}) \sqrt{\epsilon},
	\end{equation}
	and the result follows.
\end{proof}

\begin{proposition}\label{prop:approximate-comparison}
	Suppose $\big((\M_n, X_n)\big)_{n=1}^\infty$ is a sequence of factorial tracially complete $C^*$-algebras with CPoU and $p, q \in \prod^\omega \M_n$ are projections.
	\begin{enumerate}
		\item\label{approx-comp-1} If $\tau(p) \leq \tau(q)$ for all $\tau \in \sum^\omega X_n$, then $p$ is Murray--von Neumann subequivalent to $q$.
		\item\label{approx-comp-2} If $\tau(p) = \tau(q)$ for all $\tau \in \sum^\omega X_n$, then $p$ and $q$ are unitarily equivalent.
	\end{enumerate}
\end{proposition}

\begin{proof}
	Lemma~\ref{lem:approximate-comparison} immediately implies \ref{approx-comp-1}.  For \ref{approx-comp-2}, use \ref{approx-comp-1} to obtain $v \in \prod^\omega \M_n$ with $v^*v = p$ and $vv^* \leq q$.  Since $q - vv^* \geq 0$ and
	\begin{equation}
		\tau(q - vv^*) = \tau(q) - \tau(p) = 0
	\end{equation}
	for all $\tau \in \sum^\omega X_n$, we have $vv^* = q$.  Repeating this argument with $p^\perp$ and $q^\perp$ in place of $p$ and $q$ produces a partial isometry $w \in \prod^\omega \M_n$ such that $w^*w = p^\perp$ and $ww^* = q^\perp$.  Then $u \coloneqq v + w$ is a unitary with $upu^* = q$.
\end{proof}

Combining Propositions~\ref{prop:real-rank-zero} and~\ref{prop:approximate-comparison} resolves the trace problem (Question~\ref{Q:traces}) for reduced products of tracially complete $C^*$-algebras with CPoU. This generalises the case of McDuff $W^*$-bundles from \cite[Proposition 3.32]{BBSTWW} (taking $A\coloneqq\mathbb C$ in that proposition).\footnote{This will be extended by the third-named author in the forthcoming work \cite{Ev24}, to resolve the trace problem for factorial tracially complete $C^*$-algebras with CPoU.}

\begin{theorem}\label{thm:no-silly-traces}
	If $\big((\M_n, X_n)\big)_{n=1}^\infty$ is a sequence of factorial tracially complete $C^*$-algebras with CPoU, then $T\big(\prod^\omega \M_n) = \sum^\omega X_n$.
\end{theorem}

\begin{proof}
	Let $(\M, X) \coloneqq \big(\prod^\omega \M_n, \sum^\omega X_n\big)$.  Then $\M$ has real rank zero by Proposition~\ref{prop:real-rank-zero} and comparison of projections with respect to $X$ by Proposition~\ref{prop:approximate-comparison}.  Further, for all $d \geq 1$, $\M \otimes M_d$ has comparison of projections with respect to $X \otimes \{\mathrm{tr_d}\}$ since Proposition~\ref{prop:matrix-ultrapower} identifies $(\M \otimes M_d, X \otimes \{\mathrm{tr}_d\})$ with $\prod^\omega (\M_n \otimes M_d, X_n \otimes \{\mathrm{tr}_d\})$ and each of the algebras $(\mathcal M_n \otimes M_d, X_n \otimes \{\mathrm{tr}_d\})$  are factorial and have CPoU by Proposition~\ref{prop:matrix-CPoU}.
	
	Let $V(\M)$ denote the Murray--von Neumann semigroup of $\M$ and let $[p] \in V(\M)$ denote the class of a projection $p \in \M \otimes M_d$ for an integer $d \geq 1$.  For $\tau \in T(\M)$, let $\hat{\tau}$ denote the induced state on $V(\M)$ given by
	\begin{equation}
		\hat{\tau}([p]) \coloneqq (\tau \otimes \mathrm{Tr}_d)(p)
	\end{equation}
	for a projection $p$ in $\M \otimes M_d$, where $\mathrm{Tr}_d$ is the unnormalised trace on $M_d$ so that $\mathrm{Tr}_d(1_{M_d}) = d$. By comparison of projections in matrices over $\M$, the natural map
	\begin{equation}
		V(\M) \rightarrow \mathrm{Aff}(X) \colon x \mapsto (\tau \mapsto \hat{\tau}(x))
	\end{equation}
	is an order embedding.
	
	Now suppose $\sigma \in T(\M)$.  By \cite[Corollary~2.7]{Blackadar-Rordam92}, the state $\hat{\sigma}$ on $V(\M)$ extends to a state $\phi$ on $\mathrm{Aff}(X)$.  Since all states on $\mathrm{Aff}(X)$ are point evaluations, (see Proposition \ref{prop:pointwise-to-uniform}\ref{prop:pointwise-to-uniform1}), there is a trace $\tau \in X$ such that $\phi(f) = f(\tau)$ for all $f \in \mathrm{Aff}(X)$.  Then, by construction, $\hat{\sigma} = \hat{\tau}$, and hence $\sigma(p) = \tau(p)$ for all projections $p \in \M$.  As $\M$ has real rank zero by Proposition~\ref{prop:real-rank-zero}, $\M$ is the $\|\cdot\|$-closed span of its projections, and hence $\sigma = \tau \in X$.
\end{proof}

\begin{remark}\label{rmk:commutators}
	Versions of Theorem~\ref{thm:no-silly-traces} have appeared previously in \cite[Proposition~3.22]{BBSTWW} and \cite[Proposition~4.6]{CETWW} in the settings of ultrapowers of factorial McDuff $W^*$-bundles and uniform tracial ultrapowers of $C^*$-algebras with CPoU, respectively.  The proofs in these references are based on a result of Fack and de la Harpe (\cite[Th{\'e}or{\`e}me~2.3]{FH80}). This style of argument could also be adapted to the setting of tracially complete $C^*$-algebras, providing an alternative proof of Theorem~\ref{thm:no-silly-traces}.
\end{remark}

Theorem~\ref{thm:no-silly-traces} also proves that factoriality passes to reduced products in the presence of CPoU.  This allows us to prove both property $\Gamma$ and CPoU pass to reduced products as promised in Remarks~\ref{rem:GammaUltrapower} and~\ref{rem:reduced-products-CPoU}.

\begin{corollary}\label{cor:factorial-ultrapower}
	Suppose $\big((\M_n, X_n)\big)_{n=1}^\infty$ is a sequence of factorial tracially complete $C^*$-algebras with reduced product $(\M, X)$.
	\begin{enumerate}
		\item\label{cor:factorial-ultrapower:CPoU} If each $(\M_n, X_n)$ has CPoU, then $(\M, X)$ is factorial and has CPoU.
		\item\label{cor:factorial-ultrapower:Gamma} If each $(\M_n, X_n)$ satisfies property $\Gamma$, then $(\M, X)$ is factorial and satisfies property $\Gamma$.
	\end{enumerate}
\end{corollary}

\begin{proof}
	To see \ref{cor:factorial-ultrapower:CPoU}, note that $(\M, X)=(\M,T(\M))$ by Theorem~\ref{thm:no-silly-traces}, and so is factorial.  The approximation property in Proposition~\ref{CPoU:Ultra:FiniteSet}\ref{item:CPOU-local} characterising CPoU passes to reduced products, so $(\M, X)$ satisfies CPoU.  For \ref{cor:factorial-ultrapower:Gamma}, Theorem~\ref{introthmgammaimpliescpou} implies each $(\M_n, X_n)$ satisfies CPoU, so $(\M, X)$ is factorial by \ref{cor:factorial-ultrapower:CPoU}.  The approximation property in Proposition~\ref{prop:Gamma}\ref{item:Gamma-approx} characterising $\Gamma$ passes to reduced products, and hence $(\M, X)$ satisfies property $\Gamma$.
\end{proof}

Corollary \ref{cor:factorial-ultrapower} and Kirchberg's $\epsilon$-test (Lemma~\ref{lem:EpsTest}) allow for the following variation of CPoU for reduced products.  The point is that the projections $p_i$ in the definition of CPoU can be chosen in the reduced product instead of the reduced power of the reduced product.

\begin{corollary}[cf.\ Proposition~\ref{CPoU:Ultra:FiniteSet}\ref{item:CPOU++}]
\label{cor:CPoU-in-ultrapower}
	Suppose $\big((\M_n, X_n)\big)_{n=1}^\infty$ is a sequence of factorial tracially complete $C^*$-algebras with CPoU.  If $a_1, \ldots, a_k \in (\prod^\omega \M_n)_+$, $S \subseteq \prod^\omega \M_n$ is $\|\cdot\|_{2, X^\omega}$-separable, and
	\begin{equation}
		\sup_{\tau \in X} \min_{1 \leq i \leq k} \tau(a_i) < \delta,
	\end{equation}
	then there are projections $p_1, \ldots, p_k \in \prod^\omega \M_n \cap S'$ such that
	\begin{equation}
		\sum_{i=1}^k p_i = 1_{\M^\omega} \qquad \text{and} \qquad \tau(a_i p_i) \leq \delta \tau(p_i)
	\end{equation}
	for all $1 \leq i \leq k$ and $\tau \in X^\omega$.
\end{corollary}

\subsection{Structure of unitaries}\label{sec:app-CPoU-unitaries}

With Corollary~\ref{cor:factorial-ultrapower} in hand, we derive some properties of unitaries in tracially complete $C^*$-algebras with CPoU.  The following result and its proof are analogous to \cite[Proposition~2.1]{CETW-classification}.

\begin{proposition}\label{prop:unitary-exponentials}
	Suppose $\big((\M_n, X_n)\big)_{n=1}^\infty$ is a sequence of factorial tracially complete $C^*$-algebras with CPoU and $S \subseteq \prod^\omega \M_n$ is a $\|\cdot\|_{2, X^\omega}$-separable subset.  If $u \in \prod^\omega \M_n \cap S'$ is a unitary, then there is a self-adjoint $h \in \prod^\omega \M_n \cap S'$ such that $u = e^{ih}$ and $\|h\| \leq \pi$.
\end{proposition}

\begin{proof}
	Since this theorem involves the numbers $\pi$ and $i=\sqrt{-1}$, we shall set $\sigma_\tau$ to denote the GNS representation corresponding to a trace $\tau$ and use the letter $j$ for our summation index.
	For the sake of brevity, we write $(\M, X) \coloneqq \big(\prod^\omega \M_n, \sum^\omega X_n\big)$.

	Fix $\epsilon > 0$ and a finite set $\mathcal F \subseteq \M$. By Kirchberg's $\epsilon$-test (Lemma~\ref{lem:EpsTest}), it suffices to show that  there is a self-adjoint $h \in \M$ with $\|h\| \leq \pi$ such that
	\begin{equation}\label{eq:approx-polar-decomp}
		\|u - e^{ih}\|_{2, X} \leq \epsilon \qquad \text{and} \qquad \max_{x \in \mathcal F} \|[h, x]\|_{2, X} \leq \epsilon.
	\end{equation}
	For each $\tau \in X$, there is a self-adjoint $\bar h_\tau \in \sigma_\tau(\M)'' \cap \sigma_\tau(\mathcal F)'$ with $\|\bar h_\tau\| \leq \pi$ such that $\sigma_\tau(u) = e^{i\bar h_\tau}$.  By Kaplansky's density theorem, there is a self-adjoint $h_\tau \in \M$ with $\|h_\tau\| \leq \pi$ such that
	\begin{equation}
    \begin{split}
		\|u - e^{ih_\tau}\|_{2, \tau} &< (|\mathcal F| + 1)^{-1/2} \epsilon \qquad \text{and}\\ 
		\max_{x \in \mathcal F} \|[h_\tau, x]\|_{2, \tau} &< (|\mathcal F| + 1)^{-1/2} \epsilon.
	\end{split}
   \end{equation}
Then define
	\begin{equation}
		a_\tau \coloneqq |u - e^{ih_\tau}|^2 + \sum_{x \in \mathcal F} |[h_\tau, x]|^2 \in \M_+,
	\end{equation}
so that $\tau(a_\tau) < \epsilon^2$.  By the compactness of $X$, there are $\tau_1, \ldots, \tau_k \in X$ such that
	\begin{equation}
		\sup_{\tau \in X} \min_{1 \leq j \leq k} \tau(a_{\tau_j}) < \epsilon^2.
	\end{equation}
	By CPoU (in the form of Corollary~\ref{cor:CPoU-in-ultrapower}), there are  mutually orthogonal projections 
	\begin{equation}
		p_1, \ldots, p_k \in \M \cap \mathcal F' \cap \{u\}' \cap \{h_{\tau_1}, \ldots, h_{\tau_k}\}'
	\end{equation}
	such that
	\begin{equation}
		\sum_{j=1}^k p_j = 1_{\M} \qquad \text{and} \qquad \tau(a_{\tau_j} p_j)  \leq \epsilon^2 \tau(p_j)
	\end{equation}
	for all $1 \leq j \leq k$ and $\tau \in X$.
	
	Define $h \coloneqq \sum_{j=1}^k h_{\tau_j} p_j \in \M$ and note that $\|h\| \leq \pi$ as $\|h_{\tau_j}\| \leq \pi$ for all $1 \leq j \leq k$ and the $p_j$ are mutually orthogonal projections commuting with the $h_{\tau_j}$.  Also, for each $\tau \in X$, we have
	\begin{equation}
		\|u - e^{ih}\|_{2, \tau}^2 \leq \sum_{j=1}^k \tau(a_{\tau_j} p_j) \leq \epsilon^2 \sum_{j=1}^k \tau(p_j) = \epsilon^2,
	\end{equation}
	and for each $\tau \in X$ and $x \in \mathcal F$,
	\begin{equation}
		\|[h, x]\|_{2, \tau}^2 \leq \sum_{j=1}^k \tau(a_{\tau_j} p_j) \leq \epsilon^2 \sum_{j=1}^k \tau(p_j) = \epsilon^2.
	\end{equation}
	This verifies \eqref{eq:approx-polar-decomp}.
\end{proof}

Applying Proposition~\ref{prop:unitary-exponentials} to matrix algebras (with $S = \emptyset$) yields the following $K$-theoretic computation.

\begin{corollary}
	If $\big((\M_n, X_n)\big)_{n=1}^\infty$ is a sequence of factorial tracially complete $C^*$-algebras with CPoU, then $K_1\big(\prod^\omega \M_n\big) = 0$.
\end{corollary}

\begin{proof}
	Proposition~\ref{prop:unitary-exponentials} implies that the unitary group of $\prod^\omega \M_n$ is path-connected, and Propositions~\ref{prop:matrix-ultrapower} and~\ref{prop:matrix-CPoU} imply the same result for matrices over $\prod^\omega \M_n$.
\end{proof}

The following result facilitates the use of Theorem~\ref{thm:reparameterisation} in the presence of CPoU. 

\begin{corollary}\label{cor:unitary-stable-relation}
	If $\big((\M_n, X_n)\big)_{n=1}^\infty$ is a sequence of factorial tracially complete $C^*$-algebras with CPoU and $u \in \prod^\omega \M_n$ is a unitary, then there is a sequence of unitaries $(u_n)_{n=1}^\infty \in \prod_{n=1}^\infty \M_n$ lifting $u$.
\end{corollary}

Using a sequence of counterexamples argument (see \cite{Loring97}, for example), Corollary~\ref{cor:unitary-stable-relation} also provides a uniform 2-norm stability result for unitaries.

\begin{corollary}
	For all $\epsilon, c > 0$, there is $\delta > 0$ such that if $(\M, X)$ is a factorial tracially complete $C^*$-algebra with CPoU and $v \in \M$ with $\|v\| \leq c$ and $\|v^* v - 1_\mathcal M\|_{2, X} < \delta$, there is a unitary $u \in \M$ such that $\|u - v\|_{2, X} < \epsilon$.
\end{corollary}

\begin{proof}
	Suppose the result is false and that $\epsilon, c > 0$ provide a counterexample.  For each $n \in \mathbb N$, let $(\M_n, X_n)$ be a tracially complete $C^*$-algebra with CPoU and let $v_n \in \M_n$ be such that $\|v_n\| \leq c$, $\|v_n^* v_n - 1\|_{2, X_n} < \frac1n$, but for all unitaries $u \in \M_n$, we have
	\begin{equation}\label{eq:far-from-unitary}
		\|u - v_n\|_{2, X_n} \geq \epsilon.
	\end{equation}
	
	The sequence $(v_n)_{n=1}^\infty$ induces an element $v \in \prod^\omega \M_n$.  We have $\|v^*v - 1\|_{2, \sum^\omega X_n} = 0$, and hence $v$ is an isometry.  Since $1 -vv^* \geq 0$ and $\tau(1 - vv^*) = 0$ for all $\tau \in X$, we further have that $v$ is unitary.  By Corollary~\ref{cor:unitary-stable-relation}, there is a sequence of unitaries $(u_n)_{n=1}^\infty \in \prod_{n=1}^\infty \M_n$ lifting $v$.  Now
	\begin{equation}
		\lim_{n \rightarrow \omega} \|u_n - v_n \|_{2, X_n} = 0,
	\end{equation}
	which contradicts \eqref{eq:far-from-unitary}.
\end{proof}

Proposition~\ref{prop:unitary-exponentials} also implies that if $(\M, X)$ is a factorial tracially complete $C^*$-algebra with CPoU, then there is a $\|\cdot\|_{2, X}$-dense set of unitaries in $\M$ that have the form $e^{ih}$ for a self-adjoint $h \in \M$ with $\|h\| \leq \pi$.  We do not know if every unitary in $\M$ has this form.  Without CPoU, there are certainly commutative counterexamples such as $\big(C(\mathbb T), T(C(\mathbb T))\big)$, but we do not know of any counterexamples among type II$_1$ tracially complete $C^*$-algebras $(\M, X)$ -- with or without CPoU -- or even among tracially complete $C^*$-algebras $(\M,X)$ for which $\pi_\tau(\M)''$ is diffuse for each $\tau\in X$.

\begin{question}\label{question:unitary-exponentials}
	If $(\M, X)$ is a factorial tracially complete $C^*$-algebra with CPoU, is every unitary in $\M$ an exponential?  Slightly less ambitiously, is the unitary group of $\M$ path connected in the operator norm topology?
\end{question}

In every finite von Neumann algebra, every element has a unitary polar decomposition.  The same holds in reduced products of factorial tracially complete $C^*$-algebras with CPoU.  In fact, this can be done in an approximately central way. 

\begin{proposition}\label{prop:polar-decomposition}
	Suppose $\big((\M_n, X_n)\big)_{n=1}^\infty$ is a sequence of factorial tracially complete $C^*$-algebras with CPoU and $S \subseteq \prod^\omega \M_n$ is a $\|\cdot\|_{2, X^\omega}$-separable subset.  If $a \in \prod^\omega \M_n \cap S'$, then there is a unitary $u \in \prod^\omega\M_n\cap S'$ such that $a = u|a|$.  Further, if $a = a^*$, we may arrange that $u = u^*$.
\end{proposition}

\begin{proof}
	For the sake of brevity, define $(\M, X) \coloneqq \big(\prod^\omega \M_n, \sum^\omega X_n\big)$.
	We may assume that $S$ is a set of contractions.  Fix $\epsilon > 0$ and a finite set $\mathcal F \subseteq S$.  For the first part, by Kirchberg's $\epsilon$-test, it suffices to show that if $a \in \M$, then there is a unitary  $u \in \M$ such that
	\begin{equation}\label{eqn:polar-docomp-target}
		\|u|a|-a\|_{2, X} \leq \epsilon\qquad \text{and}\qquad \max_{b \in \mathcal F}\|[u, b]\|_{2, X} \leq \epsilon.
	\end{equation}
	
	By the existence of unitary polar decompositions in finite von Neumann algebras, for each $\tau \in X$, there is a unitary $\bar u_\tau \in \pi_\tau(\M)'' \cap \pi_\tau(\mathcal F)'$ such that
	\begin{equation}\label{eq:polar-decomp-1}
		\pi_\tau(a) = \bar u_\tau\pi_\tau(|a|).
	\end{equation}
	Since $\pi_\tau(\M)'' \cap \pi_\tau(\mathcal F)'$ is a von Neumann algebra, $\bar{u}_\tau = e^{i\bar{h}_\tau}$ for some self-adjoint $\bar{h}_\tau \in \pi_\tau(\M)'' \cap \pi_\tau(\mathcal F)'$. 
	
	By Kaplansky's density theorem, we can approximate $\bar{h}_\tau$ by a self-adjoint in $\pi_\tau(\M)$ and lift to a self-adjoint element $h_\tau \in \M$ such that unitary $u_\tau \coloneqq e^{ih_\tau} \in \M$ satisfies
	\begin{equation}\begin{split}
		\| u_\tau|a|-a  \|_{2, \tau} &< (1+|\mathcal F|)^{-1/2} \epsilon, \\
		\max_{b \in \mathcal F} \|[u_\tau, b]\|_{2, \tau} &<  (1+|\mathcal F|)^{-1/2} \epsilon. \label{eq:polar-decomp-2}
	\end{split}\end{equation}
	
	For $\tau \in X$, define
	\begin{equation}
		c_\tau \coloneqq \big|u_\tau|a|-a\big|^2 + \sum_{b \in \mathcal F} \big|[u_\tau, b]\big|^2 \in \M_+,
	\end{equation}
	and note that $\tau(c_\tau) < \epsilon^2$ for all $\tau \in X$.  By the compactness of $X$, there are $\tau_1, \ldots, \tau_k \in X$ such that
	\begin{equation}\label{eq:polar-decomp-7}
		\sup_{\tau \in X} \min_{1 \leq j \leq k} \tau(c_{\tau_j}) < \epsilon^2.
	\end{equation}
	
	By CPoU (in the form of Corollary~\ref{cor:CPoU-in-ultrapower}), there are mutually orthogonal projections 
	\begin{equation}
		p_1, \ldots, p_k \in \M \cap \mathcal F' \cap \{a\}' \cap \{ u_{\tau_1}, \ldots, u_{\tau_k}\}'
	\end{equation}
	such that
	\begin{equation}\label{eq:polar-decomp-8}
		\sum_{j=1}^k p_j = 1_{\M} \qquad \text{and} \qquad \tau(c_{\tau_j} p_j) \leq \epsilon^2 \tau(p_j),
	\end{equation}
	for all $1 \leq j \leq k$ and $\tau \in X$.  Define $u \coloneqq \sum_{j=1}^n u_{\tau_j} p_j$.  As the $u_{\tau_j}$ are unitaries commuting with the $p_j$, we have that $u$ is a unitary.
	We compute that for all $\tau \in X$ and $b \in \mathcal F$,
	\begin{align}
		\|u|a|-a\|_{2,\tau}^2 &\leq \sum_{j=1}^k \tau(c_{\tau_j} p_j) \leq \epsilon^2 \sum_{j=1}^k \tau(p_j) = \epsilon^2
		\intertext{and}
		\|[u, b]\|_{2, \tau}^2 &\leq \sum_{j=1}^k \tau(c_{\tau_j} p_j) \leq \epsilon^2 \sum_{j=1}^k \tau(p_j) = \epsilon^2,
	\end{align} 
	which implies \eqref{eqn:polar-docomp-target}, as required.
	
	In the case that $a^* = a$, the unitaries $\bar{u}_\tau$ can be taken to be self-adjoint. For this, by ensuring the unitaries $u_\tau$ are chosen with $\|u_\tau - u_\tau^*\|_{2, \tau}$ sufficiently small, adjusting the bounds in \eqref{eq:polar-decomp-2}, and replacing $c_\tau$ with $c_\tau + |u_\tau - u_\tau^*|^2$ in the above proof, one may also arrange $\|u - u^*\|_{2, X^\omega} \leq \epsilon$ in \eqref{eqn:polar-docomp-target}.
\end{proof}

As promised, we have the following strengthened version of Proposition~\ref{prop:real-rank-zero}.

\begin{corollary}\label{cor:stable-rank-one}
	If $\big((\M_n, X_n)\big)_{n=1}^\infty$ is a sequence of factorial tracially complete $C^*$-algebras with CPoU and $S \subseteq \prod^\omega\M_n$ is a $\|\cdot\|_{2, X^\omega}$-separable subset, then $\prod^\omega\M_n \cap S'$ has real rank zero and stable rank one.
\end{corollary}

\begin{proof}
	Suppose $a \in \prod^\omega\M_n \cap S'$ and write $a = u|a|$ for a unitary $u \in \prod^\omega\M_n \cap S'$ as in Proposition~\ref{prop:polar-decomposition}.  Then for each $\epsilon > 0$, the element $b \coloneqq u(|a|+\epsilon)$ is invertible and $\|a - b\| \leq \epsilon$.  Therefore, $\prod^\omega\M_n \cap S'$ has stable rank one.  If $a$ is self-adjoint and we take $u = u^*$, then $u$ and $|a|$ commute.  It follows that $b$ is self-adjoint, and this shows $\prod^\omega\M_n \cap S'$ has real rank zero.
\end{proof}

In the same spirit as the comments after Corollary~\ref{cor:unitary-stable-relation}, an approximate version of stable rank one can be obtained for all factorial tracially complete $C^*$-algebras $(\M, X)$ with CPoU.  Indeed, combining Proposition~\ref{prop:polar-decomposition} (with $S = \emptyset$) and Corollary~\ref{cor:unitary-stable-relation} shows that a $\|\cdot\|_{2, X}$-dense set of elements $a \in \M$ have the form $a = u|a|$ for some unitary $u \in \M$.  Then the proof of Corollary~\ref{cor:stable-rank-one} shows that a $\|\cdot\|_{2, X}$-dense set of elements in $\M$ are invertible.  The analogous statement for self-adjoint elements is less clear -- this would require a version of Corollary~\ref{cor:unitary-stable-relation} for self-adjoint unitaries (or, equivalently, for projections).

\subsection{Classification of projections}\label{sec:app-CPoU-projections}

Question~\ref{question:unitary-exponentials} exposes a common drawback of the CPoU technique -- even when exact results or $\|\cdot\|$-approximate results are possible in finite von Neumann algebras, a direct application of CPoU will always introduce a uniform 2-norm error.  In this section, we show that in the case of classification of projections, we can overcome this defect by controlling the 2-norm distance from the unitary implementing the equivalence to the unit. This will prove Theorems~\ref{introthm-classprojections2} and \ref{IntroThmCtsProj}.

The following lemma and its proof are analogous to \cite[Lemma~XIV.2.1]{Tak03}, where the result is shown for finite von Neumann algebras.  It is possible to prove this result by combining the von Neumann algebra result with a CPoU argument similar to those above, but with the structural results of reduced products already obtained, a more direct proof is possible.

\begin{lemma}\label{lem:equivalence-close-to-unit}
	Suppose $(\M, X)$ is a factorial tracially complete $C^*$-algebra with CPoU.  If $p, q \in \M$ are projections with $\tau(p) = \tau(q)$ for all $\tau \in X$ and $\epsilon > 0$, then there is a unitary $u \in \M$ such that
	\begin{equation}
		\| upu^* - q \|_{2, X} < \epsilon \qquad \text{and} \qquad \| u - 1_\M \|_{2, X} < 2 \sqrt{2} \|p - q\|_{2, X} + \epsilon.
	\end{equation}
\end{lemma}

\begin{proof}
	By the liftability of unitaries in reduced products (Corollary~\ref{cor:unitary-stable-relation}), it suffices to show there is a unitary $u \in \M^\omega$ such that
	\begin{equation}
		upu^* = q \qquad \text{and} \qquad \| u - 1_{\M^\omega}\|_{2, X^\omega} \leq 2 \sqrt{2} \|p - q\|_{2, X}.
	\end{equation}
	By Proposition~\ref{prop:approximate-comparison}, there is a unitary $v \in \M^\omega$ such that $vpv^* = q$.  Set $a \coloneqq pvp + p^\perp v p^\perp$. As $a \in \M^\omega \cap \{p\}'$, Proposition~\ref{prop:polar-decomposition} implies there is a unitary $w \in \M^\omega \cap \{p\}'$ such that $a = w|a|$.  Then $u \coloneqq vw^*$ is a unitary with $upu^* = q$.
	
	Note that $(pvp - pv)^*(p^\perp v p^\perp - p^\perp v) = 0$.  Therefore, by the Pythagorean identity,\footnote{This says $\|a+b\|_{2,\tau}^2=\|a\|_{2,\tau}^2+\|b\|_{2,\tau}^2$ provided $a^*b=0$, and implies a corresponding inequality with a uniform trace norm $\|\cdot\|_{2,X}$.}
	\begin{equation}\begin{split}\label{eq:equiv-estimate-1}
		\| a - v\|_{2, X^\omega}^2 
        &=\|pvp-pv+p^\perp v p^\perp - p^\perp v\|_{2,X^\omega}^2 \\
		&\leq \|pvp - pv\|_{2, X^\omega}^2 + \|p^\perp v p^\perp - p^\perp v\|_{2, X^\omega}^2  \\
		&= \|p(q - p)v\|_{2, X^\omega}^2 + \|p^\perp(q^\perp - p^\perp)v\|_{2, X^\omega}^2  \\
		&\leq 2 \|p - q\|_{2, X}^2.
	\end{split}\end{equation}
	Further, since $\|a\| \leq 1$, we have $(1_{\M^\omega} - |a|)^2 \leq (1_{\M^\omega} - |a|^2)^2$ in $\M^\omega$.  Therefore,
	\begin{equation}\begin{split}\label{eq:equiv-estimate-2}
		\|w - a\|_{2, X^\omega}^2
		&= \|1 - |a|\|_{2, X^\omega}^2  \\
		&\leq \| 1 - |a|^2 \|_{2, X^\omega}^2  \\
		&= \|p - pv^*pvp\|_{2, X^\omega}^2 + \|p^\perp - p^\perp v^* p^\perp v p^\perp\|_{2, X^\omega}^2  \\
		&= \|pv^*(q - p)vp\|_{2, X^\omega}^2 + \|p^\perp v^*(q^\perp - p^\perp)v p\|_{2, X^\omega}^2  \\
		&\leq 2 \|p - q\|_{2, X}^2,
	\end{split}\end{equation}
    using the Pythagorean identity again on the third line.
	Combining \eqref{eq:equiv-estimate-1} and \eqref{eq:equiv-estimate-2} shows 
	\begin{equation}
		\|u - 1_{\M^\omega}\|_{2, X^\omega} = \|v - w\|_{2, X^\omega} \leq 2 \sqrt{2} \|p-q\|_{2, X}. \qedhere
	\end{equation}
\end{proof}

The second part of the following result is Theorem~\ref{introthm-classprojections2}\ref{intro-class-proj-exist} from the introduction.  Note also that taking $(\mathcal M, X)$ to be a trivial $W^*$-bundle over a compact Hausdorff space $K$ with fibre being a II$_1$ factor with property $\Gamma$ in the following result proves Theorem~\ref{IntroThmCtsProj} from the overview, using Theorem~\ref{introthmgammaimpliescpou} to verify the CPoU hypothesis.

\begin{theorem}\label{thm:comparison}
	Suppose $(\M, X)$ is a factorial tracially complete $C^*$-algebra with CPoU and $p, q \in \M$ are projections.
	\begin{enumerate}
		\item\label{comparison1} If $\tau(p) \leq \tau(q)$ for all $\tau \in X$, then $p$ is Murray--von Neumann subequivalent to $q$.
		\item\label{comparison2} If $\tau(p) = \tau(q)$ for all $\tau \in X$, then $p$ and $q$ are unitarily equivalent.
	\end{enumerate}
\end{theorem}

\begin{proof}
	First we show \ref{comparison2}.  Let $(\epsilon_n)_{n=1}^\infty \subseteq (0, \infty)$ be a decreasing sequence with $\sum_{n=1}^\infty \epsilon_n < \infty$.  We will construct a sequence of unitaries $(u_n)_{n=1}^\infty \subseteq \M$ such that
	\begin{equation}\begin{aligned}
		\|u_n p u_n^* - q\|_{2, X} &< \epsilon_n, &\hspace{2ex}&n \geq 1, \text{ and} \\
		\|u_{n+1} - u_n\|_{2, X} &< 4 \epsilon_n &\hspace{2ex}&n \geq 1.
	\end{aligned}\end{equation}
	Then $(u_n)_{n=1}^\infty \subseteq \M$ is a $\|\cdot\|$-bounded, $\|\cdot\|_{2, X}$-Cauchy sequence, and hence converges to some $u \in \M$. Since multiplication and taking adjoints are $\|\cdot\|_{2,X}$-continuous on $\|\cdot\|$-bounded sets, it follows that $u$ is a unitary and $upu^* = q$.
	
	The sequence $(u_n)_{n=1}^\infty$ will be constructed inductively.  Lemma~\ref{lem:equivalence-close-to-unit} implies there is a unitary $u_1 \in \M$ such that
	\begin{equation}
		\|u_1 p u_1^* - q\|_{2, X} < \epsilon_1.
	\end{equation}
	Assuming $u_n$ has been constructed, applying Lemma~\ref{lem:equivalence-close-to-unit} to the projections $u_npu_n^*$ and $q$, there is a unitary $v_n \in \M$ such that
	\begin{equation}\begin{split}
		\|v_nu_n p u_n^*v_n^* - q\|_{2, X} &< \epsilon_{n+1}, \\ 
		\|v_n - 1_{\M} \|_{2, X} &< 2\sqrt{2}\epsilon_n + \epsilon_{n+1} \leq 4\epsilon_n.
	\end{split}\end{equation}
	Now set $u_{n+1} \coloneqq v_nu_n$. Then $\|u_{n+1} p u_{n+1} - q\|_{2, X} < \epsilon_{n+1}$ and
	\begin{equation}
		\|u_{n+1} - u_n \|_{2, X} \leq \|u_n\|\|v_n - 1\|_{2, X}
		< 4\epsilon_n,
	\end{equation}
	as required.
	
	Now we show \ref{comparison1}.  By Proposition~\ref{prop:approximate-comparison}, there is a partial isometry $v \in \M^\infty$ such that $v^* v = p$ and $vv^* \leq q$.  Suppose $r \colon \mathbb N \rightarrow \mathbb N$ is a function with $\lim_{n \rightarrow \infty} r(n) = \infty$ and consider the reparameterisation morphism $r^* \colon \M^\infty \rightarrow \M^\infty$ as defined just before Theorem~\ref{thm:reparameterisation}.  Since $p, q \in \M$, we have $r^*(p) = p$ and $r^*(q) = q$.  For all $\tau \in X^\infty$,
	\begin{equation}
		\tau(r^*(q - vv^*)) = \tau(r^*(q)) - \tau(r^*(p)) = \tau(q) - \tau(p) = \tau(q - vv^*).
	\end{equation}
	By Proposition~\ref{prop:approximate-comparison}, $q - vv^*$ and $r^*(q - vv^*)$ are unitarily equivalent in $\M^\infty$.  By Corollary~\ref{cor:unitary-stable-relation}, unitaries in $\M^\infty$ lift to unitaries in $\ell^\infty(\M)$, so we may apply intertwining through reparameterisation (Theorem~\ref{thm:reparameterisation}) in the case where $S$ is a one point space and obtain a projection $p' \in \M$ that is unitarily equivalent to $q-vv^*$ in $\M^\infty$. The projections $p \oplus p'$ and $q \oplus 0$ in $M_2(\M)$ agree on traces. By Proposition~\ref{prop:matrix-CPoU}, $M_2(\M)$ has CPoU. Hence, $p \oplus p'$ and $q \oplus 0$ are unitarily equivalent by \ref{comparison2}.  Therefore, $p \precsim q$ in $M_2(\M)$ and so also in $\M$.
\end{proof}

The uniqueness theorem for projections is complemented by a corresponding existence theorem building projections with prescribed behaviour on traces. This is obtained by applying CPoU to Proposition~\ref{prop:TracesFiniteVNAS}\ref{prop:TracesFiniteVNAS.4}.  Note that the following result is precisely Theorem~\ref{introthm-classprojections2}\ref{intro-class-proj-exist} from the introduction.

\begin{theorem}\label{thm:existence-projections}
	Suppose $(\M, X)$ is a type {\rm II}$_1$ factorial tracially complete $C^*$-algebra with CPoU.  If $f \in \mathrm{Aff}(X)$ with $0 \leq f \leq 1$, then there is a projection $p \in \mathcal M$ such that $\tau(p) = f(\tau)$ for all $\tau \in X$.
\end{theorem}

\begin{proof}
	It suffices to construct a projection $p \in \mathcal \M^\infty$ such that \begin{equation}\label{eq:correct-trace}
		\tau(p) = f(\tau|_\M), \qquad \tau \in X^\infty.
	\end{equation}
	Indeed, assuming this has been done, for every function $r \colon \mathbb N \rightarrow \mathbb N$ with $\lim_{n \rightarrow \infty} r(n) = \infty$, we have
	\begin{equation}
		\tau(r^*(p)) = f((\tau \circ r^*)|_\M) = f(\tau|_{\M}) = \tau(p)
	\end{equation}
	for all $\tau \in X^\infty$, and hence by Lemma~\ref{lem:approximate-comparison}, $r^*(p)$ and $p$ are unitarily equivalent.  Since unitaries in $\M^\infty$ lift to unitaries in $\ell^\infty(\M)$ (Corollary~\ref{cor:unitary-stable-relation}), intertwining through reparameterisation (Theorem~\ref{thm:reparameterisation}) implies that there is a projection $p_0 \in \mathcal M$ that is unitarily equivalent to $p$.  Then $\tau(p_0) = f(\tau)$ for all $\tau \in X$.
	
	We work to build a projection $p \in \mathcal M^\infty$ satisfying \eqref{eq:correct-trace}.  First, we extend $f$ to all of $T(\M)$ using Theorem~\ref{thm:extending-affine}. Then, by Proposition~\ref{prop:CP}, there is a self-adjoint $c \in \M$ such that $\tau(c) = f(\tau)$ for all $\tau \in T(\M)$. 
	
	Fix $\tau \in X$ and $\epsilon>0$. Since $\pi_\tau(\M)''$ is type II$_1$, Proposition~\ref{prop:TracesFiniteVNAS}\ref{prop:TracesFiniteVNAS.4} implies there is a projection $\bar{p}_\tau \in \pi_\tau(\M)''$ such that $\sigma(\bar{p}_\tau) = \sigma(\pi_\tau(c))$ for all $\sigma \in T(\pi_\tau(\M)'')$.  Therefore, by \cite[Th\'eor\`eme~3.2]{FH80} and Kaplansky's density theorem, there are a positive contraction $p_\tau \in \M$ and elements $x_{j, \tau}, y_{j, \tau} \in \M$ of norm at most $12(\|c\| + 1)$ for $j = 1, \ldots, 10$ such that
	\begin{equation}
		\Big\| p_\tau - c - \sum_{j=1}^{10} [ x_{j, \tau}, y_{j, \tau} ] \Big\|^2_{2, \tau} < \epsilon \quad \text{and} \quad \tau(p_\tau - p_\tau^2) < \epsilon
	\end{equation}
	for all $\tau \in X$.
	Define
	\begin{equation}
		a_\tau \coloneqq p_\tau - p_\tau^2 + \Big| p_\tau - c - \sum_{j=1}^{10} [ x_{j, \tau}, y_{j, \tau} ] \Big|^2 \in \mathcal M_+
	\end{equation}
	and note that $\tau(a_\tau) < 2\epsilon$.  
	
	As $X$ is compact, there are $\tau_1, \ldots, \tau_k \in X$ such that
	\[ \sup_{\tau \in X} \min_{1 \leq i \leq k} \tau(a_{\tau_i}) < 2 \epsilon. \]
	By CPoU, there are projections $q_1, \ldots, q_k \in \M^\infty$ summing to $1_{\M^\infty}$ and commuting with $c$, $p_{\tau_i}$, $x_{j, \tau_i}$ and $y_{j, \tau_i}$ for all $i = 1, \ldots, k$ and $j = 1, \ldots, 10$, such that
	\begin{equation}
		\tau( a_{\tau_i} q_i) \leq 2 \epsilon \tau(q_i), \qquad  \tau \in X^\infty,\ i = 1, \ldots, k.
	\end{equation}
	Define 
	\begin{equation}
		p \coloneqq \sum_{i=1}^k q_i p_{\tau_i}, \quad x_j \coloneqq \sum_{i=1}^k q_i x_{j, \tau_i}, \quad \text{and} \quad y_j \coloneqq \sum_{i=1}^k q_i y_{j, \tau_i}
	\end{equation}
	for $j= 1, \ldots, 10$.  Then we have
	\begin{equation}
		\Big\| p - c - \sum_{i=1}^{10} [ x_j, y_j ] \Big\|_{2, X^\infty} \leq \sqrt{2\epsilon} \quad \text{and} \quad \sup_{\tau \in X^\infty} \tau(p - p^2) \leq \sqrt{2\epsilon}.
	\end{equation}
	
	The result of the previous paragraph and Kirchberg's $\epsilon$-test implies there are a projection $p \in \mathcal M^\infty$ and elements $x_j, y_j \in \M^\infty$ such that 
	\begin{equation} 
		p = c + \sum_{j=1}^{10} [x_j, y_j]. 
	\end{equation}
	Then for all $\tau \in X^\infty$, $\tau(p) = \tau(c) = f(\tau|_\M)$ as required in \eqref{eq:correct-trace}.
\end{proof}

The classification of projections given in Theorems~\ref{thm:comparison} and~\ref{thm:existence-projections} allows us to compute the Murray--von Neumann semigroup and the $K_0$-group for factorial tracially complete $C^*$-algebra with CPoU. This is immediate from the previous two results and Propositions~\ref{prop:matrix-ultrapower} and~\ref{prop:matrix-CPoU}, which allow us to apply the previous two results to projections in matrix amplifications.

\begin{corollary}
	Suppose $(\mathcal M, X)$ is a type {\rm II}$_1$ factorial tracially complete $C^*$-algebra with CPoU.  The natural maps
	\begin{equation}
		V(\mathcal M) \overset\cong\longrightarrow \mathrm{Aff}(X)_+ \qquad \text{and} \qquad K_0(\mathcal M) \overset\cong\longrightarrow \mathrm{Aff}(X)
	\end{equation}
	are isomorphisms of ordered monoids and ordered groups, respectively.
\end{corollary}

\section{Hyperfiniteness}\label{sec:hyperfinite}

	Section~\ref{sec:hyp-fin-statements} contains the definition of hyperfiniteness along with the statements of all the results needed to prove the regularity theorem: namely, hyperfiniteness implies {}amenability and CPoU.  The proof of the latter implication occupies most of this section.  En route to proving CPoU, we prove an inductive limit decomposition for hyperfinite tracially complete $C^*$-algebras in the separable setting (Theorem~\ref{thm:hyperfinite-limit}).  Section~\ref{sec:perturbations} contains the finite dimensional perturbation lemmas required to obtain the inductive limit decomposition. The proofs of Theorems~\ref{thm:hyperfinite-implies-CPoU} and~\ref{thm:hyperfinite-limit} are given in Section~\ref{sec:hyp-fin-proofs}.
	
	\subsection{Definition and main properties}\label{sec:hyp-fin-statements}
	
	We define hyperfiniteness analogously to Murray and von Neumann's notion for II$_1$ factors in \cite{MvN43}.
	
	\begin{definition}
		\label{def:hyperfinite}
		We say that a tracially complete $C^*$-algebra $(\mathcal M, X)$ is \emph{hyperfinite} if for every finite set $\mathcal F \subseteq \mathcal M$ and $\epsilon > 0$, there is a finite dimensional $C^*$-algebra $F \subseteq \M$ such that for every $a \in \mathcal F$, there is  $b \in F$ with $\|a - b\|_{2, X} < \epsilon$.
	\end{definition}
	
	The passage from hyperfiniteness to {}amenability is a little subtle.  The obvious approach is to obtain the completely positive approximations directly by taking the ``downward'' maps to be conditional expectations onto finite dimensional subalgebras and the ``upward'' maps to be the inclusions.  For the estimates to match up correctly, one must arrange for the conditional expectations to be continuous in the uniform 2-norms, and it is not clear if this can be done. This issue does not arise in the setting of tracial von Neumann algebras (as the trace-preserving conditional expectation onto a subalgebra is necessarily $2$-norm contractive), so we can obtain {}amenability by working locally.

	\begin{theorem}\label{thm:hyperfinite-implies-semidiscrete}
		Every hyperfinite tracially complete $C^*$-algebra is {}amenable.
	\end{theorem}
	
	\begin{proof}
		Suppose $(\M, X)$ is a hyperfinite tracially complete $C^*$-algebras.  For all $\tau \in X$, the tracial von Neumann algebra $\pi_\tau(\M)''$ is hyperfinite and hence is semidiscrete.  The result follows from Theorem \ref{thm:introamenable}.
	\end{proof}
	
	The following result obtaining CPoU from hyperfiniteness is more subtle yet.  
	
	\begin{theorem}\label{thm:hyperfinite-implies-CPoU}
		Every hyperfinite factorial tracially complete $C^*$-algebra satisfies CPoU.
	\end{theorem}
	
	It is easy to see that finite dimensional $C^*$-algebras satisfy CPoU in the sense that if $F$ is a finite dimensional $C^*$-algebra, then $\big(F, T(F)\big)$ is a factorial tracially complete $C^*$-algebras with CPoU (this is a very special case of Corollary~\ref{cor:vNa-CPoU}).  If one strengthens the definition of hyperfiniteness (Definition~\ref{def:hyperfinite}) to require that the finite dimensional subalgebras $F \subseteq \mathcal M$ are factorial when viewed as a tracially complete $C^*$-algebra with the traces inherited from $(\mathcal M, X)$ (or equivalently, that every trace on $F$ extends to a trace in $X$), then CPoU for $(\mathcal M, X)$ would follow directly from CPoU from $\big(F, T(F)\big)$.
	
	For a general finite dimensional $C^*$-subalgebra $F \subseteq \mathcal M$, we do not have a way of relating general traces in $T(F)$ to traces of the form $\tau|_F$ for $\tau \in X$, so the above argument does not work.  A somewhat related issue was addressed by Murray and von Neumann in the setting of $\mathrm{II}_1$ factors in \cite{MvN43}.  They defined a $\mathrm{II}_1$ factor $\mathcal M$ to be ``approximately finite (B)'' if it satisfies for all finite set $\mathcal F \subseteq \mathcal M$ and $\epsilon > 0$, there is a finite dimensional $C^*$-algebra $F \subseteq \mathcal M$ such that for all $a \in \mathcal F$, there is a $b \in F$ such that $\|a - b\|_2 < \epsilon$, and they define $\mathcal M$ to be ``approximately finite (A)'' if one can further arrange for $F$ to be a factor.  It is a non-trivial result that these two notions coincide (see \cite[Lemma~4.6.2]{MvN43}).\footnote{In the end, for type II$_1$ factorial tracially complete $C^*$-algebras, the two analogous notions of hyperfiniteness are equivalent as a consequence of Theorem~\ref{thm:hyperfinite} using that the tracially complete $C^*$-algebras $(\mathcal R_X, X)$ of Example~\ref{Ex:ConcreteModels} satisfy the stronger notion of hyperfiniteness.  However, this is circular since Theorem~\ref{thm:hyperfinite} depends on Theorem~\ref{thm:hyperfinite-implies-CPoU}.}
	
	We circumvent the need to establish the equivalence discussed in the previous paragraph by showing that (in the separable setting) hyperfiniteness is equivalent to the a priori stronger property of being a sequential inductive limit of finite dimensional algebras.  Murray and von Neumann established such a result for separably acting hyperfinite II$_1$ factors (using the terminology ``approximately finite (C)'' for sequential inductive limits), and Bratteli proved a $C^*$-version of this result in his seminal paper on AF $C^*$-algebras (\cite{Bratteli72}).
	
	\begin{theorem}\label{thm:hyperfinite-limit}
		If $(\M, X)$ is a hyperfinite tracially complete $C^*$-algebra such that $\M$ is $\|\cdot\|_{2, X}$-separable, then $(\M, X)$ is an inductive limit of tracially complete $C^*$-algebras $(\M_n, X_n)_{n=1}^\infty$ such that each $\M_n$ is finite dimensional. 
	\end{theorem}
	
	The proof has the same flavour as the corresponding results for $C^*$-algebras and tracial von Neumann algebras.  We will prove Theorem \ref{thm:hyperfinite-limit} in Section \ref{sec:hyp-fin-proofs} after we set out the relevant uniform 2-norm perturbation and near-containment lemmas in the next section.   

Once Theorem \ref{thm:hyperfinite-limit} is in place, it will allow us to realise a separable factorial hyperfinite $C^*$-algebras as a tracial completion of the tracial completion of an AF algebra with respect to its traces.  Such tracial completions have CPoU by the permanence properties of Section \ref{sec:cpou-permanence}. This will prove Theorem \ref{thm:hyperfinite-implies-CPoU} in the separable case, and we will deduce the non-separable case from there.

\subsection{Perturbation lemmas}\label{sec:perturbations}

The goal of this subsection is to prove a uniform 2-norm `near-containment' result (Proposition~\ref{prop:near-inclusion}) for finite dimensional subalgebras.  Similar results were used by Murray and von Neumann in the von Neumann algebra setting, and Glimm and Bratteli for $C^*$-algebras (cf.\ \cite{MvN43, Glimm60, Bratteli72}).  The main observation is that although the Borel functional calculus used in the von Neumann algebraic results does not generally exist in a tracially complete $C^*$-algebra, it is defined in all finite dimensional subalgebras.

We start with a perturbation result for almost orthogonal projections.

\begin{lemma}
	\label{lem:orthogonal-projections}
	For all $\epsilon > 0$ and $n \in \N$, there is $\delta > 0$ such that for all finite dimensional tracially complete $C^*$-algebras $(F, X)$ and all self-adjoint contractions $q_1, \ldots, q_n \in F$ with
	\begin{equation}
		\| q_i - q_i^2 \|_{2, X} < \delta \qquad \text{and} \qquad \|q_i q_j\|_{2, X} < \delta,
	\end{equation}
	for all $i, j = 1, \ldots, n$ with $i \neq j$, there are mutually orthogonal projections $p_1, \ldots, p_n \in F$ such that $\|p_i - q_i\|_{2, X} < \epsilon$.
\end{lemma}

\begin{proof}
	We will prove the result by induction on $n$.  When $n = 1$, set $\delta \coloneqq \epsilon/2$, let $f \colon \mathbb R \rightarrow \mathbb R$ denote the characteristic function of $[1/2, \infty)$, and define $p_1 \coloneqq f(q_1)$.  Then
	\begin{equation}
		|f(t) - t| \leq 2 |t - t^2|, \qquad t \in \mathbb R,
	\end{equation}
	and so $\|p_1 - q_1\|_{2, X} \leq 2 \|q_1 - q_1^2\|_{2, X} < \epsilon$.
	
	Assuming the result has been proven for $n \in \mathbb N$, let $\delta' > 0$ be given by applying the lemma with this $n$ with $\epsilon/(8n)$ in place of $\epsilon$.  Define
	\begin{equation}
		\delta \coloneqq \min\Big\{ \delta', \frac{\epsilon}{8n+4} \Big\}.
	\end{equation}
	Let $q_1, \ldots, q_{n+1} \in F$ be given as in the statement.  As $\delta \leq \delta'$, the choice of $\delta'$ implies there are mutually orthogonal projections $p_1, \ldots, p_n \in F$ such that $\|p_i - q_i\|_{2, X} < \epsilon/(8n)$ for $i = 1, \ldots, n$.
	
	Define $p \coloneqq \sum_{i=1}^n p_i$, $q \coloneqq q_{n+1}$, and $p_{n+1} \coloneqq f(p^\perp q p^\perp)$.  Then $p_{n+1}$ is a projection and is orthogonal to each of $p_1, \ldots, p_n$.  Using the bimodule property of \eqref{eq:SpecialHolderIneq}, we can estimate
	\begin{equation}\begin{split}
		\|p_{n+1} - q_{n+1}\|_{2, X}
		&\leq \|f(p^\perp q p^\perp) - p^\perp q p^\perp\|_{2, X} + \|p^\perp q p^\perp - q\|_{2, X}  \\
		&\leq 2 \| p^\perp q p^\perp - (p^\perp q p^\perp)^2 \|_{2, X} + 2 \|p q\|_{2, X} \\
		&\leq 2 \|  q  - qp^\perp q \|_{2, X} + 2 \|p q\|_{2, X} \\
		&\leq 2 \|q - q^2\|_{2, X} + 4 \|pq\|_{2, X} \\
		&< 2 \delta + 4 \sum_{i=1}^n \|p_i q\|_{2, X} \\
		&\leq 2 \delta + 4 \sum_{i=1}^n \big( \|p_i - q_i\|_{2, X} + \|q_iq\|_{2, X} \big) \\
		&< 2 \delta + 4n\Big(\frac{\epsilon}{8n} + \delta \Big) \leq \epsilon,
	\end{split}\end{equation}
    which completes the proof.
\end{proof}

The following perturbation result for partial isometries will be used to perturb the off-diagonal matrix units of a near inclusion of finite dimensional $C^*$-algebras.

\begin{lemma}
	\label{lem:partial-isometry}
	Suppose $(F, X)$ is a finite dimensional tracially complete $C^*$-algebra and $p, q \in F$ are projections.  If $w \in F$ is a contraction, then there is a partial isometry $v \in F$ such that $vp = v = qv$ and 	
	\begin{equation}
		\|v - w\|_{2, X} \leq \sqrt{38} \max \big\{\|w^*w - p\|_{2, X}^{1/2}, \|ww^* - q\|_{2, X}^{1/2} \big\}.
	\end{equation}
\end{lemma}

\begin{proof}
	Define
	\begin{equation}\label{eq:partial-isom-delta}
		\delta \coloneqq \max \big\{\|w^*w - p\|_{2, X}, \|ww^* - q\|_{2, X} \big\}.
	\end{equation}
	Let $u \coloneqq q w p$ and note that $u$ is a contraction and 
	\begin{equation}\label{eq:approx-source-range}
		\max \big\{\|u^*u - p\|_{2, X}, \|uu^* - q\|_{2, X} \big\} \leq 3 \delta.
	\end{equation}
	Indeed,
	\begin{equation}\begin{split}
		\|u^*u - p\|_{2, X} &= \|pw^*qwp - p\|_{2, X} \\
		&\leq \|pw^*(q - ww^*)w p\|_{2, X} + \| pw^*w(w^*w - p)p\|_{2, X}  \\ &\hphantom{==} + \|p (w^*w - p)p\|_{2, X}  \\ &\leq 3\delta, \label{eq:approx-source-range-2}
	\end{split}\end{equation}
	where the last inequality uses \eqref{eq:partial-isom-delta} and the bimodule property of \eqref{eq:SpecialHolderIneq}.  The other inequality in \eqref{eq:approx-source-range} follows similarly.
	
	Define $g \colon \mathbb [0,1] \rightarrow \mathbb R$ by
	\begin{equation}
		g(t) \coloneqq \begin{cases} 0 & 0 \leq t \leq \frac14; \\ t^{-1/2} & \frac14 < t \leq 1, \end{cases}
	\end{equation}
	and define $v \coloneqq u g(u^*u)$.  Then $v^*v$ is a projection, and hence $v$ is a partial isometry.  As $up = u = qu$, we also have $vp = v = qv$.  After solving some polynomial inequalities, we get that for $t \in [0, 1]$, we have 
	\begin{equation}\label{eq:bound-on-g}
		0 \leq |t - tg(t)| \leq \frac{4}{3} (t - t^2) \qquad \text{and} \qquad 0 \leq g(t) \leq 2.
	\end{equation}
	Therefore, for each $\tau \in X$,
	\begin{equation}\begin{array}{rcl}
		\| u - v\|_{2, \tau}^2 
		&=& \tau\big((u - u g(u^*u))^*(u - ug(u^*u))\big) \\
		&=& \tau\big(u^*u - 2u^*ug(u^*u) + u^*ug(u^*u)^2\big) \\
		&=& \tau\big(u^*u - u^*ug(u^*u)\big) \\ &&\qquad + \tau\big((u^*ug(u^*u) - u^*u)g(u^*u)\big) \\
		&\stackrel{\eqref{eq:bound-on-g}}\leq& \tau\big(|u^*u - u^*ug(u^*u)|\big) \\ &&\qquad + 2\tau\big(|u^*u - u^*ug(u^*u)|\big) \\
		&\stackrel{\eqref{eq:bound-on-g}}{\leq}& 4 \tau(u^*u - (u^*u)^2)\\
		&\leq& 4(\|u^*u-p\|_{2,\tau}+\|p-(u^*u)^2\|_{2,\tau})\\
  &\stackrel{\eqref{eq:approx-source-range}}{\leq}& 36\delta,
	\end{array}\end{equation}
	using the Cauchy--Schwarz inequality and the fact that $p$ is a projection.  Now note that for each $\tau \in X$,
 \begin{equation}
 \begin{split}
    \| v - w\|^2_{2, \tau}
    &= \|v- pw\|^2_{2, \tau} + \|p^\perp w\|^2_{2, \tau} \\
    &= \|v -p w q \|^2_{2, \tau} + \| p^\perp w q^\perp \|^2_{2, \tau} + \| p^\perp w\|^2_{2, \tau} \\
    &\leq 38\delta,
\end{split}
 \end{equation}
 as required.
\end{proof}

Since we don't get $v^*v=p$ in the previous lemma, obtaining off-diagonal matrix units will require us to further perturb the diagonal matrix units.
The following allows us to manage accumulated errors incurred in this way.

\begin{lemma}
	\label{lem:meet-of-subprojections}
	If $(F, X)$ is a finite dimensional tracially complete $C^*$-algebra and $p_1, \ldots, p_n, p \in F$ are projections with $p_i \leq p$ for all $i = 1, \ldots, n$, then
	\begin{equation}\label{eqn:p-minus-meet}
		\Big\| p - \bigwedge_{i=1}^n p_i \Big\|_{2, X} \leq \Big( \sum_{i=1}^n \| p - p_i \|_{2, X}^2 \Big)^{1/2} \leq \sum_{i=1}^n \|p - p_i\|_{2, X}.
	\end{equation}
\end{lemma}

\begin{proof}
	The second inequality amounts to $\sum_{i=1}^nt_i^2\leq(\sum_{i=1}^nt_i)^2$ for positive $t_i$, so it is immediate.  We prove the first inequality by induction on $n$ with the case $n = 1$ being trivial.  Let $q_1 \coloneqq \bigwedge_{i=1}^n p_i$ and $q_2 \coloneqq p_{n+1}$.  By \cite[Proposition~III.1.1.3]{Bl06},
	\begin{equation}
		(q_1 - (q_1 \wedge q_2)) \quad \text{and}\quad ((q_1 \vee q_2) - q_2),
	\end{equation}
	are Murray--von Neumann equivalent.	Therefore, for each $\tau \in X$,
	\begin{equation}\begin{split}
		\|p - (q_1 \wedge q_2)\|_{2, \tau}^2 
		&= \tau(p - (q_1 \wedge q_2)) \\
		&= \tau(p - q_1) + \tau(q_1 - (q_1 \wedge q_2)) \\
		&= \tau(p - q_1) + \tau((q_1 \vee q_2) - q_2) \\
		&\leq \tau(p - q_1) + \tau(p - q_2) \\
		&= \|p - q_1\|_{2, \tau}^2 + \|p - q_2\|_{2, \tau}^2. 
	\end{split}\end{equation}
	Hence we have
	\begin{equation}\begin{split}
		\Big\| p - \bigwedge_{i=1}^{n+1} p_i \Big\|_{2, X}^2
		&\leq \Big\| p - \bigwedge_{i=1}^{n} p_i \Big\|_{2, X}^2 + \|p - p_{n+1}\|_{2, X}^2,
	\end{split}\end{equation}	
	as required.
\end{proof}

The last three lemmas combine to prove the following stability result for near inclusions of finite dimensional tracially complete $C^*$-algebras.

\begin{proposition}
	\label{prop:near-inclusion}
	Suppose $(\M, X)$ is a tracially complete $C^*$-algebra, let $F$ be a finite dimensional $C^*$-subalgebra of $\M$, and let $\mathcal F \subseteq F$ be a system of matrix units.  For all $\epsilon > 0$, there exists $\delta > 0$ such that if $G \subseteq \M$ is a finite dimensional $C^*$-algebra and for all $a \in \mathcal F$, there exists $b \in G$ with $\|a - b\|_{2, X} < \delta$, then there is a (not necessarily unital) $^*$-homomorphism $\phi \colon F \rightarrow G$ such that 
	\begin{equation}
		\| \phi(a) - a \|_{2, X} \leq \epsilon \|a\|, \qquad  a \in F.
	\end{equation}
\end{proposition}

\begin{proof}
	Write
	\begin{equation}
		\mathcal F = \big(e_{i, j}^{(k)}\big)_{i, j = 1, \ldots, d_k}^{k = 1, \ldots, m}.
	\end{equation}
	By replacing $\epsilon$ with a scalar multiple (depending only on the dimension of $F$), it suffices to construct a corresponding system of matrix units 
	\begin{equation}
		\big(f_{i, j}^{(k)}\big)_{i, j = 1, \ldots, d_k}^{k = 1, \ldots, m} \subseteq G
	\end{equation}
	such that $\big\| e_{i, j}^{(k)} - f_{i, j}^{(k)} \big\|_{2, X} < \epsilon$ for all $i, j, k$.
	
	Let $d \coloneqq \max \{d_1, \ldots, d_m \}$ and choose $\epsilon' > 0$ with 
	\begin{equation}
		\label{eq:epsilon-prime}
		(4\epsilon' + 4(38\epsilon')^{1/2})(d + 1) < \epsilon.
	\end{equation}
	Apply Lemma~\ref{lem:orthogonal-projections} with $\epsilon'$ in place of $\epsilon$ and with $\sum_{k=1}^m d_k$ in place of $n$ to obtain $\delta' > 0$, and define $\delta \coloneqq \min\{ \epsilon'/3, \delta'/9 \} > 0$.  By assumption and Lemma~\ref{lem:UnitBallDensity}, for each $i =1, \ldots, d_k$ and $k = 1, \ldots, m$, there are contractions $q_{i}^{(k)}, w_{i}^{(k)} \in G$ such that 
	\begin{equation}
		\label{eq:close-to-matrix-units}
		\big\|  q_{i}^{(k)} - e_{i,i}^{(k)} \big\|_{2, X} < 3 \delta \leq \epsilon' \quad \text{and} \quad \big\| w_{i}^{(k)} - e_{i,1}^{(k)} \big\|_{2, X} < 3 \delta \leq \epsilon'.
	\end{equation}
	
	For all  $k, k' = 1, \ldots, m$, $i = 1, \ldots, d_k$, and $i' = 1, \ldots, d_{k'}$ such that $(i,k)\neq (i',k')$, we have
 \begin{equation}
		\|(q_{i}^{(k)})^2 - q_{i}^{(k)} \|_{2, X} < 9\delta \leq \delta', \quad \text{and} \quad \|q_{i}^{(k)} q_{i'}^{(k')} \|_{2, X} < 6 \delta < \delta'.
	\end{equation}
	By the choice of $\delta'$, we may invoke the conclusion of Lemma~\ref{lem:orthogonal-projections} to produce mutually orthogonal projections $p_{i}^{(k)} \in G$ such that for all $i$ and $k$,
	\begin{equation}
		\label{eq:p-and-q}
		\| p_{i}^{(k)} - q_{i}^{(k)} \|_{2, X} < \epsilon'.
	\end{equation}
	Now for each $i$ and $k$, we have
	\begin{equation}
		\label{eq:w-source-and-p}
		w_{i}^{(k)}{}^*w_{i}^{(k)} \stackrel{\eqref{eq:close-to-matrix-units}}{\approx}_{\!\!\!2\epsilon'} e_{1, i}^{(k)} e_{i, 1}^{(k)} = e_{1, 1}^{(k)} \stackrel{\eqref{eq:close-to-matrix-units}}{\approx}_{\!\!\!\epsilon'} q_{1}^{(k)} \stackrel{\eqref{eq:p-and-q}}{\approx}_{\!\!\!\epsilon'} p_{1}^{(k)},
	\end{equation}
	and similarly $w_{i}^{(k)} w_{i}^{(k)}{}^* \approx_{4\epsilon'} p_{i}^{(k)}$, where the approximations are in $\|\cdot\|_{2, X}$.  By Lemma~\ref{lem:partial-isometry}, for each $i$ and $k$, there is a partial isometry $v_{i}^{(k)}\in G$ with source contained in $p_{1}^{(k)}$ and range contained in $p_{i}^{(k)}$ and such that
	\begin{equation}
		\label{eq:v-and-w}
		\| v_{i}^{(k)} - w_{i}^{(k)} \|_{2, X} \leq 2 (38\epsilon')^{1/2}.
	\end{equation}
	
	For $k = 1, \ldots, m$ and $i, j = 1, \ldots, d_k$, define
	\begin{equation}
		r^{(k)} \coloneqq \bigwedge_{i=1}^{d_k} v_{i}^{(k)}{}^* v_{i}^{(k)} \leq p_{1}^{(k)} \quad \text{and} \quad f_{i, j}^{(k)} \coloneqq v_{i}^{(k)} r^{(k)} v_{j}^{(k)}{}^*.
	\end{equation}
	Then the $f_{i, j}^{(k)}\in G$ satisfy the matrix unit relations, so it suffices to show that $\big\| e_{i, j}^{(k)} - f_{i, j}^{(k)} \big\|_{2, X} < \epsilon$ for all $i, j, k$.  For each $k = 1, \ldots, m$, we have
	\begin{equation}
		v_{i}^{(k)}{}^* v_{i}^{(k)} \stackrel{\eqref{eq:v-and-w}}{\approx}_{\!\!\!4(38\epsilon')^{1/2}} w_{i}^{(k)}{}^* w_{i}^{(k)} \stackrel{\eqref{eq:w-source-and-p}}{\approx}_{\!\!\!4\epsilon'} p_{1}^{(k)},
	\end{equation}
	where the approximations are in the $\|\cdot\|_{2, X}$-norm.  Therefore, Lemma~\ref{lem:meet-of-subprojections} implies 
	\begin{equation}\label{eq:p-and-r}
		\| p_{1}^{(k)} - r^{(k)} \|_{2, X} < (4\epsilon' + 4(38\epsilon')^{1/2}) d_k.
	\end{equation}
	
	For all $k = 1, \ldots, m$ and $i, j = 1, \ldots, d_k$, we have
	\begin{equation}\label{eq:e-and-f}
		e_{i, j}^{(k)} = e_{i, 1}^{(k)} e_{1, j}^{(k)} \stackrel{\eqref{eq:close-to-matrix-units}}{\approx}_{\!\!\!2\epsilon'} w_{i}^{(k)} w_{i}^{(k)}{}^* \stackrel{\eqref{eq:v-and-w}}{\approx}_{\!\!\!4(38\epsilon')^{1/2}} v_{i}^{(k)}v_{i}^{(k)}{}^* = v_{i}^{(k)} p_{1}^{(k)} v_{i}^{(k)}{}^*.
	\end{equation}
	Combining \eqref{eq:p-and-r} and \eqref{eq:e-and-f}, we have
	\begin{equation}
		\big\| e_{i, j}^{(k)} - f_{i, j}^{(k)} \big\|_{2, X} \leq (4 \epsilon' + 38 (\epsilon')^{1/2})(d_k + 1) \stackrel{\eqref{eq:epsilon-prime}}{<} \epsilon. \qedhere
	\end{equation}
\end{proof}

\subsection{Hyperfinite implies CPoU}\label{sec:hyp-fin-proofs}

The perturbation results of the previous subsection allow us to prove Theorem~\ref{thm:hyperfinite-limit}, from which we will deduce Theorem~\ref{thm:hyperfinite-implies-CPoU}.

\begin{proof}[Proof of Theorem~\ref{thm:hyperfinite-limit}]
	Suppose $(\M, X)$ is hyperfinite and $\M$ is $\|\cdot\|_{2, X}$-separable. Let $\mathcal G_n$ be an increasing sequence of finite subsets of $\M$ with $\|\cdot\|_{2, X}$-dense union and let $\epsilon_n > 0$ be such that $\sum_{n=1}^\infty \epsilon_n < \infty$.  We will inductively construct sequences of finite dimensional subalgebras $F_n \subseteq \M$ and $^*$-homomorphisms $\phi_n^{n+1} \colon F_n \rightarrow F_{n+1}$ such that for all $n \geq 1$,
	\begin{enumerate}
		\item\label{afd-proof1} for all $a \in \mathcal G_n$, there is $b \in F_n$ such that $\|a - b\|_{2, X} < \epsilon_n$, and
		\item\label{afd-proof2} for all $b \in F_n$, $\|\phi_n^{n+1}(b) - b\|_{2, X} < \epsilon_n \|b\|$.
	\end{enumerate}
	
	We can find $F_1$ satisfying \ref{afd-proof1} by the definition of hyperfiniteness.  Assuming that $F_n$ has been constructed, let $\mathcal F_n$ be a system of matrix units for $F_n$, and let $\delta_n > 0$ be given by applying Proposition~\ref{prop:near-inclusion} with $F_n$ and $\epsilon_n$ in place of $F$ and $\epsilon$. By the definition of hyperfiniteness, there is a finite dimensional $C^*$-algebra $F_{n+1} \subseteq \M$ such that \ref{afd-proof1} holds and for all $a \in \mathcal F_n$, there is a $b \in F_{n+1}$ such that $\|a - b\|_{2, X} < \delta_n$.  Then Proposition~\ref{prop:near-inclusion} yields the $^*$-homomorphism $\phi_n^{n+1} \colon F_n \rightarrow F_{n+1}$ required in \ref{afd-proof2}, which completes the construction.
	
	For $m, n \in \mathbb N$ with $m < n$, let $\phi_m^n \colon F_m \rightarrow F_n$ be given by composing the maps $\phi_k^{k+1}$ for $m \leq k < n$.  For $m \in \mathbb N$ and $b \in F_m$, condition \ref{afd-proof2} above and the choice of $\epsilon_n$ imply that the bounded sequence $(\phi_m^n(b))_{n=m}^\infty$ is $\|\cdot\|_{2, X}$-Cauchy.  Define $\psi_m \colon F_m \rightarrow \M$ by
	\begin{equation}
		\psi_m(b) \coloneqq \lim_{n\rightarrow \infty}
		\phi_m^n(b), \qquad b \in F_m.
	\end{equation}
	Note that $\psi_n \circ\phi_m^n = \psi_m$ for all $m, n \in N$ with $m < n$.  If \mbox{$A \coloneqq \varinjlim \,(F_n, \phi_n)$} is the $C^*$-algebraic inductive limit of the algebras $F_n$ and \mbox{$\phi_{\infty, n} \colon F_n \rightarrow A$} are the natural maps, then there is an induced $^*$-homomorphism $\psi \colon A \rightarrow \M$ such that $\psi \circ \phi_{\infty, n} = \psi_n$ for all $n \in \mathbb N$.
	
	Since $A$ is an AF algebra and quotients of AF algebras are also AF algebras (see \cite[Theorem III.4.4]{Davidson96}, for example), we have that $\psi(A)$ is an AF subalgebra of $\mathcal M$.  It suffices to show that $\psi(A) \subseteq \mathcal M$ is $\|\cdot\|_{2, X}$-dense.  For each $n \in \mathbb N$, condition \ref{afd-proof2} and the definition of $\psi_n$ implies
	\begin{equation}
		\|\psi_n(b) - b \|_{2, X} \leq \sum_{m=n}^\infty \epsilon_m \|b\|, \qquad  b \in F_n. 
	\end{equation}
	As $\sum_{n=1}^\infty \epsilon_n < \infty$ and the sets $\mathcal G_n$ are increasing, the above inequality and condition \ref{afd-proof1} imply that $\mathcal G_n$ is contained in the $\|\cdot\|_{2, X}$-closure of $\psi(A)$ for all $n \geq 1$.  As the sets $\mathcal G_n$ have $\|\cdot\|_{2, X}$-dense union in $\mathcal M$, this completes the proof that $(\M,X)$ is an inductive limit of finite-dimensional tracially complete $C^*$-algebras. The converse is straightforward.
\end{proof}

We now have the pieces to show hyperfinite implies CPoU in the factorial setting.

\begin{proof}[Proof of Theorem~\ref{thm:hyperfinite-implies-CPoU}]
	Let $(\mathcal M, X)$ be a hyperfinite factorial tracially complete $C^*$-algebra. Suppose first that $\M$ is $\|\cdot\|_{2,X}$-separable.  Then, by Theorem~\ref{thm:hyperfinite-limit}, there is an increasing sequence $(F_n)_{n=1}^\infty$ of finite dimensional unital $C^*$-subalgebras of $\mathcal M$ with $\|\cdot\|_{2, X}$-dense union.  Let $A \subseteq \mathcal M$ denote the norm closure of the union of the $F_n$ so that $A$ is an AF $C^*$-algebra.  If $X_A \coloneqq \{\tau|_A : \tau \in X \}$, then Corollary~\ref{cor:dense-subalgebra} implies that $X_A$ is a closed face in $T(A)$ and $\big(\completion{A}{X_A}, X_A\big) \cong (\mathcal M, X)$.  

Note that $\big(\completion{A}{T(A)},T(A)\big)$ is the inductive limit of the factorial tracially complete $C^*$-algebras $\big(F_n, T(F_n)\big)$ and each of these has CPoU by Corollary~\ref{cor:vNa-CPoU}. Therefore $\big(\completion{A}{T(A)},T(A)\big)$ has CPoU by Proposition~\ref{prop:InductiveLimitCPoU}.  Further, note that the tracial completion of $\completion{A}{T(A)}$ with respect to $X_A$ is $(\completion{A}{X_A},X_A)\cong (\M,X)$ by Proposition~\ref{prop:tracial-completion}\ref{item:completion-subset}. This has CPoU by Proposition~\ref{prop:CPoU-quotient}.

The general case follows from the $\|\cdot\|_{2,X}$-separable case as hyperfiniteness is separably inheritable and factoriality and CPoU are strongly separably inheritable (see the discussion immediately following Definition \ref{defn:sep} in Appendix \ref{sec:sep}).  The proofs of the (strong) separable inheritability are given in Theorem \ref{thm:sep}.
\end{proof}

\section{Classification and Regularity}\label{sec:Classification}

In this final section, we combine the pieces to obtain our main results.
We establish the classification of tracially nuclear $^*$-homomorphisms (Theorem~\ref{IntroThmClassMap}) through uniqueness and existence results for morphisms in Sections~\ref{sec:uniqueness} and~\ref{sec:existence}, respectively.
The classification theorem (Theorem~\ref{InformalClassification}) is then obtained by standard intertwining arguments in Section~\ref{sec:classification-theorems}, and the structure theorem (Theorem~\ref{InformalStructureThm}) follows.

\subsection{Uniqueness results for morphisms}\label{sec:uniqueness}

This subsection gives a uniqueness result for tracially nuclear $^*$-homomorphisms into factorial tracially complete $C^*$-algebras with CPoU.  Uniqueness will be up to approximate unitary equivalence in uniform 2-norm, as follows.

\begin{definition}
	If $(\mathcal M, X)$ is a tracially complete $C^*$-algebra and $\phi, \psi \colon A \rightarrow \M$ are functions, we say that $\phi$ and $\psi$ are \emph{approximately unitarily equivalent} if there is a net of unitaries $(u_\lambda) \subseteq \mathcal M$ such that \begin{equation} \|u_\lambda \phi(a) u_\lambda^* - \psi(a)\|_{2, X} \rightarrow 0, \qquad a \in A.
 \end{equation}
Note that when $A$ is a separable $C^*$-algebra and $\phi$ and $\psi$ are $\|\cdot\|$-continuous we may arrange the net $(u_\lambda)$ to be a sequence. 
\end{definition}

We begin by recording the following uniqueness result for von Neumann algebras which is a consequence of Connes' theorem.  Several variations and special cases of this can be found in the literature (see \cite[Proposition~2.1]{CGNN13} or \cite{DH05} for example);  the result most often appears when $A$ is nuclear and $\mathcal M$ is a II$_1$ factor.  Most, if not all, variations of this argument (including the one below) follow the same strategy:
\begin{enumerate}
	\item  use Connes' theorem to replace the domain with a hyperfinite von Neumann algebra;
	\item use an $\epsilon/3$-argument to replace the domain with a finite dimensional $C^*$-algebra;
	\item  use the Murray--von Neumann classification of projections by traces to compare matrix units in the codomain.
\end{enumerate}

\begin{proposition}
	\label{prop:uniqueness-to-vNa}
	Suppose $A$ is a $C^*$-algebra, $\mathcal N$ is a finite von Neumann algebra, and $\phi, \psi \colon A \rightarrow \mathcal N$ are weakly nuclear $^*$-homomorphisms.  Then there is a net of unitaries $(u_\lambda) \subseteq \mathcal N$ such that 
	\begin{equation}\label{e9.2}
		\text{\rm $\sigma$-strong$^*$-}\lim_\lambda\, u_\lambda\phi(a) u_\lambda^* = \psi(a), \qquad a \in A,
	\end{equation}
	if and only if $\tau\circ\phi=\tau\circ\psi$ for all $\tau\in T(\mathcal N)$.
\end{proposition}

\begin{proof}
If $\phi$ and $\psi$ are approximately unitarily equivalent in the strong$^*$ topology, then $\phi$ and $\psi$ agree on all normal traces.  As the normal traces are weak$^*$-dense in $T(\mathcal N)$ (Propositions~\ref{prop:TracesFiniteVNAS}\ref{prop:TracesFiniteVNAS.2}), $\phi$ and $\psi$ agree on all traces $\tau \in T(\mathcal N)$.

 Suppose conversely that $\tau\circ\phi=\tau\circ\psi$ for all $\tau\in T(\mathcal N)$. We may assume that $\mathcal N$ has a faithful normal trace, as in general we can decompose $\mathcal N$ as a (possibly uncountable) direct sum $\oplus_i\mathcal N_i$, where each $\mathcal N_i$ has a faithful normal trace,\footnote{Given a non-zero central projection $p$ in a finite von Neumann algebra $\mathcal N$, there is a normal trace $\tau$ on $\mathcal N$ with $\tau(p)=1$. The support projection $q$ of $\tau$ is the minimal central projection with $\tau(1-q)=0$; it is necessarily a non-zero subprojection of $p$ and $\tau$ is faithful on $\mathcal Nq$.  Accordingly Zorn's lemma gives the required decomposition of $\mathcal N$.} and combine unitaries witnessing approximate unitary equivalence in each summand. Now, let $\mathcal M \subseteq \mathcal N$ denote the von Neumann subalgebra generated by $\phi(A)$. As $\mathcal N$ has a faithful trace, the hypothesis ensures \mbox{$\ker(\phi) = \ker(\psi)$}, and so $\psi$ factorises through $\phi(A)$.  Let $\bar\phi \colon \mathcal M \rightarrow \mathcal N$ denote the inclusion map and let $\bar \psi \colon \mathcal M \rightarrow \mathcal N$ denote the unique normal $^*$-homomorphism with $\bar \psi \circ \phi = \psi$.  Then $\tau\circ\bar\phi = \tau\circ\bar\psi$ for all normal traces $\tau \in T(\mathcal N)$ (and hence also for all $\tau \in T(\mathcal N)$ by Proposition~\ref{prop:tracial-completion}\ref{prop:TracesFiniteVNAS.2}).
	
	As $\mathcal N$ is finite with a faithful normal trace, there is a normal conditional expectation \mbox{$\mathcal N \rightarrow \mathcal M$} (see \cite[Lemma~1.5.10]{Br08}; for example), and so, the corestriction of $\phi$ to a $^*$-homomorphism $A \rightarrow \mathcal M$ is weakly nuclear.  Then $\M$ is hyperfinite by the proof of (6) implies (5) of \cite[Theorem~3.2.2]{Bro06}.\footnote{This reference handles the case that $\mathcal M \coloneqq \pi_\tau(A)''$ and $\phi \coloneqq \pi_\tau$ for a trace $\tau \in T(A)$.  To prove the claim in our setting, fix a faithful normal $^*$-homomorphism $\pi \colon \mathcal M \rightarrow \mathcal B(H)$ so that $\mathcal M \cong \pi(\mathcal M) = \pi(\phi(A))''$, and in the proof in \cite{Bro06}, replace $\pi_\tau$ with $\pi \circ \phi$.}

 Let $F \subseteq \mathcal \M$ be a finite dimensional $C^*$-algebra and note that $\tau\circ\bar\phi|_F=\tau\circ\bar\psi|_F$.  It is a standard consequence of the classification of projections in the finite von Neumann algebra $\mathcal N$ by traces that $\bar\phi|_F$ and $\bar\psi|_F$ are unitarily equivalent.\footnote{\label{fd-unique} 
	Let $e_{i, j}^{(k)}$ be a system of matrix units for $F$.
	For each $k = 1, \ldots, m$, we have
	\begin{equation*}
		\tau\big(\bar\phi\big(e_{1,1}^{(k)}\big)\big) = \tau\big(\bar\psi\big(e_{1, 1}^{(k)}\big)\big), \qquad \tau \in T(\mathcal N),
	\end{equation*}
	and hence there is a partial isometry $v_k \in \mathcal N$ with source $\bar\phi(e_{1,1}^{(k)})$ and range $\bar\psi(e_{1, 1}^{(k)})$.  Define 
	\begin{equation*}
		u \coloneqq \sum_{k = 1}^m \sum_{i = 1}^{d_k} \bar\psi(e_{i, 1}^{(k)})v_k \bar\phi(e_{1, i}^{(k)}) \in \mathcal N. 
	\end{equation*}
	Then $u \in \mathcal N$ is a unitary with $u \bar\phi(x)u^* = \bar\psi(x)$ for all $x \in F$.}
	Therefore $\bar\phi$ and $\bar\psi$ are strong$^*$-approximate unitarily equivalent, and hence so too are $\phi$ and $\psi$.
\end{proof}

Our main uniqueness result is obtained by a CPoU argument to pass from  local uniqueness at the fibres to a global statement.  This works in essentially the same way as \cite[Theorem 2.2]{CETW-classification} (which handles the case when $A$ is nuclear).\footnote{The setup in \cite{CETW-classification} considers uniform tracial sequence algebras $B^\infty$ as codomains when $B$ is a separable $C^*$-algebra with CPoU, but there are no additional difficulties in working in the abstract framework of tracially complete $C^*$-algebras.}

\begin{theorem}
	\label{thm:uniqueness}
	Suppose $A$ is a $C^*$-algebra, $(\mathcal N, Y)$ is a factorial tracially complete $C^*$-algebra with CPoU, and  $\phi, \psi \colon A \rightarrow \mathcal N$ are tracially nuclear $^*$-homomorphisms.  Then $\phi$ and $\psi$ are approximately unitarily equivalent if and only if $\tau \circ \phi = \tau \circ \psi$ for all $\tau \in Y$.
\end{theorem}

\begin{proof}
	The forward direction is immediate.  For the reverse direction, suppose $\tau \circ \phi = \tau \circ \psi$ for all $\tau \in Y$ and fix a finite set $\mathcal F \subseteq A$ and $\epsilon > 0$.  For each pair of traces $\tau \in Y$ and $\sigma \in T(\pi_\tau(\mathcal N)'')$, we have $\sigma \circ \pi_\tau \in Y$ by Lemma~\ref{lem:GNSTraceFace}, and hence 
	\begin{equation}
		\sigma \circ \pi_\tau \circ \phi = \sigma \circ \pi_\tau \circ \psi.
	\end{equation}
	As $\pi_\tau \circ \phi$ and $\pi_\tau \circ \psi$ are weakly nuclear (see Theorem~\ref{thm:amenable}), Proposition~\ref{prop:uniqueness-to-vNa} implies there is a unitary $\bar u_\tau \in \pi_\tau(\mathcal N)''$ such that 
	\begin{equation}
		\max_{a \in \mathcal F} \|\bar u_\tau \pi_\tau(\phi(a)) \bar u_\tau^* - \pi_\tau (\psi(a)) \|_{2, \tau} < \epsilon, \qquad \tau \in Y.
	\end{equation}
	As $\pi_\tau(\mathcal N)''$ is a von Neumann algebra, we know that $\bar u_\tau = e^{i\bar{h}_{\tau}}$ for some self-adjoint $\bar{h}_{\tau} \in \pi_\tau(\mathcal N)''$. Applying Kaplansky's density theorem to $\bar{h}_{\tau} \in \pi_\tau(\mathcal N)''$ and making use of the existence of self-adjoint lifts, we deduce that for each $\tau \in Y$, each unitary $\bar u_\tau$ is a $\|\cdot\|_{2, \tau}$-limit of unitaries of the form $\pi_\tau(u_\tau)$ for a unitary $u_\tau = e^{ih_\tau} \in \mathcal N$, and in particular, we may find unitaries $u_\tau \in \mathcal N$ satisfying
	\begin{equation}
		\max_{a \in \mathcal F} \| u_\tau \phi(a) u_\tau^* - \psi(a)\|_{2, \tau} < \epsilon, \qquad \tau \in Y.
	\end{equation}

	For $\tau \in Y$, define
	\begin{equation}
		a_\tau \coloneqq \sum_{a \in \mathcal F} \big|u_\tau \phi(a) u_\tau^* - \psi(a) \big|^2 \in \mathcal N_+
	\end{equation}
	and note that $\tau(a_\tau) < |\mathcal F|\epsilon^2$ for all $\tau \in Y$.  As $Y$ is compact, there are traces $\tau_1, \ldots, \tau_n \in Y$ such that 
	\begin{equation}
		\sup_{\tau \in Y} \min_{1 \leq i \leq n} \tau(a_{\tau_i}) < |\mathcal F| \epsilon^2. 
	\end{equation}
	Define $S \coloneqq \phi(\mathcal F) \cup \psi(\mathcal F) \cup \{u_{\tau_1}, \ldots, u_{\tau_n} \} \subseteq \mathcal N$.  As $(\mathcal N, Y)$ has CPoU, there are projections $p_1, \ldots, p_n \in \mathcal N^\omega \cap S'$ such that 
	\begin{equation}
		\sum_{j=1}^n p_j = 1_{\mathcal N^\omega} \qquad \text{and} \qquad \tau(a_{\tau_i} p_i) \leq |\mathcal F|\epsilon^2 \tau(p_i)
	\end{equation}
	for all $i= 1, \ldots, n$ and $\tau \in Y^\omega$.  Then, using the fact that the $p_i$ commute with $S$ and are mutually orthogonal, $u \coloneqq \sum_{i=1}^n p_i u_{\tau_i} \in \mathcal N^\omega$ is a unitary and 
	\begin{equation}
		\max_{a \in \mathcal F} \|u \phi(a) u^* - \psi(a) \|_{2, Y^\omega} \leq |\mathcal F|^{1/2} \epsilon.
	\end{equation}
	By Corollary~\ref{cor:unitary-stable-relation}, there is a sequence of unitaries $(u_n)_{n=1}^\infty \subseteq \mathcal N$ representing $u$.  Then
	\begin{equation}
		\limsup_{n \rightarrow \omega} \max_{a \in \mathcal F} \|u_n \phi(a) u_n^* - \psi(a) \|_{2, Y} \leq |\mathcal F|^{1/2} \epsilon,
	\end{equation}
	which proves the theorem.
\end{proof}

In the special case of reduced products of factorial tracially complete $C^*$-algebras with CPoU, the uniqueness theorem gives on-the-nose unitary equivalence.

\begin{corollary}[{cf.\ \cite[Theorem 2.2]{CETW-classification}}]
	\label{cor:uniqueness}
	Let $A$ be a separable $C^*$-algebra and $(\mathcal N, Y)$ is a factorial tracially complete $C^*$-algebra with CPoU.  If $\phi, \psi \colon A \rightarrow \mathcal N^\omega$ are tracially nuclear $^*$-homomorphisms, then $\phi$ and $\psi$ are unitarily equivalent if and only if $\tau \circ \phi = \tau \circ \psi$ for all $\tau \in Y^\omega$.
\end{corollary}
\begin{proof}
	As $(\mathcal N, Y)$ is a factorial tracially complete $C^*$-algebra with CPoU, the same is true for $(\mathcal{N}^\omega, Y^\omega)$ by Corollary~\ref{cor:factorial-ultrapower}. Hence, applying Theorem~\ref{thm:uniqueness}, we get that $\phi$ and $\psi$ are $\|\cdot\|_{2,Y^\omega}$-approximately unitary equivalent if and only if $\tau \circ \phi = \tau \circ \psi$ for all $\tau \in Y^\omega$. By the separability of $A$ and a standard application of Kirchberg's $\epsilon$-test (Lemma~\ref{lem:EpsTest}), $\phi, \psi \colon A \rightarrow \mathcal N^\omega$ are $\|\cdot\|_{2,Y^\omega}$-approximately unitary equivalent if and only if they are unitary equivalent. 
\end{proof}

\subsection{Existence results for morphisms}\label{sec:existence}

In this subsection, we will give a general existence result showing that morphisms can be constructed from {}amenable tracially complete $C^*$-algebras to factorial tracially complete $C^*$-algebras with CPoU with prescribed tracial information (see Corollary~\ref{cor:existence} below).

By construction, the approximate morphisms we produce will approximately factor through finite dimensional $C^*$-algebras.  In fact, given a separable $C^*$-algebra $A$, a factorial tracially complete $C^*$-algebra $(\mathcal N, Y)$, and a continuous affine map $\gamma \colon Y \rightarrow T(A)$, we will produce approximate factorisations of $\gamma$ through the trace simplices $T(F_n)$ of finite dimensional $C^*$-algebras $F_n$.  If $\gamma(\tau)$ satisfies a suitable approximation property for all $\tau \in Y$ (e.g.\ amenability), the maps $T(F_n) \rightarrow T(A)$ may be approximately implemented by approximate morphisms $A \rightarrow F_n$.  Further, when $(\mathcal N, Y)$ has CPoU, the classification of projections in $\mathcal N$ (Theorems~\ref{thm:comparison} and~\ref{thm:existence-projections}) allows us to show the maps $Y \rightarrow T(F_n)$ are implemented by morphisms $F_n \rightarrow \mathcal N$.  The compositions $A \rightarrow F_n \rightarrow \mathcal N$ provide the desired approximate morphism.

The following lemma gives the required existence result for morphisms out of finite dimensional $C^*$-algebras.  When $\mathcal N$ is a II$_1$ factor, this is a standard result, and the proof here is essentially identical using Theorems~\ref{thm:comparison} and~\ref{thm:existence-projections} in place of the classification of projections in II$_1$ factors.

\begin{lemma}
	\label{lem:finite-dim-existence}
	If $F$ is a finite dimensional $C^*$-algebra, $(\mathcal N, Y)$ is a type~{\rm II}$_1$ factorial tracially complete $C^*$-algebra with CPoU, and $\gamma \colon Y \rightarrow T(F)$ is a continuous affine map, then there is a unital $^*$-homomorphism $\phi \colon F \rightarrow \mathcal N$ with \mbox{$\tau \circ \phi = \gamma(\tau)$} for all $\tau \in Y$.
\end{lemma}

\begin{proof}
	We first assume that $F$ is commutative and let $e_1, \ldots, e_m$ denote the minimal projections of $F$.  It suffices to construct pairwise orthogonal projections $p_1, \ldots, p_n \in \mathcal N$ such that $\tau(p_i) = \gamma(\tau)(e_i)$ for all $\tau \in Y$ and $i = 1, \ldots, m$ since we may then define a $^*$-homomorphism $\phi \colon F \rightarrow \mathcal N$ by $\phi(e_i) \coloneqq p_i$.  Further, since $\tau(\phi(1_F)) = 1$ for all $\tau \in Y$, we will have $\phi(1_F) = 1_\mathcal N$.
	
	The existence of $p_1$ follows immediately from Theorem~\ref{thm:existence-projections}.  Assuming $p_1, \ldots, p_k$ have been constructed some some $k < m$, Theorem~\ref{thm:existence-projections} provides a projection $p'_{k+1} \in \mathcal N$ such that $\tau(p'_{k+1}) = \gamma(\tau)(e_{k+1})$ for all $\tau \in Y$.  Note that
	\begin{equation}
		\tau(p'_{k+1}) = \gamma(\tau)(e_{k+1}) \leq 1 - \sum_{i=1}^k \gamma(\tau)(e_i) = \tau\Big(1_\mathcal N - \sum_{i=1}^k p_i \Big)
	\end{equation}
	for all $\tau \in Y$.  Then Theorem~\ref{thm:comparison} implies there is a partial isometry $v$ with $p'_{k+1} = v^*v$ and such that $p_{k+1} \coloneqq vv^*$ is orthogonal to each of $p_1, \ldots p_k$.  This completes the proof when $F$ is commutative.

	Next consider a general finite dimensional $C^*$-algebra $F$ with a system of matrix units 
	\begin{equation}
		\big(e_{i, j}^{(k)}\big)_{1 \leq i, j \leq d_k}^{1 \leq k \leq m} \subseteq F.
	\end{equation}
	By the first part of the proof, there are mutually orthogonal projections $p_{i}^{(k)} \in \mathcal N$ for $k = 1, \ldots, m$ and $i = 1\ldots, d_k$ such that
	\begin{equation}
		\tau(p_{i}^{(k)}) = \gamma(\tau)\big(e_{i,i}^{(k)}\big) \qquad \text{for all } \tau \in Y.
	\end{equation}
	
	By Theorem~\ref{thm:comparison}, for each $k = 1, \ldots, m$ and $i = 1\ldots, d_k$, there is a partial isometry $v_{i}^{(k)} \in \mathcal N$ such that 
	\begin{equation}
		v_{i}^{(k)}{}^* v_{i}^{(k)} = p_{i}^{(k)} \qquad \text{and} \qquad v_{i}^{(k)} v_{i}^{(k)}{}^* = p_{1}^{(k)}.
	\end{equation}
	Then we may define a $^*$-homomorphism $\phi \colon F \rightarrow \mathcal N$ by
	\begin{equation}
		\phi\big(e_{i, j}^{(k)}\big) \coloneqq v_{i}^{(k)}{}^*v_{j}^{(k)}
	\end{equation}
	for all $k = 1, \ldots, m$ and $i, j = 1, \ldots, d_k$  Then $\tau \circ \phi = \gamma(\tau)$ for all $\tau \in Y$, and since $\tau(\phi(1_F)) = 1$ for all $\tau \in Y$, we have $\phi(1_F) = 1_\mathcal N$.
\end{proof}

We now work towards an existence result for constructing morphisms $A \rightarrow \mathcal N^\omega$ with prescribed tracial data.  One of the most influential embedding results is Connes' theorem from \cite{Co76} that every separably acting injective II$_1$ factor embeds into $\mathcal R^\omega$.\footnote{This statement does not appear explicitly in \cite{Co76}, but follows by Lemma 5.2 and $7 \Rightarrow 6$ of Theorem~5.1. Condition 7 of Theorem 5.1 is verified by composing the trace on $\mathcal N$ with a conditional expectation onto $\mathcal N$.}  This is a key ingredient in showing that such factors are isomorphic to $\mathcal R$.  In the same paper, Connes poses his eponymous embedding problem asking if every separably acting II$_1$ factor embeds into $\mathcal R^\omega$ (see the paragraph above \cite[Notation~5.6]{Co76}).  Note that since every II$_1$ factor contains $\mathcal R$, we have that every II$_1$ factor which embeds into $\mathcal R^\omega$ also embeds into $\mathcal N^\omega$ for each II$_1$ factor $\mathcal N$.  

Our embedding result will be along these lines, showing that if $A$ is separable $C^*$-algebra and $(\mathcal N, Y)$ is a type~{\rm II}$_1$ factorial tracially complete $C^*$-algebra with CPoU, then there is a $^*$-homomorphism $A \rightarrow \mathcal N^\omega$ with prescribed tracial data in the case that all the relevant traces on $A$ factorise through $\mathcal R^\omega$.  We first provide the following local characterisation of traces factorising through $\mathcal R^\omega$.

\begin{definition}[{cf.\ \cite[Remark~3.7]{BS16}}]\label{defn:hyperlinear}
	A trace $\tau$ on a $C^*$-algebra $A$ is \emph{hyperlinear} if there is a net of self-adjoint linear maps $(\psi_\lambda \colon A \rightarrow M_{d(\lambda)})$ such that for all $a, b \in A$,
	\begin{enumerate}
		\item $\|\psi_\lambda(ab) - \psi_\lambda(a) \psi_\lambda(b) \|_2 \rightarrow 0$,\label{defn:hyperlinear-mult}
		\item 
		$\mathrm{tr}_{d(\lambda)}(\psi_\lambda(a))\rightarrow \tau(a)$, and\label{defn:hyperlinear-tracial}
		\item\label{hyp-bdd} $\displaystyle \limsup_\lambda \|\psi_\lambda(a)\| < \infty$.\label{defn:hyperlinear-bdd}
	\end{enumerate}
\end{definition}

For applications to classification, the hyperlinear traces of interest will be amenable, so that the maps $\psi_i$ can be taken to be c.p.c.  However, we will prove the existence theorem in terms of hyperlinear traces since it takes little extra effort.

\begin{remark}
	We collect here a few general observations about hyperlinear traces.
\begin{enumerate}
	\item As usual, when $A$ is unital, we may arrange for each $\psi_\lambda$ to be unital, and when $A$ is separable, we can arrange for the net $(\psi_\lambda)$ to be a sequence.
	\item If $A$ is separable, we choose a sequence $(\psi_n)_{n=1}^\infty$ as above and view each $M_{d(n)}$ as a unital subalgebra of $\mathcal R$.  In this way, the $\psi_n$ induce a $^*$-homomorphism $\psi \colon A \rightarrow \mathcal R^\omega$ with $\tr_{\mathcal R^\omega} \circ \psi = \tau$.  Note that of Definition~\ref{defn:hyperlinear}\ref{hyp-bdd} is needed to guarantee that $\psi$ is well-defined.
	\item If $G$ is a discrete group, then $G$ is hyperlinear if and only if the canonical trace on the reduced group $C^*$-algebra $C^*_\lambda(G)$ is hyperlinear.  This is the source for the terminology.
	\item  Definition~\ref{defn:hyperlinear}\ref{hyp-bdd} is equivalent to 
	\begin{equation}\label{eq:approx-contractive}
		\limsup_\lambda \|\psi_\lambda(a) \| \leq \|a\|, \qquad a \in A.
	\end{equation}
	Indeed, if $\Lambda$ is the index set of the net and we view each $M_{d(\lambda)}$ as a subalgebra of $\mathcal R$, , then the $\psi_\lambda$ induce a $^*$-homomorphism $\psi \colon A \rightarrow \ell^\infty(\Lambda, \mathcal R) / c_0(\Lambda, \mathcal R)$, where $c_0(\Lambda, \mathcal R)$ consists of all bounded strongly null nets.  As $^*$-homomorphisms are contractive, \eqref{eq:approx-contractive} follows.
\end{enumerate}
\end{remark}

The following is our main existence result for $^*$-homomorphisms, although in the classification theorem, it will be accessed through Corollary~\ref{cor:existence}, which gives a restatement in terms of reduced products when the domain is separable.  We keep track of the additional details regarding factorisations through finite dimensional $C^*$-algebras as it is conceivable that these will play a role, for example, in subsequent nuclear dimension calculations.

\begin{theorem}\label{thm:existence}
	Suppose $A$ is a $C^*$-algebra, $(\mathcal N, Y)$ is a type {\rm II}$_1$ factorial tracially complete $C^*$-algebra with CPoU, and $\gamma \colon Y \rightarrow T(A)$ is a continuous affine map such that $\gamma(\tau)$ is hyperlinear for all $\tau \in Y$.  For every finite set $\mathcal F \subseteq A$ and $\epsilon > 0$, there are a finite dimensional $C^*$-algebra $F$, a self-adjoint linear map $\psi \colon A \rightarrow F$, and a unital $^*$-homomorphism $\phi \colon F \rightarrow \mathcal N$ such that for all $a, b \in \mathcal F$,
	\begin{enumerate}
		\item $\|\psi(ab) - \psi(a)\psi(b) \|_{2, T(F)} < \epsilon$,
		\item $| \tau(\phi(\psi(a))) - \gamma(\tau)(a) | < \epsilon$ for all $\tau \in Y$, and
		\item $\|\psi(a)\| < \|a\| + \epsilon$.
	\end{enumerate}
Moreover, if $A$ is unital, we may arrange for $\psi$ to be unital, if $\gamma(\tau)$ is amenable for all $\tau \in Y$, we may arrange for $\psi$ to be c.p.c., and if $\gamma(\tau)$ is quasidiagonal for all $\tau \in Y$, we may arrange for
	\begin{enumerate}
		\item[\rm (i$'$)] $\|\psi(ab) - \psi(a)\psi(b)\| < \epsilon$
	\end{enumerate}
for all $a, b \in \mathcal F$.  Finally, the last three statements can be performed simultaneously.
\end{theorem}

\begin{proof}
	Since $Y$ is a closed face in $T(\mathcal N)$, we have that $Y$ is a Choquet simplex (Theorem~\ref{thm:traces-choquet}).  Applying  Theorem~\ref{thm:simplex-is-nuclear} to the continuous affine maps $\mathrm{ev}_a\circ\gamma$ for $a\in \mathcal F$, there is a finite dimensional Choquet simplex $Z$ together with continuous affine maps $\alpha' \colon Z \rightarrow Y$ and $\beta \colon Y \rightarrow Z$ so that
	\begin{alignat}{2}
		|\gamma(\alpha'(\beta(\tau)))(a) - \gamma(\tau)(a) | &< \epsilon, & \hspace{4ex} &a \in \mathcal F,\ \tau\in Y.
	\intertext{Define $\alpha \coloneqq \gamma \circ \alpha' \colon Z \rightarrow T(A)$ so that $\alpha(\tau)$ is hyperlinear for all $\tau \in Z$ and }
		|\alpha(\beta(\tau))(a) - \gamma(\tau)(a) | &< \epsilon, & &a \in \mathcal F,\ \tau\in Y.
	\end{alignat}
	
	Let $\rho_1, \ldots, \rho_n \in Z$ denote the extreme points of $Z$.  For $i = 1, \ldots, n$, since $\alpha(\rho_i) \in T(A)$ is hyperlinear (by the assumption on $\gamma$), there are $d_i \in \mathbb N$ and $^*$-linear maps $\psi_i \colon A \rightarrow M_{d_i}$ such that for all $a, b \in \mathcal F$,
	\begin{equation}\begin{split}
		\|\psi_i(ab) - \psi_i(a)\psi_i(b)\|_{2, \tr_{d_i}} &< \epsilon,  \\
		|\mathrm{tr}_{d_i}(\psi_i(a)) - \alpha(\rho_i)(a)| &< \epsilon, \quad \text{and} \\
		\|\psi_i(a)\| - \|a\|&< \epsilon.\label{eqn:existence:neweqn}
	\end{split}\end{equation}

	Define $F \coloneqq \bigoplus_{i=1}^n M_{d_i}$ and let $\psi\coloneqq\bigoplus_{i=1}^n\phi_i\colon A\to F$. We identify $Z$ with $T(F)$ via the affine map given by
	\begin{equation}
		Z \rightarrow T(F) \colon \rho_i \mapsto \mathrm{tr}_{d_i}\circ \pi_i,
	\end{equation}
where $\pi_i \colon F \rightarrow M_{d_i}$ denotes the projection map for $i = 1, \ldots, n$.   
	By Lemma~\ref{lem:finite-dim-existence}, there is a unital $^*$-homomorphism $\phi \colon F \rightarrow \mathcal N$ such that $\tau \circ \phi = \beta(\tau)$ for all $\tau \in Y$.  Then all three conditions in the theorem follow from the corresponding conditions in \eqref{eqn:existence:neweqn}.
	
	The additional claims in the theorem follow by choosing the $\psi_i$ with the appropriate properties in the second paragraph of the proof.
	\end{proof}

\begin{corollary}
	\label{cor:existence}
	Suppose $A$ is a separable unital $C^*$-algebra and $(\mathcal N, Y)$ is a type~{\rm II}$_1$ factorial tracially complete $C^*$-algebra with CPoU.  Given a continuous affine map $\gamma \colon Y \rightarrow T(A)$ such that $\gamma(\tau)$ is hyperlinear for all $\tau \in Y$, there is a unital $^*$-homomorphism $\theta \colon A \rightarrow \mathcal N^\omega$ such that $\tau \circ \theta = \gamma(\tau|_\mathcal N)$ for all $\tau \in Y^\omega$.
\end{corollary}

\begin{proof}
	Since $A$ is separable, we may choose a countable dense $\mathbb Q[i]$-subalgebra $A_0 \subseteq A$.  By Theorem~\ref{thm:existence}, there are sequences $(F_n)_{n=1}^\infty$ of finite dimensional $C^*$-algebras, $(\psi_n \colon A_0 \rightarrow F_n)_{n=1}^\infty$ of self-adjoint linear maps, and $(\phi_n \colon F_n \rightarrow \mathcal N)_{n=1}^\infty$ of unital $^*$-homomorphisms such that for all $a, b \in A_0$ and $\tau \in Y$,
	\begin{enumerate}
		\item $\|\psi_n(ab) - \psi_n(a)\psi_n(b)\|_{2, T(F_n)} \rightarrow 0$,
		\item $|\tau(\phi(\psi(a))) - \gamma(\tau)(a)| \rightarrow 0$, and
		\item $\displaystyle\limsup_{n \rightarrow \infty} \|\psi(a_n)\| \leq \|a\|$.
	\end{enumerate}
	The sequences $(\phi_n)_{n=1}^\infty$ and $(\psi_n)_{n=1}^\infty$ induce $^*$-homomorphisms
	\begin{equation}
		A_0 \xrightarrow{\psi} \prod^\omega \big(F_n, T(F_n)\big) \xrightarrow{\phi} \mathcal N^\omega.
	\end{equation}
	Define $\theta_0 \coloneqq \phi \circ \psi$.  By (iii), $\theta_0$ is contractive and hence extends to a $^*$-homomorphism $\theta \colon A \rightarrow \mathcal N^\omega$ with the required property.
\end{proof}

If the range of $\gamma$ is contained in the uniformly amenable traces on $A$, then $\theta$ will necessarily be tracially nuclear by Theorem~\ref{thm:amenable}.  In the next subsection, we will use this observation together with the uniqueness result in Corollary~\ref{cor:uniqueness} to strengthen this existence theorem -- namely, $\theta$ can be chosen to take values in $\mathcal N$ (Theorem~\ref{thm:classification-morphism}).

In case one wants a version of the above result where the domain of $\gamma$ is $Y^\omega$ (instead of $Y$), this comes in the classification result below (Theorem~\ref{thm:classification-morphism}\ref{classif1}), at least when $\gamma(Y^\omega)$ is contained in the amenable traces on $A$, by taking $(\mathcal N^\omega, Y^\omega)$ in place of $(\mathcal N,Y)$.

As a first application of Theorem~\ref{thm:existence}, we establish a $W^*$-bundle version of the Connes' embedding problem, assuming a positive solution in each fibre.
We may consider $\omega$ to be a free ultrafilter on $\mathbb N$ (as opposed to a general free filter as above) in the following.

\begin{corollary}\label{cor:ConnesEmbedding}
	Suppose $\mathcal M$ is a $W^*$-bundle over a compact metrisable space $K$ with separably acting fibres.  If each fibre of $\mathcal M$ admits a trace-preserving embedding into $\mathcal R^\omega$, then there is an embedding $\mathcal M \hookrightarrow C_\sigma(K, \mathcal R)^\omega$ which restricts to the identity on $C(K)$.
\end{corollary}

\begin{proof}
	As in Proposition~\ref{prop:WStarBundleToTC}, we may view $\mathcal N \coloneqq C_\sigma(K, \mathcal R)$ as a factorial tracially complete $C^*$-algebra $(\mathcal N, Y)$ where $Y \cong \mathrm{Prob}(K)$ is the set of traces given by integrating the trace on $\mathcal R$ over $K$.  It is easy to see that $\mathcal N$ is McDuff, and hence $(\mathcal N, Y)$ has CPoU by Theorem~\ref{introthmgammaimpliescpou}.  Similarly, we may view $\mathcal M$ as a tracially complete $C^*$-algebra $(\mathcal M, X)$ where $X \cong \mathrm{Prob}(K)$, where again a Radon probability measure on $K$ induces on trace on $\mathcal M$ by integrating the traces on the fibres of $\mathcal M$.  
	
	Let $\gamma \colon Y \rightarrow X$ be the continuous affine map induced by the identity map on $\mathrm{Prob}(K)$.  By hypothesis, $\gamma(\tau) \in X$ is hyperlinear for all $\tau \in K$.  It is easy to show that the hyperlinear traces on a unital $C^*$-algebra form a closed convex set, and hence $\gamma(\tau) \in X$ is hyperlinear for all $\tau \in Y$. Since $K$ is metrisable and each fibre of $\mathcal M$ is separably acting, we have that $\mathcal M$ is $\|\cdot\|_{2, X}$-separable.  Let $A \subseteq \mathcal M$ be a unital  $\|\cdot\|$-separable $C^*$-subalgebra that is $\|\cdot\|_{2,X}$-dense in $\M$.  
	
	By Corollary~\ref{cor:existence}, there is a $^*$-homomorphism $\theta_0 \colon A \rightarrow \mathcal N^\omega$ such that $\gamma(\tau)|_A = \tau \circ \theta_0$ for all $\tau \in Y^\omega$.  Then $\theta_0$ extends by continuity to a morphism $(\mathcal M, X) \rightarrow (\mathcal N^\omega, Y^\omega)$.  By construction, the inclusion map $C(K) \hookrightarrow \mathcal N^\omega$ and $\theta|_{C(K)} \colon C(K) \hookrightarrow \mathcal N^\omega$ agree on traces, and hence are unitarily equivalent by Corollary~\ref{cor:uniqueness}. 
 But since $C(K)$ is the centre of $\mathcal N$, the copy of $C(K)$ commutes with any unitary in $\mathcal  N^\omega$, and so $\theta|_{C(K)}$ agrees with the canonical inclusion $C(K) \hookrightarrow \mathcal N^\omega$. 
\end{proof}

\begin{question}
    Can Corollary \ref{cor:ConnesEmbedding} be improved to provide an embedding $\M\to C_\sigma(K,\mathcal R^\omega)$ (which restricts to the identity on $C(K)$)?
\end{question}
 
\subsection{Classification and consequences}\label{sec:classification-theorems}

The existence and uniqueness results for morphisms in the previous two subsections together can be combined with tracially complete versions of standard intertwining arguments to produce our main classification theorems.  First, we use the intertwining via reparameterisation technique (Theorem~\ref{thm:reparameterisation}) to obtain a classification result for tracially nuclear morphisms.  The second part of Theorem~\ref{IntroThmClassMap} from the overview follows from part~\ref{alg-classif2} below, with the first part of Theorem~\ref{IntroThmClassMap} having already been proven in Theorem~\ref{thm:amenable}.

\begin{theorem}
	\label{thm:classification-morphism}
Let $(\mathcal N,Y)$ be a  a type~{\rm II}$_1$ factorial tracially complete $C^*$-algebra with CPoU.
	\begin{enumerate}
		\item\label{classif1} Let $A$ be a separable $C^*$-algebra. If $\gamma \colon Y \rightarrow T(A)$ is a continuous affine map such that $\gamma(\tau)$ is uniformly amenable for all $\tau \in Y$, then there is a tracially nuclear $^*$-homomorphism $\phi \colon A \rightarrow \mathcal N$ such that \mbox{$\tau \circ \phi = \gamma(\tau)$} for all $\tau \in Y$, and this $\phi$ is unique up to approximate unitary equivalence.
		\item \label{classif2} Let $(\M,X)$ be a $\|\cdot\|_{2,X}$-separable tracially complete $C^*$-algebra. If $\gamma \colon Y \rightarrow X$ is a continuous affine function such that $\gamma(\tau)$ is uniformly amenable for all $\tau \in Y$, then there is a tracially nuclear $^*$-homomorphism $\phi \colon (\mathcal M, X) \rightarrow (\mathcal N, Y)$ such that $\tau \circ \phi = \gamma(\tau)$ for all $\tau \in Y$, and this $\phi$ is unique up to approximate unitary equivalence.
  \item\label{classif3} Let $(\M,X)$ be a $\|\cdot\|_{2,X}$-separable {}amenable tracially complete $C^*$-algebra. If $\gamma \colon Y \rightarrow X$ is a continuous affine function, then there is a tracially nuclear $^*$-homomorphism $\phi \colon (\mathcal M, X) \rightarrow (\mathcal N, Y)$ such that $\tau \circ \phi = \gamma(\tau)$ for all $\tau \in Y$, and this $\phi$ is unique up to approximate unitary equivalence.
	\end{enumerate}
\end{theorem}

\begin{proof}
\ref{classif1}. The uniqueness part of \ref{classif1} is Theorem~\ref{thm:uniqueness}. For the existence part of \ref{classif1}, since uniformly amenable traces are hyperlinear, Corollary~\ref{cor:existence} implies there is a $^*$-homomorphism $\theta_\infty \colon A \rightarrow \mathcal N^\infty$ with \mbox{$\tau \circ \theta_\infty = \gamma(\tau|_\mathcal N)$} for all $\tau \in Y^\infty$.  By Theorem~\ref{thm:amenable}, $\theta_\infty$ is tracially nuclear.  Note that if $r \colon \mathbb N \rightarrow \mathbb N$ is a function with $\lim_{n \rightarrow \infty} r(n) = \infty$ and $r^* \colon \mathcal N^\infty \rightarrow \mathcal N^\infty$ is the reparameterisation map as in Theorem~\ref{thm:reparameterisation}, then  
 \begin{equation}
    \tau \circ \theta_\infty=\gamma(\tau|_{\mathcal N})=\gamma((\tau\circ r^*)|_{\mathcal N})=\tau \circ r^* \circ \theta_\infty,\qquad\tau \in Y^\infty.   
 \end{equation}
By Corollary~\ref{cor:uniqueness}, $r^* \circ \theta_\infty$ and $\theta_\infty$ are unitarily equivalent.  By Corollary~\ref{cor:unitary-stable-relation}, unitaries in $\mathcal N^\infty$ lift to unitaries in $\ell^\infty(\mathcal N)$, and hence Theorem~\ref{thm:reparameterisation} provides a $^*$-homomorphism $\theta \colon A \rightarrow \mathcal N$ such that $\iota_\mathcal N \circ \theta$ and $\theta_\infty$ are unitarily equivalent.  Then $\theta$ satisfies $\tau \circ \theta = \gamma(\tau)$ for all $\tau \in \mathcal N$, and Theorem~\ref{thm:amenable} implies that $\theta$ is tracially nuclear.

\ref{classif2}.  Let $A \subseteq \M$ be a separable $\|\cdot\|_{2, X}$-dense $C^*$-subalgebra. Then uniqueness follows from uniqueness in \ref{classif1} as all morphisms between tracially complete $C^*$-algebras are automatically contractive between the uniform 2-norms.  Part \ref{classif1} gives a (tracially nuclear) $^*$-homomorphism $\hat{\phi}\colon A\to\mathcal N$ with $\tau\circ\hat{\phi}=\gamma(\tau)$ for $\tau\in Y$, which extends uniquely to $\phi\colon (\M,X)\to(\mathcal N,Y)$ by Corollary \ref{cor:dense-subalgebra}\ref{item:dense-subalg3}. The extension is tracially nuclear by Lemma \ref{lem:semidiscrete-dense}. As $A$ is $\|\cdot\|_{2,X}$-dense in $\M$, $\tau\circ\phi=\gamma(\tau)$ for $\tau\in Y$. 

\ref{classif3}. This is immediate from \ref{classif2} as {}amenability of $\M$ forces all traces in $X$ to be uniformly amenable by Theorem~\ref{thm:introamenable}.
\end{proof}

The classification theorem for regular amenable tracially complete $C^*$-algebras  is obtained from a tracially complete version of the two-sided Elliott intertwining argument (the $C^*$-version of which can be found as  \cite[Corollary~2.3.4]{Rordam-Book}, for example; note that \cite[Theorem~3]{El09} abstracts this $C^*$-version to an abstract intertwining result for categories with notions of inner automorphisms and metric structure on the morphism sets, and one could apply it to the category of separable tracially complete $C^*$-algebras).  As with $C^*$-classification results, any isomorphism at the level of the invariants -- in this case, the designated traces -- lifts to an isomorphism of tracially complete $C^*$-algebras.

\begin{theorem}
	\label{thm:classification}
	Suppose that $(\mathcal M, X)$ and $(\mathcal N, Y)$ are type~{\rm II}$_1$ {}amenable factorial tracially complete $C^*$-algebras with property $\Gamma$ such that $\mathcal M$ is \mbox{$\|\cdot\|_{2, X}$}-separable and $\mathcal N$ is $\|\cdot\|_{2, Y}$-separable. 
	\begin{enumerate}
		\item\label{alg-classif1} $(\mathcal M, X) \cong (\mathcal N, Y)$ if and only if $X \cong Y$, and moreover,
		\item\label{alg-classif2} For any affine homeomorphism $\gamma \colon Y \rightarrow X$, there is an isomorphism $\theta \colon (\mathcal M, X) \rightarrow \mathcal (\mathcal N, Y)$ such that $\tau \circ \theta = \gamma(\tau)$ for all $\tau \in Y$.
	\end{enumerate}
\end{theorem}

\begin{proof}
	The first part follows from the second, so fix an isomorphism $\gamma$ as in \ref{alg-classif2}. Both $(\M,X)$ and $(\mathcal N,Y)$ have CPoU by Theorem~\ref{introthmgammaimpliescpou}. So, by two applications of the existence portion of Theorem~\ref{thm:classification-morphism}\ref{classif2}, there are tracially nuclear morphisms $\phi_0 \colon (\mathcal M, X) \rightarrow (\mathcal N, Y)$ and $\psi_0 \colon (\mathcal N, Y) \rightarrow (\mathcal M, X)$ such that
	\begin{equation}
		\tau \circ \phi_0 = \gamma(\tau) \qquad \text{and} \qquad \sigma \circ \psi_0 = \gamma^{-1}(\sigma)
	\end{equation}
	for all $\tau \in Y$ and $\sigma \in X$.  In particular,
	\begin{equation}
		\tau \circ \phi_0 \circ \psi_0 = \tau \qquad \text{and} \qquad \sigma \circ \psi_0 \circ \phi _0= \sigma
	\end{equation}
	for all $\tau \in Y$ and $\sigma \in X$.  Using the uniqueness portion of Theorem~\ref{thm:classification-morphism}\ref{classif2} twice, we have $\phi_0 \circ \psi_0$ is approximately unitarily equivalent to $\mathrm{id}_\mathcal N$ and $\psi_0 \circ \phi_0$ is approximately unitarily equivalent to $\mathrm{id}_\mathcal M$.	
	
	Let $(\mathcal F_n')_{n=1}^\infty$ and $(\mathcal G_n')_{n=1}^\infty$ be increasing sequences of finite subsets of $\mathcal M$ and $\mathcal N$, respectively, whose unions are dense in the respective uniform 2-norms.  We will inductively construct increasing sequences $(\mathcal F_n)_{n=1}^\infty$ and $(\mathcal G_n)_{n=1}^\infty$ of finite subsets of $\mathcal M$ and $\mathcal N$, respectively, and sequences of $^*$-homomorphisms $(\phi_n \colon \mathcal M \rightarrow \mathcal N)_{n=1}^\infty$ and $(\psi_n \colon \mathcal N \rightarrow \mathcal M)_{n=1}^\infty$ such that for all $n \geq 1$,		
	\begin{enumerate}
		\item $\mathcal F_n' \subseteq \mathcal F_n$ and $\mathcal G_n' \subseteq \mathcal G_n$,
		\item $\phi_n$ and $\psi_n$ are unitarily equivalent to $\phi_0$ and $\psi_0$, respectively,
		\item $\phi_n(\mathcal F_n) \subseteq \mathcal \mathcal G_n$ and $\psi_n(\mathcal G_n) \subseteq \mathcal F_{n+1}$, and
		\item for all $a \in \mathcal F_n$ and $b \in \mathcal G_n$, we have
		\begin{align}
			\|\psi_n(\phi_n(a)) - a\|_{2, X} &< 2^{-n},\quad \text{and}\label{eq:classification-1} \\
			\|\phi_n(\psi_{n-1}(b)) - b\|_{2, Y} &< 2^{-n}.\label{eq:classification-2}
		\end{align}
	\end{enumerate}
	
	Set $\mathcal F_0 \coloneqq \emptyset$ and $\mathcal G_0 \coloneqq \emptyset$.  Assuming that $\mathcal F_{n-1}$, $\mathcal G_{n-1}$, $\phi_{n-1}$, and $\psi_{n-1}$ have been constructed, let $\mathcal F_n \coloneqq \mathcal F_n' \cup \mathcal F_{n-1} \cup \psi_{n-1}(\mathcal G_{n-1})$.  Since $\phi_{n-1}$ and $\psi_{n-1}$ are unitarily equivalent to $\phi_0$ and $\psi_0$, respectively, and $\phi_0 \circ \psi_0$ is approximately unitarily equivalent to $\mathrm{id}_\mathcal M$, we have $\phi_{n-1} \circ \psi_{n-1}$ approximately unitarily equivalent to $\mathrm{id}_\mathcal M$.  Therefore, there is a $^*$-homomorphism $\phi_n \colon \mathcal M \rightarrow \mathcal N$ which satisfies \eqref{eq:classification-2}, and is unitarily equivalent to $\phi_{n-1}$, and hence also $\phi_0$.  The construction of $\mathcal G_n$ and $\psi_n$ is similar.
	
	If $m \geq n \geq 1$ and $a \in \mathcal F_n$, then 
	\begin{equation}\begin{split}
		\|\phi_{m+1}(a) - \phi_m(a) \|_{2, Y} &\leq \|\phi_{m+1}(a - \psi_m(\phi_m(a)))\|_{2, Y}  \\
		&\hphantom{==} + \| \phi_{m+1}(\psi_m(\phi_m(a))) - \phi_m(a)\|_{2, Y}  \\
		&< 2^{-m} + 2^{-m-1}.
	\end{split}\end{equation}
	In particular, $(\phi_m(a))_{m=1}^\infty$ is norm-bounded and $\|\cdot\|_{2,Y}$-Cauchy for all $a \in \mathcal F_n$.  So $(\phi_n(a))_{n=1}^\infty$ is norm-bounded and $\|\cdot\|_{2,Y}$-Cauchy for all $a \in \bigcup_{n=1}^\infty \mathcal F_n$, and hence also for all $a\in \mathcal M$.  Similarly, $(\psi_n(b))_{n=1}^\infty$ is norm-bounded and $\|\cdot\|_{2,X}$-Cauchy for all $b \in \mathcal N$.
	
	Define $\phi \colon \mathcal M \rightarrow \mathcal N$ and $\psi \colon \mathcal N \rightarrow \mathcal M$ by 
	\begin{equation}
		\phi(a) \coloneqq \lim_{n \rightarrow \infty} \phi_n(a) \qquad \text{and} \qquad \psi(a) \coloneqq \lim_{n \rightarrow \infty} \psi_n(a).
	\end{equation}
	By construction, $\phi$ and $\psi$ are approximately unitarily equivalent to $\phi_0$ and $\psi_0$, respectively, and in particular, $\tau \circ \phi = \tau \circ \phi_0$ for all $\tau \in Y$ and $\sigma \circ \psi = \sigma \circ \psi_0$ for all $\sigma \in X$.  To complete the proof, note that $\phi$ and $\psi$ are mutual inverses using \eqref{eq:classification-1} and \eqref{eq:classification-2}.
\end{proof}

Theorem \ref{Main-A} from the overview is an immediate consequence of Theorem \ref{thm:classification}, provided (as set out in Footnote~\ref{FootnoteThmA}) we interpret it to mean that for unital separable nuclear $\Z$-stable $C^*$-algebras $A$ and $B$, one has $\big(\completion{A}{T(A)},T(A)\big)\cong \big(\completion{B}{T(B)},T(B)\big)$ if and only if $T(A)$ and $T(B)$ are affinely homeomorphic. The point is that the `easy direction' -- recovering the trace space from an isomorphism -- is a tautology when the isomorphism is in the category of tracially complete $C^*$-algebras.\footnote{To recover $T(A)$ from $\completion{A}{T(A)}$ as a $C^*$-algebra requires the forthcoming solution (\cite{Ev24}) to the trace problem (Question~\ref{Q:traces}) for tracially complete $C^*$-algebras with CPoU due to the third-named author.}  To see the `hard direction', note that if $A$ is a unital separable nuclear $\mathcal Z$-stable $C^*$-algebra, then $\big(\completion{A}{T(A)}, T(A) \big)$ is $\|\cdot\|_{2, T(A)}$-separable, semidiscrete by Corollary~\ref{cor:semidiscrete-completion}, factorial by Proposition~\ref{prop:tracial-completion}\ref{item:completion-factorial}, and satisfies property $\Gamma$ by Propositions~\ref{prop:z-stable-mcduff} and~\ref{prop:McDuff-implies-Gamma}.  The result then follows from Theorem~\ref{thm:classification}\ref{alg-classif1}. 

The classification result for morphisms provides the following strengthened completely positive approximation, in which the `upward maps' can be taken to be $^*$-homomorphisms (and the downward maps are approximately multiplicative).  Such a result holds for injective von Neumann algebras, and this was the starting point to obtaining strengthened forms of the completely positive approximation property involving \emph{decomposable approximations} in the $C^*$-algebraic setting (\cite{BCW16, CS20, HKW12}).\footnote{In \cite[Section 1]{HKW12}, a local reflexivity argument is given which shows how to use hyperfiniteness of a von Neumann algebra $\M$ to obtain nets of finite dimensional algebras $F_i$ together with u.c.p.\  maps $\psi_i\colon \M\to F_i$, $\psi_i\colon F_i\to \M$ such that $\phi_i$ a $^*$-homomorphism and $\phi_i(\psi_i(x))\to x$ in the weak$^*$ topology for all $x\in\M$.  In fact all of Connes', Popa's, and Haagerup's proofs of injectivity implies hyperfiniteness (\cite{Co76,Po86,Ha85}) output such approximations in the case of factors.  }

\begin{theorem}\label{thm:decomposable}
	Suppose $A$ is a $C^*$-algebra, $(\mathcal N, Y)$ is a type~{\rm II}$_1$ factorial tracially complete $C^*$-algebra with CPoU, and $\theta \colon A \rightarrow \mathcal N$ is a tracially nuclear $^*$-homomorphism.  Then there are nets
	\begin{equation}
		A \overset{\psi_\lambda}{\longrightarrow} F_\lambda \overset{\phi_\lambda}{\longrightarrow} \mathcal N
	\end{equation}
	of finite dimensional $C^*$-algebras $F_\lambda$ and c.p.c.\ maps $\phi_\lambda$ and $\psi_\lambda$ such that
	\begin{enumerate}
		\item $\| \phi_\lambda(\psi_\lambda(a)) - \theta(a)\|_{2, Y} \rightarrow 0$ for all $a \in A$,
		\item $\| \psi_\lambda(ab) - \psi_\lambda(a) \psi_\lambda(b)\|_{2, T(F_\lambda)} \rightarrow 0$ for all $a, b \in A$, and
		\item each $\phi_\lambda$ is a unital $^*$-homomorphism.
	\end{enumerate}
	If $\|\tau \circ \theta\|^{-1} (\tau \circ \theta)$ is quasidiagonal whenever $\tau \in Y$ with $\tau \circ \theta \neq 0$, we may further arrange for 
	\begin{enumerate}
		\item[\rm (ii$'$)] $\| \psi_\lambda(ab) - \psi_\lambda(a)\psi_\lambda(b)\| \rightarrow 0$ for all $a, b \in A$. 
	\end{enumerate}
	In either case, when $A$ and $\theta$ are unital, we may arrange for each $\psi_\lambda$ to be unital.
\end{theorem}

\begin{proof}
	Adding a unit to $A$, we may assume that $A$ and $\theta$ are unital (Lemma~\ref{lem:amenable-unitise}).  Further, as the conclusion is a local condition, it suffices to prove it when $A$ is separable.  By Theorem~\ref{thm:amenable}, $\tau \circ \theta$ is uniformly amenable for each $\tau \in T(A)$.  By Theorem~\ref{thm:existence} (applied to the affine map $\tau \mapsto \tau \circ \theta$), there are finite dimensional $C^*$-algebras $F_n$ and u.c.p.\ maps
	\begin{equation}
		A \overset{\psi_n}{\longrightarrow} F_n \overset{\phi'_n}{\longrightarrow} \mathcal N
	\end{equation}
	such that for all $a, b \in A$ and $\tau \in Y$, we have
	\begin{equation}
		\|\psi_n(ab) - \psi_n(a)\psi_n(b)\|_{2, T(F_n)} \rightarrow 0\label{eq:decomp-1}
	\end{equation}
	and
	\begin{equation}
		| \tau(\phi'_n(\psi_n(a))) - \tau(\theta(a))| \rightarrow 0,
	\end{equation}
	and such that each $\phi'_n$ is a unital $^*$-homomorphism.  Further, if $\tau \circ \theta$ is quasidiagonal for all $\tau \in T(A)$, we may replace \eqref{eq:decomp-1} with
	\begin{equation}
		\|\psi_n(ab) - \psi_n(a)\psi_n(b)\| \rightarrow 0
	\end{equation}
	for all $a, b \in A$.
	
	The maps $(\phi'_n \circ \psi_n)_{n=1}^\infty$ induce a $^*$-homomorphism $\theta' \colon A \rightarrow \mathcal N^\infty$ with $\tau \circ \theta' = \tau \circ \theta$ for all $\tau \in Y^\infty$.  By Theorem~\ref{thm:amenable}, $\theta'$ is tracially nuclear, and hence by Corollary~\ref{cor:uniqueness}, $\theta$ and $\theta'$ are unitarily equivalent.  If $u \in \mathcal N^\omega$ is a unitary with $\mathrm{ad}(u) \circ \theta' = \theta$, then by Corollary~\ref{cor:unitary-stable-relation}, there is a sequence of unitaries $(u_n)_{n=1}^\infty \subseteq \mathcal N$ lifting $u$.  The result follows by setting $\phi_n \coloneqq \mathrm{ad}(u_n) \circ \phi_n'$.
\end{proof}

We end by proving the structure and classification theorems from the overview (Theorem~\ref{InformalStructureThm} and~\ref{InformalClassification}).  Most of the implications involved are already in place, and it remains to  use the classification theorem to pass back from hyperfiniteness to the McDuff property via CPoU.

\begin{theorem}\label{thm:hyperfinite}
	Let $(\mathcal M, X)$ be a type~{\rm II}$_1$ factorial tracially complete $C^*$-algebra. Then the following conditions are equivalent:
	\begin{enumerate}
		\item\label{hypfin} $(\mathcal M, X)$ hyperfinite;
		\item\label{cpou}$(\mathcal M, X)$ is {}amenable and has CPoU;
		\item\label{gamma} $(\mathcal M, X)$ is {}amenable and satisfies property $\Gamma$;
		\item\label{mcduff} $(\mathcal M, X)$ is {}amenable and McDuff.
	\end{enumerate}
	In this setting, if $\M$ is also assumed to be $\|\cdot\|_{2, X}$-separable, then $(\M, X) \cong (\mathcal R_X, X)$ (see Example~\ref{Ex:ConcreteModels}).
\end{theorem}

\begin{proof}
	The implications \ref{mcduff}$\Rightarrow$\ref{gamma}, \ref{gamma}$\Rightarrow$\ref{cpou}, and \ref{hypfin}$\Rightarrow$\ref{cpou} hold by Proposition~\ref{prop:McDuff-implies-Gamma}, Theorem~\ref{introthmgammaimpliescpou}, and Theorem~\ref{thm:hyperfinite-implies-CPoU}, respectively.  Further, Theorem~\ref{thm:decomposable} applied to $\theta \coloneqq \mathrm{id}_\mathcal M$ shows \ref{cpou}$\Rightarrow$ \ref{hypfin}.  Suppose \ref{cpou} holds and $(\mathcal M, X)$ is $\|\cdot\|_{2,   X}$-separable. Both $(\mathcal M, X)$ and $(\mathcal R_X, X)$ satisfy the conditions of Theorem~\ref{thm:classification}, and hence $(\mathcal M, X) \cong (\mathcal R_X, X)$.  As $(\mathcal R_X, X)$ is McDuff, this also shows $(\mathcal M, X)$ is McDuff and finishes the proof in the separable setting.
	
	It remains to show \ref{cpou} implies \ref{mcduff} without separability.  Assume $(\mathcal M, X)$ satisfies CPoU.  By the local characterisation of McDuff's property (Proposition~\ref{prop:McDuff}\ref{item:McDuff-approx}), it suffices to show that every finite set $\mathcal F \subseteq \mathcal M$ is contained in a unital $\|\cdot\|_{2, X}$-closed $C^*$-subalgebra of $\mathcal M$ that is factorial and McDuff as a tracially complete $C^*$-algebra.  Using the separable inheritability of the conditions in \ref{cpou} (Theorem~\ref{thm:sep}\ref{thm:sep-semidiscrete} and~\ref{thm:sep-CPoU}, together with Proposition~\ref{prop:sep-countable}), this follows from the fact that \ref{cpou} implies \ref{mcduff} in the separable setting.
\end{proof}

\appendix

\section{Separabilisation}\label{sec:sep}

In this appendix we collect the machinery needed to reduce our main structural results -- Theorems \ref{InformalStructureThm} and \ref{introthmgammaimpliescpou} -- to the case of separable tracially complete $C^*$-algebras and prove the non-metrisable version of Theorem \ref{thm:simplex-is-nuclear}.

\subsection{Separabilising tracially complete \texorpdfstring{$C^*$}{C*}-algebras}

The following definition is modelled on Blackadar's notion of separable inheritability for $C^*$-algebras (\cite[Section II.8.5]{Bl06}).\footnote{We note that the second part of Blackadar's definition -- corresponding to part \ref{defn:sep:closed} of Definition \ref{defn:sep} -- is ambiguous about whether the inductive systems are over the natural numbers or more general directed sets. In our analogous definition, we opt for the more conservative version, with inductive systems over the natural numbers only.}  Recall from Section~\ref{sec:tracially-complete-defs} (before Definition~\ref{def:embedding}) that if $(\M, X)$ is a tracially complete $C^*$-algebra then a tracially complete $C^*$-subalgebra of $(\M, X)$ is $(\M_0,X_0)$ where $\M_0 \subseteq \M$ is a unital $\|\cdot\|_{2, X}$-closed $C^*$-subalgebra and $X_0 \subseteq T(\M)$ is the set of traces arising as restrictions of traces in $X$ to $\M_0$.

\begin{definition}[cf.\ {\cite[Section II.8.5]{Bl06}}]\label{defn:sep}
	We say that a property $(P)$ of tracially complete $C^*$-algebras is \emph{separably inheritable} if
	\begin{enumerate}
		\item\label{defn:sep:unbounded} whenever $(\M, X)$ is a tracially complete $C^*$-algebra satisfying $(P)$ and $S \subseteq \M$ is a $\|\cdot\|_{2, X}$-separable subset of $\M$, there is a tracially complete $C^*$-subalgebra $(\M_0,X_0)$ satisfying $(P)$ and such that $S\subseteq \M_0$ and $\mathcal M_0$ is $\|\cdot\|_{2, X}$-separable.
		\item\label{defn:sep:closed} if $\big((\M_n, X_n)\big)_{n=1}^\infty$ is a sequence of tracially complete $C^*$-algebras satisfying $(P)$ and $\phi_n^{n+1} \colon (\M_n, X_n) \rightarrow (\M_{n+1}, X_{n+1})$ is an embedding for each $n \geq 1$, then $\varinjlim\, \big((\M_n, X_n), \phi_n^{n+1}\big)$ also satisfies $(P)$.
	\end{enumerate}
 We say that $(P)$ is \emph{strongly separably inheritable} if, in addition,
 \begin{enumerate}
     \setcounter{enumi}{2}
     \item\label{defn:sep:sepS} if $(\M,X)$ is a tracially complete $C^*$-algebra such that for every $\|\cdot\|_{2,X}$-separable subset $S$ of $\M$, there exists a tracially complete $C^*$-subalgebra $(\M_0,X_0)$ satisfying $(P)$ with $S \subseteq \mathcal M_0$, then $(\M,X)$ satisfies $(P)$.
 \end{enumerate}
\end{definition}

Adapting the language of \cite[Definition 1.4]{Schaf20} to tracially complete $C^*$-algebras, condition \ref{defn:sep:sepS} asks that $(\M,X)$ satisfies $(P)$ whenever it separably satisfies $(P)$.

If $(P)$ is a separably inheritable property and $(Q)$ satisfies condition \ref{defn:sep:sepS}, then in order to prove $(P)\Rightarrow(Q)$ for all tracially complete $C^*$-algebras, it is enough to prove it for $\|\cdot\|_{2,X}$-separable tracially complete $C^*$-algebras. 

The following will allow us to construct separable subalgebras that satisfy several separably inheritable properties simultaneously.  The proof is an easy modification of the corresponding result for $C^*$-algebras given in \cite[Proposition~II.8.5.3]{Bl06}.

\begin{proposition}\label{prop:sep-countable}
	The conjunction of countably many (strongly) separably inheritable properties is (strongly) separably inheritable.
\end{proposition}

\begin{proof}
	Let $\big((P_\lambda)\big)_{\lambda \in \Lambda}$ be a collection of separably inheritable properties of tracially complete $C^*$-algebras indexed by some countable set $\Lambda$, and let $(P)$ be the conjunction of the $(P_\lambda)$.  
	
	It is clear that Definition~\ref{defn:sep}\ref{defn:sep:closed} holds for $(P)$ as it holds for each $(P_\lambda)$.  To see Definition~\ref{defn:sep}\ref{defn:sep:unbounded}, let $(\M, X)$ be a tracially complete $C^*$-algebras satisfying $(P)$ and let $S \subseteq \M$ be a $\|\cdot\|_{2, X}$-separable subset.  Fix a surjective map $f \colon \mathbb N \rightarrow \Lambda$ such that each $\lambda\in\Lambda$ has infinitely many preimages.  Using Definition~\ref{defn:sep}\ref{defn:sep:unbounded} for the properties $(P_\lambda)$, inductively construct an increasing sequence of $\|\cdot\|_{2, X}$-closed and $\|\cdot\|_{2,X}$-separable $C^*$-subalgebras $\M_n \subseteq \M$ such that $S \subseteq \M_1$ and $\M_n$ satisfies $(P_{f(n)})$.  As each $(P_\lambda)$ satisfies Definition~\ref{defn:sep}\ref{defn:sep:closed}, the $\|\cdot\|_{2, X}$-closed union of the $\M_n$ satisfies $(P)$.

    For strong separable inheritability, it is clear that Definition~\ref{defn:sep}\ref{defn:sep:sepS} is closed under (arbitrary) conjunctions.
\end{proof}

We will show that many of the properties of tracially compete $C^*$-algebras defined in this paper are strongly separably inheritable.

\begin{theorem}\label{thm:sep}
	The following properties are strongly separably inheritable:
	\begin{enumerate}
		\item\label{thm:sep-factorial} factoriality,
		\item\label{thm:sep-semidiscrete} {}amenability,
		\item\label{thm:sep-hyp-fin} hyperfiniteness,
		\item\label{thm:sep-McDuff} McDuff's property
		\item\label{thm:sep-Gamma} factoriality and property $\Gamma$, and
		\item \label{thm:sep-CPoU} factoriality and CPoU.
	\end{enumerate}
\end{theorem}

In the case of the last two conditions, we include the factoriality condition as we have not defined property $\Gamma$ or CPoU in the non-factorial setting.

The rest of this subsection is devoted to the proof.  We will show each condition separately.  For factoriality, most of the work is contained in Lemma~\ref{lem:separable-face}.

\begin{proof}[Proof of Theorem \ref{thm:sep}\ref{thm:sep-factorial}]
	A sequential direct limit of factorial tracially complete $C^*$-algebras is factorial by Proposition~\ref{prop:UTCInductiveLimit}.  Now suppose $(\M, X)$ is a factorial tracially complete $C^*$-algebra and $S \subseteq \M$ is a $\|\cdot\|_{2, X}$-separable subset of $\M$.
	
	By Lemma~\ref{lem:separable-face}, there is a $\|\cdot\|$-separable unital $C^*$-algebra $A \subseteq \mathcal M$ containing a $\|\cdot\|_{2, X}$-dense subset of $S$ such that 
	\begin{equation}
		X_A \coloneqq \{ \tau|_A : \tau \in X \} \subseteq T(A)
	\end{equation}
	is a closed face.  Let $(\mathcal N, Y)$ denote the tracial completion of $A$ with respect to $X_A$ and note that $(\mathcal N, Y)$ is factorial by Proposition~\ref{prop:tracial-completion}\ref{item:completion-factorial}.  The inclusion $A \rightarrow \M$ extends to a morphism $\phi \colon (\mathcal N, Y) \rightarrow (\mathcal M, X)$ by Proposition~\ref{prop:extend-by-continuity}.  Note that $\phi$ is an embedding whose range contains $S$.  Also, the unit ball of $\phi(\mathcal N)$ is $\|\cdot\|_{2, X}$-closed in the unit ball of $\mathcal M$, and it follows from Lemma~\ref{lem:UnitBallDensity} that $\phi(\mathcal{N})$ is $\|\cdot\|_{2, X}$-closed in $\mathcal M$. 

Finally, suppose that $(\M,X)$ is a tracially complete $C^*$-algebra such that every $\|\cdot\|_{2,X}$-separable subset $S$ is contained in a factorial tracially complete $C^*$-subalgebra of $(\M,X)$.
To show that $(\M,X)$ is factorial, let $\tau_1,\tau_2,\tau \in T(\M)$ be such that $\tau \in X$ and $\tau$ is a non-trivial convex combination of $\tau_1$ and $\tau_2$, and we will check that $\tau_1 \in X$.
For any finite subset $\mathcal F$ of $\M$, we can find a factorial tracially complete $C^*$-subalgebra $(\M_0,X_0)$ such that $\mathcal F\subseteq \M_0$.
Since $\tau|_{\M_0}$ is a non-trivial convex combination of $\tau_1|_{\M_0}$ and $\tau_2|_{\M_0}$, and $X_0$ is a face, it follows that $\tau_1|_{\M_0} \in X_0$.
By definition of $X_0$, this means that there exists $\sigma_{\mathcal F} \in X$ such that $\tau_1|_{\M_0} = \sigma_{\mathcal F}|_{\M_0}$, so in particular, $\tau_1|_{\mathcal F} = \sigma_{\mathcal F}|_{\mathcal F}$.
By doing this over all finite subsets $\mathcal F$, we see that $\tau_1$ is a weak$^*$-limit of traces in $X$, and therefore, $\tau_1\in X$, as required.
\end{proof}

It is somewhat subtle to show that {}amenability is preserved by inductive limits.  Even in the tracial von Neumann algebra setting, the only known proof that the weak closure of an increasing union of semidiscrete von Neumann algebras is semidiscrete relies on the equivalence of semidiscreteness and injectivity from Connes' work (\cite{Co76}). Our argument for parts \ref{defn:sep:unbounded} and \ref{defn:sep:closed} of Definition~\ref{defn:sep} goes via the extension result for tracially nuclear morphisms (Lemma \ref{lem:semidiscrete-dense}), proved using the local-to-global characterisations of {}amenability in Section \ref{sec:amenable}.

\begin{proof}[Proof of Theorem \ref{thm:sep}\ref{thm:sep-semidiscrete}]

For condition \ref{defn:sep:sepS} of Definition \ref{defn:sep}, fix a finite subset $\mathcal F$ of $\mathcal M$ and $\epsilon>0$.  Use the hypothesis to find an {}amenable tracially complete $C^*$-subalgebra $(\mathcal M_0,X_0)$ of $\mathcal M$ containing $\mathcal F$.  By Arveson's extension theorem, the maps witnessing {}amenability of $(\mathcal M_0,X_0)$ can be extended to give a finite dimensional $C^*$-algebra $F$ and c.p.c.\ maps 
\begin{equation}
    \mathcal M \overset\psi\longrightarrow F \overset\phi\longrightarrow \mathcal M
\end{equation} 
with $\|\phi(\psi(x))-x\|<\epsilon$ for $x\in \mathcal F$. Working with a net indexed over finite sets $\mathcal F$ and $\epsilon>0$ gives {}amenability of $\mathcal M$.

	Suppose $\big((\M_n, X_n)\big)_{n=1}^\infty$ is a sequence of {}amenable tracially complete $C^*$-algebras and $\phi_n^{n+1} \colon (\M_n, X_n) \rightarrow (\M_{n+1}, X_{n+1})$ is an embedding for each $n \geq 1$ and consider $(\M, X) \coloneqq \varinjlim\, \big((\M_n, X_n), \phi_n^{n+1}\big)$.  If $A \subseteq \mathcal M$ is the $C^*$-algebraic direct limit of the $\mathcal M_n$, then the inclusion $A \hookrightarrow \mathcal M$ is tracially nuclear (using Arveson's extension theorem as above).  By Lemma~\ref{lem:semidiscrete-dense}, $(\mathcal M, X)$ is {}amenable, and this shows Definition~\ref{defn:sep}\ref{defn:sep:closed}.

	Suppose now that $(\mathcal M, X)$ is an {}amenable tracially complete $C^*$-algebra.  We'll first show that for every unital  $\|\cdot\|_{2, X}$-separable, $\|\cdot\|_{2, X}$-closed  $C^*$-subalgebra $\M_0 \subseteq \M$, there is a   $\|\cdot\|_{2, X}$-separable, $\|\cdot\|_{2, X}$-closed $C^*$-subalgebra $\mathcal N \subseteq \M$ containing $\M_0$ so that the inclusion  $\M_0  \hookrightarrow \mathcal N$ is tracially nuclear.
	
	Fix a $\|\cdot\|$-separable, $\|\cdot\|_{2, X}$-dense $C^*$-algebra $A \subseteq \mathcal M_0$ and an increasing sequence $(\mathcal F_n)$ of finite subsets of $A$ with $\|\cdot\|$-dense union, and for each $n \geq 1$, use Proposition~\ref{prop:bounded-cpap} to construct a finite dimensional $C^*$-algebra $F_n$ and c.p.c.\ maps 
	\begin{equation}
		A \overset{\psi_n}{\longrightarrow} F_n \overset{\phi_n}{\longrightarrow} \mathcal M
	\end{equation}
	so that
	\begin{equation}
		\|\phi_n(\psi_n(a)) - a \|_{2, X} < \frac1n, \qquad a \in \mathcal F_n.
	\end{equation}
	Let $\mathcal N$ denote the tracially complete $C^*$-subalgebra of $\M$ generated by $\mathcal M_0 \cup \bigcup_{n=1}^\infty \phi_n(\mathcal F_n)$.  By construction, $\mathcal N$ is $\|\cdot\|_{2, X}$-separable and the inclusion $A \hookrightarrow \mathcal N$ is tracially nuclear and, by Lemma~\ref{lem:semidiscrete-dense}, so is the inclusion $\mathcal M_0 \hookrightarrow \mathcal N$.
	
	Applying this result inductively, given a $\|\cdot\|_{2, X}$-separable subset $S \subseteq \mathcal M$, we may construct an increasing sequence $(\mathcal M_n)_{n=1}^\infty$ of $\|\cdot\|_{2, X}$-separable, $\|\cdot\|_{2, X}$-closed subalgebras of $\mathcal M$ such that $S \subseteq \mathcal M_1$ and the inclusions $\mathcal M_n \hookrightarrow \mathcal M_{n+1}$ are tracially nuclear.  If $\mathcal N$ is the $\|\cdot\|_{2, X}$-closure of the union of the $\mathcal M_n$, then $\mathcal N$ is $\|\cdot\|_{2, X}$-separable.  If $A \subseteq \mathcal N$ is the $\|\cdot\|$-closure of the union of the $\mathcal M_n$, then the inclusion $A \hookrightarrow \mathcal N$ is tracially nuclear (by the same application of Arveson's extension theorem used in the previous parts), and hence $\mathcal N$ is {}amenable by Lemma~\ref{lem:semidiscrete-dense}.  This shows Definition~\ref{defn:sep}\ref{defn:sep:unbounded}.
\end{proof}

The proof that hyperfiniteness is separably inheritable is standard.

\begin{proof}[Proof of Theorem \ref{thm:sep}\ref{thm:sep-hyp-fin}]
	Verifying Definition~\ref{defn:sep}\ref{defn:sep:closed} and \ref{defn:sep:sepS} is a standard $\epsilon/3$ argument.  To see Definition~\ref{defn:sep}\ref{defn:sep:unbounded} suppose $(\M, X)$ is a hyperfinite tracially complete $C^*$-algebra and let $S$ be a $\|\cdot\|_{2, X}$-separable subset of $\mathcal M$.
 
    Let $\mathcal M_1$ denote the tracially complete $C^*$-subalgebra of $\M$ generated by $S$.  As $\mathcal M_1$ is $\|\cdot\|_{2,X}$-separable, fix an increasing sequence $(\mathcal F_k)$ of finite sets with dense union in $\M_1$. For  each $k \geq 1$, let $F_k \subseteq \M$ be a finite dimensional $C^*$-algebra such that for every $a \in \mathcal F_k$, there is a $b \in F_k$ with $\|a - b\|_{2, X} < \frac1k$.  Let $\mathcal M_2$ denote the  tracially complete $C^*$-subalgebra generated by $\M_1$ and the $F_k$.
	
	Iterating this construction, there is an increasing sequence of $\|\cdot\|_{2, X}$-closed $C^*$-algebras $\mathcal M_n \subseteq \M$ such that for every $n \geq 1$, finite set $\mathcal F \subseteq \M_n$, and $\epsilon > 0$, there is a finite dimensional $C^*$-algebra $F \subseteq \M_{n+1}$ such that for all $a \in \mathcal F$ there is a $b \in F$ with $\|a - b\|_{2, X} < \epsilon$.  Then the $\|\cdot\|_{2,X}$-closure of the union of the $\M_n$ is a tracially complete hyperfinite subalgebra of $\M$ containing $S$.
\end{proof}

The separable inheritability of McDuff's property and property $\Gamma$ are easy consequences of the local characterisation of these properties in Proposition~\ref{prop:McDuff}\ref{item:McDuff-approx} and Proposition~\ref{prop:Gamma-indlim} respectively.

\begin{proof}[Proof of Theorem \ref{thm:sep}\ref{thm:sep-McDuff} and \ref{thm:sep-Gamma}]
	McDuff's property is preserved by sequential inductive limits by Corollary~\ref{cor:McDuff-indlim-ultraproducts}, and the combination of factoriality and property $\Gamma$ is preserved under sequential inductive limits by Propositions~\ref{prop:UTCInductiveLimit} and~\ref{prop:Gamma-indlim}, respectively. In particular, Definition~\ref{defn:sep}\ref{defn:sep:closed} holds for both conditions.  Likewise, Proposition~\ref{prop:McDuff}\ref{item:McDuff-approx} shows that Definition~\ref{defn:sep}\ref{defn:sep:sepS} holds for McDuff's property, while   Proposition~\ref{prop:Gamma}\ref{item:Gamma-approx} (together with the fact that Definition~\ref{defn:sep}\ref{defn:sep:sepS} holds for factoriality) handles this condition for factorial tracially complete $C^*$-algebras with property $\Gamma$.

For Definition~\ref{defn:sep}\ref{defn:sep:unbounded}, let $(\mathcal M, X)$ be a McDuff tracially complete $C^*$-algebra and let $S \subseteq \M$ be a $\|\cdot\|_{2, X}$-separable set; write $\mathcal M_1$ for the $\|\cdot\|_{2, X}$-closed $C^*$-subalgebra of $\M$ generated by $S$ which is necessarily $\|\cdot\|_{2,X}$-separable. Inductively, given a $\|\cdot\|_{2,X}$-separable, $\|\cdot\|_{2,X}$-closed $C^*$-subalgebra $\mathcal M_n$ of $\mathcal M$,   Proposition~\ref{prop:McDuff}\ref{item:McDuff-approx} provides a sequence $(v_m)_{m=1}^\infty \subseteq \M$ such that for any finite subset $\mathcal F \subseteq \mathcal M_n$ and $\epsilon>0$, there is $m \geq 1$ such that for all $a \in \mathcal F$,
	\begin{equation}
		\|[v_m, a]\|_{2, X} < \epsilon,\ \|v_m^*v_m + v_mv_m^* - 1_\mathcal M\|_{2, X} <\epsilon,\ 
		\text{and}\
		\|v_m^2\|_{2, X} < \epsilon.
	\end{equation}
	Let $\M_{n+1}$ be the $\|\cdot\|_{2, X}$-closed $C^*$-subalgebra of $\M$ generated by $\M_n\cup\{v_m:m\in\mathbb N\}$. Let $\mathcal N$ denote the $\|\cdot\|_{2, X}$-closure of the union of the $\M_n$.  Then $\mathcal N$ is a $\|\cdot\|_{2, X}$-separable, $\|\cdot\|_{2, X}$-closed $C^*$-subalgebra of $\M$ containing $S$ and satisfying the approximation property in Proposition~\ref{prop:McDuff}\ref{item:McDuff-approx}.  So $\mathcal N$ is McDuff, and Definition~\ref{defn:sep}\ref{defn:sep:unbounded} holds.

 The argument for Definition~\ref{defn:sep}\ref{defn:sep:unbounded} works very similarly for the combination of factoriality and property $\Gamma$. For the inductive step, given a $\|\cdot\|_{2,X}$-separable, $\|\cdot\|_{2,X}$-closed $C^*$-subalgebra of $\mathcal M$ which is factorial with the tracially complete structure induced from $\mathcal M$, we replace the sequence $(v_m)_{m=1}^\infty$ with a sequence $(p_m)_{m=1}^\infty$ of contractions such that for any finite subset $\mathcal F\subset\mathcal M_n$ and $\epsilon>0$, there is $m$ with 
 	\begin{equation}
		\|p_m - p_m^2 \|_{2, X}<\epsilon,\ \|[p_m, a]\|_{2, X}<\epsilon\ \text{and}\ \sup_{\tau\in X}\big|\tau(ap_m) - \frac12 \tau(a)\big|<\epsilon
	\end{equation}
 for all $a\in \mathcal F$, which is given by the characterisation of property $\Gamma$ in Proposition \ref{prop:Gamma}\ref{item:Gamma-approx}. Now use the corresponding part of Theorem \ref{thm:sep}\ref{thm:sep-factorial} to obtain a $\|\cdot\|_{2,X}$-separable, $\|\cdot\|_{2,X}$-closed $C^*$-subalgebra $\mathcal M_{n+1}$ of $\mathcal M$ containing $\mathcal M_{n}\cup\{p_m:m\in\mathbb N\}$ which is factorial in the induced tracially complete structure.  Just as in the previous paragraph, the $\|\cdot\|_{2,X}$-closure of the union of the $\mathcal M_n$ is factorial (by Theorem \ref{thm:sep}\ref{thm:sep-factorial}) with property $\Gamma$ (via Proposition~\ref{prop:Gamma}\ref{item:Gamma-approx}).
\end{proof}

Finally we establish that CPoU is strongly separably inheritable for factorial tracially complete $C^*$-algebras.  This works in a very similar fashion to separable inheritability of the McDuff property and property $\Gamma$.

\begin{proof}[Proof of Theorem \ref{thm:sep}\ref{thm:sep-CPoU}]
	The combination of factoriality and CPoU  is preserved under inductive limits by Propositions~\ref{prop:UTCInductiveLimit} and~\ref{prop:InductiveLimitCPoU}. Similarly, Proposition~\ref{CPoU:Ultra:FiniteSet}\ref{item:CPOU-local} (together with the fact that Definition~\ref{defn:sep}\ref{defn:sep:sepS} holds for factoriality) shows that the conjunction of factoriality and CPoU satisfies Definition~\ref{defn:sep}\ref{defn:sep:sepS}.
 
	Suppose now that $(\M, X)$ is a factorial tracially complete $C^*$-algebra satisfying CPoU.  Let $S \subseteq \M$ be a  $\|\cdot\|_{2, X}$-separable subset of $\M$.  We start by constructing an increasing sequence $(\mathcal{N}_n)_{n=1}^\infty $ of $\|\cdot\|_{2, X}$-separable, $\|\cdot\|_{2, X}$-closed $C^*$-subalgebras of $\M$ such that
	\begin{itemize}
		\item  $S \subseteq \mathcal{N}_1$,
		\item each $\mathcal{N}_n$ is factorial as a tracially complete $C^*$-algebra, and
		\item for each $k, n \geq 1$, if $a_1, \ldots, a_k \in (\mathcal{N}_n)_+$ and $\delta > 0$ with
		\begin{equation}
			\sup_{\tau \in X} \min_{1 \leq i \leq k} \tau(a_i) < \delta,
		\end{equation}
		then there are projections $p_1, \ldots, p_k \in (\mathcal{N}_{n+1})^\omega \cap \mathcal{N}_n'$ such that
		\begin{equation}
			\sum_{j=1}^k p_j = 1_{\mathcal M^\omega} \qquad \text{and} \qquad \tau(a_i p_i) \leq \delta \tau(p_i)
		\end{equation}
		for all $i = 1, \ldots, k$ and $\tau \in X^\omega$.
	\end{itemize}
	We will construct the algebras $\mathcal{N}_n$ by induction starting with $\mathcal{N}_1$ being any unital factorial $\|\cdot\|_{2, X}$-closed $C^*$-subalgebra of $\M$ containing $S$, which exists as factoriality is separably inheritable (Theorem \ref{thm:sep}\ref{thm:sep-factorial}).
	
	Suppose $n \geq 1$ and $\mathcal{N}_n$ has been constructed.  Let $S_n$ be a countable, $\|\cdot\|_{2, X}$-dense subset of $(\mathcal{N}_n)_+$ and let $K_n$ be the set of all pairs \mbox{$\kappa = (\mathcal G_\kappa, \delta_\kappa)$} where $\mathcal G_\kappa \subseteq S_n$ is a finite set and $\delta_\kappa > 0$ is rational with
	\begin{equation}
		\sup_{\tau \in X} \min_{a \in \mathcal G_\kappa} \tau(a) < \delta_\kappa.
	\end{equation}
	Note that $K_n$ is a countable set.  Since $(\M, X)$ satisfies CPoU, for each $\kappa \in K_n$ and $a \in \mathcal G_\kappa$, there are projections $p_{\kappa, a} \in \M^\omega \cap \mathcal{N}_n'$ such that
	\begin{equation}
		\sum_{b \in \mathcal G_\kappa} p_{\kappa, b} = 1_{\M^\omega} \qquad \text{and} \qquad \tau(a p_{\kappa, a}) \leq \delta \tau(p_{\kappa, a})
	\end{equation}
	for all $a \in \mathcal G_\kappa$ and $\tau \in X^\omega$.  For each $\kappa \in K_n$ and $a \in \mathcal G_\kappa$, let $(p_{\kappa, a, m})_{m=1}^\infty \subseteq \M$ be a sequence in $\M$ representing $p_{\kappa, a}$.  As factoriality is separably inheritable (part \ref{thm:sep-factorial} of the theorem), there is a unital $\|\cdot\|_{2, X}$-separable, $\|\cdot\|_{2, X}$-closed $C^*$-subalgebra $\mathcal{N}_{n+1} \subseteq \M$ which contains $\mathcal{N}_n$ and each $p_{\kappa, a, m}$ for all $m \geq 1$, $\kappa \in K_n$ and $a \in \mathcal G_\kappa$ and which is factorial as a tracially complete $C^*$-algebra.  Then $\mathcal{N}_{n+1}$ satisfies the required properties.

	Let $\mathcal N \subseteq \M$ be the $\|\cdot\|_{2, X}$-closure of the union of the algebras $\mathcal{N}_n$.  Then $\mathcal{N}$ is a $\|\cdot\|_{2, X}$-closed, $\|\cdot\|_{2, X}$-separable $C^*$-subalgebra of $\M$ which contains $S$. As each $\mathcal N_n$ is factorial, so too is $\mathcal N$ by Proposition \ref{prop:UTCInductiveLimit}.  The third condition on the $\mathcal N_n$ ensures that $\mathcal{N}$ satisfies CPoU, and this completes the proof.
\end{proof}

\subsection{Proof of Theorem~\ref{thm:simplex-is-nuclear}}\label{sec:sep-choquet} The final section of the appendix is devoted to finite dimensional approximations of non-metrisable Choquet simplices. We show how to reduce Theorem \ref{thm:simplex-is-nuclear} to the metrisable case (\cite[Lemma~2.8]{Elliott-Niu15}, which is a corollary of the fundamental work of Lazar and Lindenstrauss in \cite{Lazar-Lindenstrauss71}).  The strategy is another variation of Blackadar's separable inheritability, this time for compact convex sets; roughly, we want to know the property of being a Choquet simplex is ``metrisably inheritable'' among compact convex sets.  We find it conceptually easier to do this in the dual picture and consider separably inheritable properties of Archimedean order unit spaces.

Choquet simplices can be characterised in terms of Archimedean order unit spaces using Kadison duality as follows.  An ordered vector space $V$ satisfies \emph{Riesz interpolation} if for all $f_1, f_2, g_1, g_2 \in V$ with $f_i \leq g_j$ for $i, j = 1,2$, there is $h \in V$ with $f_i \leq h \leq g_j$ for $i, j = 1, 2$.

\begin{theorem}[{\cite[Corollary~II.3.11]{Alf71}}]\label{thm:Riesz-interp}
	Suppose $X$ is a compact convex set and $V$ is an Archimedean order unit space.
	\begin{enumerate}
		\item $X$ is a Choquet simplex if and only if $\mathrm{Aff}(X)$ has Riesz interpolation.
		\item\label{item:riesz-implies-choquet} $V$ has Riesz interpolation if and only if $S(V)$ is a Choquet simplex.
	\end{enumerate}
\end{theorem}

\begin{proof}
	The first statement is \cite[Corollary~II.3.11]{Alf71}, and the second statement follows from the first by Kadison duality (see Section~\ref{sec:simplices}).
\end{proof}

We will show Riesz interpolation is a separably inheritable property in Proposition~\ref{prop:sep-choquet}.  To facilitate this, we show that Riesz interpolation can be detected on a dense set.

\begin{lemma}\label{lem:dense-interpolation}
	Let $V$ be an Archimedean order unit space.  Then $V$ satisfies Riesz interpolation if and only if there exists a unital dense rational subspace $V_0 \subseteq V$ with Riesz interpolation.
\end{lemma}

\begin{proof}
	The forward direction is clear by taking the subspace to be $V$.  In the other direction, suppose $V_0 \subseteq V$ is a unital dense rational subspace of $V$.  Since all states (i.e.~positive and unital functionals) on both $V_0$ and $V$ are bounded and $V_0$ is dense in $V$, the restriction map $S(V) \rightarrow S(V_0)$ is an affine homeomorphism.  As $V_0$ has Riesz interpolation, $S(V_0)$ is a Choquet simplex by \cite[Corollary~10.6]{Go86}, and hence $S(V)$ is a Choquet simplex.  So $V$ has Riesz interpolation.
\end{proof}

Finally, we arrive at the main separabilisation result needed in this subsection.  The proof is similar to the proofs in the previous subsection, and also to the proofs of separabilisation results in other settings, such as those in \cite[Section~II.8.5]{Bl06}

\begin{proposition}\label{prop:sep-choquet}
	Riesz interpolation is a separably inheritable property of Archimedean order unit spaces in the following sense.
	\begin{enumerate}
		\item\label{item:sep-choquet-subspace} If $V$ is an Archimedean order unit space with Riesz interpolation and $S$ is a separable subset of $V$, then there is a separable unital closed subspace $V_0 \subseteq V$ that contains $S$ and has Riesz interpolation.
		\item\label{item:sep-choquet-lim} For a sequence $(V_n)_{n=1}^\infty$ of separable Archimedean order unit spaces with Riesz interpolation and unital order embeddings $V_n \hookrightarrow V_{n+1}$, the inductive limit of the $V_n$ also has Riesz interpolation.
	\end{enumerate}
\end{proposition}

\begin{proof}
	Part \ref{item:sep-choquet-lim} follows from Lemma~\ref{lem:dense-interpolation} since the non-closed union of the $V_n$ will have Riesz interpolation.
	
	Assume in \ref{item:sep-choquet-subspace} that $V$ and $S$ are given, and without loss of generality, assume $S$ is countable and $1 \in S$.  Let $V_1$ be the rational vector space spanned by $S$ so that $V_1$ is countable.  
	
	We can inductively construct an increasing sequence of countable rational subspaces $(V_n)_{n=1}^\infty$ of $V$ such that for all $n \geq 1$ and $f_1, f_2, g_1, g_2 \in V_n$ with $f_i \leq g_j$ for $1 \leq i, j \leq 2$, there is an $h \in V_{n+1}$ such that $f_i \leq h \leq g_j$ for all $1 \leq i, j \leq 2$.  Indeed, as $V$ has Riesz interpolation, for any such quadruple $(f_1, f_2, g_1, g_2)$ from $V_n$, we may choose a corresponding $h \in V$.  As the set of all such quadruples in $V_n$ is countable, we may define $V_{n+1}$ as the rational span of $V_n$ and the functions $h$ corresponding to these quadruples.  
	
	Now, $\bigcup_{n=1}^\infty V_n$ has Riesz interpolation.  Let $V_0$ be the closure of this union and note that $V_0$ is separable, contains $S$ by construction, and has Riesz interpolation by Lemma~\ref{lem:dense-interpolation}.
\end{proof}

With these results in hand, we prove Theorem~\ref{thm:simplex-is-nuclear} by reducing it to the metrisable setting.  We restate the theorem for the convenience of the reader.

\simplexisnuclear*

\begin{proof}[Proof of Theorem~\ref{thm:simplex-is-nuclear}]
	View $\mathbb R^d$ as an ordered vector space with the coordinatewise order and with order unit $(1, \ldots, 1)$.  By Kadison duality, it suffices to show that for every finite set $\mathcal F \subseteq \mathrm{Aff}(X)$ and $\epsilon > 0$, there are an integer $d$ and unital positive linear maps 
	\begin{equation}
		\mathrm{Aff}(X) \overset\psi\longrightarrow \mathbb R^{d} \overset\phi\longrightarrow \mathrm{Aff}(X)
	\end{equation}
	such that $\|\phi(\psi(f)) - f\| < \epsilon$ for all $f \in \mathcal F$.
	
	By Proposition~\ref{prop:sep-choquet}, there is a separable unital closed subspace $V_0$  of $\mathrm{Aff}(X)$ that contains $\mathcal F$ and satisfies Riesz interpolation.  By the Hahn--Banach theorem (applied in each coordinate), any unital positive map $V_0 \rightarrow \mathbb R^d$ extends to a unital positive map $\mathrm{Aff}(X) \rightarrow \mathbb R^d$ (since unital functionals are positive exactly when they have norm 1).  Therefore, it suffices to find an integer $d$ and unital positive linear maps
	\begin{equation}
		V_0 \overset\psi\longrightarrow \mathbb R^{d} \overset\phi\longrightarrow V_0
	\end{equation}
	such that $\|\phi(\psi(f)) - f\| < \epsilon$ for all $f \in \mathcal F$.  Since $V_0$ has Riesz interpolation, $S(V_0)$ is a Choquet simplex by Theorem~\ref{thm:Riesz-interp}\ref{item:riesz-implies-choquet}.  Using Kadison duality to identify $V_0$ with $\mathrm{Aff}(S(V_0))$, the result follows from the separable case (\cite[Lemma~2.8]{Elliott-Niu15}).
\end{proof}

\end{document}